\documentclass[a4paper,12pt]{article}
\usepackage[margin=1in,footskip=0.5in]{geometry}
 \usepackage[bottom]{footmisc}%
 \usepackage{setspace}
 \usepackage{amsmath}
\usepackage{graphicx}  
\usepackage{float}  
\usepackage{subfigure}  
\usepackage[utf8]{inputenc} 
\usepackage[T1]{fontenc}    
\usepackage{hyperref}       
\usepackage{url}            
\usepackage{booktabs}       
\usepackage{amsfonts}       
\usepackage{nicefrac}       
\usepackage{microtype}      
\usepackage{lipsum}		
\usepackage{amsthm}
 \usepackage{appendix}
 \usepackage{mathrsfs}
 \usepackage{amssymb}
 \usepackage{xcolor}
  \usepackage{natbib}
  \usepackage{enumitem}

 \theoremstyle{plain}
\newtheorem{theorem}{Theorem}

\newtheorem{remark}{Remark}
\newtheorem{corollary}{Corollary}
\newtheorem{lemma}{Lemma}

\newtheorem{definition}{Definition}
\newtheorem*{definition*}{Definition}

\DeclareMathOperator*{\argmin}{arg\,min}

\def\m{\mathcal}
\def\dd{{\rm d}}
\def\mb{\mathbb}

\def\ms{\mathscr}
\def\mf{\mathfrak}

\def\wt{\widetilde}
\def\wh{\widehat}
\def\sm{{[m]}}
\def\num{\nu^{\ast}_{\sm, \wh Q_\sm}}

\title{Minimax Rate of Distribution Estimation on Unknown Submanifold under Adversarial Losses}

 \author{Rong Tang and Yun Yang}
 \date{University of Illinois Urbana-Champaign}

\begin{document}
\maketitle
\begin{abstract}
 Statistical inference from high-dimensional data with low-dimensional structures has recently attracted lots of attention.
 In machine learning, deep generative modeling approaches implicitly estimate distributions of complex objects by creating new samples from the underlying distribution, and have achieved great success in generating synthetic realistic-looking images and texts. A key step in these approaches is the extraction of latent features or representations (encoding) that can be used for accurately reconstructing the original data (decoding). In other words, low-dimensional manifold structure is implicitly assumed and utilized in the distribution modeling and estimation.
 To understand the benefit of low-dimensional manifold structure in generative modeling, we build a general minimax framework for distribution estimation on unknown submanifold under adversarial losses, with suitable smoothness assumptions on the target distribution and the manifold. The established minimax rate elucidates how various problem characteristics, including intrinsic dimensionality of the data and smoothness levels of the target distribution and the manifold, affect the fundamental limit of high-dimensional distribution estimation. To prove the minimax upper bound, we construct an estimator based on a mixture of locally fitted generative models, which is motivated by the partition of unity technique from differential geometry and is necessary to cover cases where the underlying data manifold does not admit a global parametrization. We also propose a data-driven adaptive estimator that is shown to simultaneously attain within a logarithmic factor of the optimal rate over a large collection of distribution classes.

 \end{abstract}
{Keywords:} adversarial training, generative model, distribution estimation, manifold, minimax rate, partition of unity.

\section{Introduction}
High-dimensional statistical models arise in various areas of science, including computer vision, astrophysics, social science, genetics, and computational biology, among others. 
In order to make the accompanied ``large $D$, small $n$'' inference problem solvable, or, in other words, guarantee the existence of a consistent estimator, some low-dimensional structural assumptions need to be imposed. Here $D$ refers to the ambient dimension of the problem and $n$ refers to the sample size.
Sparsity, one common low-dimensional structure, assumes that only a small number $s$ ($\ll n$) of dimensions contributes to the model, whereas the subset corresponds to these active dimensions is unknown. Popular sparsity motivated statistical methods, such as LASSO~\citep{tibshirani1996regression}, SCAD~\citep{fan2001variable} and MCP~\citep{zhang2010nearly}, have received impressive success in various prediction related tasks in many applications.
In other applications, all variables may collectively influence the model, but the variables themselves may exhibit some low-dimensional structure, such as lying on an unknown submanifold whose intrinsic dimension $d$ is much smaller than the ambient dimension $D$, thus avoids the ``curse of dimensionality''.  For example, a manifold structure is naturally assumed and utilized in computer vision problems~\citep{lui2012advances} such as face recognition, action recognition and visual tracking. Despite the high-dimensional appearance of the object image data represented as a matrix collecting all pixels levels, the vectorized pixel matrix usually lies on a low-dimensional manifold parameterized by global characteristics such as camera projection, lighting condition, texture, object position and orientation. In bioinformatics, protein-protein interaction networks are often assumed to lie on or near some low-dimensional manifold embedded in the high-dimensional unorganized observation space~\citep{you2010using}, since proteins interact with other proteins based on a limited number of biochemical and structural properties~\citep{terradot2004biochemical}. 

In this paper, we consider the statistical problem of distribution estimation on an unknown sub-manifold embedded in an ambient Euclidean space, where the target distribution $\mu$ is implicitly defined through a (mixture of) generative model. In the machine learning literature, generative models, such as Generative Adversarial Network (GAN,~\cite{goodfellow2014generative, li2015generative,10.1214/19-AOS1858}), Wasserstein GAN (WGAN,~\cite{arjovsky2017wasserstein}) and Wasserstein Auto-Encoder (WAE,~\cite{tolstikhin2019wasserstein, zhao2018infovae}), have received great success in generating synthetic realistic-looking images and texts~\citep{brock2018large,oord2016wavenet}, which is an implicit manner of distribution estimation over complex data space. The success of these unsupervised machine learning methods for complex distribution estimation can be largely attributed to two key factors.
First, these methods apply deep neural networks for extracting latent features or representations (i.e.~encoding) that can be used for accurately reconstructing the original data (i.e.~decoding). In other words, low-dimensional manifold structures are implicitly utilized in the distribution estimation via encoder-decoder pairs. The superior performance of these methods over classical fully nonparametric methods again reinforces the fact that complex objects such as images and texts, despite their high-dimensional appearance, are low-dimensional in nature --- they lie on some sub-manifold embedded in the original data space.
Second, these methods are different from classical distribution estimation approach that aims at forming a parametric or nonparametric estimate of the probability density function at a point or the probability of a set; instead they fit a generative model that specifies a stochastic process whose simulated data look indistinguishable to real data. Methodology-wise, this generative modeling framework for distribution estimation automatically promotes low-dimensional data representation without explicitly estimating the unknown data manifold. Computation-wise, a best generative model can be naturally fitted by minimizing certain discrepancy measure between the real data and the synthetic data generated from the model. Moreover, sampling is often more useful and important than explicit distribution estimation in practical applications, as a known distribution (up to normalizing constant) may still require substantial effort to sample from (for example,~sampling from Bayesian posteriors).
A formal definition of a generative model is described in Section~\ref{se:Gen_Mol}, where further details about comparisons with traditional explicit distribution estimation approaches are also discussed. 

Despite the recent surge of works (see Section~\ref{se:literature} for a selective review) on generative model learning, there is a lack of theoretical results quantifying the fundamental limit of these procedure to estimate a distribution supporting on an unknown manifold lying on a high-dimensional ambient Euclidean space and how various problem characteristics affect the limit. 
In this paper, we aim to close this gap by identifying the minimax rate of distribution estimation on unknown submanifold under adversarial losses.
Here, an adversarial loss~\citep{arjovsky2017wasserstein,singh2018nonparametric, tolstikhin2019wasserstein,liang2020generative} is defined as $d_{\mathcal{F}}(\mu_1, \mu_2)={\sup}_{f\in \mathcal{F}}  |\int_\m X f(x) \, \dd\mu_1-\int_\m X f(x)\, \dd\mu_2|$, for two distributions $\mu_1$ and $\mu_2$ over data space $\m X$, where $\mathcal{F}$ is pre-specified set, called the discriminator class, composed of functions over $\m X$ (c.f.~Section~\ref{se:adv_loss} for further details). Popular choices of $\m F$ includes Lipschitz continuous function class (WGAN,~\cite{arjovsky2017wasserstein}), Sobolev function class (Sobolev GAN,~\cite{mroueh2017sobolev}) and reproducing kernel Hilbert space (MMD GAN,~\cite{DBLP:journals/corr/LiCCYP17}). Note that conventional discrepancy measures such as $\ell_p$ ($p\geq 1$) distance, Hellinger distance and Kullback-Leibler (KL) divergence that are widely adopted in nonparametric density estimation theory~\citep{Tsybakov2009} are no longer applicable to define the risk in our context since
the (implicitly) estimated distribution is not absolutely continuous with respect to the Lebesgue measure of the ambient data space $\m X$ and may be singular to the estimation target, denoted as $\mu^\ast$.  

One distinct feature of our framework from the generative modeling literature in machine learning is that we do not require the unknown data manifold to admit a global parametrization (single chart). For example, for compact manifolds without boundary, such as the sphere, at least two parametrizations are needed in order to cover the whole surface. We avoid this stringent assumption, often implicitly assumed in existing methods, by the technique of partition of unity (c.f.~Section~\ref{sec:pou} for details). Specifically, we show that the minimax rate under adversarial loss $d_{\mathcal{F}}$, whose discriminator class $\m F$ has $\gamma\in (0,\infty)$ smoothness level, scales with sample size $n$ as \footnote{$a\vee b$ and $a\wedge b$ are the respectively shorthand of $\max\{a,b\}$ and $\min\{a,b\}$.}$n^{-\frac{1}{2}}\vee n^{-\frac{\alpha+\gamma}{2\alpha+d}}\vee n^{-\frac{\beta\gamma}{d}}$ modulo logarithm terms, where $\beta\in(1,\infty)$ is the smoothness of the manifold, $\alpha\in [0,\beta-1]$ the smoothness of the probability density function relative to the volume measure of the manifold, and recall that $d$ is the intrinsic dimension of the manifold. 
In the rate, the term $n^{-\frac{1}{2}}\vee n^{-\frac{\alpha+\gamma}{2\alpha+d}}$ is attributed to the risk of estimating an unknown $\alpha$-smooth density when the $d$-dimensional $\beta$-smooth manifold is known under the same adversarial loss, and the term $n^{-\frac{\beta\gamma}{d}}$ to risk of estimating an unknown $\beta$-smooth manifold. Note that when $\gamma=1$, the $n^{-\frac{\beta}{d}}$ term matches the minimax rate of estimating a $\beta$-smooth manifold under the Hausdorff loss~\citep{10.1214/18-AOS1685}.

\subsection{Related work}\label{se:literature}

\noindent {\bf Generative model learning:}
In the machine learning literature, a generative modeling procedure aims to find a distribution $\mu$ in the generator class  $\mathcal{D}_G$ (c.f.~Section~\ref{se:Gen_Mol} for a precise definition) that is closest to the target distribution $\mu^{\ast}$ over the data space $\mathbb{R}^D$ with respect to an adversarial loss defined by a discriminator class $\mathcal{F}$ composed of test functions (c.f.~Section~\ref{se:adv_loss} for a precise definition), that is, solve the minimax optimization problem
\begin{align}\label{eqn:minimax_prob}
\inf_{\mu \in\mathcal{D}_G} d_{\m F}(\mu, \mu^\ast) = \inf_{\mu \in\mathcal{D}_G} \sup_{f\in \mathcal{F}}  \Big|\int_\m X f(x) \, \dd \mu-\int_\m X f(x)\, \dd \mu^\ast\Big|.
\end{align}
In practice, we only have access to a finite number $n$ of i.i.d.~samples $X_{1:n}=\{X_1, X_2, \cdots, X_n\}$ from $\mu^{\ast}$. 
To estimate $\mu^\ast$ from $X_{1:n}$ based on~\eqref{eqn:minimax_prob}, we need a finite sample surrogate $\mathcal{J}(f;\,X_{1:n})$ (as a functional from $\m F$ to $\mb R$) to approximate $\int f(x) \, \dd \mu^{\ast}$ for any $f\in\m F$. In the generative adversarial network literature~\citep{goodfellow2014generative, li2015generative,10.1214/19-AOS1858,arjovsky2017wasserstein,tolstikhin2019wasserstein, zhao2018infovae}, 
$\mathcal{J}(f;\, X_{1:n})$ is often simply chosen as the empirical average $n^{-1}\sum_{i=1}^n f(X_i)$.  Although the empirical average is easy to compute, it leads to statistical inefficiency in estimating the solution to problem~\eqref{eqn:minimax_prob} due to the failure of taking the smoothness of target distribution $\mu^{\ast}$ and test function $f$ into consideration. \cite{liang2020generative} and \cite{singh2018nonparametric} show that when $\mu^{\ast}$ admits an $\alpha$-smooth density function relative to the Lebesgue measure on $[0,1]^D$, test functions in the discriminator class are $\gamma$-smooth, and
$\mu^\ast$ belongs to the generator class $\m D_G$ (model correctly specified), 
then solving an empirical version of the minimax problem~\eqref{eqn:minimax_prob} with $\int_\m X f(x)\, \dd \mu^\ast$ being replaced by $\mathcal{J}(f;\, X_{1:n})=\int f\, \dd\widetilde{\mu}_n$, where $\widetilde\mu_n$ is a regularized estimator defined through kernel smoothning, leads to a better estimator of $\mu^\ast$ than simply replacing $\int_\m X f(x)\, \dd \mu^\ast$ with its empirical average. Furthermore, the resulting estimator attains the minimax optimal rate $n^{-\frac{\alpha+\gamma}{2\alpha+D}}\vee n^{-\frac{1}{2}}$ of learning an $\alpha$-smooth density function on $[0,1]^D$ under the same adversarial loss $d_{\m F}$. The non-parametric rate $n^{-\frac{\alpha+\gamma}{2\alpha+D}}$ may suffer from the curse of dimensionality as the ambient space dimension $D$ can be enormous in machine learning applications involving images and texts, for which their methods do not adapt to the underlying low-dimensional manifold structure. 

\medskip

\noindent {\bf Distribution estimation on manifold:}
Some literature~\citep{NIPS2009_2ac2406e,10.1214/21-EJS1826} considers the problem of probability density estimation on an unknown manifold $\mathcal{M}^{\ast} $, where the density function is defined as the Radon–Nikodym derivative of the underlying data distribution relative to the volume measure of the manifold.
For example,~\cite{NIPS2009_2ac2406e} proposes a simple modification of the classical kernel density estimator (KDE) in the ambient space $\mb R^D$ for obtaining pointwise estimation of the density function on $\mathcal{M}^{\ast} $. The authors show that with an optimal choice of the bandwidth parameter, the pointwise mean squared error of the resulting estimator only depends on the intrinsic dimension $d$ instead of the ambient dimension $D$.  \cite{10.1214/21-EJS1826} investigates several non-parametric kernel methods with data-driven bandwidths that are adaptive to the unknown manifold structure. They show that, when the target density function is $\alpha$-smooth and manifold $\mathcal{M}^{\ast} $ is $\beta$-smooth (c.f.~Section~\ref{sec:pou} for a precise definition), their estimator achieves an $n^{-\frac{\alpha\wedge (\beta-1)}{2(\alpha\wedge (\beta-1))+d}}$ error bound under the maximal pointwise $\ell_p$ loss over the manifold. They also illustrates that their procedure is asymptotically minimax optimal when $\beta\geq \alpha+1$. Unfortunately, the KDE procedures developed in~\cite{NIPS2009_2ac2406e} and~\cite{10.1214/21-EJS1826} only recover density values at points lying on the unknown manifold $\mathcal{M}^{\ast} $. Without the knowledge of the support of the manifold, their estimator can not be used to generate (approximate) samples from $\mathcal{M}^{\ast} $, which limits their practical applicability.
Along a different line, \cite{genovese2012minimax} and~\cite{10.1214/18-AOS1685} consider the problem of manifold estimation which corresponds to support estimation of $P$. \cite{genovese2012minimax} shows that under the strong assumption that observations are subject to perpendicular noises to the
manifold, the minimax rate relative to the Hausdorff distance of estimating a boundaryless manifold is $n^{-\frac{2}{d+2}}$. 
\cite{10.1214/18-AOS1685} shows that in the noise-free setting, the minimax rate of estimating a boundaryless $\beta$-smooth($\beta\geq 2$) manifold relative to the Hausdorff distance is $n^{-\frac{\beta}{d}}$. However, the estimator constructed in these papers is an unstructured union of $d$-dimensional
balls in $\mb R^D$. Consequently, their estimator does not recover the topology of $\m M^\ast$ as the estimator based on generative model learning.

\subsection{Organization}
The rest of the paper is organized as follows. In Section 2, we review some important concepts, such as generative models, adversarial loss and Riemannian manifold, and setup the problem. In Section 3, we introduce our main result on the minimax rate, describe a construction of rate-optimal estimator based on a mixture of generative models, and propose a data-driven adaptive estimator. Roadmap for the proof of our main result is provided in Section 4. In Appendix~\ref{sec:Wavelet_review}, we provide a brief review of Wavelets and Besov function space that are used in our estimator construction and analysis. In Appendix~\ref{sec:gen_model_class}, we describe a larger class of distributions explicitly defined via (mixture of) generative models, for which the same minimax rate applies. Some extensions of our results, including applications in two-sample tests and local distribution estimation constrained on compact sets are discussed in Appendix~\ref{ext:main result}.  All technical results and proofs are collected in Appendices~\ref{sec:remaining_results},~\ref{App:technical} and~\ref{sec:proof_extension}.

 \section{Background and Problem Formulation}
 In this section, we begin with notation and a brief introduction to generative models. The advantages of using adversarial losses for error quantification in manifold distribution estimation over conventional discrepancy measures such as the total variation distance or KL divergence are then discussed.  After that, we review the concept of partition of unity for the manifold, and give a specific construction of partition of unity for submanifolds embedded in ambient Euclidean spaces. We then formally setup the problem of distribution estimation on submanifold under adversarial losses.

\subsection{Notation}

We use $\bold{1}_{A}$ to denote the indicator function of a set $A$ so that $\bold{1}_A(x)=1$ if $x\in A$ and zero otherwise. For any positive integer $m$,  we use the shorthand $[m]:=\{1,\cdots,m\}$.   For $\alpha\in \mathbb{R}$, the floor and ceiling functions are denoted by $\lfloor \alpha\rfloor$ and $\lceil \alpha\rceil$, indicating rounding $\alpha$ to the next smaller and larger integer.  For two sequences $\{a_n\}$ and $\{b_n\}$, we use the notation $a_n \lesssim b_n$ and $a_n \gtrsim b_n$ to mean $a_n \leq Cb_n$ and $a_n \geq C b_n$, respectively, for some constant $C>0$ independent of $n$. In addition, $a_n \asymp b_n$ means that both $a_n \lesssim b_n$ and $a_n\gtrsim b_n$ hold.  For a probability measure $\mu$, the support  ${\rm supp}(\mu)$ of $\mu$ is defined as the complement of the largest open set on which $\mu$ vanishes. For a probability measure $\mu$ and a measurable set $\Omega$, we use $\mu|_{\Omega}$ to denote the  restriction of $\mu$ on $\Omega$. 
For two probability measures $\mu$ and $\nu$ where $\mu$ is absolutely continuous with respect to $\nu$, we use $\frac{\dd \mu}{\dd \nu}$ to denote the Radon-Nikodym derivative of $\mu$ with respect to $\nu$, and $D_{\rm KL}(\mu\,||\,\nu)=\int \log(\frac{\dd \mu}{\dd \nu}) \,\dd \mu$ the KL divergence between them. We use $\mathcal{P}(\mathbb{R}^d)$ to denote the set of probability measures on $\mathbb{R}^d$.  When no ambiguity arises, for an absolutely continuous probability measure $\nu$,  we may also use $\nu$ to refer its density function.


We use $\|\cdot\|_p$ to denote the usual vector $\ell_p$ norm, and reserve $\|\cdot\|$ for the $\ell_2$ norm (that is, suppress the subscript when $p=2$). For a vector $x=(x_1,x_2,\cdots,x_D)$ we use $x_{i:j}=(x_i,x_{i+1},\cdots,x_{j})$ to denote  the vector composed of the $i$ to $j$ elements of $x$. We use $\bold{0}_d$ to denote the $d$-dimensional all zero vector, and $\mb B_r(x)$ the closed ball centered at $x$ with radius $r$ (under the $\ell_2$ distance) in the Euclidean space; in particular, we use $\mb B_r^d$ to denote $\mb B_r(\bold{0}_d)$ when no ambiguity may arise. For a measurable set $S$, we use $S^{\circ}$ to denote the interior of  $S$ and $\partial S$ to denote the \footnote{The boundary of a set $S$ is the set of all points in the closure of $S$ not belonging to its interior $S^\circ$.}boundary of $S$.
For a vector-valued function in several variables $f: \mathbb{R}^d\to \mathbb{R}^D$, we use $f_i$  with $i\in [D]$ to denotes its $i$th component and $\bold{J}_{f}(x)$  to denote the $D\times d$ Jacobian matrix  of $f$ evaluated at point $x$.
For a scalar-valued multivariate function $f:\mathbb{R}^d\to \mathbb{R}$, we use ${\rm supp}(f)$ to denote its support, defined as ${\rm supp}(f)=\overline{\{x\in \mathbb{R}^d\,|\,f(x)>0\} }$, and $\|f\|_L=\sup_{x,y\in\mb R^d,\, x\neq y}\frac{|f(x) - f(y)|}{\|x-y\|_2}$ its Lipschitz constant (if the supreme is finite).   For a measurable set $\Omega\subset \mb R^d$, we use $f|_{\Omega}$ to denote restriction of $f$ on $\Omega$.
 For a multi-index $a=(a_1,\cdots,a_d)\in\mathbb{N}_0^d=\{(a_1,\cdots,a_d)\,|\, \forall j\in [d], \, a_j\in \mathbb{N}_0\}$, we define $|a| = \sum_{k=1}^d a_j$ and $a!=\prod_{i=1}^da_i !$. For two vectors $x,y \in \mathbb{R}^d$, we use $(x-y)^{a}$ to denote $\prod_{i=1}^d (x_i-y_i)^{a_i}$.  For a function $f:\,\mb R^d\to\mb R$, we use $f^{(a)}$ to denote its mixed partial derivative $\partial^{|a|} f/ \partial x_{1}^{a_{1}} \cdots \partial x_{d}^{a_{d}}$. We define the $\alpha$-smooth H\"{o}lder (function) class (see e.g.,~\cite{evans10}) with radius $r>0$ over $\Omega$ as $C^{\alpha}_r(\Omega):=\big\{f:\, \Omega \rightarrow \mathbb{R}\,\big|\,\|f\|_{C^{\alpha}(\Omega)}=\sum_{|a| \leq\lfloor\alpha\rfloor}\max_{x\in \Omega}| f^{(a)}(x)|+\sum_{|a| =\lfloor\alpha\rfloor}\max _{x, y\in\Omega,\,x\neq y} \left|f^{(a)}(x)-f^{(a)}(y)\right| /\|x-y\|^{\alpha-\lfloor\alpha\rfloor} \leq r \big\}$.
Similarly, we use $C^{\alpha}_r(\Omega; \mb R^D)=\big\{f=(f_1,\ldots,f_D):\, \Omega\to \mb R^D\,\big|\, \forall \,j \in [D],\, f_j\in C^{\alpha}_{r}(\Omega)\big\}$ to denote the vector valued function space counterpart. For an $f\in C^{\alpha}_r(\Omega; \mb R^D)$ and a multi-index $a\in \mb N_0^d$, we denote $f^{(a)}$ as the $D$ dimensional vector whose $j$-th component is the mixed partial derivative $[f_j]^{(a)}$ of $f_j$ for $j\in[D]$.

\subsection{Generative models}\label{se:Gen_Mol}
Mathematically, we define a \emph{generative model} as a pair $(\nu, G)$, where $\nu$ is a distribution on a low-dimensional latent space $\m Z\subset \mb R^d$, called \emph{generative distribution}, that is easy to sample from; and $G: \m Z\to \mb R^D$ is a map from $\m Z$ to the data space $\mb R^D$, called \emph{generative map}, so that if $Z \sim \nu$, then $G(Z) \sim \mu$. In order words, the target distribution $\mu$ can be expressed via the generative model $(\nu, G)$ via $\mu= G_\# \nu$, the \footnote{For any measure $\nu$ on $\m Z$ and map $G:\, \m Z\to \m X$, the pushforward measure $\mu = G_\#\nu$ is defined as the unique measure on $\m X$ such that $\mu(A) = \nu\big( G^{-1}(A)\big)$ holds for any measurable set $A$ on $\m X$.}pushforward measure of $\mu$ using map $G$. The set $\m D_G=\{G_{\#} \nu:\, \nu \in \Upsilon,\, G\in \m G\}$ of all generative models $(\nu, G)$ with $\nu\in\Upsilon$ and $G\in\m G$ for some distribution family $\Upsilon$ on $\m Z$ and function class $\mathcal{G}$ (consists of maps from $\m Z$ to $\mathbb{R}^D$) is called a \emph{generator class}. In practice, $\Upsilon$ can be chosen to contain a single and simple distribution such as the standard Gaussian or uniform distribution, so that sampling from any generative model in $\m D_G$ is efficient and easy.

Defining an intrinsically low-dimensional distribution on a high-dimensional ambient space implicitly through a generative model enjoys multiple benefits. First, such a distribution is otherwise difficult to describe: on the one hand, it cannot be defined as usual through a density function as the distribution only admits a density function relative to the volume measure of the manifold, but not to the Lebesgue measure of the ambient space; on the other hand, the \footnote{The support of a measure on $\m X$ is defined as the largest (closed) subset of $\m X$ for which every open neighbourhood of every point of the set has positive measure.}support of the distribution (i.e.~the underlying manifold) is unknown, which further complicates the characterization. In comparison, a generative model captures the intrinsic low-dimensional structure of the distribution via transforming from a latent space $\m Z$, while the support of the distribution corresponds to the range of map $G$. Consequently, a generative model learning procedure naturally decouples the distribution estimation problem into manifold learning (estimation of $G$) plus density estimation on the manifold (estimation of $\nu$). Second, in many applications generating samples from an underlying distribution is more important and useful than estimating the distribution. Moreover, summaries or functionals of a distribution can be easily calculated from sampling via Monte Carlo methods; while sampling can be extremely difficult even with the full knowledge of the distribution (for example, sampling from Bayesian posteriors). 
Third, map $G$ in the generative model can capture highly nonlinear structures that may lead to singularities (such as jumps and point mass) in the distribution and are hard to characterize via a density or distribution function. 
Last but not least, representing a distribution through a generative model has the computational benefit of facilitating efficient implementation, as functions tend to be easier to handle in optimization than distributions with constraints. In addition, generative models have the natural adversarial tranining framework of minimizing certain discrepancy measure between the empirical distributions of the real data and the generated synthetic data~\citep{goodfellow2014generative}.

\subsection{Adversarial loss}\label{se:adv_loss}
Conventional discrepancy measures based on Radon–Nikodym derivatives relative to the Lebesgue measure are not suitable for characterizing the closeness between mutually singular probability measures on data space $\m X=\mb R^D$. For distributions with different supports, one commonly used class of discrepancy measures in the machine learning literature are adversarial losses, which are also known as integral probability metrics~\citep{muller1997integral} in the probability literature. For a discriminator class $\m F$ of of bounded and Borel-measurable functions, the adversarial loss between probability measures $\mu$ and $\nu$ is defined as
\begin{equation}\label{eqn:adv_loss}
d_{\mathcal{F}}(\mu,\nu)=\underset{f\in \mathcal{F}}{\sup} \, \Big|\int_{\mb R^D} f(x) \, \dd \mu-\int_{\mb R^D} f(x) \, \dd \nu\Big|.
\end{equation}
If the discriminator class satisfies $\m F=-\m F$, then taking the absolute value inside the supreme of~\eqref{eqn:adv_loss} is not necessary. Many common probability metrics can be realized as an adversarial loss. For example, the Wasserstein-$1$ metric corresponds to the choice of $\m F=\{\mbox{all $1$-Lipchitz functions}\}$; the total variation metric corresponds to $\m F=\{\mbox{all measurable functions bounded by $1$}\}$; and the maximum mean discrepancy (MMD,~\cite{gretton2012kernel,tolstikhin2017minimax}) metric corresponds to $\m F$ as the unit ball of a reproducing kernel Hilbert space.

Adversarial losses with suitable $\m F$ are often adopted in formulating machine learning methods (e.g.~WGAN and WAE) as $d_{\mathcal{F}}$ can be numerically approximated by feeding empirical samples from $\mu$ and $\nu$ into a discriminator neural network. This computational ease is particularly beneficial for problems involving distributions that are implicitly defined through generative models where samples are relatively cheap to obtain.
Theoretical-wise, since many distributional characteristics can be defined as an integral of some function $f$ with respect to the underlying probability measure, probability metrics based on the comparison of integrals are natural candidates for the discrepancy measure in finite-sample error analysis.

In this work, we focus on the following adversarial loss, whose discriminator class $\m F$ is $C^{\gamma}_{1}(\mathbb{R}^D)$, the unit ball of the $\gamma$-smooth H\"{o}lder class with $\gamma>0$,
\begin{equation}\label{dbeta}
d_{\gamma} (\mu,\nu)=\underset{f\in C_1^{\gamma}(\mathbb{R}^D)}{\sup} \Big(\int_{\mb R^D} f(x) \, \dd \mu-\int_{\mb R^D} f(x) \, \dd \nu\Big).
\end{equation}
Note that $d_\gamma$ satisfies the triangle inequality and by the Weierstrass approximation theorem~\citep{10.2307/3029750}, $d_{\gamma}(\mu,\nu)=0$ if and only if $\mu=\nu$. Consequently, $d_\gamma$ is a valid metric over all probability measures on $\mb R^D$. When restricted to distributions over a bounded set such as ball $\mb B_r^D$ with radius $r$, metric $d_{\gamma}$ with $\gamma=1$ is equivalent to the Wasserstein-$1$ metric. Moreover, metric $d_\gamma$ becomes stronger as $\gamma$ decreases, and approaches the total variation metric $d_{\rm TV}$ as $\gamma\to 0_+$.

Deployed as the discrepancy measure for distribution estimation on unknown submanifolds, the smoothness parameter $\gamma$ in $d_\gamma$ characterizes a trade-off between supporting manifold recovery and density estimation on the manifold. A smaller $\gamma$ makes $d_\gamma(\mu,\nu)$ more sensitive to the misalignment between the supports of $\mu$ and $\nu$. To see this, define ${\rm dist}(x, A)=\inf_{y\in A}\|x-y\|_2$ as the distance from a point $x\in\mb R^d$ to a set $A\subset\mb R^D$. Note that ${\rm dist}(\cdot, A)^\gamma$ belongs to $ C^{\gamma}(\mathbb{R}^D)$ for any $\gamma>0$.
For two distributions $\mu$ and $\nu$ with bounded supports, we may take $f(x)=c\, {\rm dist}(x,{\rm supp}(\nu))^{\gamma} - c\, {\rm dist}(x,{\rm supp}(\mu))^{\gamma}$ for some sufficiently small constant $c$ such that $f\in C_1^{\gamma}(\mathbb{R}^D)$, leading to
\begin{align*}
  d_\gamma^{\rm S}(\mu,\nu):\,= \mathbb{E}_{\mu} \big[{\rm dist}(X,{\rm supp}(\nu))^{\gamma}\big] +\mathbb{E}_{\nu} \big[{\rm dist}(X,{\rm supp}(\mu))^{\gamma}\big] \leq c^{-1}   d_{\gamma}(\mu,\nu).
\end{align*}
Consequently, an upper bound of $d_\gamma$ implies an error bound on the supporting manifold recovery through discrepancy measure $d_\gamma^{\rm S}$.
As $\gamma$ tends to zero, $d_\gamma^{\rm S}(\mu,\nu)$ approaches $\mathbb{P}_{\mu}\big(X\notin{\rm supp}(\nu)\big)+ \mathbb{P}_{\nu} \big(X\notin{\rm supp}(\mu)\big)$, which vanishes only if $\mu$ and $\nu$ have perfectly aligned supports. When $\gamma=1$, $\frac{1}{2}\,d_\gamma^{\rm S}$ can be viewed as the limiting average Hausdorff distance~\citep{Aydin_2021},
\begin{equation*}
   d_{\rm AH}\big(\{Y_i\}_{i=1}^m,\,\{Y'_i\}_{i=1}^m\big)=\frac{1}{2m}\sum_{i=1}^m \underset{j\in [m]}{\min}\, \|Y_i-Y'_j\|_2+\frac{1}{2m}\sum_{i=1}^m \underset{j\in [m]}{\min}\, \|Y'_i-Y_j\|_2,
\end{equation*}
as sample size $m$ tends to infinity, where $\{Y_i\}_{i=1}^m$ and $\{Y'_i\}_{i=1}^m$ are i.i.d.~samples from $\mu$ and $\nu$, respectively.

\subsection{Smooth submanifolds and partition of unity}\label{sec:pou}
Intuitively speaking, a manifold is a topological space that locally resembles the Euclidean space. A submanifold in the ambient space $\mb R^D$ can be viewed as a nonlinear ``subspace''. Formally, a $\beta$-smooth ($\beta\geq 1$) $d$-dimensional manifold $\m M$ is defined as a topological space satisfying:
\begin{enumerate}[topsep=0.5em,itemsep=0.5em,partopsep=0em,parsep=0em]
  \item There exists an atlas on $\m M$ consisting of a collection of $d$-dimensional charts $\ms A = \{(U_\lambda, \varphi_\lambda)\}_{\lambda\in \Lambda}$ covering $\m M$, that is, $\m M = \bigcup_{\lambda\in\Lambda} U_\lambda$.
  \item Each chart \footnote{Subscript $\lambda$ is suppressed for the simplicity of notation.}$(U, \varphi)$ in atlas $\ms A$ consists of a homeomorphism $\varphi:\, U \to \widetilde U$, called coordinate map, from an open set $U \subset \m M$ to an open set $\widetilde U \subset \mb R^d$, that is, $\varphi$ is bijective and both $\varphi$ and $\varphi^{-1}$ are continuous maps.
  \item Any two charts $(U, \varphi)$ and $(V,\psi)$ in atlas $\ms A$ are compatible, meaning that the transition map $\varphi\circ \psi^{-1}:\, \psi(U\cap V) \to \varphi(U\cap V)$ is an $\beta$-smooth diffeomorphism.
\end{enumerate}
\noindent  The manifold structure is an intrinsic property that does not rely on the choice of the atlas. For a submanifold embedded in $\mb R^D$, the second and third conditions can be combined into a single condition that the coordinate map $\varphi$ in each chart is a $\beta$-smooth map when identified as a vector-valued function from subset $U$ of $\mb R^D$ to subset $\widetilde U$ of $\mb R^d$.
The $\alpha$-smooth H\"{o}lder (function) class $C^\alpha(\m M)$ for $\alpha\in(0,\beta]$ over a $\beta$-smooth manifold $\m M$ consists of all functions $f:\m M\to \mb R$ whose localization $f\circ\varphi^{-1}:\, \varphi(U)\to \mb R$ to each local chart $(U,\varphi)$ is $\alpha$-H\"{o}lder smooth in the usual Euclidean sense. From this definition, the coordinate map $\varphi$ in each chart $(U,\phi)$ belongs to $C^\beta(U)$ by identifying $U$ as an embedded submanifold of $\m M$ inheriting the same differentiable structure.
Note that here for a $\beta$-smooth submanifold, it is not meaningful to talk about functions with smoothness level $\alpha$ beyond $\beta$ since the definition of a higher-order smoothness level may not be compatible between charts if the atlas is at most $\beta$-smooth.

Most generative model based distribution estimation procedures in the literature (e.g.~\cite{arjovsky2017wasserstein,mroueh2017sobolev,DBLP:journals/corr/LiCCYP17}) uses a single generative model $(\nu,G)$ in modeling the underlying data distribution. This implicitly requires the underlying submanifold $\m M$ that supports the target distribution $\mu = G_\# \nu$ to admit a \emph{global parametrization}, or a single chart description. 
However, many commonly encountered manifolds such as spheres cannot be covered by a single chart in any of its representing atlas. One technical advance of the current paper is to allow multiple charts in the underlying data manifold representation through the mathematical technique of \emph{partition of unity} as defined below.

\begin{definition}
 A partition of unity on a $\beta$-smooth manifold $\mathcal{M}$ is a collection of $\beta$-smooth functions $\{\rho_\lambda\}_{\lambda\in \Lambda}$ on $\mathcal M$ so that 
 \begin{enumerate}[topsep=0.5em,itemsep=0.5em,partopsep=0em,parsep=0em]
     \item $0\leq \rho_\lambda\leq 1$ for all $\lambda\in \Lambda$, and $\sum_{\lambda\in \Lambda}\rho_\lambda(x)=1$ for all $x\in \mathcal{M}$.
     \item Each point $x\in \mathcal{M}$ has a neighborhood which intersects ${\rm supp}(\rho_\lambda)$ for only finitely many $\lambda\in \Lambda$.
 \end{enumerate}
 \end{definition}

\noindent Using the partition of unity, one can glue constructions in the local charts to form a global construction on the manifold. Such a global construction usually does not rely on the choice of the partition of unity. Conversely, the partition of unity enables the decomposition of a global estimation problem into local ones, which resembles the data localization in local (polynomial) regression~\citep{loader2006local,bickel2007local}. A partition of unity can be constructed from any open cover $\{U_\lambda\}_{\lambda\in\Lambda}$ of the manifold in a way where the partition $\{\rho_\lambda\}_{\lambda\in \Lambda}$ is indexed over the same set and ${\rm supp}(\rho_\lambda)\subset U_\lambda$ for any $\lambda\in\Lambda$. Such a partition of unity is said to be \emph{subordinate to} the open cover $\{U_\lambda\}_{\lambda\in\Lambda}$. When no ambiguity may arise, we also say that a partition of unity is subordinate to an atlas $\ms A = \{(U_\lambda, \varphi_\lambda)\}_{\lambda\in \Lambda}$ if it is subordinate to its incurred open cover $\{U_\lambda\}_{\lambda\in\Lambda}$.
For a submanifold of $\mb R^D$, any open cover of ambient space $\mb R^D$ induces a partition of unity on the submanifold. This leads to the following construction that will be used throughout the rest of the paper.

Assume $\mathcal{M}$ is contained in the closed ball $\mb B_{L}^{D}$ for some sufficiently large radius $L$, we construct a partition of unity of $\mathcal{M}$ as follows. Firstly, we find a set of points $a_{1:M}=\{a_1,a_2,\cdots, a_M\}$ in $\mathbb{R}^D$ and a set of positive radii $r_{1:M}=\{r_1,r_2,\cdots,r_M\}$ such that $\ms O_M = \{\mb B_{r_m}(a_m)^{\circ}\}_{m\in[M]}$ forms a finite open cover of $\mb B_{L}^D$. Let $\chi:\,\mathbb{R}\to [0,\infty)$ denote the commonly used mollifier defined by $\chi(t)=e^{-1/t}$ for $t>0$ and $\chi(t)=0$ for $t\leq 0$. For each $m\in [M]$, we define a local partition function as
 \begin{equation*}
      \widetilde{\rho}_{m}(x)=\frac{\chi(r_m-\|x-a_m\|_2)}{\chi(r_m-\|x-a_m\|_2)+\chi(\|x-a_m\|_2-r_m/2)}, \quad x\in\mb R^D.
 \end{equation*}
It is straightforward to check that for each $m \in [M]$, $\widetilde{\rho}_{m}(x)\in C^{\infty}(\mb R^D)$, $\widetilde{\rho}_{m}(x)=1$ for $x\in B_{r_m/2}(a_m)$, and $\widetilde{\rho}_{m}$ vanishes outside $\mb B_{r_m}(a_m)$.  Therefore, $\{{\rho}_m\}_{m\in[M]}$ forms a partition of unity for $\mathcal{M}$ with $\rho_m={\widetilde\rho}_m/\big(\sum_{m'=1}^M \widetilde{\rho}_{m'}\big)$ for $m\in[M]$.

\subsection{Smooth distributions on submanifold and generative model class}
For a smooth submanifold $\m M$ with atlas $\ms A = \{(U_\lambda, \varphi_\lambda)\}_{\lambda\in \Lambda}$, one can define a distribution $\mu$ on $\m M$ by specifying how it acts on all smooth functions $f\in C^\beta(\m M)$ through its expectation $\mb E_\mu[f]$ (duality between distributions and bounded continuous functions).
Specifically, the global characterization of $\mb E_\mu[f]$ as an integral over $\m M$ can be decomposed into local ones as in the following via a partition of unity argument~\citep{do1992riemannian} (second equality), and the local integrals can be characterized using charts (third equality):
\begin{equation}\label{eqnpou}
    \begin{aligned}
    \mb E_\mu[f] = \int_\m M f \,\dd \mu  = \sum_{\lambda \in\Lambda}\int_{U_\lambda} f\,\dd (\rho_{\lambda} \mu) = \sum_{\lambda \in\Lambda}\int_{\varphi_\lambda(U_\lambda)} f\circ\varphi_\lambda^{-1} \, \dd \big[(\varphi_\lambda)_\# (\rho_\lambda \mu)\big],
\end{aligned}
\end{equation}
where $\{\rho_\lambda\}_{\lambda\in\Lambda}$ is a partition of unity subordinate to atlas $\ms A$, and $\rho_\lambda \mu$ stands for the non-negative measure whose Radon-Nikodym derivative relative to $\mu$ is $\rho_\lambda$.
In particular, if $(\varphi_\lambda)_\# (\rho_\lambda \mu)$ admits an $\alpha$-smooth density function for $\alpha \in (0,\beta-1]$ relative to the Lebesgue measure on $\mb R^d$ for each $\lambda\in\Lambda$, then $\mu$ is said to be an $\alpha$-smooth distribution on $\m M$. Note that here similar to the definition of smooth functions on a $\beta$-smooth manifold, it is not meaningful to talk about distributions with smoothness level beyond $\beta-1$ since the change of measure formula (we have abused the notation of a measure to denote its density function),
\begin{align*}
     \big[(\varphi_1)_\# (\rho_1\rho_2\mu)\big]\big(\varphi_1(x)\big)= \big[(\varphi_2)_\# (\rho_1\rho_2\mu)\big]\big(\varphi_2(x)\big)  \cdot \big|{\rm det}\big(\dd [\varphi_2\circ\varphi_1^{-1}]_{\varphi_1(x)}\big) \big|,\ \ x\in U_1\cap U_2,
\end{align*}
may lead to incompatible smoothness definitions over the intersection of two charts $(U_1,\varphi_1)$ and $(U_2,\varphi_2)$ if the atlas is at most $\beta$-smooth --- the differential \footnote{Here we have identified both tangent spaces $T_y \mb R^d$ and $T_{\varphi_2\circ\varphi_1^{-1}(y)} \mb R^d$ of $\mb R^d$ at $y$ and $\varphi_2\circ\varphi_1^{-1}(y)$ with $\mb R^d$.}$\dd [\varphi_2\circ\varphi_1^{-1}]_y:\,\mb R^d\to\mb R^d$ at $y\in \varphi_1(U_1\cap U_2)$ of the transition map $\varphi_2\circ\varphi_1^{-1}$ is at most $(\beta-1)$-smooth in $y$.
An $\alpha$-smooth distribution on $\m M$ can be equivalently defined as a distribution whose density function with respect to the volume measure of $\m M$ exists and belongs to $C^\alpha(\m M)$~\citep{lee2013smooth}. 
Consequently, the smoothness level of the distribution $\mu$ is an intrinsic quantity that does not reply on the choice of the partition of unity. 


\subsubsection{Smooth distributions on smooth compact submanifold}
 Now, we are in place to define the family of smooth distributions on smooth compact submanifold without boundaries on $\mb R^D$ as the set $\m P^\ast = \m P^\ast(d,D,\alpha,\beta,L^\ast)$ with $d\leq D$, $\beta>1$ and $\alpha\in(0,\beta-1]$ composed of all probability measures $\mu\in \m P(\mb R^D)$ satisfying:
\begin{enumerate}[topsep=0.5em,itemsep=0.5em,partopsep=0em,parsep=0em]
\item $\mu$ is an $\alpha$-smooth distribution on  a $\beta$-smooth $d$-dimensional compact submanifold $\m M$ embedded in $\mb R^D$ .

\item The density $\mu$ relative to the volume measure of $\m M$ is uniformly bounded from below by $1/L^\ast$ on $\m M$.

\item  $\m M$ is covered by an atlas  $\ms{A}=\{(U_{\lambda},\phi_{\lambda})\}_{\lambda\in \Lambda}$ on $\m M$ such that: a)~each chart $(U,\phi)$ in atlas $\ms A$ satisfies $\|\phi^{-1}\|_{C^\beta(\phi(U))}\leq L^\ast$ and $\|\mu\circ \phi^{-1}\|_{C^\alpha(\phi(U))} \leq L^\ast$; b)~for any $z\in \phi(U)$, the Jacobian of $\phi^{-1}(z)$ is full rank and all its singular values are lower bounded by $1/L^\ast$ in absolute values. Moreover, for any $x\in \m M$, there exists a $\lambda\in \Lambda$ such that $U_{\lambda}$ and $\phi_{\lambda}(U_{\lambda})$ covers $B_{1/L^\ast}(x)\cap \m M$ and $B_{1/L^\ast}(\phi_{\lambda}(x))$ respectively.
\end{enumerate}
\noindent In Appendices~\ref{sec:gen_model_class} and~\ref{ext:main result}, we discuss extensions to manifolds with boundaries and unbounded manifolds.

\begin{remark} \normalfont 
A similar class of smooth distributions on submanifolds is also considered in~\cite{10.1214/21-EJS1826}, where
their regularity of the manifold is characterized by the notion of reach\footnote{The reach of a manifold $\m M$ is the supremum of all $r\geq 0$ such that the orthogonal projection on $\m M$ is well-defined on the $r$-neighbourhood of $\m M$.}.
In fact, for any $\beta$-smooth $d$-dimensional compact submanifold with reach uniformly bounded from below, one can always find an atlas $\ms A$ satisfying above conditions with a sufficiently large $L^\ast$~\citep{10.1214/18-AOS1685,10.1214/21-EJS1826}. One particular choice of the local parametrization is the exponential map\footnote{The exponential map: $ \exp_x: T_{x}\m M\to \m M$ of $\m M$ at $x$ is defined by $\exp_x(v)=\gamma_{x,v}(1)$, where $\gamma_{x,v}$ is the unique constant speed geodesic path of $\m M$ with initial value $x$ and velocity $v$.} $\exp_{x}$, and Condition 3 above holds if $\exp_x$ and $\mu\circ \exp_x$ have the corresponding smoothness and the injectivity radius\footnote{The injectivity radius ${\rm inj}(x)$ of $\m M$ at $x$ is the supremum of values of $r$ such that the exponential map defines a global diffeomorphism from $\{v\in T_{x}\m M\,|\, \|v\|\leq r\}$ onto its image in $\m M$. The injectivity radius of $\m M$ is defined as the infimum of ${\rm inj}(x)$ over all $x\in \m M$. } of $\m M$ is lower bounded away from zero. 
\end{remark}

\begin{remark} \normalfont 
A manifold without boundary that is compact is called a closed manifold. Examples of closed submanifolds in $\mb R^D$ include a $d$-dimensional sphere lying in a $(d+1)$-dimensional affine subspace of $\mb R^D$ and a $d$-dimensional torus $\mb T^d$ embedded in $\mb R^D$ that is diffeomorphic to the product of $d$ circles. Any closed submanifold requires at least two covering charts in its describing atlas since it is not homeomorphic to any open set of $\mb R^d$. Mathematically, the Lusternik-Schnirelmann category~\citep{10.2307/1968905,cornea2003lusternik} of a topological manifold $\m M$ can be used to provide a lower bound on the smallest number of charts to cover $\m M$. For example, the $d$-dimensional sphere requires at least two charts (using the stereographic projection) and the $d$-dimensional torus $\mb T^d$ cannot be covered with $d$ or fewer charts. 
\end{remark}

 \subsubsection{Distribution estimator class: mixture of generative models}\label{approximationfamily}
 
To describe the statistical model for representing probability measures $\mu$ on unknown submanifolds, we consider two (mixture of) generative model classes, $\m S^{\rm ap}  = \m S^{\rm ap}(d,D,\alpha,\beta,\ms O_M ,L)$ and $\m S^{\rm ap}_{\nu_0}  = \m S^{\rm ap}_{\nu_0}(d,D,\alpha,\beta,\ms O_M ,L)$, where $\ms O_M =\{\mb B_{r_m}(a_m)^{\circ}\}_{m\in [M]}$ is a pre-specified open cover of $\mb B_L^D$ that contains the submanifold. The first generative model class $\m S^{\rm ap}$ consists of  mixtures of generative models with rejection sampling:  $\mu =\sum_{m=1}^M w_{[m]} \mathcal{A}(G_{[m]},\nu_{[m]},\rho_m)$ where 
$\{\rho_m\}_{m\in [M]}$ is the partition of unity  subordinate to $\ms O_M$ defined in Section~\ref{sec:pou}, $\{w_{[m]}\}_{m\in[M]}$ are non-negative mixing weights with $\sum_{m=1}^M w_{[m]}=1$, and for any $m\in [M]$: (1) each component of $G_{[m]}$ is a $\beta$-smooth function over $\mb R^d$ with $\beta$-H\"{o}lder norm bounded by $L$; (2) $\nu_{[m]}$ is an $\alpha$-smooth probability density on $\mb B_1^d$ with $\alpha$-H\"{o}lder norm bounded by $L$; (3) $\mathcal{A}(G_{[m]}, \nu_{[m]}, \rho_m)$ denotes the probability measure induced by the data generating process where $X\sim [G_{[m]}]_{\#}\nu_{[m]}$ is accepted with probability $\rho_m(X)\in[0,1]$. 
To summarize, we have
\begin{equation*}
\begin{aligned}
    \m S^{\rm ap}=&\Big\{\sum_{m=1}^M w_{[m]} \mathcal{A}(G_{[m]},\nu_{[m]},\rho_m)\,:\, \sum_{m=1}^M w_{[m]}=1;\, \forall m\in [M],\, 0\leq w_{[m]}\leq 1;\\
    &\qquad\qquad G_{[m]}\in C^{\beta}_{L}(\mb R^d;\mb R^{D});\, \nu_{[m]}\in \m P(\mb B_1^d) \text{ and } \nu_{[m]}\in C^{\alpha}_L(\mb B_1^d)\ \Big\}.
    \end{aligned}
\end{equation*}
The decomposition $\mu =\sum_{\lambda\in\Lambda} \rho_\lambda \mu$ of $\mu$ is the dual counterpart of definition~\eqref{eqnpou} for expectation $\mb E_{\mu}[f]$ through the partition of unity with $\Lambda=[M]$. To avoid explicit estimation of the local densities $\nu_{[m]}$, we also consider a second generative model class $\m S^{\rm ap}_{\nu_0}$ 
\begin{equation*}
\begin{aligned}
   \m S^{\rm ap}_{\nu_0}=&\Big\{\sum_{m=1}^M w_{[m]} \mathcal{A}(G_{[m]},\nu_{[m]},\rho_m)\,:\, \sum_{m=1}^M w_{[m]}=1;\, \forall m\in [M],\, 0\leq w_{[m]}\leq 1;\\
    &\qquad\qquad G_{[m]}\in C^{\beta}_{L}(\mb R^d;\mb R^{D});\, \nu_{[m]}=\big(V_{[m]}\big)_\#\nu_0 \text{ and } V_{[m]}\in C^{\alpha+1}_L(\mb B_1^d)\ \Big\},
    \end{aligned}
\end{equation*}
where $\nu_0$ is a prespecified distribution on $\mb B_1^d$ that is easy to sample from. In other words, $\m S^{\rm ap}_{\nu_0}$ further replaces the local latent variable distribution $\nu_{[m]}$ by a generative model $(\nu_0,V_{[m]})$ for each $m\in[M]$ with a common generative distribution $\nu_0$, so that $\mu\in \mathcal{S}^{\rm ap}_{\nu_0}$ can be equivalently expressed as $\mu =\sum_{m=1}^M w_{[m]} \mathcal{A}([G_{[m]}\circ V_{[m]}],\nu_0,\rho_m)$ that is easy to generate samples: first draw a categorical variable $V$ in $[M]$ with probabilities $\mb P(V=m)=w_{[m]}$ for $m\in[M]$; then draw a sample $U$ from $\nu_0$ and set $X=G_{[V]}\circ V_{[V]}(U)$; finally, flip a coin and accept $X$ with probability $\rho_V(X)$.

We will show in our main result (Theorem~\ref{maintheorem}) that for smooth distributions on unknown closed submanifold with positive density, we are able to construct a minimax-optimal estimator (modulo logarithm terms) $\wh\mu\in \m S^{\rm ap}_{\nu_0}$ with $\nu_0$ being the uniform distribution on $\mb B_1^d$ (the uniform distribution can also be replaced with any other smooth distributions). While when the target distribution $\mu^\ast$ belongs to a more general (mixture of) generative model class $\m S^\ast$ as we considered in Theorem~\ref{upperboundgenerative}, which contains all distributions in $\m P^\ast$ and also distributions induced by a single generative model $(\nu^\ast, G^\ast)$\footnote{Note that such a single generative model is not included in $\m P^\ast$, as any closed submanifold requires at least two covering charts in its describing atlas.} with $\nu^\ast$ being smoothly decaying to zero around the boundary of its support, we will need the more flexible approximation family $\m S^{\rm ap}$ to cover $\m S^\ast$.

\section{Minimax Rate of Convergence}
In this section, we establish the minimax rate of convergence for the adversarial risk on $\mu^\ast \in \m P^\ast$ of distribution estimation on unknown submanifold with i.i.d.~samples $X_1,X_2,\ldots,X_n\sim \mu^\ast$, and
propose an optimal procedure based on learning a (mixture of) generative model in class $\m S^{\rm ap}$ (or $\m S^{\rm ap}_{\nu_0}$) via minimizing a carefully constructed empirical surrogate risk. After that, we also provide a data-driven adaptive estimator that does not require prior knowledge about intrinsic dimension $d$, manifold smoothness $\beta$ and distribution smoothness $\alpha$.

The following theorem summarizes our main result on the minimax rate of convergence.
 \begin{theorem}[Minimax rate of distribution estimation]\label{maintheorem}
Fix $L^\ast>0$, $\gamma\geq 0$, $0\leq \alpha\leq \beta-1$, $\beta>1$, and $D,d \in \mathbb{N}^+$ with $D>d$, write $\mathcal{P}^{\ast}=\mathcal{P}^{\ast}(d,D,\alpha,\beta,L^\ast)$, then
\begin{enumerate}
    \item  there exists a constant $L_0$  such that when $L^\ast \geq L_0$, then
\begin{equation}
\begin{aligned}
 \underset{\wh{\mu}\in \m P(\mb R^D)}{\inf} \,\underset{\mu\in \mathcal{P}^{\ast}} {\sup}\mathbb{E} \big[d_{\gamma}(\wh{\mu},\mu) \big] \geq C \, n^{-\frac{1}{2}}\vee n^{-\frac{\alpha+\gamma}{2\alpha+d}}\vee n^{-\frac{\gamma\beta}{d}};
\end{aligned}
\end{equation}
 \item there exist positive constants $L_1,L_2$ such that  for any $L\geq L_1$ and open cover $\ms O_{M}=\{\mb B_{r_m}(a_m)^{\circ}\}_{m\in [M]}$ of $\mb B_{L}^D$ with $\max\{r_1,r_2,\cdots,r_M\}\leq L_2$,  it holds that 
\begin{equation}
\begin{aligned}
\underset{\wh{\mu}\in S^{\rm ap}_{\nu_0}(d,D,\alpha,\beta,\ms O_M ,L)}{\inf} \,\underset{\mu\in \mathcal{P}^{\ast}} {\sup}\mathbb{E} \big[ d_{\gamma}(\wh{\mu},\mu)\big] \leq C\, \Big(\frac{n}{\log n}\Big)^{-\frac{1}{2}}\vee \Big(\frac{n}{\log n}\Big)^{-\frac{ \alpha+\gamma}{2\alpha+d}}\vee \Big(\frac{n}{\log n}\Big)^{-\frac{\gamma\beta}{d}},
\end{aligned}
\end{equation}
\end{enumerate}
where $\nu_0={\rm Unif}(\mb B_1^d)$ and the infimum are both taken over all distribution estimators $\wh{\mu}$ belonging to corresponding families based on data $X_{1:n}$.
 \end{theorem}
 
 We make several brief comments on the minimax rate from Theorem~\ref{maintheorem}.  First, the logarithmic terms appearing in the upper bound of Theorem~\ref{maintheorem} enable us to obtain a high probability bound for further bounding the expected loss. Second, the ambient space dimension $D$ does not appear in the exponents of the minimax rate, so the problem of estimating a distribution on a low-dimensional submanifold does not suffer from the ``curse of dimensionality'' due to a large $D$.
 Third, recall that the adversarial loss $d_\gamma$ employed in this paper for distribution estimation captures two aspects of the data generating process: supporting manifold recovery and density estimation on the manifold (c.f.~Section~\ref{se:adv_loss}). Both of these two aspects are reflected in the derived minimax rate, as we describe in the following.
 
In fact,~\cite{10.1214/18-AOS1685} proves the minimax optimal rate $n^{-\frac{\beta}{d}}$ of estimating a $d$-dimensional $\beta$-smooth submanifold under the Hausdorff distance, which is related to the third term $n^{-\frac{\gamma\beta}{d}}$ in our rate under $\gamma=1$. As we discussed in Section~\ref{se:adv_loss}, our adversarial loss $d_\gamma$ under $\gamma=1$ can be interpreted as an average version of the Hausdorff distance. Therefore, term $n^{-\frac{\gamma\beta}{d}}$ is coming from estimating the unknown support of $\mu^\ast$, or supporting manifold recovery (see Remark~\ref{rmk:gamma} for more discussions). Moreover, in absence of a low-dimensional submanifold structure, the derived rate reduces to the minimax rate $n^{-\frac{1}{2}}\vee n^{-\frac{\alpha+\gamma}{2\alpha+d}}$ (by taking $\beta=\infty$) of estimating a $\alpha$-smooth density on $[0,1]^d$, for $\alpha\in[0,\infty)$, proved in~\cite{liang2020generative}.
In another related work,~\cite{10.1214/21-EJS1826} prove that a carefully constructed kernel density estimator achieves the rate $n^{-\frac{\alpha}{2\alpha+d}}$, when\footnote{The manifold regularity $\alpha>0$ defined in~\cite{10.1214/21-EJS1826} is related to our manifold smoothness level $\beta>1$ via $\alpha=\beta-1$.}$\alpha\in[0,\beta-1]$, for the pointwise $\ell_p$ loss of estimating a smooth density supported by an unknown $d$-dimensional submanifold. Notice that their pointwise loss only concerns the density difference evaluated on the submanifold, which already uses the knowledge of the manifold in defining the loss, explaining why their rate does not involve a term like $n^{-\frac{\beta}{d}}$ due to supporting manifold estimation.
According to these reasons, the second term $n^{-\frac{\alpha+\gamma}{2\alpha+d}}$ in our rate can be interpreted as a consequence of smooth density estimation on the $d$-dimensional submnaifold as if the manifold known, where the extra $\gamma$ in the exponent is due to the smoothness of the discriminator class (pointwise loss can be viewed as a discontinuous discriminator, or $\gamma=0$).  
 
 \begin{figure}[t]
\centering  
 \label{figure:minimaxrate1}
\includegraphics[width=0.6\textwidth]{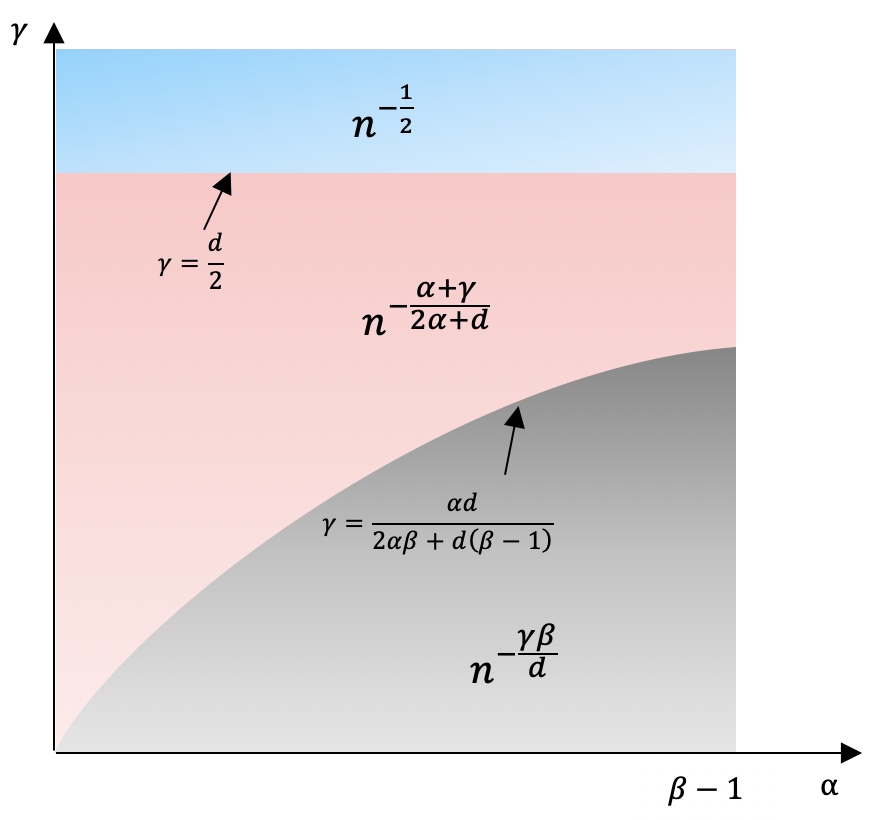}
 \caption{Diagram for the minimax rate $n^{-\frac{1}{2}}\vee n^{-\frac{\alpha+\gamma}{2\alpha+d}}\vee n^{-\frac{\gamma\beta}{d}}$ for fixed $d\in \mathbb{N}^+$ and $\beta>1$.}
\label{figure:minimaxrate}
\end{figure}
 
Figure~\ref{figure:minimaxrate} depicts the three regimes of the
problem characteristics identified by Theorem~\ref{maintheorem}, defined by which of the three terms in the minimax rate $n^{-\frac{1}{2}}\vee n^{-\frac{\alpha+\gamma}{2\alpha+d}}\vee n^{-\frac{\gamma\beta}{d}}$ dominates.
From the diagram, there exist transitions at $\gamma=d/2$ and $\gamma=\alpha d/[2\alpha\beta+d(\beta-1)]$. When the discriminator smoothness level $\gamma$ satisfies $\gamma\geq d/2$ so that the discriminator class is relatively small, the rate is the parametric root-$n$ rate $n^{-\frac{1}{2}}$. When the discriminator smoothness level is moderate, or $  {\alpha d}/[2\alpha\beta+d(\beta-1)]\leq  \gamma<d/2$, the term $n^{-\frac{\alpha+\gamma}{2\alpha+d}}$ due to $d$-dimensional density estimation dominates the minimax rate. When $0\leq \gamma< {\alpha d}/[2\alpha\beta+d(\beta-1)]$, the minimax rate becomes $n^{-\frac{\gamma\beta}{d}}$ since with a small $\gamma$ the adversarial loss $d_\gamma(\mu,\nu)$ between two distributions $\mu$ and $\nu$ tends to be more sensitive to the misalignment between their supports supp$(\mu)$ and supp$(\nu)$ than to the discrepancy between the probability mass allocations on their respective supports (c.f.~Section~\ref{se:adv_loss}). Overall, in the regime of $\gamma\leq d/2$, increasing $\gamma$ leads to a faster rate, while the evaluation metric becomes weaker. It is also worthwhile highlighting that the transition boundary $\gamma =d/2$ between the parametric regime and the density estimation regime only depends on the discriminator smoothness $\gamma$ and intrinsic dimension $d$, while the transition boundary $\gamma = {\alpha d}/[2\alpha\beta+d(\beta-1)]$ between the density estimation regime and the supporting manifold estimation regime depends on all problem characteristics $(\alpha,\beta,\gamma,d)$ except for the ambient dimension $D$ --- the transition threshold on $\gamma$ converges to $0$ as the manifold smoothness $\beta \to\infty$ and becomes $d/(2\beta+d)$ if the distribution $\mu^\ast$ has the maximal (well-defined) smoothness degree $\alpha=\beta-1$.

 \begin{remark}\label{rmk:gamma} \normalfont  
 Taking $\gamma=0$ in Theorem~\ref{maintheorem} implies that ${\inf}_{\wh{\mu}}\,{\sup}_{\mu\in \mathcal{P}^{\ast}} \mathbb{E} \big[d_{\rm {TV}}(\wh{\mu},\mu)\big]$ is lower bounded away from zero, meaning that no estimator can consistently estimate $\mu^\ast$ relative to total variation metric. In addition, by Pinsker’s inequality, the minimax rate relative to the Jensen–Shannon (JS) divergence or KL divergence is also lower bounded away from zero. This lack of estimation consistency is due to the misalignment in supports between $\wh\mu$ and $\mu^\ast$, and theoretically explains the empirical findings made in~\citep{arjovsky2017principled, goodfellow2014generative}: the training of the original GAN, which minimizes the JS divergence~\citep{goodfellow2014generative} at the population level, tends to be unstable, while the training of GAN's using some weaker discrepancy measures, such as the Wasserstein GAN~\citep{arjovsky2017principled} using the $1$-Wasserstein distance (corresponding to $d_\gamma$ with $\gamma=1$), are more stable and the resulting generators are more accurate and reliable. In particular,~\cite{arjovsky2017principled} shows that, in the original GAN,
when the support supp$(\mu^\ast)$ of target distribution $\mu^\ast$ lies on a low-dimensional manifold and does not perfectly aligned with the support supp$(\wh \mu)$ of the output distribution $\wh \mu$ from the generator class, then there always exists a perfect discriminator in the discriminator class that separates real samples and fake samples produced by the generator with $100\%$ accuracy. As a consequence, $\wh \mu$ tends to max out in a neighborhood around supp$(\mu^\ast)$ in the ambient space due to the exploding generator gradient whose expectation and variance are infinite as the discriminator becomes closer to optimality, leading to a notorious decrease in sample quality.
 \end{remark}

 \begin{remark} \normalfont 
A more general setting adopted by some authors considers the deconvolution problem that allows the observed data to contain noises, that is, we observe a set of $n$ i.i.d.~random samples $\{Y_i\}_{i=1}^n$ generated from model $Y_i=X_i+\varepsilon_i$, where $\{X_i\}_{i=1}^n$ are samples from the target distribution $\mu^\ast$ supported on an unknown $d$-dimensional submanifold $\m M$ in $\mb R^D$,
and $\{\varepsilon_i\}_{i=1}^n$ are typically independent errors with a known distribution~\citep{10.1214/11-EJS646,10.1214/12-AOS994}.  In this noisy setting,~\cite{10.1214/11-EJS646} propose a modified kernel deconvolution estimator and show the convergence rate $O\big((\log n)^{-1}\big)$ under the $2$-Wasserstein distance with various manifolds and various noise distributions (e.g.~isotropic Gaussian noise or Gaussian noise ``perpendicular" to the manifold).
In addition,~\cite{10.1214/12-AOS994} shows that when the error follows the standard $D$-dimensional Gaussian distribution, the minimax rate of estimating manifold $\m M$ under the Hausdorff distance is extremely slow: it is lower bounded by $O\big((\log n)^{-1}\big)$.
These results suggest that estimating a low-dimensional distribution or its supporting manifold based on noisy observations is an intrinsically hard problem, which is why we focus on the noiseless case in this paper. One way commonly adopted in the literature of circumventing this slow convergence is by assuming the noise variance $\sigma^2$ to decay with the same size. If we also allow isotropic Gaussian noise in the data with variance scales as $\sigma^2 =O\big( n^{-1+\frac{2\beta}{d}}\big)$, then Corollary~\ref{noisecorollary} in Appendix~\ref{app:noisy case} shows that the estimation procedure via generative models developed in Section~\ref{sec:optimal_proce} has the same rate of convergence as if the samples are noiseless.
 \end{remark}
 
 

\subsection{Minimax-optimal estimation via generative models}\label{sec:optimal_proce}
In this subsection, we describe an estimator $\wh \mu$ constructed via generative models that achieves the minimax rate upper bound in Theorem~\ref{maintheorem}.
Generative model learning has recently become popular \citep{goodfellow2014generative, li2015generative,10.1214/19-AOS1858,arjovsky2017wasserstein,tolstikhin2019wasserstein, zhao2018infovae} due to its great practical success in generating new examples such as images and texts that are indistinguishable from real objects, and impressive computational scalability to complex and massive datasets. In the conventional framework, a single generative model $(\widetilde\nu, \widetilde G)$ is learned by solving the following minimax optimization problem, 
\begin{align}\label{Eqn:Old_est}
    \widetilde \mu= \widetilde G_\# \widetilde\nu = \argmin_{\mu \in \m D_G} \,\sup_{f\in\m F} \Big| \mb E_{\mu} [f(x)] - \frac{1}{n} \sum_{i=1}^n f(X_i)\Big|,
\end{align}
where recall that $\m D_G=\big\{G_\#\nu:\, \nu\in\Upsilon,\, G\in\m G\big\}$ is a generic generator class and $\m F$ is a generic discriminator class. Here the empirical average $n^{-1}\sum_{i=1} f(X_i)$ is a sample surrogate to the population level expectation $\mb E_{\mu^\ast}[f]$. A successful application of this procedure replies on the implicit assumption that the underlying data manifold $\m M$ admits a single chart representation. In this work, we propose a new generative model learning procedure with two improvements --- first, we employ mixtures of generative models in $\m S^{\rm ap}$ (or $\m S^{\rm ap}_{\nu_0}$) to estimate distributions on those submanifolds that cannot be covered by a single chart; second, we use a regularized surrogate $\wh{ \m J}(f)=\wh{ \m J}(f;\,X_{1:n})$ to replace $n^{-1}\sum_{i=1} f(X_i)$ in~\eqref{Eqn:Old_est}, which improves the estimation accuracy by utilizing smoothness structures in $\mu^\ast$ and discriminator $f$ and thus mitigates overfitting. 

Our procedure takes the form of
\begin{align}\label{Eqn:New_est}
    \wh\mu = \argmin_{\mu \in \m S} \sup_{f\in\m F}\Big| \mb E_{\mu} [f(X)] - \wh{ \m J}(f)\Big|,
\end{align}
where $\m S$ is the approximation family that will be chosen later. The main ingredient of our minimax upper bound proof is to bound $\sup_{f\in\m F} \big|\mb E_{\mu^\ast}[f(X)]- \wh{ \m J}(f) \big|$ since by the optimality of $\wh \mu$ in~\eqref{Eqn:New_est} and the definition of adversarial loss $d_\gamma(\cdot,\cdot)$, if the approximation family is  correctly specified so that $\mu^\ast\in \m S$, we have the following basic inequality
\begin{equation}\label{eqn:basic_ineq}
    \begin{aligned}
    d_\m F(\wh \mu,\mu^\ast) = &\, \sup_{f\in\m F}\Big| \mb E_{\wh \mu} [f(x)] - \mb E_{\mu^\ast}[f(X)]\Big|
    \leq \sup_{f\in\m F}\Big| \mb E_{\wh \mu} [f(x)] - \wh{ \m J}(f)\Big| \\
    & \qquad\qquad\qquad + \sup_{f\in\m F}\Big| \mb E_{\mu^\ast}[f(X)] - \wh{ \m J}(f)\Big| \leq 2 \sup_{f\in\m F}\Big| \mb E_{\mu^\ast}[f(X)] - \wh{ \m J}(f)\Big|.
\end{aligned}
\end{equation}
Therefore, the problem of finding an optimal estimator of $\mu^\ast$ boils down to the simultaneous estimation of functional $\mb E_{\mu^\ast}[f(X)]$ for all $f\in\m F$ with smallest worst case error. In this paper, we focus on the H\"{o}lder discriminator class $\m F= \m C_1^\gamma(\mb R^D)$. Note that the empirical average $\wh{ \m J}_{\rm ave}(f):\,=n^{-1}\sum_{i=1}^n f(X_i)$ is not an optimal choice for $\wh{ \m J}(f)$ since the following two sided high probability bound of the supremum of empirical process (c.f.~Lemma~\ref{lemmaempiricalmean} in Appendix~\ref{App:technical}
)
\begin{align}\label{Eqn:EP_bound}
   C_1 \,\sqrt{ \frac{1}{n}} \vee \frac{n^{-\frac{\gamma}{d}}}{\log n} \leq  \sup_{f\in C_1^\gamma(\mb R^D)}\Big| \mb E_{\mu^\ast}[f(X)] - \frac{1}{n} \sum_{i=1}^n f(X_i)\Big| \leq C_2 \,\sqrt{ \frac{\log n}{n}} \vee n^{-\frac{\gamma}{d}}
\end{align}
implies its rate of convergence to be strictly worse than the optimal rate (modulo $\log n$ factors) $n^{-\frac{1}{2}} \vee n^{-\frac{\alpha+\gamma}{2\alpha+d}} \vee n^{-\frac{\gamma\beta}{d}}$ ($\beta>1$) inferred by the minimax lower bound in Theorem~\ref{maintheorem}, where the third term $n^{-\frac{\gamma\beta}{d}}$ is due to the estimation of unknown submanifold $\m M$ and will disappear if $\m M$ is known. It is worthwhile noting that despite the discriminator class $\m C_1^\gamma(\mb R^D)$ being defined on $\mb R^D$, the upper bound in~\eqref{Eqn:EP_bound} only depends on the intrinsic dimension $d$ of the support of $\mu^\ast$, which is due to a covering argument. We provide two proofs to~\eqref{Eqn:EP_bound} in the appendix: one is based on the usual chaining technique in the empirical process theory; the other is based on embedding the discrimator space $C_1^\gamma(\mb R^D)$ into Besov space $B_{\infty,\infty}^\alpha(\mb R^D)$ and truncating the wavelet expansion of $f$ to a proper degree (c.f.~Appendix~\ref{sec:Wavelet_review} for a brief review about Besov spaces and wavelet expansions). The first approach based on chaining is succinct and leads to a tighter bound (no $\log n$ factors), but not easily generalizable to analyze more complicated surrogate $\wh{ \m J}(f)$ beyond the empirical average; the second approach incurs extra $\log n$ factors and is technically more involved, but its proof is more insightful and motivates our improved surrogate $\wh{ \m J}(f)$ leading to a minimax-optimal (modulo $\log n$ factors) estimator $\wh \mu$.

The primary reason for the empirical average $\wh{ \m J}_{\rm ave}(f)$ not achieving the optimal bound for $\sup_{f\in\m F}\big| \mb E_{\mu} [f(x)] - \wh{ \m J}(f)\big|$ is that $\wh{ \m J}_{\rm ave}(f)$ does not utilize the smoothness structure on the true underlying distribution $\mu^\ast$ and submanifold $\m M$. In a nutshell, our improvements on surrogate $\wh{ \m J}$ come from two sources: 
\begin{enumerate}[topsep=2pt,itemsep=0ex]
\item we plug-in a smoothness regularized empirical distribution $\widetilde \nu$ to improve the estimation on the expectation $\mb E_{\mu^\ast}[f_{\rm high}]$ of the high frequency part $f_{\rm high}$ of the discriminator $f=f_{\rm low} + f_{\rm high}$ with $f_{\rm low}$ denoting the low frequency part. This improvement reduces part of the $n^{-\frac{1}{2}}\vee n^{-\frac{\gamma}{d}}$ error in~\eqref{Eqn:EP_bound} to $n^{-\frac{1}{2}}\vee n^{-\frac{\alpha+\gamma}{2\alpha+d}}$ due to utilizing the $\alpha$-smoothness of true distribution $\mu^\ast$.
Specifically, $\widetilde \nu$ is constructed by using partition of unity and truncating the wavelet expansion of localized empirical distributions (restricted to
the open cover) to filter out the high frequency components that are unstable due to relatively high variances. 

\item we add a higher-order correction term to account for the misalignment between the effective support of regularized distribution $\widetilde \nu$ and the support of true distribution $\mu^\ast$, both having intrinsic dimension $d$. This improvement reduces part of the $n^{-\frac{1}{2}}\vee n^{-\frac{\gamma}{d}}$ error in~\eqref{Eqn:EP_bound} to $n^{-\frac{1}{2}}\vee n^{-\frac{\gamma\beta}{d}}$ due to utilizing the $\beta$-smoothness of submanifold $\m M$ ($\beta\geq 1$). Specifically, this correction is constructed using partition of unity and compensating a remainder term from the Taylor expansion of discriminator $f\in C_1^\gamma(\mb R^D)$ up to order $\lfloor \gamma \rfloor$ when estimating its expectation.
\end{enumerate}
\noindent Combining these two modifications on $\widehat{\m J}$ together improves the overall worst case error rate from $n^{-\frac{1}{2}} \vee n^{-\frac{\gamma}{d}}$ to $n^{-\frac{1}{2}} \vee n^{-\frac{\alpha+\gamma}{2\alpha+d}} \vee n^{-\frac{\gamma\beta}{d}}$.

\smallskip

To summarize, our proposed minimax-optimal estimator $\wh \mu$ is constructed in three steps.
Let $[n]= I_1\cap I_2$ be a random splitting of the data indices into two sets with $|I_1|= \lceil n/2 \rceil$ and $|I_2| = n-|I_1|$. Let $(L,\,M)$ be sufficiently large positive constants and recall that $\{\rho_m\}_{m\in[M]}$ is the partition of unity subordinate to the open cover $\ms O_M=\{\mb B_{r_m}(a_m)^{\circ}\}_{m\in[M]}$ of $\mb B_L^D$ constructed in Section~\ref{sec:pou}. 

\noindent {\bf Step 1}: (Submanifold estimation) For each open set $\mb B_{r_m}(a_m)^\circ$ in the open cover $\ms O_M$, we form estimators 
$\wh G_{[m]}$ and $\wh Q_{[m]}$ of a \footnote{Coordinate maps are not unique. For examples, the composition of any coordinate map with any $C^\infty$ diffeomorphism of $\mb R^d$ remains a coordinate map.}coordinate map $\varphi_m:\,\mb \mb B_{r_m}(a_m)^\circ \to \mb R^d$ and its inverse $\varphi_m^{-1}$ respectively by minimizing the squared reconstruction loss on samples $\{X_i:\,i\in I_1\}$,
 \begin{equation}\label{estimator1}
 (\wh{G}_{[m]},\wh{Q}_{[m]})=\underset{G\in\ms{G},\, Q\in \ms{Q}}{\arg\min}\left( \frac{1}{|I_1|}\sum_{i\in I_1} \|X_i-G\circ Q(X_i)\|_2^2 \cdot \bold{1}(X_i\in S_m^{\dagger})\right),
 \end{equation}
 where $S_m^{\dagger}=\mb B_{r_m+0.5/L}(a_m)$ is an enlargement of $\mb B_{r_m}(a_m)$ for avoiding the technical issue due to the boundary of $\mb B_{r_m}(a_m)$, 
 $\ms{G}=C_{L}^{\beta}(\mathbb{R}^d;\, \mathbb{R}^D)$ and $\ms {Q}= C_{L}^{\beta}(\mathbb{R}^D;\, \mathbb{R}^d)$. Let $\wh p_m = \frac{1}{|I_1|} \sum_{i\in I_1} \bold{1}(X_i\in S_m^{\dagger})$ denote the sample frequency of falling into $S_m^{\dagger}$ and let $\wh{\mb M}=\{m\in [M]\,:\,\wh p_m\geq \sqrt{\frac{\log n}{n}}\}$.

\noindent {\bf Step 2}: (Surrogate functional construction) For each $m\in [M]$, let $\widetilde \nu_{[m],\wh Q_{[m]}}$ be a smoothness regularized estimator of $\big(\wh Q_{[m]}\big)_\# (\rho_m \mu^\ast)$ by truncating a wavelet expansion to a finite degree. Let $J$ be the largest integer such that $2^J \leq (\frac{n}{\log n})^{\frac{1}{2\alpha+d}}$, $\Pi_J f$ denote the projection of any $f\in\m F$ onto the first scale $J$ wavelet coefficients and $\Pi_J^\perp f = f-\Pi_J f$. The precise definitions of $\widetilde \nu_{[m],\wh Q_{[m]}}$ and $\Pi_J f$ are available in Appendix~\ref{sec:Wavelet_review} . 
We form a regularized estimator of localized expectation $\mb E_{\mu^\ast}[f(X)\,\rho_m(X)]$ using samples $\{X_i:\,i \in I_2\}$ as follows. If $m\notin \wh{\mb M}$, then define $\wh {\m J}_m (f) = 0$ to avoid estimation degeneracy (see Section~\ref{Sec:more_proofs_upperbound} for more details); otherwise, define $\wh {\m J}_m (f) = \wh {\m J}_{m,s}(f) + \wh {\m J}_{m,l}(f) + \wh {\m J}_{m,h}(f)$, where
\begin{subequations}\label{eqn:surrogate_m}
\begin{align}
    \wh {\m J}_{m,l}(f) &= \frac{1}{|I_2|} \sum_{i\in I_2} (\Pi_J f)\circ\wh G_{[m]}\circ \wh Q_{[m]}(X_i) \,\rho_m(X_i), \label{eqn:low_freq}\\
    \wh {\m J}_{m,h}(f) &= \mb E_{\widetilde \nu_{[m], \wh Q_{[m]}}}\big[ (\Pi_J^\perp f)\circ \wh G_{[m]}\big],\quad \mbox{and} \label{eqn:high_freq}\\
    \wh {\m J}_{m,s}(f) &= -\frac{1}{|I_2|} \sum_{i\in I_2} \sum_{j\in \mb N_0^D\atop 1\leq |j|\leq \lfloor\gamma\rfloor} \frac{1}{j!} f^{(j)}(X_i) \, \big[\wh G_{[m]} \circ \wh Q_{[m]}(X_i) - X_i\big]^j \, \rho_m(X_i). \label{eqn:smooth_corr}
\end{align}
\end{subequations}
The first two terms $\wh {\m J}_{m,l}(f)$ and $\wh {\m J}_{m,h}(f)$ together form a sample approximation to $\mb E_{\mu^\ast}\big[ f\big(\wh G_{[m]} \circ \wh Q_{[m]}(X)\big)\,\rho_m(X)\big)\big]$, where $\wh {\m J}_{m,l}(f)$ estimates the expectation of the low frequency components collected in $f_{\rm low}=\Pi_J f$ of $f$ and $\wh {\m J}_{m,h}(f)$ estimates the high frequency components collected in $f_{\rm high}=\Pi^\perp_J f$;
the third term $\wh {\m J}_{m,s}(f)$ corresponds to a (sample version of) higher-order smoothness correction to $\mb E_{\mu^\ast}\big[ f\big(\wh G_{[m]} \circ \wh Q_{[m]}(X)\big)\,\rho_m(X)\big)\big]$ for approximating $\mb E_{\mu^\ast}\big[f(X)\,\rho_m(X)\big]$. Finally, we construct a regularized estimator of $\mb E_{\mu^\ast}[f(X)]$ as $\wh{ \m J}(f) = \sum_{m=1}^M \wh{ \m J}_m(f)$.

\noindent {\bf Step 3}: (Generative model estimation) Given an approximation family $\m S$, the estimator $\wh \mu$ is defined as
\begin{equation}\label{estimatorsmallbeta}
    \wh{\mu}=\underset{\mu\in \mathcal{S}}{\arg\min} \underset{f\in C^{\gamma}_1(\mathbb{R}^D)}{\sup}\Big(\mb E_\mu\big[f(X)] - \wh {\m J}(f)\Big).
\end{equation}


\noindent In practice, we may also use the kernel density estimator with proper bandwidth to regularize the empirical distribution of local latent variables (c.f.~Remark~\ref{rem:KDE} in Appendix~\ref{sec:Wavelet_review}) in step 2; however, the wavelet truncation is technically easier to analyze.
The following theorem shows that the estimator $\wh \mu$ in~\eqref{estimatorsmallbeta} can achieve a minimax-optimal rate (modulo logarithmic term) when the target distribution $\mu^\ast$ belongs to a larger generative model class $\m S^\ast=\m S^\ast(d,D,\alpha,\beta,\ms O_{M},L)$ that contains the distribution family $\m P^\ast$ considered in Theorem~\ref{maintheorem} as a subset, where $\ms O_M=\{\mb B_{r_m}(a_m)^{\circ}\}_{m\in[M]}$ forms a open cover for $\mb B_L^D$ and for any distribution $\mu\in \m S^\ast$ and $m\in [M]$, there exists a set $\widetilde{S}_m$ containing $\mb B_{r_m}(a_m)$, such that $\mu|_{\widetilde{S}_m}$ can be written as a single generative model $(\nu_{[m]},G_{[m]})$ with an invertible and $\beta$-smooth generative function $G_{[m]}:\mb R^d\to \mb R^D$, and an $\alpha$-smooth density $\nu_{[m]}\in \m P(\mb B_1^d)$. The precise definition of $\m S^\ast$ is given  in Appendix~\ref{sec:gen_model_class}. In particular, Lemma~\ref{lemmaboundaryless} in Appendix~\ref{sec:gen_model_class} shows that for a suitable choice of the open cover $\ms O_M$, $\m P^\ast$ is a subset of $\m S^\ast$; moreover, the distribution estimator class $\m S_{\nu_0}^{\rm ap}$ with $\nu_0$ being ${\rm Unif}(\mb B_1^d)$ is sufficient to cover $\m P^\ast$, while we need the more flexible approximation family $\m S^{\rm ap}$ to cover the generative model class $\m S^\ast$, thus the following theorem gives a stronger result than the upper bound in Theorem~\ref{maintheorem}. See Lemma~\ref{lemmaboundaryless} in Appendix~\ref{sec:gen_model_class} for a precise relationship among these distribution classes. It is worth mentioning that our proof of Lemma~\ref{lemmaboundaryless} applies the Caffarelli’s global regularity theory~\citep[Theorem 12.50 of][]{Villani2009,10.2307/2118564} from optimal transport theory to construct the generative maps $\{V_{[m]}\}_{m=1}^M$ in the definition of class $\m S_{\nu_0}^{\rm ap}$ (c.f.~Remark~\ref{remarl:OT} in in Appendix~\ref{sec:gen_model_class} for further details).

\begin{theorem}[Minimax upper bound in generative model class]\label{upperboundgenerative}
Let the approximation family $\m S$ to be $\m S^{\rm ap}(d,D,\alpha,\beta,\ms O_M,L)$.  Suppose $\mu^\ast\in \m S^\ast(d,D,\alpha,\beta, \ms O_{M},L)$ and $X_{1:n}$ are $i.i.d.$ samples from $\mu^\ast$. If $D>d$, $\gamma,\alpha\geq 0$ and $\beta>1$, then for any positive constant $c$, there exist positive constants $c_1$ and $n_0$ such that when $n\geq n_0$, it holds with probability larger than $1-n^{-c}$ that 
\begin{equation*}
   \underset{f\in C^{\gamma}_1(\mathbb{R}^D)}{\sup}\Big(\mb E_{\mu^\ast}\big[f(X)] - \sum_{m=1}^M\wh {\m J}_m(f)\Big)\leq  c_1\, \Big(\frac{n}{\log n}\Big)^{-\frac{1}{2}}\vee \Big(\frac{n}{\log n}\Big)^{-\frac{ \alpha+\gamma}{2\alpha+d}}\vee \Big(\frac{n}{\log n}\Big)^{-\frac{\gamma\beta}{d}}.
\end{equation*}
As a result 
\begin{equation*}
    \mb E [d_{\gamma}(\wh\mu, \mu^\ast)]\leq C\, \Big(\frac{n}{\log n}\Big)^{-\frac{1}{2}}\vee \Big(\frac{n}{\log n}\Big)^{-\frac{ \alpha+\gamma}{2\alpha+d}}\vee \Big(\frac{n}{\log n}\Big)^{-\frac{\gamma\beta}{d}}.
\end{equation*}
\end{theorem}
Based on Lemma~\ref{lemmaboundaryless} and the first statement of Theorem~\ref{upperboundgenerative}, when the approximation family $\m S$ is chosen to be $\m S_{\nu_0}^{\rm ap}(d,D,\alpha,\beta,\ms O_M,L)$ with $\nu_0={\rm Unif}(\mb B_1^d)$, the estimator $\wh{\mu}$ can achieve the minimax upper bound in Theorem~\ref{maintheorem}  for a large enough constant $L$ and a suitable choice of the open cover $\ms O_M$.

 \begin{remark} \normalfont
  The most essential conditions in the definition of the generative model class $\m S^\ast$ for obtaining the minimax-optimal rate  are the following: (1)
  the invertibility and the H\"{o}lder regularity of the generative functions $G_{[m]}$ in each generative model, which enable us to obtain a feasible estimator $(\wh G_{[m]},\wh Q_{[m]})$ for reconstructing the data around $\mb B_{r_m}(a_m)$ in Step 1 of the construction of $\wh \mu$; (2) the  H\"{o}lder regularity of reweighted densities $\nu_{[m]}\cdot (\rho_m\circ G_{[m]})$ of the local variables over the entire space $\mb R^d$, as this can lead to sufficiently high smoothness 
  of $(\wh{Q}_{[m]})_{\#}(\rho_m\mu^\ast)$ (with high probability) and thus enables us to construct a smoothness regularized estimator $\wt{\nu}_{[m],\wh{Q}_{[m]}}$ to the density of $(\wh{Q}_{[m]})_{\#}(\rho_m\mu^\ast)$ in Step 2 of the construction of $\wh \mu$.
  Moreover, our theoretical development relaxes the common assumption made in the manifold learning literature~\citep[e.g.\!][]{10.1214/18-AOS1685} on the smoothness of the underlying data submanifold from $C^2$ to $C^\beta$ with $\beta>1$.
 \end{remark}
  \begin{remark}\normalfont
 A similar estimator can be constructed to achieve the same rate in Theorem~\ref{upperboundgenerative} for estimating the constrained distribution $\mu^\ast|_{K_0}$\footnote{Note that in general $d_{\gamma}(\wh{\mu}|_{K}, \mu^\ast|_{K})$ can not be simply upper bounded by $d_{\gamma}(\wh{\mu},\mu^\ast)$ for positive $\gamma$, as the indicator function $\bold{1}_K$ is discontinuous.}, where $K_0$ is a fixed compact set and the intersection of $K_0$ and $\m M={\rm supp}(\mu^\ast)$ is away from the boundary of $\m M$ if it exists. The main difference is that, in this case, after estimating a parametrization $(\wh G,\wh Q)$ of $\m M \cap K_0$, we need an extra step to estimate the support of $\wh{Q}_{\#}(\mu^\ast|_{K_0})$ before constructing a finite sample surrogate to $\mb E_{\mu^\ast|_{K}}[f(\wh{G}\circ \wh{Q}(X))]$. Further details are included in Appendices~\ref{sec:unbounded} and~\ref{proof:unbounded}.
  \end{remark}
 \begin{remark}\normalfont
 The surrogate functional $\widehat{\m J}(f)$ can also be used to construct a test statistic in the two sample hypothesis testing problem, where we reject the hypothesis $H_0\,:\,\mu_1=\mu_2$ if $\sup_{f\in C^{\gamma}_1(\mb R^D)}\big(\widehat{\m J} (f;\,X_{1:n})-\widehat{\m J} (f;\,Y_{1:n})\big)\geq c\,\delta_n^\ast \equiv c\,\big(\frac{n}{\log n}\big)^{-\frac{1}{2}}\vee \big(\frac{n}{\log n}\big)^{-\frac{ \alpha+\gamma}{2\alpha+d}}\vee \big(\frac{n}{\log n}\big)^{-\frac{\gamma\beta}{d}}$ given two set of $n$ random samples $X_{1:n}$ and $Y_{1:n}$ independently obtained from two underlying distributions $\mu_1,\mu_2\in \m S^\ast$, respectively. The test statistic can successfully detect any local alternatives that separate from the null by at least $\delta_n^\ast$ (up to a multiplicative constant) in the $d_{\gamma}$ metric; see Appendix~\ref{sec:two-sample} for further details.
 \end{remark}

\subsection{Data-driven adaptive distribution estimation}
 Since the problem of adaptive estimation of the intrinsic dimensionality of (noisy) manifold-valued data is a fairly well-studied topic~\citep{camastra2002estimating,carter2009local,farahmand2007manifold,levina2004maximum,little2009estimation,yang2016bayesian}, we may apply any of these methods to obtain a a high probability consistent estimator of $d$ beforehand. Therefore, we will only focus on the adaption of our estimator to the unknown manifold smoothness level $\beta^\ast$ and distribution smoothness level $\beta^\ast$.
 Suppose the true smoothness levels $\beta^\ast\in[\beta_{\min},\beta_{\max}]$ and $\alpha^\ast\in [\alpha_{\min},\alpha_{\max}]$, where $\alpha_{\min}\geq 0$ and $\beta_{\min}>1$. Choose discrete grids
 \begin{equation*}
     \m B_1=\{\beta_{\min}=\beta_1<\cdots<\beta_{N_1}=\beta_{\max}\} \ \ \mbox{and}\ \
     \m B_2=\{\alpha_{\min}=\alpha_1<\cdots<\alpha_{N_2}=\alpha_{\max}\},
 \end{equation*}
 where  $\beta_{j}-\beta_{j-1}= c_1/\log n$ and $\alpha_{k}-\alpha_{k-1}= c_2/\log n$ for $j\in [N_1]$ and $k\in [N_2]$. Recall that $\wh{\mb M}=\{m\in [M]\,:\, \wh{p}_m\geq \sqrt{\frac{\log n}{n}}\}$ from Section~\ref{sec:optimal_proce}. For $m\in \wh{\mb M}$  and $\beta_j\in \m B_1$, let $(\wh{G}_{[m]}^{[\beta_j]},\wh{Q}_{[m]}^{[\beta_j]})$ denote the submanifold estimator defined in equation~\eqref{estimator1} with $\beta=\beta_j$. Also let $\wh\beta_{[m]}$ to be the maximal smoothness level in $\m B_1$ that minimizes the reconstruction error, or
 \begin{equation*}
     \wh\beta_{[m]}= \max \,\Big\{\beta\in \m B_1: \sum_{i\in I_1} \|X_i-\wh G^{[\beta]}_{[m]}\circ \wh Q^{[\beta]}_{[m]}(X_i)\|_2^2 \cdot \bold{1}(X_i\in S_m^{\dagger})=0\Big\}.
 \end{equation*}
For a fixed $\alpha_k\in \m B_2$ and $m\in \wh{\mb M}$,  we use  $\wh {\m J}_{m,h}^{[\alpha_k]}(f)$, $\wh {\m J}_{m,l}^{[\alpha_k]}(f)$ and $\wh {\m J}_{m,s}^{\dagger}(f)$ to denote the surrogate functionals $\wh{\m J}_{m,h}(f)$, $\wh{\m J}_{m,l}(f)$ and $\wh{\m J}_{m,s}(f)$ defined in~\eqref{eqn:surrogate_m} respectively with $\alpha=\alpha_k$ and $(\wh{G}_{[m]},\wh{Q}_{[m]})=\big(\wh{G}_{[m]}^{[\wh\beta_{[m]}]},\wh{Q}_{[m]}^{[\wh\beta_{[m]}]}\big)$. The Lepski's estimator~\citep{lepskii1991problem} is then defined as 
 \begin{equation*}
     \wh{\mu}^{\dagger}=\underset{\mu\in \m S}{\arg\min} \underset{f\in C^{\gamma}_1(\mathbb{R}^D)}{\sup}\Big(\mb E_\mu\big[f(X)] - \sum_{m\in \wh{\mb M}}\big[\wh {\m J}_{m,h}^{[\wh\alpha_{[m]}]}(f)+\wh {\m J}_{m,l}^{[\wh\alpha_{[m]}]}(f)+\wh {\m J}_{m,s}^{\dagger}(f)\big]\Big),
 \end{equation*}
 where $\m S= \m S^{\rm ap}(d,D,\alpha_{\max},\beta_{\max},\ms O_{M},L)$, and
 \begin{equation*}
 \begin{aligned}
     \wh\alpha_{[m]}=\max&\Big\{\alpha\in \m B_2:\text{ for all } \alpha'\leq \alpha, \, \alpha'\in \m B_2,\\
     & \underset{f\in C^{\gamma}_1(\mb R^D)}{\sup}\big|\wh{\m J}^{[\alpha]}_{m,h}(f)+\wh{\m J}^{[\alpha]}_{m,l}(f)-\wh{\m J}^{[\alpha']}_{m,h}(f)-\wh{\m J}^{[\alpha']}_{m,l}(f)\big|\leq c_0 \, \sqrt{\frac{\log n}{n}}\vee \big(\frac{\log n}{n}\big)^{\frac{\alpha'+\gamma}{2\alpha'+d}} 
     \Big\}.
      \end{aligned}
 \end{equation*}
 The following corollary shows that such an estimator $\wh{\mu}^{\dagger}$ simultaneously attains the optimal rate (within a possibly
logarithmic factor) over all smoothness levels in the range $\beta^\ast\in[\beta_{\min},\beta_{\max}]$ and $\alpha^\ast\in [\alpha_{\min},\alpha_{\max}]$.
 \begin{corollary}\label{Lepski}
Suppose $\mu^\ast\in \m S^\ast(d,D,\alpha^\ast,\beta^\ast, \ms O_{M},L)$. If $D>d$, $\gamma\geq 0$, $\beta^\ast\in[\beta_{\min},\beta_{\max}]$ and  $\alpha^\ast\in [\alpha_{\min},\alpha_{\max}] \cap [0,\beta^\ast-1]$,  then there exists a  positive constant $C$ such that
\begin{equation*}
    \mb E \big[d_{\gamma}(\wh\mu^{\dagger}, \mu^\ast)\big]\leq C\, \Big(\frac{n}{\log n}\Big)^{-\frac{1}{2}}\vee \Big(\frac{n}{\log n}\Big)^{-\frac{ \alpha^\ast+\gamma}{2\alpha^\ast+d}}\vee \Big(\frac{n}{\log n}\Big)^{-\frac{\gamma\beta^\ast}{d}}.
\end{equation*}
 \end{corollary}

\section{Proof of Main Results}\label{Sec:main_proofs}
In this section, we prove the main results in Theorem~\ref{maintheorem} and Theorem~\ref{upperboundgenerative}. Proofs of other theorems and technical details are provided in the supplement. 

\subsection{Proof of minimax lower bound in Theorem~\ref{maintheorem}}\label{sec:prooflowerbound}
We use the standard Fano's method and Le Cam's method~\citep{yu1997assouad,wainwright_2019} to establish the minimax lower bound by identifying a subset of distributions within the considered distribution family $\m P^\ast(d,D,\alpha,\beta,L^\ast)$ that are statistically hard to distinguish.
In a nutshell, the term $n^{-\frac{\gamma\beta}{d}}$ in the lower bound is obtained by fixing a smooth distribution $\nu_0$ on $\mathbb{R}^d$ and consider a class of $\beta $-smooth generative maps $G$ whose pushforward measures $G_\# \nu_0$ constitute the candidate ``hardest'' distribution subset. Since the generative map $G$ determines the position and shape of the supporting submanifold, the term $n^{-\frac{\gamma\beta}{d}}$ reflects the statistical hardness of estimating an unknown $\beta$-smooth submanifold.  
In contrast, the term $n^{-\frac{\alpha+\gamma}{2\alpha+d}}\vee n^{-\frac{1}{2}}$ in the lower bound is obtained by fixing the submanifold to be a $d$-dimensional sphere $\m S^d$ (or any other smooth compact submanifold) and consider those distributions whose probability density functions $f$ relative to the volume measure of $\m S^d$ are $\alpha$-smooth functions on the manifold, or $f\in C^\alpha(\m S^d)$. Therefore, the term  $n^{-\frac{\alpha+\gamma}{2\alpha+d}}\vee n^{-\frac{1}{2}}$ reflects the statistical hardness of estimating an unknown $\alpha$-smooth density as if the submanifold is known.

\subsubsection{Lower bound of $n^{-\frac{\gamma\beta}{d}}$}\label{ratelower1}
Under the assumption that $\beta>1$, we only need to consider $\gamma\in[0,1)$. To see this, we show that otherwise $n^{-\frac{\gamma\beta}{d}}$ is always dominated by the other two terms $n^{-\frac{\alpha+\gamma}{2\alpha+d}}\vee n^{-\frac{1}{2}}$ in the minimax lower bound. In fact, if $\gamma \geq d/2$, then by $\beta>1$, we have $n^{-\frac{\gamma\beta}{d}}\leq n^{-\frac{1}{2}}$; if $1\leq \gamma<d/2$, then using $\alpha\leq \beta-1$ we obtain the following sequence of inequalities
$$n^{-\frac{\alpha+\gamma}{2\alpha+d}}\geq n^{-\frac{ \beta-1+\gamma}{2(\beta-1)+d}} \geq n^{-\frac{\beta-1+\gamma}{d}}
\geq n^{-\frac{\gamma\beta}{d}},$$
where the last inequality is due to $\gamma\beta-(\beta-1+\gamma)=(\beta-1)(\gamma-1)\geq 0$. Thus we focus on $\gamma\in[0,1)$ below.

As mentioned before, we will construct a subset of distributions that are statistically hard to distinguish in the sense that their KL divergences to one same distribution (their average) is bounded by constant, while their mutual $d_\gamma$ distances are at least $O\big(n^{-\frac{\gamma\beta}{d}}\big)$, so that we may apply the standard reduction argument by reducing the estimation problem into a multiple testing problem, and use the Fano's lemma (see for example, proposition 15.12 of~\cite{wainwright_2019}) to bound the multiple testing error from below. Specifically, let $\m M_0=\mb S_2^d\times \bold{0}_{D-d-1}=\big\{x\in \mb R^D:\,\|x_{1:(d+1)}\|^2=2,\,x_{(d+2):D}=\bold{0}_{D-d-1}\big\}$ denote a $d$-dimensional sphere with radius $\sqrt{2}$ embedded in $\mb R^D$. Let $\mu_0$ be the uniform distribution over $\m M_0$, that is, the distribution whose probability density function relative to the volume measure of $\m M_0$ is constant, say $C^{-1}$. Our constructed subset of ``hardest'' distributions is obtained by perturbing the support of $\mu_0$ via adding small bumps, as summarized in the following lemma. 

\begin{lemma}[Hardest instances via manifold perturbation]\label{Lemma:MP} Assume $\gamma<1$ and $\beta>1$. There exist $H$ distributions $\{\mu_h\}_{h=1}^H\subset\mathcal{P}^{\ast}(d,D,\alpha,\beta,L^\ast)$  with $\log H \geq 4\,n$ based on perturbing $\mu_0$ such that:
\begin{enumerate}[topsep=2pt,itemsep=0ex]
\item $D_{\rm KL}(\mu_h\,||\, \bar\mu) \leq \log 2$ holds for each $h\in[H]$, where $\bar\mu=H^{-1} \sum_{h=1}^H \mu_h$ is the averaged distribution;
\item there exists a constant $c$ only depending on $(d,\gamma)$ such that for any distinct pair $h,\ell\in[H]$, $d_\gamma(\mu_h,\,\mu_\ell) \geq c n^{-\frac{\gamma\beta}{d}}$.
\end{enumerate}
\end{lemma}

\noindent Note that since the $\mu_h$'s have non-overlapping supports, it is inevitable to consider the KL divergence between $\mu_h$ and $\bar \mu$ in property 1 of the lemma --- $\mu_h$ is absolutely continuous with respect to $\bar\mu$ while the pairwise KL diverges, i.e.~$D_{\rm KL}(\mu_h\,||\,\mu_\ell) =\infty$ for $h\neq \ell$. Fortunately, this issue will not affect the application of Fano's lemma.
A proof of this lemma is provided in Section~\ref{Sec:proof_Lemma_MP} in the supplement. The construction of $\mu_h$ is based on gluing two distributions together (any compact submanifold requires at least two patches to cover): the first part is a measure over a perturbed manifold of the upper area $\wt{\m M}_{0} =\big\{x\in \mb{R}^D:\, x_{1:d}\in \mb B_1^d,\, x_{d+1}=\sqrt{2-\|x_{1:d}\|^2},\, x_{(d+2):D}=\bold{0}_{D-d-1}\big\}$ of the sphere $\m M_0$; and the second part is the restriction of $\mu_0$ onto the remaining spherical cap $\m M_0\setminus \wt{\m M}_{0}$ (see Figure~\ref{figure:manifold_support} for an illustration). 
The measure in the first part is constructed via generative modeling where
we fix a smooth distribution $\nu_0$ over the $d$-dimensional unit ball $\mb B^d_1$ and construct the measure as the pushforward measure through a generative map. The generative map is constructed by adding small bumps to the fixed generative map $G_0:\,\mb R^d\to \mb R^D$, $z\mapsto \big(z,\sqrt{2-\|z\|_2^2}\,,\,\bold{0}_{D-d-1}\big)$ with $[G_0]_\# \nu_0$ being the uniform distribution over $\wt{\m M}_{0}$, or the restriction of $\mu_0$ over $\wt{\m M}_{0}$.
Property 1 in Lemma~\ref{Lemma:MP} can be satisfied by controlling the size of the bumps; and property 2 can be satisfied by maximally spreading the bumps to different locations on $\wt{\m M}_{0}$.

 \begin{figure}[h]
 \centering  
 \subfigure[Unperturbed manifold]{
  \includegraphics[trim={0 1cm 0 1cm}, clip, width=0.48\textwidth]{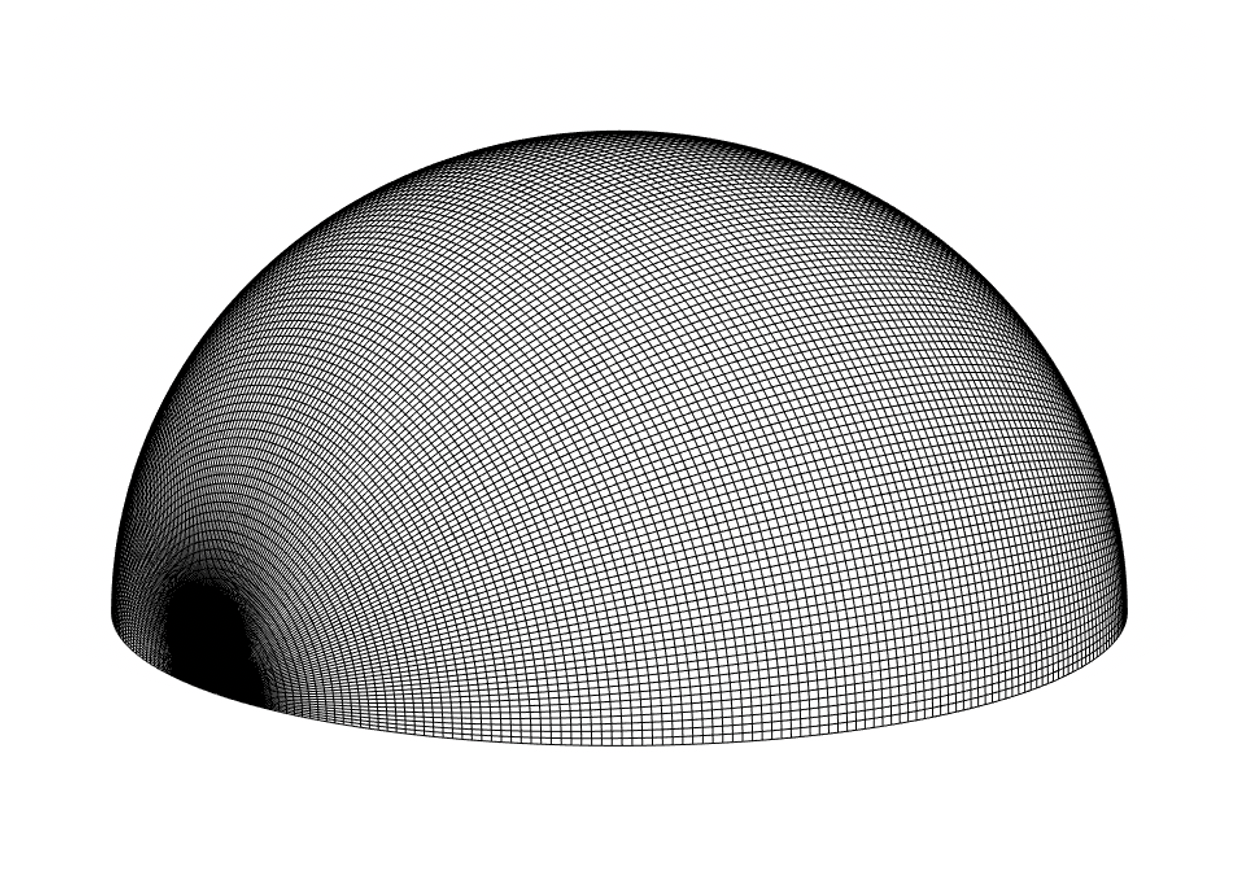}}
 \subfigure[An example of perturbed manifold]{
  \includegraphics[trim={0 0.8cm 0 1cm}, clip, width=0.48\textwidth]{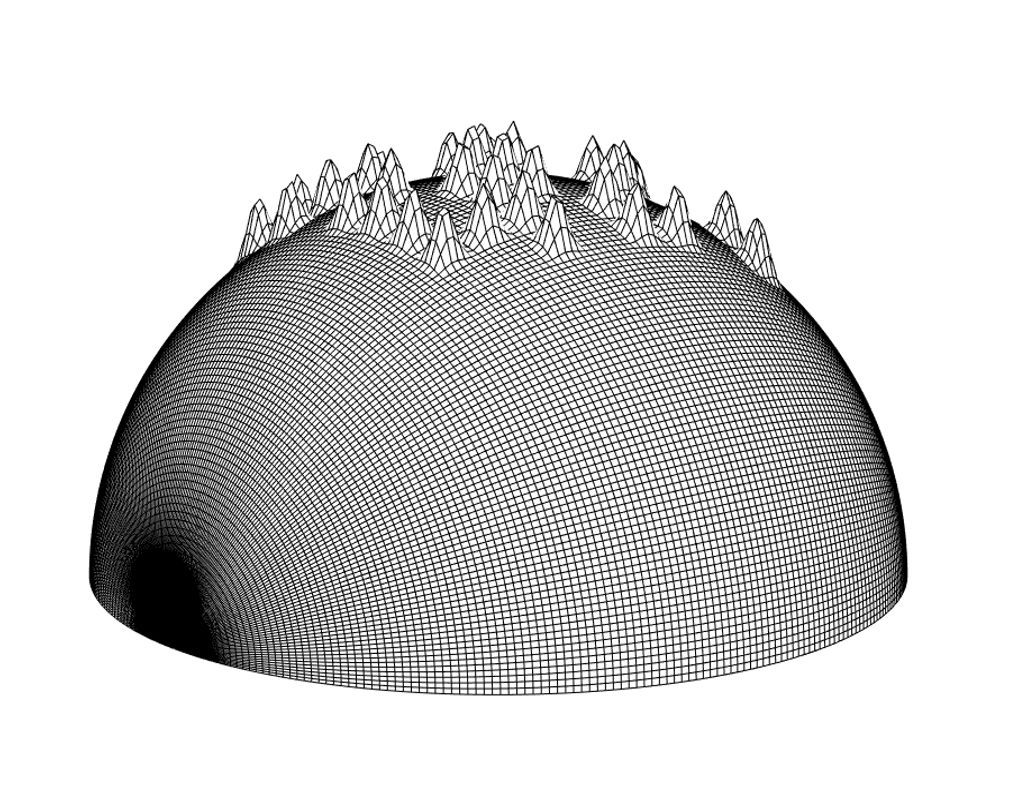}}
  \caption{Upper half of the unperturbed and perturbed manifolds ($d=2$) are shown; and the lower hemispheres are not displayed in the figures.}
 \label{figure:manifold_support}
 \end{figure}
Return to the proof of the lower bound. 
We apply part 1 of Lemma~\ref{Lemma:MP} and Fano's lemma (proposition 15.12 of~\cite{wainwright_2019}) to obtain that for any estimator $\wh h$ based on i.i.d.~observations $X_{1:n}$ from $\mu_h$, the multiple testing error probability satisfies
\begin{align*}
    \mb P(\wh h \neq h) \geq 1-\frac{\log 2+\frac{n}{H} \sum_{h=1}^{H}D_{\rm KL}(\mu_h\,||\,\bar{\mu})}{\log H} \geq \frac{1}{2}.
\end{align*}
By using part 2 of Lemma~\ref{Lemma:MP}, we further obtain
\begin{align*}
    \underset{\wh h}{\inf} \,  \underset{h\in H}{\sup}\, \mathbb{E} \big[d_{\gamma} (\mu_{\wh h}, \mu_h)\big]\geq  \frac{1}{2} \underset{h,\ell\in [H]\atop h\neq \ell} {\inf} d_{\gamma}(\mu_{h},\mu_{\ell}) \geq \frac{c}{2}\,n^{-\frac{\gamma\beta}{d}}.
\end{align*}
Finally, since $d_{\gamma}$ satisfies the triangle inequality, by a standard reduction argument~\citep{yang1999information} from estimation to multiple testing, we have 
\begin{align*}
    \underset{\wh{\mu}}{\inf}\underset{\mu\in \mathcal{P}^{\ast}} {\sup} \,\mathbb{E} \big[d_{\gamma}(\wh{\mu}, \mu)\big]
    \geq \frac{1}{2}\,\underset{\wh h}{\inf} \,  \underset{h\in H}{\sup}\, \mathbb{E} \big[d_{\gamma} (\mu_{\wh h}, \mu_h)\big] \geq \frac{c}{4}\,n^{-\frac{\gamma\beta}{d}}.
\end{align*}

\subsection{Lower bound of $n^{-\frac{\alpha+\gamma}{2\alpha+d}}$}\label{ratelower2}
\cite{liang2020generative} prove a lower bound $n^{-\frac{\alpha+\gamma}{2\alpha+d}}\vee n^{-\frac{1}{2}}$ for the minimax rate of $\alpha$-smooth density estimation on $[0,1]^d$ under the adversarial loss $d_\gamma$. 
Our proof of lower bounds $n^{-\frac{\alpha+\gamma}{2\alpha+d}}$ and $n^{-\frac{1}{2}}$ in this and next subsections  are technically similar to that in~\citep{liang2020generative} for proving nonparametric density estimation lower bounds based on constructing a subset of bumpy functions that are statistically hard to distinguish with respect to the concerned distance metric~\citep{Tsybakov2009}, which is $d_\gamma$ in our context. However, although the lower bounds appear the same, there is a non-trivial extension in our proof --- we need to construct singular distributions supported on a compact (therefore boundariless) manifold instead of the ambient space which is flat. This requires us to use a carefully constructed generative map to: 1.~pushforward bumpy distributions from $\mb R^d$ to $\mb R^D$ and to verify the resulting distributions to admit smooth density functions with respect to the volume measure of the manifold; 2.~lift the discriminator in $\mb R^d$ discriminating the bumpy functions to one in the ambient space $\mb R^D$ discriminating singular distributions supporting on the manifold; both of which complicate the proof. The lemma below summarizes the constructed subset of ``hardest'' distributions obtained by perturbing the distribution $\mu_0$ via adding small bumps on the submanifold $\m M_0$ (both defined in Section~\ref{sec:prooflowerbound}).

\begin{lemma}[Hardest instances via density perturbation]\label{Lemma:DP} Assume $\beta>1$. Then for any constant $b>0$, there exist $H'$ distributions $\{\mu_h\}_{h=1}^{H'}\subset\mathcal{P}^{\ast}(d,D,\alpha,\beta,L^\ast)$  with $\log H' \geq b^d n^{\frac{d}{2\alpha+d}}$ based on perturbing $\mu_0$ such that for constants $(c_1,c_2)$ only depending on $d$:
\begin{enumerate}[topsep=2pt,itemsep=0ex]
\item $D_{\rm KL}(\mu_h\,||\, \mu_\ell) \leq c_1 \,b^{-2\alpha} n^{-\frac{2\alpha}{2\alpha+d}}$ holds for any distinct pair $h,\ell\in[H']$;
\item $d_\gamma(\mu_h,\,\mu_\ell) \geq c_2 \,b^{-(\alpha+\gamma+d)} \, n^{-\frac{\alpha+\gamma}{2\alpha + d}}$ holds for any distinct pair $h,\ell\in[H']$.
\end{enumerate}
\end{lemma}

\noindent A proof of this lemma is provided in Section~\ref{Sec:Proof_lemma_dp} in the supplement. The rest proof in this subsection is similar to that in Section~\ref{ratelower1}. We apply above lemma and Fano's lemma to obtain
\begin{align*}
     \underset{\wh{\mu}}{\inf}\underset{\mu\in \mathcal{P}^{\ast}} {\sup} \,\mathbb{E} \big[d_{\gamma}(\wh{\mu}, \mu)\big]\geq &\,  \frac{1}{2}\,\underset{h,\ell\in [H]\atop h\neq \ell} {\inf} d_{\gamma}(\mu_{h},\mu_{\ell}) \cdot\bigg(1-\frac{\log 2+\frac{n}{H^2} \sum_{h,\ell =1}^{H}D_{\rm KL}(\mu_h\,||\,\mu_\ell)}{\log H}\bigg)\\
    \geq &\,\frac{c_2}{2}\, \,b^{-(\alpha+\gamma+d)} \, n^{-\frac{\alpha+\gamma}{2\alpha + d}} \cdot \bigg(1-
    \frac{\log 2+c_1\, b^{-2\alpha} n^{\frac{d}{2\alpha+d}}}{b^d n^{\frac{d}{2\alpha+d}}}\bigg) \geq c'\, n^{-\frac{\alpha+\gamma}{2\alpha + d}},
\end{align*}
where the last step holds by choosing a sufficiently large constant $b$.

\subsection{Lower bound of $n^{-\frac{1}{2}}$}\label{ratelower3}
The $n^{-\frac{1}{2}}$ lower bound can be obtained by Le Cam’s method of reducing the estimation problem into a two-point hypothesis testing problem~\citep{yu1997assouad,wainwright_2019}. Our proof is based on adapting the proof of~\cite{liang2020generative} for density estimation on $[0,1]^d$ to distribution estimation on the previously defined $d$-dimensional sphere $\m M_0$ embedded in $\mb R^D$.

The proof relies on the existence of two distributions $\mu_0$ (uniform distribution) and $\mu_1$ supported on $\m M_0$ with the following properties. For two distributions $\nu$ and $\mu$ such that $\nu$ is absolutely continuous with respect to $\mu$, the chi-squared distance from $\nu$ to $\mu$ is defined as $d_{\chi^2}(\nu,\mu) =\int\big(\frac{\dd \nu}{\dd\mu}-1\big)^2 \, \dd \mu$.
\begin{lemma}[Perturbation of uniform distribution]\label{Lemma:LM} There exists two distributions $\mu_0$ and $\mu_1$, both belonging to $\mathcal{P}^{\ast}(d,D,\alpha,\beta,L^\ast)$ such that: 1.~$\mu_1$ is absolutely continuous with respect to $\mu_0$; 2.~$d_{\chi^2}(\mu_1,\mu_0) \leq \frac{1}{n}$; 3.~$d_\gamma(\mu_1,\mu_0) \geq \frac{c}{\sqrt{n}}$, where $c_1$ is a constant only depending on $(d,\alpha,\gamma)$.
\end{lemma}

\noindent A proof of this lemma is provided in Section~\ref{Sec:Proof_lemma_LM} in the supplement. For a distribution $\mu$, let $\mu^{\otimes n}$ denote its $n$-fold self-product.
From the lemma and the tensorization property of $d_{\chi^2}$, we have
\begin{align*}
    &d_{\chi^2} \big(\mu_1^{\otimes n}, \mu_0^{\otimes n}\big)=\big(1+d_{\chi^2} (\mu_1, \mu_0)\big)^n-1
    \leq \Big(1+\frac{1}{n}\Big)^n-1 \leq e-1.
\end{align*}
Therefore, by Pinsker's inequality and $D_{\rm KL}(\mu\,||\,\nu)\leq d_{\chi^2}(\mu,\nu)$, we obtain that their total variation distance satisfies $d_{\rm TV} \big(\mu_1^{\otimes n}, \mu_0^{\otimes n}\big)\leq \sqrt{\frac{e-1}{2}}$.
Finally, by Le Cam's bound (see for example, Lemma 15.9 of~\cite{wainwright_2019}) we obtain
\begin{equation*}
\underset{\wh{\mu}}{\inf}\, \underset{\mu\in \mathcal{P}^{\ast}} {\sup}\,\mathbb{E} \big[d_{\gamma}(\wh{\mu}, \mu)\big] \geq    d_{\gamma}(\mu_1, \mu_0) \,\Big(1-\sqrt{\frac{e-1}{2}}\Big) \geq \Big(1-\sqrt{\frac{e-1}{2}}\Big) \,\frac{c}{\sqrt{n}},
\end{equation*}
where the last step is due to Lemma~\ref{Lemma:LM}.

\subsection{Proof of minimax upper bound}\label{sec:proofupperbound}
We provide in this section the proof of the more general upper bound in Theorem~\ref{upperboundgenerative}.  Lemma~\ref{lemma:gen_model_as_mix} in Appendix~\ref{Sec:gen_model_as_mix} shows that $\m S^\ast$ is included in the approximation family $\m S^{\rm ap}$. Recall that from the basic inequality~\eqref{eqn:basic_ineq}, it suffices to show that
\begin{align}\label{Eqn:EP_newbound}
    \sup_{f\in C_1^\gamma(\mb R^D)}\Big| \mb E_{\mu^\ast}[f(X)] - \wh{ \m J}(f)\Big| \leq C \, \Big(\frac{n}{\log n}\Big)^{-\frac{1}{2}}\vee \Big(\frac{n}{\log n}\Big)^{-\frac{ \alpha+\gamma}{2\alpha+d}}\vee \Big(\frac{n}{\log n}\Big)^{-\frac{\gamma\beta}{d}},
\end{align}
where the regularized surrogate $\wh{ \m J}(f)=\wh{\m J}_l(f)+\wh{\m J}_h(f)+\wh{\m J}_s(f)$ for approximating $\mb E_{\mu^\ast}[f(X)]$ is composed of three terms:
\begin{align}
    \wh{\m J}_l(f) = \sum_{m=1}^M \wh{\m J}_{m,l}(f), \quad \wh{\m J}_h(f) = \sum_{m=1}^M \wh{\m J}_{m,h}(f) \quad\mbox{and}\quad \wh{\m J}_s(f) = \sum_{m=1}^M \wh{\m J}_{m,s}(f),
\end{align}
where for each $m\in[M]$ corresponding to the index in the partition of unity, the triplet $\big(\wh{\m J}_{m,l}(f),\wh{\m J}_{m,h}(f),\wh{\m J}_{m,s}(f)\big)$ is defined in~\eqref{eqn:surrogate_m} when $\wh{p}_m>\sqrt{\frac{\log n}{n}}$ and otherwise we set $\wh{\m J}_{m,l}(f)=\wh{\m J}_{m,h}(f)=\wh{\m J}_{m,s}(f)=0$. In particular, $\wh{\m J}_l(f)$ estimates the expectation of $\Pi_J f$ that collects the low frequency components in the wavelet expansion of $f$; $\wh{\m J}_h(f)$ estimates the expectation of $\Pi_J^\perp f$ that collects the high frequency component; and $\wh{\m J}_s(f)$ corresponds to a high-order smoothness correction due to the submanifold estimation error from the local coordinate map (and its inverse) estimator $\big(\wh G_{[m]},\wh Q_{[m]}\big)$ for $m\in[M]$, defined in~\eqref
{estimator1}. We may similarly decompose the target functional into three terms as $\mb E_{\mu^\ast}[f(X)]  = \m J_l(f) + \m J_h(f)+ \m J_s(f)$, where
\begin{align*}
    \m J_l(f)  &= \sum_{m=1}^M \m J_{m,l}(f) \ \ \mbox{with} \ \ \m J_{m,l}(f)= \mb E_{\mu^\ast}\big[\Pi_J f\big(\wh G_{[m]} \circ \wh Q_{[m]}(X)\big)\cdot \rho_m(X)],\\
     \m J_h(f)  &= \sum_{m=1}^M \m J_{m,h}(f)\ \ \mbox{with} \ \ \m J_{m,h}(f)= \mb E_{\mu^\ast}\big[\Pi^\perp_J f\big(\wh G_{[m]} \circ \wh Q_{[m]}(X)\big)\cdot \rho_m(X)],\\
      \m J_s(f)  &= \sum_{m=1}^M \m J_{m,l}(f)\ \ \mbox{with} \ \ \m J_{m,l}(f)= \mb E_{\mu^\ast}\big[f(X)\cdot \rho_m(X)]-\mb E_{\mu^\ast}\big[f\big(\wh G_{[m]} \circ \wh Q_{[m]}(X)\big)\cdot \rho_m(X)].
\end{align*}
Let $\mb M=\{m\in [M]\,:\,\mathbb{P}_{\mu^{\ast}}(X\in S^{\dagger}_m)\geq \frac{1}{2}\sqrt{\frac{\log n}{n}}\}$ and recall $\wh{\mb M}=\{m\in [M]\,:\,\wh p_m\geq \sqrt{\frac{\log n}{n}}\}$, in the following proofs, we only consider  $m\in \mb M\cap \wh{\mb M}$. In fact, by applying Bernstein's inequality for a binomial random variable and a union bound argument, for any constant $c$, there exists a constant $n_0$ such that when $n\geq n_0$, it holds with probability at least $1-n^{-c}$ that for any $m\in [M]$ such that $\mathbb{P}_{\mu^{\ast}}(X\in S^{\dagger}_m)< \frac{1}{2}\sqrt{\frac{\log n}{n}}$ or $\wh{p}_m< \sqrt{\frac{\log n}{n}}$,
\begin{equation*} 
 \begin{aligned}
  \underset{f\in C^{\gamma}_1(\mathbb{R}^D)}{\sup} \Big|\mb E_{\mu^\ast}[f(X)\cdot\rho_m(X)]- \wh{\m J}_m(f)\Big| = \underset{f\in C^{\gamma}_1(\mathbb{R}^D)}{\sup} \big|\mb E_{\mu^\ast}[f(X)\cdot\rho_m(X)]\big|\leq 2\,\sqrt{\frac{\log n}{n}},
   \end{aligned}
 \end{equation*}
see Appendix~\ref{Sec:more_proofs_upperbound} for further detail. The following three lemmas show that $\wh{\m J}_{m,l}$, $\wh{\m J}_{m,h}$ and $\wh{\m J}_{m,s}$ are good estimators for $\m J_{m,l}$, $\m J_{m,h}$ and $\m J_{m,s}$, respectively, for each $m\in \mb M\cap \wh{\mb M}$. 
\begin{lemma}[Low frequency components]\label{lemma:low_freq}
  With probability at least $1-n^{-c}$, for any $m\in \mb M\cap \wh{\mb M}$, the functional $\wh{\m J}_{m,l}:\, C^\gamma(\mb R^D) \to \mb R$ 
  defined in~\eqref{eqn:low_freq} 
  satisfies
  \begin{align*}
      \sup_{f\in C_1^\gamma(\mb R^D)} \big|\wh{\m J}_{m,l}(f) - {\m J}_{m,l}(f)\big| \leq C\,\sqrt{\frac{\log n}{n}} +  C\, \Big(\frac{\log n}{n}\Big)^{\frac{\alpha+\gamma}{2\alpha+d}}.
  \end{align*}
\end{lemma}
\begin{lemma}[High frequency components]\label{lemma:high_freq}
  With probability at least $1-n^{-c}$, for any $m\in \mb M\cap \wh{\mb M}$, the functional $\wh{\m J}_{m,h}:\, C^\gamma(\mb R^D) \to \mb R$ defined in~\eqref{eqn:high_freq} satisfies
  \begin{align*}
      \sup_{f\in C_1^\gamma(\mb R^D)} \big|\wh{\m J}_{m,h}(f) - {\m J}_{m,h}(f)\big| \leq  C\, \Big(\frac{\log n}{n}\Big)^{\frac{\alpha+\gamma}{2\alpha+d}}.
  \end{align*}
\end{lemma}
\begin{lemma}[Smoothness correction]\label{lemma:smooth_corr}
  With probability at least $1-n^{-c}$, for any $m\in \mb M\cap \wh{\mb M}$, the functional $\wh{\m J}_{m,s}:\, C^\gamma(\mb R^D) \to \mb R$ defined in~\eqref{eqn:smooth_corr} satisfies
  \begin{align*}
      \sup_{f\in C_1^\gamma(\mb R^D)} \big|\wh{\m J}_{m,s}(f) - {\m J}_{m,s}(f)\big| \leq   C\,\sqrt{\frac{\log n}{n}} + C\, \Big(\frac{\log n}{n}\Big)^{\frac{\gamma\beta}{d}}+ C\,\Big(\frac{\log n}{n}\Big)^{\frac{\gamma+\beta-1}{d}}.
  \end{align*}
\end{lemma}
\noindent Proofs of these lemmas are provided in Sections~\ref{sec:proof_low_freq},~\ref{sec:proof_high_freq} and~\ref{sec:proof_smooth_corr} in the supplementary material. It is worthwhile mentioning that an important intermediate result used in the proof of Lemma~\ref{lemma:smooth_corr} is the following lemma, which characterizes the estimation error of $\big(\wh G_{[m]},\wh Q_{[m]}\big)$ defined in step 1 of submanifold estimation in Section~\ref{sec:optimal_proce}. Its proof, which is provided in Section~\ref{sec:proof_mani_est} in the supplementary material, is quite technical and involved, and uses many techniques from empirical process theory.
\begin{lemma}[Submanifold estimation]\label{lemma:mani_est}
  For any fixed constant $\eta>0$, it holds with probability at least $1-n^{-c}$ that,
  \begin{align*}
      \mb E_{\mu^\ast} \Big[\big\|X- \wh G_{[m]}\circ \wh Q_{[m]}(X)\big\|_2^\eta \, \rho_m(X)\Big] \leq C\, \Big(\frac{\log n}{n}\Big)\vee \Big(\frac{\log n}{n}\Big)^{\frac{\eta\beta}{d}},\quad \forall m\in \mb M,
  \end{align*}
  where the expectation is taken with respect to the randomness in $X$ (not the randomness in $\wh G_{[m]}$ and $\wh Q_{[m]}$).
\end{lemma}
\noindent 
By combining Lemmas~\ref{lemma:low_freq},~\ref{lemma:high_freq} and~\ref{lemma:smooth_corr}, we obtain
\begin{align*}
    \sup_{f\in C_1^\gamma(\mb R^D)}\Big| \mb E_{\mu^\ast}[f(X)] - \wh{ \m J}(f)\Big| \leq C\, M\, \Big(\frac{n}{\log n}\Big)^{-\frac{1}{2}}\vee \Big(\frac{n}{\log n}\Big)^{-\frac{ \alpha+\gamma}{2\alpha+d}}\vee \Big(\frac{n}{\log n}\Big)^{-\frac{\gamma\beta}{d}} \vee \Big(\frac{n}{\log n}\Big)^{-\frac{\gamma+\beta-1}{d}}.
\end{align*}
Finally,  under the assumption that $\beta\geq\alpha-1$, the last term $\big(\frac{n}{\log n}\big)^{-\frac{\gamma+\beta-1}{d}}$ above is always dominated by the second term $\big(\frac{n}{\log n}\big)^{-\frac{ \alpha+\gamma}{2\alpha+d}}$, which completes the proof of the claimed inequality~\eqref{Eqn:EP_newbound}.

\section{Discussion} 
In this paper, we studied the minimax rate of distribution estimation on unknown submanifold under adversarial losses, covering cases where the manifold, the density, and the discriminator class have various H\"{o}lder regularities. In conclusion, the minimax rate shows that the curse of dimensionality can be overcome for data with low intrinsic dimension, smooth density and regular support, which partly explains the empirical successes of generative model based approaches for generating realistic objects in real applications. Apart from the H\"{o}lder class, some other function spaces, such as Sobolev class and reproducing kernel Hilbert space may also be considered for the discriminator class when defining the adversarial loss, which we leave for future research. Moreover, the rate-optimal procedure developed in this study is mainly for the theoretical purpose of proving a minimax upper bound, and a modification towards it to make it computationally feasible may also be left to our future work.


\bibliography{references}{}
\bibliographystyle{abbrvnat}

\newpage
\appendix
\begin{center}
{\bf\Large Appendix}
\end{center}
\textbf{Notations}: We adopt the notations in the manuscript, and further introduce the following additional notations for the technical proofs.  For two symmetric matrices $A$ and $B$, we use $A\preccurlyeq B$ to mean that $B-A$ is a positive semi-definite matrix. We use $N(\mathcal{F},\,\widetilde{d}, \,\epsilon)$ to denote the $\epsilon$-covering number of function space $\mathcal{F}$ with respect to pseudo-metric $\widetilde{d}$. Throughout, $C$, $c$, $C_0$, $c_0$, $C_1$, $c_1$, $C_2$, $c_2$,\ldots are generically used to denote positive constants whose values might change from one line to another, but are independent from everything else.

\section{Wavelet and Besov Function Space}\label{sec:Wavelet_review}
In this section, we give a brief introduction to the wavelet and Besov function Space, and then define the smoothness regularized empirical distribution $\wt\nu_{[m],\wh{Q}_{[m]}}$ used in Step 2 of the construction of the  minimax-optimal estimator $\wh{\mu}$ in Section~\ref{sec:optimal_proce} based on wavelet expansion.

Let $\phi_{\mf M}\in C^{\zeta}(\mb R)$ and $\phi_{\mf F}\in C^{\zeta}(\mb R)$ be a compactly supported wavelet and scaling function, respectively, for example Daubechies wavelets~\citep{doi:10.1080/03610926.2015.1019144, hutter2020minimax}. This implies that 
\begin{equation*}
\Psi_{k}^j=\left\{
\begin{array}{ll}
   \psi_{\mf F} (x-k) & j=0, k\in \mb Z,  \\
    2^{(j-1)/2} \psi_{\mf M}(2^{j-1}x-k), &  j\in \mb N^{+}, k\in \mb Z,
\end{array}
\right.
\end{equation*}
is an orthonormal basis of $\m L^2(\mb R)$, where we use $\m L^2$ to denote the set of square integrable functions. To obtain a basis of $\m L^2(\mb R^d)$ for an integer $d>1$, set 
\begin{equation*}
     \mf G = \{\mf F,\,\mf M\}^d\setminus \{(\mf F,\ldots,\mf F)\}.
\end{equation*}
Then for any multi-index $k\in \mb Z^d$, the level zero basis $\phi_k^{[d]}$ is obtained by translating the $d$-fold tensor product $\phi_{\mf F}^{\otimes d}$ by $k$ as $\phi_{k}^{[d]}(x) = \prod_{i=1}^d \phi_{\mf F}(x_i-k_i)$ for $x=(x_1,\ldots,x_d)\in\mb R^d$, and for any $j\geq 1$, the level $j$ basis $\big\{\psi_{ljk}^{[d]}:\, l\in[2^d-1]\big\}$ with translation $k$ is any ordering of the following $2^d-1$ functions,
\begin{align*}
    \psi_{k}^{j,g}(x)=2^{\frac{d(j-1)}{2}} \,\prod_{i=1}^d \phi_{g_i}\big(2^{j-1}x_i - k_i\big), \quad \forall g\in \mf G.
\end{align*}
When no ambiguity arises, we suppress the superscript $[d]$ in $\phi_{k}^{[d]}(x)$ and $\psi_{ljk}^{[d]}(x)$ for $x\in \mb R^d$. This gives the orthornormal basis  
\begin{equation*}
\Psi_{k}^{j,l}=\left\{
\begin{array}{ll}
   \phi_k(x), & j=0,l=0, k\in \mb Z^d,  \\
    \psi_{ljk}(x), &  j\in \mb N^{+},l\in [2^d-1], k\in \mb Z^d.
\end{array}
\right.
\end{equation*}
Then  let $1\leq p,q\leq \infty$, $s\geq 0$ and let the regularity of the above wavelets satisfy $\zeta> s\vee \big(\frac{2d}{p}+\frac{d}{2}-s)$.   We are then ready to define the Besov space $B^s_{p,q}(\mb R^d)$ consists of functions $f$ that admits the wavelet expansion
\begin{equation*}
  f(x)= \sum_{k\in \mathbb{Z}^d} b_k\, \phi_k(x)+\sum_{l=1}^{2^d-1} \sum_{j=1}^{\infty}\sum_{k\in \mathbb{Z}^d} f_{ljk}\, \psi_{ljk}(x).
\end{equation*}
 and equipped with the norm 
 \begin{equation*}
     \|f\|_{B^s_{p,q}}:=\bigg[\Big(\sum_{k\in Z^d} |b_k|^p\Big)^{\frac{q}{p}}+\sum_{j=1}^{\infty} 2^{jq(s+\frac{d}{2}-\frac{d}{p})}\sum_{l\in [2^d-1]}\Big(\sum_{k\in Z^d}|f_{ljk}|^p\Big)^{\frac{q}{p}}\bigg]^{\frac{1}{q}}<\infty.
 \end{equation*}
 Moreover, for any positive integer $J$, we use $\Pi_J f$ to denote the projection of any $f\in B^s_{p,q}(\mb R^d)$ onto the first scale $J$ wavelet coefficients, given by 
 \begin{equation*}
     \Pi_J f(x)= \sum_{k\in \mathbb{Z}^d} b_k\, \phi_k(x)+\sum_{l=1}^{2^d-1} \sum_{j=1}^{J}\sum_{k\in \mathbb{Z}^d} f_{ljk}\, \psi_{ljk}(x)
 \end{equation*}
 and $\Pi_J^\perp f = f-\Pi_J f$. 
 
 The following Theorem collects the relationship between the Besov space and H\"{o}lder space.
\begin{theorem}
(Theorem 1.122 of~\cite{Triebel2006} and Proposition 4.3.30 of~\cite{gine_nickl_2015}) Let $\alpha>0$, if $\alpha$ is not integer, then
\begin{equation*}
    C^{\alpha}(\mb R^d)=B^{\alpha}_{\infty,\infty}(\mb R^d);
\end{equation*}
if $\alpha$ is integer, then
\begin{equation*}
    B^{\alpha}_{1,\infty}(\mb R^d)\subset C^{\alpha}(\mb R^d)\subset B^{\alpha}_{\infty,\infty}(\mb R^d).
\end{equation*}
\end{theorem}
\paragraph{Smoothness regularized empirical distribution based on wavelet expansion:}
Suppose the target density $\num(z)$ of $(\wh{Q}_{[m]})_{\#}(\rho_m\mu^\ast)$ is $\alpha$-smooth (which is true with high probability, see Lemma~\ref{lemma:density_regularity} for further detail), then it has the following wavelet expansion:
\begin{align*}
\num(z)&=\sum_{k \in  \mathbb{S}}a^{\wh{Q}_{[m]}}_k \,\phi_k(z) + \sum_{l=1}^{2^{d}-1}\sum_{j=0}^{+\infty}\sum_{k\in \mathbb{S}_{lj}}\theta_{ljk}^{\wh Q_{[m]}}\,  \psi_{ljk}(z),\quad\mbox{with}\\
 \mathbb{S}&=\big\{k\in\mathbb{Z}^d:\,{\rm supp}(\phi_k)\cap[-L,L]^d\neq \emptyset\big\};\\
\mathbb{S}_{lj}&=\big\{k \in \mathbb{Z}^d: \,{\rm supp}(\psi_{ljk})\cap [-L,L]^d\neq \emptyset\big\},
 \end{align*}
 where
 \begin{equation*}
\begin{aligned}
{a}^{\wh{Q}_{[m]}}_{k}&=\mb E_{\mu^\ast}\big[\phi_k(\wh Q_\sm(X))\cdot \rho_m(X)\big];\\
{\theta}^{\wh{Q}_{[m]}}_{ljk} &= \mb E_{\mu^\ast}\big[\psi_{ljk}(\wh Q_\sm(X))\cdot \rho_m(X)\big].
 \end{aligned}
\end{equation*}
 Here, the expectation is taken with respect to $X\sim \mu^\ast$ (not the randomness in $\wh Q_\sm$). The smoothness regularized estimator $\widetilde \nu_{\sm,\wh Q_{\sm}}$ corresponds to a truncated empirical version of $\nu_{\sm,\wh Q_{\sm}}^\ast$ by truncating the expansion at a finite level $J$ such that $2^J$ is the largest integer not exceeding $(\frac{n}{\log n})^{\frac{1}{2\alpha+d}}$, and replacing the wavelet coefficients with their sample averages,
 \begin{equation} 
 \widetilde{\nu}_{[m],\wh Q_\sm}(y)=\sum_{k \in  \mathbb{S}}\widetilde{a}^{\wh{Q}_{[m]}}_k \phi_k(y) +\sum_{l=1}^{2^{d}-1}\sum_{j=0}^J\sum_{k\in \mathbb{S}_{lj}}\widetilde{\theta}_{ljk}^{\wh{Q}_{[m]}} \psi_{ljk}(y),\\
 \end{equation}
 where
\begin{equation*}
\begin{aligned}
\widetilde{a}^{\wh{Q}_{[m]}}_{k}&=\frac{1}{|I_2|}\sum_{i\in I_2} \phi_k(\wh Q_\sm(X_i))\cdot \rho_m(X_i);\\
\widetilde{\theta}^{\wh{Q}_{[m]}}_{ljk} &= \frac{1}{|I_2|}\sum_{i\in I_2}  \psi_{ljk}(\wh Q_\sm(X_i))\cdot \rho_m(X_i).
 \end{aligned}
\end{equation*}
\begin{remark}\label{rem:KDE} \normalfont
Another possible choice of the smoothness regularized empirical distribution $\wt \nu_{\sm,\wh{Q}_\sm}$ is the kernel density estimator (KDE) for $(\wh{Q}_{[m]})_{\#}(\rho_m\mu^\ast)$ that is easier to implement in practice. Specifically, we may choose any kernel function $\bar{k}(x)\in C^{1+ \alpha}_{c_1}(\mathbb{R})$ over $\mb R$ such that: (1)~${\rm supp}(\bar{k}(x))\subset [0,1]$; (2)~$\int \bar{k}(x)\dd x=1$; (3)~for any $j\in \mb N^{+}$ with $ j\leq \lceil  \alpha\rceil$, $\int x^j \bar{k}(x)\, \dd x=0$. Then we define the KDE of $(\wh{Q}_{[m]})_{\#}(\rho_m\mu^\ast)$ as
 \begin{equation*}
    \wt \nu_{\sm,\wh{Q}_\sm}(y)=\frac{1}{h^d|I_2|} \sum_{i\in I_2}\bigg[\Big\{ \prod_{j=1}^d \bar{k}\Big(\frac{\wh{Q}_{\sm,j}(X_i)-y_j}{h}\Big)\Big\}\cdot \rho_m(X_i)\bigg],
  \end{equation*}
where bandwidth parameter $h=\big(\frac{\log n}{n}\big)^{\frac{1}{2 \alpha+d}}$ and $\wh{Q}_{\sm,j}(X_i)$ denotes the $j$-th dimension of $\wh{Q}_{\sm}(X_i)\in\mb R^d$. We show in Appendix~\ref{sec:proof_high_freq} that such a KDE based regularization also leads to the same upper bound as in Lemma~\ref{lemma:high_freq}.
\end{remark}

\section{Generative Model Class of Target Distribution}\label{sec:gen_model_class}

In this section, we consider  generative model classes that include the distribution class $\m P^\ast$ considered in Theorem~\ref{maintheorem} as a special case.  Similar as the approximation family $\m S^{\rm ap}$ and $\m S^{\rm ap}_{\nu_0}$, we consider two (mixture of) generative model classes for the target distribution, $\m S^\ast = \m S^\ast(d,D,\alpha,\beta,\ms O_M,L)$ and $\m S_{\nu_0}^\ast = \m S^\ast_{\nu_0}(d,D,\alpha,\beta,\ms O_M,L)$, where recall that $\ms O_M=\{S_m=\mb B_{r_m}(a_m)^{\circ}\}_{m\in [M]}$ forms a open cover for $\mb B_L^D$. The first generative model class $\m S^\ast$ consists of all probability measures $\mu\in \m P(\mb R^D)$ satisfying:
\begin{enumerate}[topsep=0.5em,itemsep=0.5em,partopsep=0em,parsep=0em]
 \item $\mu$ has a support contained in $\mb B_L^D$. 
\item For any $m\in [M]$, there exists a set $\widetilde{S}_m\supseteq \mb B_{r_m+1/L}(a_m)$ such that
\begin{enumerate}
    \item  there exists a map $G_{[m]}\in C^{\gamma}_L(\mathbb{R}^d; \mathbb{R}^D)$ such that $\mu|_{\widetilde{S}_m} = [G_{[m]}]_{\#} \nu_{[m]}$ for some distribution $\nu_{[m]} \in  \m P(\mb R^d)$ supported on $\mb B_1^d$. Moreover, there exists $Q_{[m]}\in C^{\gamma}_L(\mathbb{R}^D; \mathbb{R}^d)$ such that $Q_{[m]}\circ [G_{[m]}|_{\mb B_1^d}] = {\rm id}_{\mb B_1^d}$, the identity map on $\mb B_1^d$.
    \item $\nu_{[m]}$ is absolutely continuously w.r.t. the Lebesgue measure with density also denoted as $\nu_{[m]}:\mathbb{R}^d\to [0,\infty)$ such that $\nu_{[m]}|_{\mb B_1^d}\in C^{\alpha}_L(\mb B_1^d)$ and there exists $L_0$ depending on $m$ such that $0\leq L_0\leq L$ and $\sup_{z\in \mb B_1^d} \big|\log \frac{\nu_{[m]}(z)}{(1-\|z\|_2)^{L_0}}\big|\leq L$.
    \item Either $G_{[m]}(\mb B_1^d\setminus \mb B^d_{1-\epsilon})\cap \mb B_{r_m}(a_m)=\emptyset$ for some $\epsilon$ satisfying $0<\log\frac{1}{\epsilon} \leq L$ or $\nu_{[m]}\in C^{\alpha}_L(\mathbb{R}^d)$.
\end{enumerate}
 
\end{enumerate}

\begin{remark}\normalfont 
If we choose $L_0=0$ for each $m\in [M]$ in point 2(b) of the above assumption, it recovers the assumption in the distribution family $\m P^\ast$, where the density of $\mu$ is $\alpha$-smooth and bounded from above and below. The point 2(c) guarantees that for any partition of unity $\{\rho_m\}_{m\in [M]}$ subordinate to $\ms O_{M}$, the probability measure of the latent variable reweighted by $\rho_m$, given by $\nu_{[m]}\cdot(\rho_m\circ G_{[m]})$ is $\alpha$-smooth on $\mb R^d$. In particular, the assumption  $G_{[m]}(\mb B_1^d\setminus \mb B^d_{1-\epsilon})\cap \mb B_{r_m}(a_m)=\emptyset$ requires the support of $\mu$ to be   boundaryless; on the other hand, the  assumption $\nu_{[m]}\in C^{\alpha}_L(\mathbb{R}^d)$ allows the manifold to have a boundary while the distribution should smoothly decay to zero around the boundary.
\end{remark}

\begin{remark} \normalfont 
Given a partition of unity  $\{\rho_m\}_{m\in[M]}$ constructed based on the open over $\ms O_{M}=\{\mb B_{r_m}(a_m)^\circ\}_{m\in [M]}$ of $\mb B_L^D$ (see Section~\ref{sec:pou} for an example),  if for any $m\in [M]$, $\widetilde{S}_m$ contains $\mb B_{r_m}(a_m)$, then $\{\rho_m\}_{m\in[M]}$ is also a partition of unity of $\m {M}$ subordinate to $\{\widetilde{S}_m\cap \m{M}\}_{m\in [M]}$. Lemma~\ref{lemma:gen_model_as_mix} in Appendix~\ref{Sec:gen_model_as_mix} shows that the distribution $\mu\in \mathcal{S}^{\ast}$ can be expressed as a mixture of generative model with rejection sampling:  $\mu =\sum_{m=1}^M w_{[m]} \mathcal{A}(G_{[m]},\nu_{[m]},\rho_m)$ where $\{w_{[m]}\}_{m\in [M]} $ are the mixing weights given by $w_{[m]}=\mathbb{E}_{\mu}[\rho_m(X)]$ for $m\in [M]$, and recall $\mathcal{A}(G_{[m]}, \nu_{[m]}, \rho_m)$ is the probability measure induced by the data generating process where $X\sim [G_{[m]}]_{\#}\nu_{[m]}$ is accepted with probability $\rho_m(X)\in[0,1]$. Therefore, $\m S^\ast$ is included in the approximation family $\m S^{\rm ap}$.
\end{remark}
 
Given a prespecified latent space distribution $\nu_0$  
supported on $\mb B_1^d$, similar as the definition for the approximation family $\m S^{\rm ap}_{\nu_0}$, the second generative model class  for the target distribution $\m S^\ast_{\nu_0}$ is defined as the set of all probability measures $\mu\in \m P(\mb R^D)$ that satisfy properties 1 and 2 above with $\nu_{[m]} = [V_{[m]}]_{\#}\nu_0$ for $m\in[M]$, where $V_{[m]}\in C^{1+\alpha}_{L}(\mb B_1^d)$ and  it can be similarly shown that $\mathcal{S}^{\ast}_{\nu_0}\subset \m S^{\rm ap}_{\nu_0}$.\\
\quad\\
 \noindent {\bf Example (Manifold with global parametrization):} 
The simplest example of distribution  $\mu\in\m S^\ast$ considers a submanifold $\m M\subset \mb R^D$ whose atlas $\ms A$ contains a single chart $(\m M, \varphi)$, that is, $\m M$ admits a global parametrization $\varphi^{-1}:\, {\mb B_1^d}^{\circ} \to \m M$. The smoothness level $\beta$ of $\m M$ is the same as that of $\varphi^{-1}$ as a diffeomorphism between open set ${\mb B_1^d}^{\circ}\subset \mb R^d$ and $\m M\subset \mb R^D$. For example, a $d$-dimensional ball lying in a $d$-dimensional affine subspace of $\mb R^D$ and a $d$-dimensional half subspace of $\mb R^D$ are both submanifolds admitting global parametrization.  The global parametrization presumption of a submanifold implies the existence of a boundary, which incurs notoriously technical complication. To address the boundary issue, we require the density function of $\varphi_{\#}\mu$  to be $\alpha$-smooth and smoothly decay to zero around the boundary $\partial \mb B_1^d$ with polynomial tails $(1-\|z\|)^{L_0}$ for some positive constants $L_0$.  Note that in this case, the approximation family $\m S^{\rm ap}_{\nu_0}$ is not sufficient to cover $\mu$, as we need the more flexible approximation family $\m S^{\rm ap}$ containing a latent variable distribution whose density decay around the boundary of $\mb B_1^d$ matches the polynomial tail $(1-\|z\|)^{L_0}$ for a possibly unknown $L_0$. One estimation framework to avoid the boundary issue for a manifold with boundary is to consider the restriction of $\mu$ on a compact subset $K_0$ of $\m M$ that is away from the boundary $\partial \m M$ and we consider this framework in Appendix~\ref{sec:unbounded}.






 
 \paragraph{Exaxmple (Smooth distributions on unknown compact smooth boundaryless submanifold):} Another representative example considers an $\alpha$-smooth distribution supported on a closed submanifold in $\mb R^D$ with density function bounded from above and below, that is, the distribution family $\m P^\ast$ considered in Theorem~\ref{maintheorem}. Due to the assumption that the density function is bounded away from zero, the family $\m S_{\nu_0}^\ast$ with $\nu_0$ being chosen as the uniform distribution on $\mb B_1^d$ is sufficient to cover $\m P^\ast$ when the maximal radius $\{r_m\}_{m\in [M]}$ of the open cover defining the partition of unity is sufficiently small. More specifically, we have the  following lemma.
\begin{lemma}\label{lemmaboundaryless}
For any constant $d,D\in \mathbb{N}^{+}$ and $\alpha,\beta,L^\ast\in \mathbb{R}^+$ with $d<D$, $\beta>1$ and $\alpha\leq \beta-1$,  there exist positive constants $\epsilon_0$ and $L$ such that for any $\ms O_{M}=\{\mb B_{r_m}(a_m)^{\circ}\}_{m\in [M]} $ that forms a open cover for $\mb B_{L}^D$ with $\max\{r_1,r_2,\cdots,r_m\}\leq \epsilon_0$, it holds that $\mathcal{P}^{\ast}(d,D,\alpha,\beta,L^\ast)\subset \m S^{\ast}_{\nu_0}(d,D,\alpha,\beta,\ms O_{M},L)\subset \m S^{\ast}(d,D,\alpha,\beta,\ms O_{M},L)$ with $\nu_0$ being the uniform distribution on $\mb B_1^d$. Thus we also have   $\mathcal{P}^{\ast}(d,D,\alpha,\beta,L^\ast)\subset \m S^{\rm ap}_{\nu_0}(d,D,\alpha,\beta,\ms O_{M},L)$.
\end{lemma}
\begin{remark}\label{remarl:OT} \normalfont 
The proof of Lemma~\ref{lemmaboundaryless} relies on the Caffarelli’s global regularity theory~\citep[Theorem 12.50 of][]{Villani2009,10.2307/2118564} which states that for $\alpha$-smooth probability densities $\mu_P,\mu_Q$ that are bounded from above and below on their supports, if their supports are regular enough, then the unique optimal transport map from $\mu_P$ to $\mu_Q$ is $(\alpha+1)$-smooth. We first prove that ${\m P}^{\ast}$ is contained in ${\m S}^{\ast}$ where the density of the latent variables $\nu_{[m]}\, (m\in [M])$  in each local generative model are bounded away from zero, then we can use  Caffarelli’s global regularity theory to obtain that $\nu_{[m]}$ can be generated from $\nu_0$  via an $(\alpha+1)$-smooth transport map.
\end{remark}

\smallskip
  
 \section{Extension of Main result}\label{ext:main result}
\subsection{Application in Two-Sample Test} \label{sec:two-sample}


 One application of our construction in proving the minimax upper bound is in designing a test statistic for two sample hypothesis testing, which can be stated as follows. Suppose that we have two set of $n$ random samples: $X_{1:n}$ and $Y_{1:n}$, independently obtained from two populations $\mu_1$ and $\mu_2$ over ambient space $\mb R^D$, respectively. Assume both $\mu_1$ and $\mu_2$ have intrinsic dimensionality $d$, or $\mu_1$, $\mu_2\in \mathcal{S}^{\ast}$, we aim to test whether $\mu_1$ and $\mu_2$ are the same, i.e., 
 \begin{equation*}
     H_{0}: \mu_1=\mu_2\quad \mbox{versus} \quad H_{1}: \mu_1\neq \mu_2.
 \end{equation*}
 We propose the following test function (probability of rejecting the null) based on the regularized surrogate $\wh J(f)$ defined in Section~\ref{sec:optimal_proce} for estimating $\mb E_{\mu}[\,f(X)]$,
  \begin{equation*}
 \Phi_{\gamma,c}(X_{1:n},Y_{1:n})=\left\{
 \begin{array}{ll}
1,&\qquad \mbox{if} \ \ \underset{f\in C^{\gamma}_1(\mathbb{R}^D)}{\sup}  \big|\wh{\mathcal{J}} (f;\,X_{1:n})-\wh{\mathcal{J}} (f;\,Y_{1:n})\big|\\
&\qquad\qquad\geq  c\,\big(\frac{\log n}{n}\big)^{\frac{\gamma\beta}{d}}\vee \big(\frac{\log n}{n}\big)^{\frac{\alpha+\gamma}{2\alpha+d}}\vee \big(\frac{\log n}{n}\big)^{\frac{1}{2}};\\
0,&\qquad  \text {otherwise,}
\end{array}\right.
\end{equation*} 
where we estimate the distance $d_\gamma(\mu_1,\,\mu_2)$ between $\mu_1$ and $\mu_2$ by a natural estimator ${\sup}_{f\in C^{\gamma}_1(\mathbb{R}^D)} \big|\wh{\mathcal{J}} (f;\,X_{1:n})-\wh{\mathcal{J}} (f;\,Y_{1:n})\big|$, and reject the null if this estimator is larger than a threshold $\delta_n$ corresponding to the statistical error of the estimator.
To evaluate the power performance of the testing rule induced by $\Phi_{\gamma,c}=\Phi_{\gamma,c}(X_{1:n},Y_{1:n})$, we use the total error ${\rm Err}(\Phi_{\gamma,c},\delta_n)$ of $\Phi_{\gamma,c}$ under separation rate $\delta_n$,
\begin{equation*}
   {\rm Err}(\Phi_{\gamma,c},\delta_n)=\underset{\mu_1,\mu_2\in \mathcal{S}^{\ast}\atop \mu_1=\mu_2}{\sup}\mathbb{E}_{\mu_1,\,\mu_2}(\Phi_{\gamma,c})+\underset{\mu_1,\mu_2\in \mathcal{S}^{\ast}\atop d_{\gamma}(\mu_1,\mu_2)\geq\delta_n}{\sup}\mathbb{E}_{\mu_1,\mu_2}(1-\Phi_{\gamma,c}),
\end{equation*}
to measure the sum of expected (worst case) type I and type II errors of $\Phi_{\gamma,c}$ in distinguishing two distributions in $\m S^\ast$ that are at least $\delta_n$ away from each other in the $d_\gamma$ metric, or testing against local alternatives $H_{1n}: d_{\gamma}(\mu_1,\mu_2)\geq \delta_n$.
The following corollary shows that, the test rule $\Phi_{\gamma,c}$ based on the estimator developed in Section~\ref{sec:optimal_proce} can detect with diminishing type I and type II errors any local alternatives that separate from the null by as small as $\delta_n \asymp \delta_n^{\ast} = \big(\frac{\log n}{n}\big)^{\frac{\gamma\beta}{d}}\vee \big(\frac{\log n}{n}\big)^{\frac{\alpha+\gamma}{2\alpha+d}}\vee \big(\frac{\log n}{n}\big)^{\frac{1}{2}}$ in the $d_\gamma$ metric.

\begin{corollary}[Two-sample test errors]\label{twosample}
For any positive constant $r$ and $\gamma$, there exists constant $c_0=c_0(r,\gamma)$ such that for any $c\geq c_0$, there exits $c_1=c_1(r,\gamma,c)$ such that
\begin{equation*}
   {\rm Err}(\Phi_{\gamma,c}, c_1 \delta_n^{\ast})\leq n^{-r}.
\end{equation*}
\end{corollary}

\subsection{Estimation of Distribution Constrained on Compact Sets}\label{sec:unbounded}
Theorem~\ref{maintheorem} and~\ref{upperboundgenerative} address the boundary issue of the support $\mathcal{M}$ of the target distribution $\mu\in \mathcal{S}^{\ast}$ by assuming $\mathcal{M}$ to be boundaryless or $\mu$ to smoothly decay to zero around the boundary of $\m M$. Another case in which we can prevent the boundary issue is that we only care about $\mu$ constrained on a compact set $K_0$ such that $\mathcal{M}\cap K_0$ is away from the boundary of $\mathcal{M}$. Consider constants $L\in \mathbb{R}^+, d,D\in \mathbb{N}^{+}$, and a compact set $K_0\subset \mb  R^D$, define $ \overline{\mathcal{S}}^{\ast}(d,D,\alpha,\beta, K_0,L)$ to be the set of probability measure $\mu$ satisfying that
\begin{enumerate}
\item There exists a compact set $\widetilde{K}\supseteq K_1=\{x\in B_{1/L}(y)\,:\, y\in K_0\}$ such that $\mu|_{\widetilde{K}}= G_{\#}\nu$.
  \item Write $\Omega={\rm supp}(\nu)$, it holds that 
  \begin{enumerate}
      \item $ {\bigcup}_{z\in\Omega, \, G(z)\in K_1} B_{1/L}(z)\subset \Omega$;
      \item $\nu$ is absolutely continuous with respect to the Lebesgue measure; for any $z\in \Omega$, $|\log \nu(z)|\leq L$ and $\nu\in C^{\alpha}_L(\Omega)$, where we abuse the notation to use $\nu$ to denote its density function.
    \end{enumerate}
  \item There  exists $Q: \mathbb{R}^D\to \mathbb{R}^d$ such that 
  \begin{enumerate}
      \item For any $z\in \Omega$, $Q(G(z))=z$;
      \item $Q\in C^{\beta}_L(\mathbb{R}^D;\mathbb{R}^d)$ and $G\in C^{\beta}_L(\mathbb{R}^d;\mathbb{R}^D)$.
  \end{enumerate}
    \item Write $K_r=\{x\in B_{|r|/L}(y)\,:\, y\in K_0\}$ for $r>0$ and $K_r=\{x\in\mb R^D\,:\,  B_{|r|/L}(x)\subset K_0\}$ for $r<0$. For any $|r|\leq 1$, it holds that $\big|\mathbb{P}_{\nu}\big(G(z)\in K_0\big)- \mathbb{P}_{\nu}\big(G(z)\in K_r\big)\big|\leq L\,|r|$.
 \end{enumerate}
 \begin{remark} \normalfont 
 The assumptions that  ${\bigcup}_{z\in\Omega, \, G(z)\in K_1} B_{1/L}(z)\subset \Omega$ and $\big|\mathbb{P}_{\nu}\big(G(z)\in K_0\big)- \mathbb{P}_{\nu}\big(G(z)\in K_r\big)\big|=\frac{1}{\mb P_{\mu}(X\in \wt{K})}\cdot \big|\mathbb{P}_{\mu}\big(X\in K_0\big)- \mathbb{P}_{\mu}\big(X\in K_r\big)\big|\leq L\,|r|$  ensure that $\mathcal{M}\cap K_0$ is away from  $\partial \mathcal{M}$ if it exists and the intersection of $\mathcal{M}$ with $\partial K_0$ has measure $0$ with respect to $\mu$. Here we assume the support of $\mu$ has a global parametrization when constrained on $K_0$, while it can also adapt to the general case where the support of $\mu|_{K_0}$ has multiple charts by considering partition of $K_0=\sum_{m=1}^M K_{[m]}$ and estimating each $\mu|_{K_{[m]}}$.
 \end{remark}

\begin{theorem}\label{th:unbounded}
Fix $L>0$; $\gamma\geq 0$; $0\leq \alpha\leq \beta-1$; $\beta>1$; $D,d \in \mathbb{N}^+$ with $D>d$; $K_0\subset \mb R^D$ be a compact set.  Then there exists a surrogate functional $\wh{\m J}^{\circ}(f)=\wh{\m J}^{\circ}(f;X_{1:n})$ depends on data $X_{1:n}$, such that for any constant $c$, there exists a constant $n_0$ such that when $n\geq n_0$, for any target distribution $\mu^\ast\in  \overline{\mathcal{S}}^{\ast}(d,D,\alpha,\beta, K_0,L)$, it holds with probability larger than $1-n^{-c}$ that
\begin{equation*}
 \underset{f\in C^{\gamma}_1(\mb R^D)}{\sup}\Big|\wh{\m J}^{\circ}(f)-\int_{K_0}f(x)\,\dd\mu^\ast \Big|\leq C \, \big(\frac{\log n}{n}\big)^{\frac{1}{2}}\vee\big(  \frac{\log n}{n}\big)^{\frac{ \alpha+\gamma}{2\alpha+d}} \vee \big(\frac{\log n}{n}\big)^{\frac{\gamma\beta}{d}}.
\end{equation*}
Thus if $\mb P_{\mu^*}(X\in K_0)\geq C_1>0$, then there exists a distribution estimator $\wh{\mu}^{\circ}$ so that 
\begin{equation*}
 \mb {E}[d_{\gamma}(\wh{\mu}^\circ, \mu^*|_{K_0})]\leq C \, \big(\frac{\log n}{n}\big)^{\frac{1}{2}}\vee\big(  \frac{\log n}{n}\big)^{\frac{ \alpha+\gamma}{2\alpha+d}} \vee \big(\frac{\log n}{n}\big)^{\frac{\gamma\beta}{d}}.
\end{equation*}
 \end{theorem}

\section{Proofs of Remaining Results}\label{sec:remaining_results}
In this section, we provide proofs for the remaining results in the main paper.

\subsection{Proofs of lemmas in Section~\ref{sec:prooflowerbound} for minimax lower bound}\label{prooflemmalower}
In this subsection, we provide proofs for Lemmas~\ref{Lemma:MP},~\ref{Lemma:DP} and~\ref{Lemma:LM}.
\subsubsection{Proof of Lemma~\ref{Lemma:MP}}\label{Sec:proof_Lemma_MP}
Recall that $\m M_0=\mb S_2^d\times \bold{0}_{D-d-1}=\{x\in \mb R^D:\, \|x_{1:d+1}\|^2=2,\,x_{d+2:D}=\bold{0}_{D-d-1}\}$ denotes the $d$-dimensional sphere embedded in $\mb R^D$ and $\wt{\m{M}}_0=\{x\in \mb{R}^D:\, x_{1:d}\in \mb B_1^d,\, x_{d+1}=\sqrt{2-\|x_{1:d}\|^2},\, x_{d+2:D}=\bold{0}_{D-d-1}\}$ denotes its middle area. $\mu_0$ is the uniform distribution over $\m M_0$. In addition, $\wt{\m{M}}_0$
admits a global parametrization $G_0:\, \mb B_1^d\to \wt{\m{M}}_0$ defined as $G_0(z)=(z,\,\sqrt{2-\|z\|_2^2}\,,\,\bold{0}_{D-d-1})$ for $z\in \mb B_1^d$. Note that $G_0$ is $\beta$-smooth over $\mb B_1^d$ with bounded H\"{o}lder norm. Let $\nu_0$ denote the density function on $\mb B_1^d$ so that $[G_0]_\# \nu_0$ is the restriction of $\mu_0$ on $\wt{\m{M}_0}$, or
\begin{align*}
    \nu_0(z)=\frac{1}{\wt{C}}\sqrt{{\rm det}(\bold{J}_{G_0}(z)^T\bold{J}_{G_0}(z))}, \quad \forall z\in \mb B_1^d,
\end{align*}
where recall that $\bold{J}_G$ denotes the Jacobian matrix of $G$ and $\wt{C}=\int_{\mb B_1^d}\sqrt{{\rm det}(\bold{J}_{G_0}(z)^T\bold{J}_{G_0}(z))}\,\dd z$ is the normalizing constant. 
Since $\wt{\m{M}}_0$ is a compact set, there exists positive constants $a_1,a_2$ so that for any $z\in \mb B_1^d$, $a_1I_d\preccurlyeq \bold{J}_{G_0}(z)^T\bold{J}_{G_0}(z)\preccurlyeq a_2I_d$. Therefore, $\nu_0$ is an $\alpha$-smooth density function with bounded H\"{o}lder norm on $\mb B_1^d$. Next we will add small bumps to function $G_0$ to construct perturbations of $\wt{\m{M}}_0$, whose unions with the spherical cap $\wt{\m M}_1:\,=\m M_0\setminus \wt{\m{M}}_0$ form our constructed perturbed $\beta$-smooth manifolds with controlled H\"{o}lder norm. 

Let $m=\lfloor b\,n^{\frac{1}{d}}\rfloor$ with tuning parameter $b$ to be determined later. Our constructed perturbed generative maps $G_\omega$ are parametrized by a binary tensor $\omega =(\omega_{\xi})_{\xi\in [m]^d} \in \{0,1\}^{[m]^d}$ 
indexed by all $d$-dimensional grid points in $[1,m]^d$. Here, subscript $\xi\in[m]^d$ indicates the coordinates of the grid points, and $\omega$ indicates the locations of the jumps. More specifically, let 
\begin{equation}\label{eqnk.1}
 k(t)=\left\{\begin{array}{l}
(1-t)^{\beta+1} t^{\beta+1}, \quad t \in(0,1), \\
0, \quad \text { o.w. }
\end{array}\right.
 \end{equation}
defines a localized bump function. For each $d$-dimensional grid point $\xi=(\xi_1,\ldots,\xi_d)\in[m]^d$, let 
\begin{align*}
    \psi_{\xi}(z)=\prod_{i=1}^d k\Big(m\,\sqrt{\frac{d}{2}}\,z_i+\frac{m}{2}-\xi_i\Big),\quad\forall z\in \mb B_1^d,
\end{align*}
denote a localized bump function over $\mb B_1^d$ whose support is contained in the following cube
\begin{align*}
    \left[-\sqrt{\frac{1}{2d}}+\frac{\xi_1}{m}\sqrt{\frac{2}{d}},\  -\sqrt{\frac{1}{2d}}+\frac{\xi_1+1}{m}\sqrt{\frac{2}{d}}\right]\times\cdots\times \left[-\sqrt{\frac{1}{2d}}+\frac{\xi_d}{m}\sqrt{\frac{2}{d}}, \  -\sqrt{\frac{1}{2d}}+\frac{\xi_d+1}{m}\sqrt{\frac{2}{d}}\right]
\end{align*}
which has width $m^{-1}\sqrt{2/d}$ and is contained in $B_{3/4}^d$ when $b\geq 9$.
For any $\omega=(\omega_{\xi})_{\xi\in [m]^d} \in \{0,1\}^{[m]^d}$, we define the multi-bump function 
\begin{align*}
    g_{\omega}(z)=\sum_{\xi\in [m]^d}\frac{1}{m^{\beta}} \,\omega_{\xi} \, \psi_{\xi}(z),\quad z\in \mb B_1^d,
\end{align*}
whose bumps correspond to the non-zero components of $\omega$. Finally, we define $G_\omega(z) = G_0(z)+(\bold{0}_{d},\,g_{\omega}(z),\,\bold{0}_{D-d-1})$ for $z\in \mb B_1^d$ as the perturbed generative map parametrized by $\omega \in \{0,1\}^{[m]^d}$.
It is straightforward to verify that there exists some constant $L$, such that $m^{-\beta}\psi_{\xi}(z)$ belongs to $C^{\beta}_L(\mathbb{R}^d)$ for any $\xi\in [m]^d$, and for any $\omega \in \{0,1\}^{[m]^d}$, $g_\omega$ belongs to $C_L^\beta(\mb R^d)$ and $G_\omega$ belongs to $C^{\beta}_{L}(\mb B_1^d;\mathbb{R}^D)$.

By Lemma~\ref{le8} in Appendix~\ref{Sec:le8},  which is a two-sided version of the Varshamov-Gilbert lemma~\citep{Tsybakov2009}, there exists a subset $\{\omega^{(1)},\cdots , \omega^{(H_0)}\}\subset \{0,1\}^{[m]^d}$ such that:
\begin{enumerate}[topsep=2pt,itemsep=0ex]
\item $\log H_0 \geq \frac{m^d}{8}-\log 2$;
\item for any $j,k\in [H_0]$ with $j\neq k$, the Hamming distance $\rho(\omega^{(j)},\omega^{(k)})$ between $\omega^{(j)}$ and $\omega^{(k)}$ satisfies $\frac{m^d}{4}\leq \rho(\omega^{(j)},\omega^{(k)})\leq \frac{3m^d}{4}$.
\end{enumerate}
For each $\omega\in \{0,1\}^{[m]^d}$, define $\bar{\omega}=1-\omega$ in the element-wise manner. We may expand the above $H_0$ tensors into $H=2H_0$ ones, ordered as
\begin{equation*}
\{{\omega}^{(1)},\cdots, {\omega}^{(H)}\}= \{\omega^{(1)},\cdots, \omega^{(H_0)},
\bar{\omega}^{(1)},\cdots,\bar{\omega}^{(H_0)}\}.
\end{equation*}
Then $\log H \geq \frac{m^d}{8}$ and for any $i,j\in [H]$ with $i\neq j$, it holds that $\rho({\omega}^{(i)}, {\omega}^{(j)}) \geq \frac{m^d}{4}$. We do this expansion since we will use a key property later (c.f.~proof of Lemma~\ref{lemma2}) that for any fixed index $\xi\in [m]^d$, there are equal numbers of $0$'s and $1$'s in the sequence $\big({\omega}_{\xi}^{(1)},\cdots,{\omega}_{\xi}^{(H)}\big)$. 

Next, for each $i\in[H]$, let ${\m{M}}_{{\omega}^{(i)}}=G_{{\omega}^{(i)}}(\mb B_1^d)$ denote the perturbed manifold from $G_{{\omega}^{(i)}}$. We define a perturbation to $\mu_0$ by smoothly gluing together the restriction $\mu_0|_{\wt{\m M}_1}$ of $\mu_0$ onto $\wt{\m M}_1$ and $\mu_{\omega^{(i)}}:\,= [G_{\omega^{(i)}}]|_\# \nu_0$ over $\m M_{\omega^{(i)}}$ as
\begin{equation*}
    \mu_{i}=\Big(1-\frac{\wt{C}}{C}\Big)\cdot\mu_0|_{\wt{\m M}_1}
    +\frac{\wt{C}}{C}\cdot{\mu}_{{\omega}^{(i)}},
\end{equation*}
where recall that $C^{-1}$ is the density function of the uniform distribution over $\m M_0$ so that $\wt C < C$. Then 
$\mu_{i}$ is supported over the manifold $\m M_i:\,= \wt{\m M}_1\cup \m M_{\omega^{(i)}}$. Since $G_{\omega^{(i)}}\in C^{\beta}_L(\mb B_1^d,\mathbb{R}^D)$, $G_{\omega^{(i)}}^{-1}(x)=(x_1,x_2,\cdots,x_d)$ for $x\in \m M_{\omega^{(i)}}$ and by construction $G_{\omega^{(i)}}(z)=G_0(z)$ for any $z\in \mb B_1^d\setminus \mb B_{3/4}^d$, we have that $\m{M}_{i}$ is a compact $\beta$-smooth submanifold. Furthermore, the density function of distribution $\mu_{i}$ with respect to the volume measure of $\m{M}_{i}$ is given by
\begin{equation*}
    \mu_i(\dd x)=\frac{1}{C}\, \bold{1}(x\in \wt{\m M_1})+\frac{1}{C}\, \frac{\sqrt{{\rm det}(\bold{J}_{G_0}(x_{1:d})^T\bold{J}_{G_0}(x_{1:d}))}}{\sqrt{{\rm det}(\bold{J}_{G_{{\omega}^{(i)}}}(x_{1:d})^T\bold{J}_{G_{{\omega}^{(i)}}}(x_{1:d}))}}\, \bold{1}(x\in {\m{M}}_{{\omega}^{(i)}}),\ \ \forall x\in \m M_i.
    \end{equation*}
Note that function $g_i(z):\,=\frac{\sqrt{{\rm det}(\bold{J}_{G_0}(z)^T\bold{J}_{G_0}(z))}}{\sqrt{{\rm det}(\bold{J}_{G_{{\omega}^{(i)}}}(z)^T\bold{J}_{G_{{\omega}^{(i)}}}(z))}}$ is $\alpha$-smooth function with bounded H\"{o}lder norm on $\mb B_1^d$ and $g_i(z)=1$ for $z\in \mb B_1^d\setminus B_{3/4}^d$. Consequently, there exists a constant $L_1$ such that $\mu_i(\dd x)\in C^{\alpha}_{L_1}(\m{M}_{i})$\footnote{We say a function $f:\m M\to \mb R\in C^{\alpha}_{L}(\m M)$ for a manifold $\m M$ embedded in $\mb R^D$, if there exists a open set $\m U\subset \mb R^D$ containing $\m M$ and a function $\overline{f}: \m U\to \mb R$ so that $\overline{f}|_{\m M}=f$ and $f\in C^{\alpha}_{L}(\m U)$. }, or equivalently, $\mu_i$ is an $\alpha$-smooth distribution on 
the $\beta$-smooth manifold $\m M_i$.  Therefore, for sufficiently large $L^\ast$, it holds that $\mu_{i}\in\mathcal{P}^{\ast}(d,D,\alpha,\beta,L^\ast)$ for any $i\in [H]$.

 Let $\bar{\mu}=\frac{1}{H} \sum_{i=1}^{H} \mu_{i}$ be the averaged distribution. The following lemma, whose proof is deferred to Section~\ref{Sec:Proof_lemma2}, provides upper bounds to $D_{\rm KL}(\mu_{i}\,||\,\bar{\mu})$ and lower bounds to the pairwise $d_\gamma$ distances, which concludes the proof. 
\begin{lemma}\label{lemma2}
Assume $\gamma<1$ and $\beta>1$. For any $i \in [H]$, it holds that  $D_{\rm KL}(\mu_{i}\,||\,\bar{\mu})\leq \log 2$. For any $i,j\in [H]$ with $i\neq j$, it holds that $d_\gamma(\mu_i,\mu_j)\geq c_1\, m^{-\gamma\beta} \geq c_2 \,n^{-\frac{\gamma\beta}{d}}$.
\end{lemma}

\subsubsection{Proof of Lemma~\ref{Lemma:DP}}\label{Sec:Proof_lemma_dp}
We adopt the same setting and notations as in the proof of Lemma~\ref{Lemma:MP} in Section~\ref{Sec:proof_Lemma_MP}, where recall that $\mathcal{M}_0=\mathbb{S}_2^d\times \bold{0}_{D-d-1}$ is the $d$-dimensional sphere embedded in $\mb R^D$ and $\mu_0$ is the uniform distribution on $\m M_0$. In addition, we partition $\m M_0$ into its middle area $\wt{\m M}_0$ and the two remaining caps $\wt{\m M}_1$, where $(\nu_0,G_0)$ is the generative model such that $[G_0]_\# \nu_0$ is the restriction $\mu_0|_{\wt{\m M}_0}$ of $\mu_0$ on $\wt{\m M}_0$. Our construction of perturbed distributions on $\m M_0$ will be based on adding small bumps to the generative distribution $\nu_0$ on the $d$-dimensional unit ball $\mb B_1^d$.

Let $\wt m = \lfloor b n^{\frac{1}{2\alpha+d}}\rfloor$. Define the following localized bump function 
\begin{equation}\label{eqnk.2}
 \wt k(t)=\left\{\begin{array}{l}
(1-t)^{\alpha\vee \gamma+1} t^{\alpha\vee\gamma+1}(t-\frac{1}{2}), \quad t \in(0,1) \\
0, \quad \text { o.w. }
\end{array}\right.
\end{equation}
so that $\int_{-\infty}^\infty \wt k(t) \,\dd t=0$, and the corresponding localized bump function over $\mb B_1^d$ 
\begin{equation}\label{eqnk.2.2}
    \wt \psi_{\xi}(z)=\prod_{i=1}^d \wt k\Big(m\sqrt{\frac{d}{2}}z_i+\frac{m}{2}-\xi_i\Big),\quad\forall z\in \mb B_1^d,
\end{equation}
indexed by the $d$-dimensional grid point $\xi=(\xi_1,\ldots,\xi_d)\in[\wt m]^d$. In addition, we define two function sets 
\begin{align*}
\Psi_{\alpha}&=\Big\{ g_{\omega}(z)=\nu_0(z)+b_0\,\Big(\frac{1}{\wt m}\Big)^{\alpha}\sum_{\xi\in[\wt m]^d}  \omega_{\xi} \,\wt \psi_{\xi}(z):\  \omega=\{\omega_{\xi}\}_{\xi\in [\wt m]^d}\in \{0,1\}^{[\wt m]^d}\Big\},\\
\Lambda_{\gamma}&=\Big\{ f_{v}(z)=\Big(\frac{1}{\wt m}\Big)^{\gamma}\sum_{\xi\in[\wt m]^d}  v_{\xi} \,\wt\psi_{\xi}(z):\  v=\{v_{\xi}\}_{\xi\in [\wt m]^d}\in \{-1,1\}^{[\wt m]^d}\Big\},
\end{align*}
where $b_0$ is some constant to be determined later. $\Psi_\alpha$ consists of all perturbed densities around $\nu_0$ and $\Lambda_\gamma$ serves as set of discriminators defined over $\mb B_1^d$ for discriminating the densities in $\Psi_\alpha$.
Since for any $\xi \in [\wt m]^d$, the support ${\rm supp}(\psi_{\xi})$ is contained in $\big[-\sqrt{\frac{1}{2d}}+\frac{\xi_1}{m}\sqrt{\frac{2}{d}},\ -\sqrt{\frac{1}{2d}}+\frac{\xi_1+1}{m}\sqrt{\frac{2}{d}}\,\big]\times\cdots\times \big[-\sqrt{\frac{1}{2d}}+\frac{\xi_d}{m}\sqrt{\frac{2}{d}},\ -\sqrt{\frac{1}{2d}}+\frac{\xi_d+1}{m}\sqrt{\frac{2}{d}}\,\big]$, which is further contained in $\big[-\sqrt{\frac{1}{2d}}+\frac{1}{m} \sqrt{\frac{2}{d}}, \ \sqrt{\frac{1}{2d}}+\frac{1}{m} \sqrt{\frac{2}{d}}\,\big]^d$. Therefore, $\wt{\psi}_\xi$'s with distinct indices $\xi$'s have disjoint supports and if $b\geq 9$, then for each $g\in \Psi_\alpha$, we have: $g(z) =\nu_0(z)$ for all $z\in \mb R^d\setminus B_{3/4}^d$; and $g(z) \geq \inf_{z\in B^d_{3/4}}\nu_0(z)- b_0 \,b^{-\alpha} \sup_{t\in (0,1)}|\wt k(t)|^d >0$ for all $z\in B_{3/4}^d$ when $b_0$ is sufficiently small, which makes $g$ non-negative. In addition, since $\int_{-\infty}^\infty \wt k(t) \,\dd t=0$, we have $\int_{\mb B_1^d} g(z)\,\dd z = \int_{\mb B_1^d} \nu_0(z)\,\dd z =1$. Therefore, all functions in $\Psi_\alpha$ are valid probability density functions. Finally, it is straightforward to verify that there exist constants $(L_1,L_2)$ such that $g\in C^{\alpha}_{L_1}(\mathbb{R}^d)$ and $f\in C^{\gamma}_{L_2}(\mathbb{R}^d)$ for each $g\in\Psi_\alpha$ and $f\in \Lambda_\gamma$. 

For each $\omega\in\{0,1\}^{m^d}$, we define the following distribution over $\m M_0$ as
\begin{equation*}
    \mu_{\omega}=\Big(1-\frac{\wt{C}}{C}\Big)\cdot\mu_0|_{\wt{\m M}_1}+\frac{\wt{C}}{C}\cdot [G_0]_{\#}g_{\omega},
\end{equation*}
where recall that $C$ is the surface area of $\mb{S}_2^d$ and $\wt{C}=\int_{\mb B_1^d}\sqrt{{\rm det}(\bold{J}_{G_0}(z)^T\bold{J}_{G_0}(z))} \dd z$. 
Then $\mu_{\omega}$ has the following density function with respect to the volume measure of $\m{M}_0$,
\begin{equation*}
   \mu_{\omega}(\dd x)=\frac{1}{C}\, \bold{1}(x\in \wt{\m M}_1)+\frac{\wt{C}}{C}\cdot\frac{b_0\,\big(\frac{1}{\wt m}\big)^{\alpha}\sum_{\xi\in[\wt m]^d}  \omega_{\xi}\, \wt \psi_{\xi}(x_{1:d})}{\sqrt{{\rm det}(\bold{J}_{G_0}(x_{1:d})^T\bold{J}_{G_0}(x_{1:d}))}}\cdot \bold{1}(x\in {\wt{\m M}}_0),\ \ \forall x\in \m M_0.
\end{equation*}
Since $\psi_{\xi}(z)=0$ for all $z\in \mb B_1^d \setminus B_{3/4}^d$ and $G_0$ is infinitely-differentiable over the compact set $\mb B_1^d$ with non-singular Jacobian $\bold{J}_{G_0}$, there exists constant $L_3$ such that for each $\omega\in [0,1]^{[\wt m]^d}$ the density function of $\mu_{\omega}$ belongs to $C^{\alpha}_{L_1}(\mathcal{M}_0)$, implying $\mu_{\omega}\in \mathcal{P}^{\ast}(d,D,\alpha,\beta, L^\ast)$ for sufficiently large $L^\ast$.

Next, we will pick up a subset $\big\{\omega^{(0)},\ldots,\omega^{(H')}\big\}$ of $[\wt m]^d$ such that the corresponding distributions $\big\{\mu_h:\,=\mu_{\omega^{(h)}}\big\}_{h=1}^{H'}$ constitute the set of perturbed distributions in the lemma.
Concretely, by the Varshamov-Gilbert lemma~\citep{Tsybakov2009},  there exists a set $\{\omega^{(0)},\cdots , \omega^{(H')}\}\subset \{0,1\}^{[\wt m]^d}$ such that $\log H' \geq \frac{\wt m^d}{8}\log 2$ and the Hamming distance $\rho(\omega^{(j)},\omega^{(k)})\geq \frac{m^d}{8}$ for any distinct pair $j,k\in [H']$. Therefore, for any distinct $j, k\in[H']$, we have by our construction of $\mu_h$'s that 
 \begin{equation*}
\begin{aligned}
d_{\gamma}(\mu_{j}, \mu_{k})&=\frac{\wt{C}}{C} \underset{f\in C^{\gamma}_1(\mathbb{R}^D)}{\sup}\int_{\wt{\m M}_0} f(x)\cdot  \big(\dd([G_0]_{\#} g_{\omega^{(j)}})-\dd([G_0]_{\#}g_{\omega^{(k)}}]\big)\\
&=\frac{\wt{C}}{C} \underset{f\in C^{\gamma}_1(\mathbb{R}^D)}{\sup}\int_{\mb B_1^d} f\circ G_0(z)\cdot  \big(g_{\omega^{(j)}}(z)-g_{\omega^{(k)}}(z)\big)\,\dd z.
 \end{aligned}
\end{equation*}
Now since $G_0$ is infinitely-differentiable over $\mb B_1^d$ with bounded H\"{o}lder norm and the $d$-dimensional discriminator class $\Lambda_\gamma\subset C_{L_2}^\gamma(\mb B_1^d)$, we have $c_0^{-1} \Lambda_\gamma\subset C_{c_0^{-1}L_2}^\gamma(\mb B_1^d)\subset C^{\gamma}_1(\mathbb{R}^D) \circ G_0=\{f\circ G_0:\, f\in C^{\gamma}_1(\mathbb{R}^D)\}$ for some sufficiently small constant $c_0$, and
\begin{align*}
d_{\gamma}(\mu_{j}, \mu_{k})& \geq \frac{\wt{C}}{c_0 \,C}\cdot \underset{f\in\Lambda_{\gamma}}{\sup} \int_{\mb B_1^d} f(z)\cdot  \big(g_{\omega^{(j)}}(z)-g_{\omega^{(k)}}(z)\big)\,\dd z\\
& = \frac{b_0\, \wt{C}}{c_0 \,C} \cdot  \Big(\frac{1}{\wt m}\Big)^{\alpha+\gamma} \cdot  \underset{v\in\{-1,1\}^{[\wt m^d]}}{\sup} \int_{\mb B_1^d} \bigg\{\sum_{\xi\in[\wt m]^d} v_{\xi} \,\wt \psi_{\xi}(z) \bigg\}\cdot \bigg\{ \sum_{\xi\in[\wt m]^d} (\omega^{(j)}_{\xi}-\omega^{(k)}_{\xi}) \cdot \wt \psi_{\xi}(z)\bigg\}\, \dd z\\
&=  \frac{b_0\, \wt{C}}{c_0 \,C} \cdot  \Big(\frac{1}{\wt m}\Big)^{\alpha+\gamma} \cdot  \underset{v\in\{-1,1\}^{[\wt m^d]}}{\sup} \int_{\mb B_1^d} \bigg\{\sum_{\xi\in[\wt m]^d} v_{\xi} \cdot (\omega^{(j)}_{\xi}-\omega^{(k)}_{\xi}) \cdot {\wt \psi_{\xi}}^2(z)\bigg\}\, \dd z\\
&\geq c'\cdot \Big(\frac{1}{\wt m}\Big)^{\alpha+\gamma+d} \cdot \rho(\omega^{(j)},\omega^{(k)}) \geq \frac{c'}{8}\cdot b^{-(\alpha+\gamma+d)}\, n^{-\frac{\alpha+\gamma}{2\alpha+d}},
 \end{align*}
for some constant $c'$.
Moreover, we have
\begin{equation*}
\begin{aligned}
D_{\rm KL}(\mu_{j},\mu_{k})
&=\frac{\wt{C}}{C} \, D_{\rm KL}(G_{\#}g_{\omega},G_{\#}g_{\omega'})\\
&=\frac{\wt{C}}{C}\int_{\big[-\sqrt{\frac{1}{2d}}, \ \sqrt{\frac{1}{2d}}+\sqrt{\frac{2}{d}} \frac{1}{\wt m}\big]^d} -\log\bigg(\underbrace{
\frac{\nu_0(z)+b_0(\frac{1}{\wt m})^{\alpha} \sum_{\xi\in[\wt m]^d} \omega^{(k)}_{\xi} \wt \psi_{\xi}(z)}{\nu_0(z)+b_0(\frac{1}{\wt m})^{\alpha} \sum_{\xi\in[\wt m]^d} \omega^{(j)}_{\xi} \wt \psi_{\xi}(z)}}_{:\,=1+u(z)}
\bigg) g_{\omega^{(j)}}(z) \, \dd z.
 \end{aligned}
\end{equation*}
For sufficiently large $b$, we have $|u(z)| \leq 1/4$ over $\mb B_1^d$ so that $-\log(1+u(z)) \leq u^2(z) - u(z)$. 
This leads to
\begin{equation*}
\begin{aligned}
D_{\rm KL}(\mu_{j},\mu_{k})
&\leq c''\Big(\frac{1}{\wt m}\Big)^{2\alpha} + c''\,b_0\Big(\frac{1}{\wt m}\Big)^{\alpha}\int_{\big[-\sqrt{\frac{1}{2d}}, \ \sqrt{\frac{1}{2d}}+\sqrt{\frac{2}{d}} \frac{1}{\wt m}\big]^d}
 \bigg\{\sum_{\xi\in[\wt m]^d}( \omega_{\xi}^{(j)}-\omega_{\xi}^{(k)})\cdot \psi_{\xi}(z)  \bigg\}\,  \dd z
\\
&= c''\Big(\frac{1}{\wt m}\Big)^{2\alpha} = c'' \,b^{-2\alpha}n^{-\frac{2\alpha}{2\alpha+d}},
\end{aligned}
\end{equation*}
for some constant $c''$, where we used the fact that $\int_{\mb B_1^d} \psi_\xi(z)\,\dd z=0$ for any $\xi\in[\wt m]^d$.

\subsubsection{Proof of Lemma~\ref{Lemma:LM}}\label{Sec:Proof_lemma_LM}
Define function $\bar k:\mathbb{R}\to\mathbb{R}$ by
\begin{equation}\label{eqnk.3}
 \bar k(z)=\left\{\begin{array}{l}
\big(\sqrt{\frac{1}{2d}}-z\big)^{\alpha\vee \gamma+1} \,\big(z+\sqrt{\frac{1}{2d}}\big)^{\alpha\vee\gamma+1}\, z, \quad z \in \big[-\sqrt{\frac{1}{2d}}, \sqrt{\frac{1}{2d}}\,\big], \\
0, \quad \text { o.w. }
\end{array}\right.
\end{equation}
Then we have $\int_{-\infty}^\infty \bar k(z) \,\dd z=0$. For $z\in \mb B_1^d$, define
\begin{equation}\label{eqnk.4}
\begin{aligned}
 \nu_1(z)=\nu_0(z)+\frac{c}{\sqrt{n}} \prod_{j=1}^d \bar k(z_j),
\end{aligned}
\end{equation}
where $\nu_0$ is defined in the proof of Lemma~\ref{Lemma:MP} in Section~\ref{Sec:proof_Lemma_MP} and $c$ is a constant. For sufficiently small $c$, we have $\nu_1(z)\geq 0$, which combined with $\int_{\mb B_1^d} \nu_1(z) \dd z=1$ implies $\nu_1$ to be a valid probability density function on $\mb B_1^d$. Moreover, there exists some sufficiently large constant $c_1$ such that $\nu_0$, $\nu_1 \in C^\alpha_{c_1}(\mb B_1^d)$. Recall $G_0:\mb B_1^d\to \mb{R}^D$ is the generative map defined as $G(z)=(z,\,\sqrt{2-\|z\|^2}\,,\,\bold{0}_{D-d-1}\big)$ for $z\in \mb B_1^d$, $\m{M}_0=\wt{\m M}_0\cap \wt{\m M}_1$ is the partition of manifold $\m M_0$ defined in Section~\ref{Sec:proof_Lemma_MP}, and $\mu_0$ is the uniform distribution on $\m M_0$. Define another distribution on $\m M_0$ as
\begin{equation*}
    \mu_1=\Big(1-\frac{\wt{C}}{C}\Big)\cdot\mu_0|_{\wt{\m M}_1}+\frac{\wt{C}}{C}\cdot [G_0]_{\#}\nu_1.
\end{equation*}
where $C$, $\wt{C}$ are the same normalizing constants defined in Section~\ref{Sec:proof_Lemma_MP}. 
 The density function of $\mu_1$ with respect to the volume measure of $\m{M}_0$ is
 $$\mu_1(\dd x)=\frac{1}{C}\, \bold{1}(x\in \wt{\m M_1})+\frac{\wt{C}}{C}\cdot\frac{\frac{c}{\sqrt{n}} \prod_{j=1}^d \bar k(x_j)}{\sqrt{{\rm det}(\bold{J}_{G_0}(x_{1:d})^T\bold{J}_{G_0}(x_{1:d}))}}\, \bold{1}(x\in \wt{\m{M}}_0),\quad\forall x\in \m M_0.$$
Then by $\bar k(z)=0$ when $|z|>\sqrt{\frac{1}{d}}$,  we can obtain that $\mu_0, \mu_1\in \mathcal{P}^{\ast}(d,D,\alpha,\beta,L^\ast)$ for sufficiently large $L^\ast$. Moreover,
\begin{align*}
    d_{\chi^2} (\mu_1, \mu_0) &= \frac{\wt{C}}{C}\, d_{\chi^2}\big([G_0]_\#\nu_1,[G_0]_\#\nu_0\big)
    = \frac{\wt{C}}{C}\, d_{\chi^2}(\nu_1,\nu_0)\\
    &= \frac{\wt{C}}{C}\, \int_{\big[-\sqrt{\frac{1}{2d}}, \ \sqrt{\frac{1}{2d}}\,\big]^d}  
    \Big(\frac{\nu_1(z)}{\nu_0(z)}-1\Big)^2 \,\nu_0(z) \,\dd z\\
    &=\frac{\wt{C}}{C}\, \int_{\big[-\sqrt{\frac{1}{2d}}, \ \sqrt{\frac{1}{2d}}\,\big]^d}   \frac{c^2}{n}\prod_{j=1}^d \frac{\bar k^2(z_j)}{\nu_0(z)} \,\dd z\leq\frac{1}{n},
\end{align*}
for sufficiently small $c$.
Define discriminator $\bar{f}(z)=\sqrt{n}\,\big(\nu_1(z)-\nu_0(z)\big)=c \prod_{j=1}^d \bar k(z_j)$. Since $G_0$ is infinitely-differentiable with bounded higher-order derivative over $\mb B_1^d$, there exist sufficiently small constants $(c_1,c_2)$ such that $c_1\bar{f}(z) \in C_{c_2}^{\gamma}(\mathbb{R}^d) \subset C^{\gamma}_1(\mathbb{R}^D) \circ G_0=\{f\circ G_0:\, f\in C^{\beta}_1(\mathbb{R}^D)\}$. Therefore, we have
\begin{equation*}
\begin{aligned}
d_{\gamma}(\mu_1,\mu_0)
&=\underset{f\in C^{\gamma}_1(\mb{R}^D)}{\sup}\frac{\wt{C}}{C}\int_{\wt{\m M}_0} f(x)\, \big(\dd([G_0]_{\#}\nu_1]-\dd[G_0]_{\#}\nu_0\big)\\
&\geq \underset{f\in C^{\gamma}_1(\mb{R}^D)}{\sup}\frac{\wt{C}}{C}\int_{\mb B_1^d} f\circ G_0(z)\cdot 
\big(\nu_1(z)-\nu_0(z)\big)\,\dd z\\
&\geq \frac{c_1\,\wt{C}}{C} \, \sqrt{n} \int_{\mb B_1^d} \big(\nu_1(z)-\nu_0(z)\big)^2 \, \dd z\\
&=\frac{c_1\,c^2\,\wt{C}}{C\sqrt{n}} \,\int_{\mb B_1^d} \prod_{j=1}^d \bar k^2(z_j) \,\dd z\geq \frac{c_3}{\sqrt{n}},
 \end{aligned} 
\end{equation*}
for some constant $c_3$.


\subsection{Proofs of lemmas in Section~\ref{sec:proofupperbound} for minimax upper bound}\label{Sec:more_proofs_upperbound}
In this subsection, we provide proofs for Lemmas~\ref{lemma:mani_est},~\ref{lemma:low_freq},~\ref{lemma:high_freq} and~\ref{lemma:smooth_corr}. 
In the following proofs, we only consider those $m\in[M]$ such that $\mathbb{P}_{\mu^{\ast}}(X\in S^{\dagger}_m)\geq \frac{1}{2}\sqrt{\frac{\log n}{n}}$. In fact,
by applying Bernstein's inequality for a binomial random variable and a union bound argument, we obtain that, for any constant $c$, there exits a constant $c_1$ such that with probability at least $1-n^{-c}$,
 \begin{equation}\label{eqnhatp}
     \underset{m\in [M]}{\sup}\,
    \big|\, \wh{p}_m-\mathbb{P}_{\mu^{\ast}}(X\in S^{\dagger}_m)\big|\leq c_1\Big(\frac{\log n}{n}+\sqrt{\frac{\log n}{n}}\sqrt{\mathbb{P}_{\mu^{\ast}}(X\in S^{\dagger}_m)}\,\Big).
 \end{equation}
 Therefore, if $\mathbb{P}_{\mu^{\ast}}(X\in S^{\dagger}_m)\leq \frac{1}{2}\sqrt{\frac{\log n}{n}}$, then there exists some sufficiently large integer $n_0$ such that when $n\geq n_0$, it holds with probability at least $1-n^{-c}$ that  $\wh{p}_m< \sqrt{\frac{\log n}{n}}$, which leads to
 \begin{equation}\label{eqn:mub.1}
 \begin{aligned}
  \underset{f\in C^{\gamma}_1(\mathbb{R}^D)}{\sup} \Big|\mb E_{\mu^\ast}[f(X)\cdot\rho_m(X)]- \wh{\m J}_m(f)\Big| & \overset{(i)}{=} \underset{f\in C^{\gamma}_1(\mathbb{R}^D)}{\sup} \big|\mb E_{\mu^\ast}[f(X)\cdot\rho_m(X)]\big|\\
  &\leq \mathbb{P}_{\mu^{\ast}}(X\in S^{\dagger}_m) \leq  \frac{1}{2}\,\sqrt{\frac{\log n}{n}},
   \end{aligned}
 \end{equation}
 where we have used in step (i) the definition of $\wh{\m J}_m$ that $\wh{\m J}_m(f)\equiv 0$ if $\wh p_m < \sqrt{\frac{\log n}{n}}$.  On the other hand, if $\wh{p}_m< \sqrt{\frac{\log n}{n}}$, then it holds with probability larger than $1-n^{-c}$ that $\mathbb{P}_{\mu^\ast}(X\in S^{\dagger}_m)\leq 2\sqrt{\frac{\log n}{n}}$. Thus $\widehat{p}_m< \sqrt{\frac{\log n}{n}}$ can lead  to
 \begin{equation}\label{eqn:mub.2}
 \begin{aligned}
  \underset{f\in C^{\gamma}_1(\mathbb{R}^D)}{\sup} \Big|\mb E_{\mu^\ast}[f(X)\cdot\rho_m(X)]- \wh{\m J}_m(f)\Big| &  = \underset{f\in C^{\gamma}_1(\mathbb{R}^D)}{\sup} \big|\mb E_{\mu^\ast}[f(X)\cdot\rho_m(X)]\big|\\
  &\leq \mathbb{P}_{\mu^{\ast}}(X\in S^{\dagger}_m) \leq  2\,\sqrt{\frac{\log n}{n}}.
   \end{aligned}
 \end{equation}
In the proofs, we consider a fixed $\mu^{\ast}\in \mathcal{S}^{\ast}$. From the definition of $\m S^\ast$, we have that for each $m\in [M]$, there exists $\widetilde{S}_m\supseteq \mb B_{r_m+1/L}(a_m)$ such that $\mu^{\ast}|_{\widetilde{S}_m}= [G^{\ast}_{[m]}]_\#\nu^{\ast}_{[m]}$ for some generative model pair $\big(\nu^{\ast}_{[m]},G^{\ast}_{[m]}\big)$ satisfying the conditions in Section~\ref{sec:gen_model_class}. In particular, we use $Q^{\ast}_{[m]}$ to denote the smooth extension of the inverse of $G^{\ast}_{[m]}|_{\mb B_1^d}$ from submanifold $\m M$ to the entire ambient space $\mb R^D$ therein. Recall equation~\eqref{eqn:mub.2}, in the following proofs, we will abuse the notation to use $(\wh{\m J}_{m,l}(f),\wh{\m J}_{m,h}(f),\wh{\m J}_{m,s}(f))$ to denote the term in the right hand side of~\eqref{eqn:low_freq},~\eqref{eqn:high_freq} and~\eqref{eqn:smooth_corr} respectively for any $m\in \mb M=\{m\in [M]\,:\,\mb P_{\mu^*}(X\in S_m^{\dagger})\geq \frac{1}{2}\sqrt{\frac{\log n}{n}}\}$, regardless of the value $\wh{p}_m$, and we will show that $(\wh{\m J}_{m,l}(f),\wh{\m J}_{m,h}(f),\wh{\m J}_{m,s}(f))$ is close to $({\m J}_{m,l}(f),{\m J}_{m,h}(f),{\m J}_{m,s}(f))$ for each $m\in \mb M$.


\subsubsection{Proof of Lemma~\ref{lemma:mani_est}}\label{sec:proof_mani_est}
Fix an $m\in \mb M$.
Recall $\mu^{\ast}|_{\widetilde{S}_m}=[G^{\ast}_{[m]}]_{\#}\nu^{\ast}_{[m]}$, where: 
\begin{enumerate}[topsep=2pt,itemsep=0ex]
    \item $G^\ast_{[m]}|_{\mb B_1^d}$ has a smooth inverse $Q^{\ast}_{[m]}\in C_L^\gamma(\mb R^D; \mb R^d)$ such that $Q_\sm^\ast\circ [G_\sm^\ast|_{\mb B_1^d}]={\rm id}_{\mb B_1^d}$;
    \item $\nu_\sm^\ast$ is a probability distribution supported on $\mb B_1^d$ with $\alpha$-smooth density function (also denoted by $\nu_\sm^\ast$) satisfying $e^{-L}(1-\|z\|_2)^{L_0} \leq \nu_\sm^\ast(z)\leq e^{L}(1-\|z\|_2)^{L_0}$ for all $z\in \mb B_1^d$ for some $L_0\geq 0$.
\end{enumerate}
Let $\wh{f}_\sm= \wh G_\sm\circ \wh Q_\sm \circ G^\ast_\sm:\,\mb B_1^d\to \mb R^D$. Then for any $\eta>0$, we have
\begin{equation}\label{eqn:manifold_loss}
\begin{aligned}
& \mathbb{E}_{\mu^{\ast}}\big[ \|X-\wh{G}_\sm(\wh{Q}_\sm(X))\|_2^{\eta}\,\rho_m(X)\big]\\ 
=&\, \mb P_{\mu^{\ast}}(X\in \widetilde{S}_m)\cdot \int \|x-\wh{G}_\sm(\wh{Q}_\sm(x))\|_2^{\eta} \cdot \rho_m(x)\,\dd \big([G^{\ast}_{[m]}]_{\#}\nu^{\ast}_{[m]}(x)\big)\\
 \overset{(i)}{=} &\,\mb P_{\mu^{\ast}}(X\in \widetilde{S}_m)\cdot \int_{\mb B_1^d} \|G^\ast_\sm(z)-\wh{f}_\sm(z)\|_2^{\eta} \cdot \rho_m\big(G^\ast_\sm(z)\big)\cdot \nu^{\ast}_{[m]}(z)\,\dd z,
\end{aligned}
\end{equation}
where step (i) follows by applying the change of variable of $x=G^\ast_\sm(z)$ and $\nu^\ast_\sm$ is supported on $\mb B_1^d$. 

Let $\wt{n}=n\cdot \mathbb{P}_{\mu^{\ast}}(X\in \widetilde{S}_m)$,  since we only consider those $m\in [M]$ such that  $\mathbb{P}_{\mu^{\ast}}(X\in \widetilde{S}_m)\geq \frac{1}{2}\sqrt{\frac{\log n}{n}}$ (recall $\wt S_m\supset \mb B_{r_m+1/L}(a_m) \supset \mb B_{r_m+1/(2L)}(a_m)=S_m^{\dagger}$ and equation~\eqref{eqn:mub.1}), it holds that $\wt{n} \geq \frac{1}{2}\sqrt{n\log n}$.
Let $ S_m=\mb B_{r_m}(a_m)$, we resort to the following lemma that provides an upper bound on $\|G^\ast_\sm(z)-\wh{f}_\sm(z)\|_2$ for all $z\in \mb B_1^d$ such that $G^\ast_\sm(z) \in S_m$. 
\begin{lemma}\label{lemma:pointwise_error}
  It holds with probability at least $1-n^{-c}$ that for all $z\in \mb B_1^d$ such that $G^\ast_\sm(z) \in S_m$,
  \begin{align}
      \|G^\ast_\sm(z)-\wh{f}_\sm(z)\|_2 \leq C\,\min\Big\{ \Big(\frac{\log \wt n}{\wt n (1-\|z\|_2)^{L_0}}\Big)^{\frac{\beta}{d}},\, \Big(\frac{\log \wt n}{\wt n }\Big)^{\frac{\beta}{d+L_0}} \Big\}.
  \end{align}
\end{lemma}
\noindent A proof of this lemma is provided in Section~\ref{Sec:proof_lemma:pointwise_error}. Another way to state the lemma is that for some constant $c_0$, we have
\begin{align*}
    \|G^\ast_\sm(z)-\wh{f}_\sm(z)\|_2 \leq 
    \begin{cases}
     C\big(\frac{\log \wt n}{\wt n (1-\|z\|_2)^{L_0}}\big)^{\frac{\beta}{d}} & \mbox{if $\|z\|_2 \leq 1- \delta_n$};\\
     C\big(\frac{\log \wt n}{\wt n }\big)^{\frac{\beta}{d+L_0}} & \mbox{if $1- \delta_n\leq \|z\|_2 \leq 1$},
    \end{cases}
\end{align*}
where $\delta_n = c_0\big(\frac{\log \wt n}{\wt n}\big)^{\frac{1}{d+L_0}}$. The increasing pointwise estimation error around the boundary of $\mb B_1^d$ can be explained by less samples around the boundary --- $\nu_\sm^\ast(z)$ has the decay rate $(1-\|z\|_2)^{L_0}$ as $\|z\|$ approaches one.
In addition, notice that, by the condition on $\nu_\sm^\ast$, for those $z$ satisfying $1-\delta_n\leq \|z\|_2\leq 1$ we have $\nu_\sm^\ast(z) \leq e^L (1-\|z\|_2)^{L_0} \leq C \big(\frac{\log \wt n}{\wt n}\big)^{\frac{L_0}{d+L_0}}$.
Combining these two properties with equation~\eqref{eqn:manifold_loss} and the fact that $\rho_m$ is supported on $S_m=\mb B_{r_m}(a_m)$, we finally obtain that for any $\eta\in\big(0,\frac{d}{\beta}\big]$,
\begin{align*}
    & \mathbb{E}_{\mu^{\ast}}\big[ \|X-\wh{G}_\sm(\wh{Q}_\sm(X))\|_2^{\eta}\,\rho_m(X)\big]\\ 
     \leq &\, C\, \mb P_{\mu^{\ast}}(X\in \widetilde{S}_m)\cdot \big(\frac{\log \wt n}{\wt n}\big)^{\frac{\beta\eta}{d}}\, \int_{\|z\|_2 \leq 1- \delta_n} (1-\|z\|_2)^{L_0(1-\frac{\eta\beta}{d})} \,\dd z
     \\
     &\qquad\qquad\qquad+ C\, \mb P_{\mu^{\ast}}(X\in \widetilde{S}_m)\cdot \big(\frac{\log \wt n}{\wt n }\big)^{\frac{\beta\eta+L_0}{d+L_0}}\,\int_{1- \delta_n\leq \|z\|_2 \leq 1} \dd z \\
     \leq &\, C\, \mb P_{\mu^{\ast}}(X\in \widetilde{S}_m)\cdot\big(\frac{\log \wt n}{\wt n}\big)^{\frac{\beta\eta}{d}} \leq C\, \big(\frac{\log n}{n}\big)^{\frac{\beta\eta}{d}},
\end{align*}
where in the last step we used the fact that $\log \wt n\leq \log n$, $\beta\eta\leq d$ and $\wt n = n \cdot \mb P_{\mu^{\ast}}(X\in \widetilde{S}_m)$.
Since $S_m$ is a bounded set, we have that for any $\eta >\frac{d}{\beta}$,
\begin{align*}
    & \mathbb{E}_{\mu^{\ast}}\big[ \|X-\wh{G}_\sm(\wh{Q}_\sm(X))\|_2^{\eta}\,\rho_m(X)\big]\\ 
    &\qquad\qquad \leq  C_1^{\eta-\frac{d}{\beta}}\, \mathbb{E}_{\mu^{\ast}}\big[ \|X-\wh{G}_\sm(\wh{Q}_\sm(X))\|_2^{\frac{d}{\beta}}\,\rho_m(X)\big] \leq 
    C\, C_1^{\eta-\frac{d}{\beta}}\, \frac{\log n}{n}.
\end{align*}

\subsubsection{Proof of Lemma~\ref{lemma:low_freq}}\label{sec:proof_low_freq}
 Let $\zeta=1+\gamma\vee \beta\vee (\frac{d}{2}-\gamma)\vee (\frac{d}{2}-\beta)$; $\phi_{\mf M}\in C^{\zeta}(\mb R)$ and $\phi_{\mf F}\in C^{\zeta}(\mb R)$ be a compactly supported wavelet and scaling function, respectively, for example Daubechies wavelets~\citep{doi:10.1080/03610926.2015.1019144, hutter2020minimax}.
Recall that any $f \in C(\mathbb{R}^D)$ admits the following wavelet expansion (c.f.~Section~\ref{sec:Wavelet_review} for a brief review)
\begin{equation*}
  f(x)=\underbrace{\sum_{k\in \mathbb{Z}^D} b_k\, \phi_k(x)+\sum_{l=1}^{2^D-1} \sum_{j=1}^{J}\sum_{k\in \mathbb{Z}^D} f_{ljk}\, \psi_{ljk}(x)}_{\Pi_Jf} \ \  + \  \  \underbrace{\sum_{l=1}^{2^D-1} \sum_{j=J+1}^{\infty}\sum_{k\in \mathbb{Z}^D} f_{ljk}\, \psi_{ljk}(x)}_{\Pi^\perp_Jf},
\end{equation*}
where recall that for any multi-index $k\in\mb Z^D$, the level zero basis $\phi_k \in C^{\zeta}(\mb R^D)$ is obtained by translating the $D$-fold tensor product $\phi_{\mf F}^{\otimes D}$ by $k$ as $\phi_k(x) = \prod_{i=1}^D \phi_{\mf F}(x_i-k_i)$ for $x=(x_1,\ldots,x_D)\in\mb R^D$, and for any $j\geq 1$, the level $j$ basis $\big\{\psi_{ljk}:\, l\in[2^D-1]\big\}$ with translation $k$ is any ordering of the following $2^D-1$ functions,
\begin{align*}
    \psi_{k}^{j,g}(x)=2^{\frac{D(j-1)}{2}} \,\prod_{i=1}^D \phi_{g_i}\big(2^{j-1}x_i - k_i\big), \quad \forall g\in G = \{\mf F,\,\mf M\}^D\setminus \{(\mf F,\ldots,\mf F)\}.
\end{align*}
In addition, if $f\in C_1^\gamma(\mathbb{R}^D)$, then it holds for all $k\in\mb Z^D$,\, $j\in\mb N_0$ and $l\in[2^D-1]$ that
\begin{align*}
    |b_k| \leq C \quad \mbox{and}\quad |f_{ljk}| \leq C\,2^{-\frac{Dj}{2}-j\gamma}.
\end{align*}
 For ease of notation, we denote $f=f_1+f_2$, where $f_1=\Pi_J f$ is the projection of $f$ onto the first $J$ level wavelet basis, and $f_2=\Pi^\perp_J f$ the collection of all remaining ``higher-frequency'' components. Using this notation, fix an $m\in \mb M$, we can write
\begin{equation*}
\begin{aligned}
&\quad  \m J_{m,l}(f) =\mathbb{E}_{\mu^{\ast}} \big[\rho_m(X) \cdot f_1\big(\wh G_\sm \circ \wh Q_\sm (X)\big)\big]\\
&=\mathbb{E}_{\mu^{\ast}} \bigg[\rho_m(X)\bigg(\sum_{k\in \mathbb{Z}^D} b_k\, \phi_k\big(\wh G_\sm \circ \wh Q_\sm (X)\big)       +\sum_{l=1}^{2^D-1} \sum_{j=1}^{J}\sum_{k\in \mathbb{Z}^D} f_{ljk}\, \psi_{ljk}\big(\wh G_\sm \circ \wh Q_\sm (X)\big)\bigg)\bigg],
\end{aligned}
\end{equation*}
 where the expectation is taken with respect to $X\sim \mu^\ast$ (not the randomness in $\wh G_\sm$ and $\wh Q_\sm$).  Define 
 \begin{equation*}
\begin{aligned}
&\widetilde{\mathbb{S}}=\big\{k \in\mathbb{Z}^D:\, {\rm supp}(\phi_{k})\cap  \wh G_\sm \circ \wh Q_\sm\circ G^{\ast}_\sm(\mb B^d_1)\neq \emptyset\big\};\\
&\widetilde{\mathbb{S}}_{lj}=\big\{k \in \mathbb{Z}^D:\,{\rm supp}(\psi_{ljk})\cap \wh G_\sm \circ \wh Q_\sm\circ G^{\ast}_\sm(\mb B^d_1)\neq \emptyset\big\}.
 \end{aligned}
\end{equation*}
Then there exists a constant $C$ such that $|\widetilde{\mathbb{S}}|\leq C$. Moreover, 
by the Lipschitzness of $\wh G_\sm$, $\wh Q_\sm$, $G^{\ast}_\sm$, and the fact that the support of $\psi_{ljk}$ is contained in $\mb B_{2^{-j}C_0}\big(2^{1-j}k)$ for some finite constant $C_0$, we can get that $|\widetilde{\mathbb{S}}_{lj}|\leq C 2^{dj}$. 
 Under this notation, we have
\begin{equation}\label{eq1lemma2.4}
\begin{aligned}
   & \big| \wh{\m J}_{m,l}(f) - \m J_{m,l}(f)\big| = \bigg|\frac{1}{|I_2|} \sum_{i\in I_2} f_1\big(\wh G_\sm(\wh Q_\sm (X_i))\big)\cdot \rho_m(X_i)-\mathbb{E}_{\mu^{\ast}} \big[f_1\big(\wh G_\sm(\wh Q_\sm(X))\big)\cdot \rho_m(X)\big]\bigg|\\
   &\leq \sum_{k\in \widetilde{\mathbb{S}}} |b_k|\cdot\Big|\mathbb{E}_{\mu^{\ast}}\big[\phi_k\big(\wh G_\sm(\wh Q_\sm(X))\big)\cdot \rho_m(X)\big]-\frac{1}{|I_2|} \sum_{i\in I_2}\phi_k\big(\wh G_\sm(\wh Q_\sm(X_i))\big)\cdot\rho_m(X_i)\Big|\\
   &\qquad+\sum_{l=1}^{2^D-1}\sum_{j=1}^{J}\sum_{k\in\widetilde{\mathbb{S}}_{lj}}|f_{ljk}|\cdot \Big|\mathbb{E}_{\mu^{\ast}}\big[\psi_{ljk}\big(\wh G_\sm(\wh Q_\sm(X))\big)\cdot \rho_m(X)\big] \\
   &\qquad\qquad\qquad\qquad\qquad\qquad\qquad\qquad
   -\frac{1}{|I_2|} \sum_{i\in I_2}\psi_{ljk}\big(\wh G_\sm(\wh Q_\sm(X_i))\big)\cdot\rho_m(X_i)\Big|.
\end{aligned}
\end{equation}
 By a similar union bound argument plus Bernstein's inequality as the proof of~\eqref{eqn:union+Bernstein}, we obtain
  that with probability at least $1-n^{-c}$,
\begin{equation*}
    \begin{aligned}
   &\underset{k \in\widetilde{ \mathbb{S}}}{\sup} \bigg|\mathbb{E}_{\mu^{\ast}}\big[\phi_k\big(\wh G_\sm(\wh Q_\sm(X))\big)\cdot \rho_m(X)\big]-\frac{1}{|I_2|} \sum_{i\in I_2}\phi_k\big(\wh G_\sm(\wh Q_\sm(X_i))\big)\cdot\rho_m(X_i)\bigg| 
   \leq C_1\sqrt{\frac{\log n}{n}},\\
   & \mbox{and} \ \ \bigg|\mathbb{E}_{\mu^{\ast}}\big[\psi_{ljk}\big(\wh G_\sm(\wh Q_\sm(X))\big)\cdot \rho_m(X)\big]
   -\frac{1}{|I_2|} \sum_{i\in I_2}\psi_{ljk}\big(\wh G_\sm(\wh Q_\sm(X_i))\big)\cdot\rho_m(X_i)\bigg|\\
   &\qquad\qquad\qquad\leq C_1\bigg(\frac{\log n}{n}\cdot 2^{\frac{Dj}{2}}+\sqrt{\frac{\log n}{n}}\sqrt{\mathbb{E}_{\mu^{\ast}}\big[\psi_{ljk}^2\big(\wh G_\sm(\wh Q_\sm(X))\big)\cdot \rho_m^2(X)\big]}\bigg)
 \end{aligned}
\end{equation*}
holds for all $1\leq l\leq 2^D-1$, $1\leq j\leq J$, $k \in \widetilde{\mathbb{S}}_{lj}$, 
where the second inequality used the property that $\|\psi_{ljk}\|_\infty \leq C2^{\frac{Dj}{2}}$ so that the right hand side contains the term $2^{\frac{Dj}{2}}$.
By combining these two inequalities with inequality~\eqref{eq1lemma2.4} and using $2^{dJ}\leq (\frac{n}{\log n})^{\frac{d}{2 \alpha+d}}$ and $|f_{ljk}|\leq C 2^{-\frac{Dj}{2}-j\gamma}$, we obtain that with probability at least $1-n^{-c}$,
\begin{equation*} 
\begin{aligned}
   & \quad \sup_{f\in C_1^\gamma(\mb R^D)}\big| \wh{\m J}_{m,l}(f) - \m J_{m,l}(f)\big|\\
   &\leq  C\, \sqrt{\frac{\log n}{n}}+C\,
  \big(\frac{\log n}{n}\big)^{  \frac{2 \alpha+\gamma\wedge d}{2 \alpha+d}}  \\
   &\qquad\qquad  + C\,\sqrt{\frac{\log n}{n}} \cdot \sum_{l=1}^{2^D-1}\sum_{j=1}^{J}  2^{-\frac{Dj}{2}-j\gamma}\sum_{k\in\widetilde{\mathbb{S}}_{lj}}\sqrt{\mathbb{E}_{\mu^{\ast}}\big[\psi_{ljk}^2\big(\wh G_\sm(\wh Q_\sm(X))\big)\cdot \rho_m^2(X)\big]}\\
   &\leq  C\,\sqrt{\frac{\log n}{n}}+ C\,\big(\frac{\log n}{n}\big)^{  \frac{2 \alpha+\gamma\wedge d}{2 \alpha+d}}  \\
   &\qquad\qquad  + C\,\sqrt{\frac{\log n}{n}}\cdot\sum_{l=1}^{2^D-1}\sum_{j=1}^{J}  2^{-\frac{Dj}{2}-j\gamma+\frac{dj}{2}}\sqrt{\sum_{k\in\widetilde{\mathbb{S}}_{lj}}\mathbb{E}_{\mu^{\ast}}\big[\psi_{ljk}^2\big(\wh G_\sm(\wh Q_\sm(X))\big)\cdot \rho_m^2(X)\big]}\, .
\end{aligned}
\end{equation*}
Since for any $x\in G^\ast_\sm(\mb B_1^d)$, there are at most constant many $k$'s in $\widetilde{\mb S}_{lj}$ such that $\psi_{ljk}\big(\wh G_\sm(\wh Q_\sm(x))\neq 0$ and $\|\psi_{ljk}\|_\infty \leq C2^{\frac{Dj}{2}}$, we obtain
\begin{equation*} 
\begin{aligned}
   \quad \mathbb{E}_{\mu^{\ast}}  \Big[\sum_{k\in\widetilde{\mathbb{S}}_{lj}}\psi_{ljk}^2\big(\wh G_\sm(\wh Q_\sm(X))\big)\cdot \rho_m^2(X)\Big] 
   \leq C\, 2^{Dj}.
\end{aligned}
\end{equation*}
Putting all pieces together, we get that with probability at least $1-n^{-c}$,
\begin{align*}
    & \quad \sup_{f\in C_1^\gamma(\mb R^D)}\big| \wh{\m J}_{m,l}(f) - \m J_{m,l}(f)\big| \\
    &\leq C \,\sqrt{\frac{\log n}{n}} + C\,\big(\frac{\log n}{n}\big)^{  \frac{2 \alpha+\gamma\wedge d}{2 \alpha+d}}
    + C\,\sqrt{\frac{\log n}{n}} \cdot \sum_{l=1}^{2^D-1}\sum_{j=1}^{J}  2^{j\big(\frac{d}{2}-\gamma\big)}\\
    & \leq C\, \sqrt{\frac{\log n}{n}}+ C \,\big(\frac{\log n}{n}\big)^{\frac{ \alpha+\gamma}{2 \alpha+d}} ,
\end{align*}
 where we have used $2^{J}\leq (\frac{n}{\log n})^{\frac{1}{2\alpha+d}}$ in the last step.

\subsubsection{Proof of Lemma~\ref{lemma:high_freq}}\label{sec:proof_high_freq}
Fix an $m\in \mb M$, let $\nu_{\sm,\wh Q_{\sm}}^\ast=[\wh Q_\sm]_\# (\rho_m\mu^\ast)$ denote the nonnegative measure (not necessarily a probability measure) obtained as the pushforward measure of $\rho_m\mu^\ast$ via map $\wh Q_\sm$, or the measure such that for any test function $g\in C(\mb R^D)$, the following identity holds,
\begin{align}\label{eqn:change_of_v}
    \int g(z)\, \nu_{\sm,\wh Q_{\sm}}^\ast(z)\,\dd z = \mb E_{\mu^\ast}\big[g(\wh Q_\sm(X))\cdot\rho_m(X)\big].
\end{align}
As we will show, the smoothness regularized estimator $\widetilde \nu_{\sm,\wh Q_{\sm}}$ attempts to approximate $\nu_{\sm,\wh Q_{\sm}}^\ast$. This motivates us to study the regularity of $\nu_{\sm,\wh Q_{\sm}}^\ast$ first.
 The following lemma provides the form of the density function (also denoted as $\nu_{\sm,\wh Q_{\sm}}^\ast$) associated with measure $\nu_{\sm,\wh Q_{\sm}}^\ast$, and shows $\nu_{\sm,\wh Q_{\sm}}^\ast\in C_{C_0}^\alpha(\mb R^d)$ for some sufficiently large constant $C_0$. As a consequence, we can control the growth of its wavelet expansion coefficients in our analysis to follow. A proof of the lemma is provided in Section~\ref{sec:proof_lemma:density_regularity}.
 \begin{lemma}\label{lemma:density_regularity}
    Let $\wh{l}_\sm=\wh{Q}_\sm\circ G^{\ast}_\sm$ and $\Omega_m=\{z\in \mb B_1^d:\,G^{\ast}_\sm(z)\in \mb B_{r_m}(a_m)\}$. Then for all sufficiently large $n$, it holds with probability at least $1-n^{-c}$ that
 \begin{enumerate}[topsep=2pt,itemsep=0ex]
 \item  $\wh{l}_\sm$ is invertible over $\Omega_m$. If we denote its inverse as $\wh{l}_\sm^{-1}$, then $\wh l_\sm ^{-1}\in C_{C_0}^{\alpha+1}\big(\wh{l}_\sm(\Omega_m);\mb R^d)$ for some constant $C_0$;
 \item $\nu_{\sm,\wh Q_{\sm}}^\ast$ admits a density function as
   \begin{equation}\label{def:nQ}
 {\nu}^{\ast}_{\sm,\wh{Q}_\sm}(x)=
 \mathbb{P}_{\mu^{\ast}}(X\in \widetilde{S}_m)\cdot \nu^{\ast}_\sm (\wh{l}_\sm^{-1}(x))\cdot \rho_m\big(G_\sm^{\ast}(\wh{l}_\sm^{-1}(x))\big)\cdot \Big({\rm det}\big({\bold{J}^T_{\wh{l}_\sm^{-1}(x)}}\bold{J}_{\wh{l}_\sm^{-1}}(x)\big)\Big)^{\frac{1}{2}},
 \end{equation}
 for all $x \in  \wh{l}_\sm(\Omega_m)$ and zero elsewhere.
Moreover, the density function belongs to $C^{\alpha}_{C_1}(\mathbb{R}^d)$ for some constant $C_1$.
 \end{enumerate}
 \end{lemma}
 \paragraph{Smoothness regularized estimator $\widetilde \nu_{\sm,\wh Q_{\sm}}$ constructed based on Wavelet expansion:}
 
 \noindent Note that the support of $\nu_{\sm,\wh Q_{\sm}}^\ast$ is contained in $[-L, L]^d$ (since $\wh Q_\sm\in C_L^\beta(\mb R^D;\mb R^d)$), we have the following wavelet expansion,
 \begin{align*}
\num(z)&=\sum_{k \in  \mathbb{S}}a^{\wh{Q}_{[m]}}_k \,\phi_k(z) + \sum_{l=1}^{2^{d}-1}\sum_{j=0}^{+\infty}\sum_{k\in \mathbb{S}_{lj}}\theta_{ljk}^{\wh Q_{[m]}}\,  \psi_{ljk}(z),\quad\mbox{with}\\
 \mathbb{S}&=\big\{k\in\mathbb{Z}^d:\,{\rm supp}(\phi_k)\cap[-L,L]^d\neq \emptyset\big\};\\
\mathbb{S}_{lj}&=\big\{k \in \mathbb{Z}^d: \,{\rm supp}(\psi_{ljk})\cap [-L,L]^d\neq \emptyset\big\},
 \end{align*}
 where
 \begin{equation*}
\begin{aligned}
{a}^{\wh{Q}_{[m]}}_{k}&=\mb E_{\mu^\ast}\big[\phi_k(\wh Q_\sm(X))\cdot \rho_m(X)\big];\\
{\theta}^{\wh{Q}_{[m]}}_{ljk} &= \mb E_{\mu^\ast}\big[\psi_{ljk}(\wh Q_\sm(X))\cdot \rho_m(X)\big].
 \end{aligned}
\end{equation*}
 Here, the expectation is taken with respect to $X\sim \mu^\ast$ (not the randomness in $\wh Q_\sm$).
According to Lemma~\ref{lemma:density_regularity}, we have $\num\in C^{\alpha}_{C_1}(\mathbb{R}^d)$, which implies for all $k\in \mb Z^d$,\, $j\in\mb N_0$ and $l\in[2^d-1]$,
\begin{align*}
    |a_k^{\wh{Q}_\sm}| \leq C \quad \mbox{and}\quad |\theta_{ljk}^{\wh {Q}_\sm}| \leq C\,2^{-\frac{dj}{2}-j\alpha}.
\end{align*}
 Recall that the smoothness regularized estimator $\widetilde \nu_{\sm,\wh Q_{\sm}}$ corresponds to a truncated empirical version of $\nu_{\sm,\wh Q_{\sm}}^\ast$ by truncating the expansion at a finite level $J$ such that $2^J$ is the largest integer not exceeding $(\frac{n}{\log n})^{\frac{1}{2\alpha+d}}$, and replacing the wavelet coefficients with their sample averages,
 \begin{equation}\label{estimator2}
 \widetilde{\nu}_{[m],\wh Q_\sm}(y)=\sum_{k \in  \mathbb{S}}\widetilde{a}^{\wh{Q}_{[m]}}_k \phi_k(y) +\sum_{l=1}^{2^{d}-1}\sum_{j=0}^J\sum_{k\in \mathbb{S}_{lj}}\widetilde{\theta}_{ljk}^{\wh{Q}_\sm} \psi_{ljk}(y),\\
 \end{equation}
 where
\begin{equation*}
\begin{aligned}
\widetilde{a}^{\wh{Q}_\sm}_{k}&=\frac{1}{|I_2|}\sum_{i\in I_2} \phi_k(\wh Q_\sm(X_i))\cdot \rho_m(X_i);\\
\widetilde{\theta}^{\wh{Q}_{[m]}}_{ljk} &= \frac{1}{|I_2|}\sum_{i\in I_2}  \psi_{ljk}(\wh Q_\sm(X_i))\cdot \rho_m(X_i).
 \end{aligned}
\end{equation*}
 
\noindent Note that from
\begin{equation*}
 \begin{aligned}
 &|\phi_k(z)|\leq C, \quad  \int \phi^2_k(z)\, \dd z=1; \quad\mbox{and}\\
&|\psi_{ljk}(z)|\leq c 2^{\frac{dj}{2}},\quad \int \psi_{ljk}^2(z)\,\dd z=1, \quad \forall z\in \mb R^d,
 \end{aligned}
\end{equation*}
and $\num\in C^{\alpha}_{C_1}(\mathbb{R}^d)$, we can get
\begin{align*}
    &\mathbb{E}_{\mu^{\ast}}\big[\phi^2_{k}(\wt Q_\sm(X))\cdot \rho^2_m(X)\big] \leq C_1\int \phi^2_k(z)\cdot \num(z)\,\dd z\leq C_2,\quad\mbox{and}\\
   & \mathbb{E}_{\mu^{\ast}} \big[\psi^2_{ljk}(\wh Q_\sm(X))\rho_m^2(X)\big] \leq C_1 \int \psi_{ljk}^2(z)\cdot \num(z)\,\dd z\leq C_2.
\end{align*}
In addition since each additive component satisfies $|\phi_k(\wh Q_\sm(X_i))\cdot \rho_m(X_i)| \leq C_0$ and $|\psi_{ljk}(\wh Q_\sm(X_i))\cdot \rho_m(X_i)| \leq C\,2^{\frac{dJ}{2}} \leq C_0 \sqrt{n}$ for all $j\in[J]$, we can apply the Bernstein inequality plus a simple union bound argument similar to the proof of~\eqref{eqn:union+Bernstein} to obtain that with probability at least $1-n^{-c}$ (note that since $\wh Q_\sm$ is constructed from data in $I_1$, it is independent of the data in $I_2$),
\begin{equation}\label{eqnwaveletb}
\begin{aligned}
\underset{ k \in  \mathbb{S}}{\sup} \,\bigg|\underbrace{\frac{1}{|I_2|}\sum_{i\in I_2} \phi_k(\wh Q_\sm(X_i))\cdot \rho_m(X_i)}_{\widetilde{a}^{\wh{Q}_\sm}_{k}}  \ - \  \underbrace{\phantom{\frac{1}{|I_2|}\sum_{i\in I_2}} \!\!\!\!\!\!\!\!\!\!\!\!\!\!\!\!\!\!    
\mb E_{\mu^\ast}\big[\phi_k(\wh Q_\sm(X))\cdot \rho_m(X)\big]}_{{a}^{\wh{Q}_\sm}_{k}}  \bigg| &\leq C\sqrt{ \frac{\log n}{n}};\\
\underset{1\leq l\leq 2^d-1\atop{ j\in\mathbb{N},j \leq J\atop k \in \mathbb{S}_{lj}}}{\sup}\bigg|  \underbrace{ \frac{1}{|I_2|}\sum_{i\in I_2}  \psi_{ljk}(\wh Q_\sm(X_i))\cdot \rho_m(X_i)}_{\widetilde{\theta}^{\wh{Q}_{[m]}}_{ljk} }  \  - \   
\underbrace{ \phantom{\frac{1}{|I_2|}\sum_{i\in I_2}} \!\!\!\!\!\!\!\!\!\!\!\!\!\!\!\!\!\!    
\mb E_{\mu^\ast}\big[\psi_{ljk}(\wh Q_\sm(X))\cdot \rho_m(X)\big]}_{{\theta}^{\wh{Q}_{[m]}}_{ljk} } \bigg| &\leq C\sqrt{ \frac{\log n}{n}}.
\end{aligned}
\end{equation}
Recall that from our constructions we have ($f_2=\Pi^\perp_J f$)
\begin{equation*}
\begin{aligned}
 {\m J}_{m,h}(f) &=\mathbb{E}_{\mu^{\ast}} \big[f_2\big(\wh G_\sm(\wh Q_\sm(X))\big)\cdot  \rho_m(X)\big]\\
 &= \int \sum_{l=1}^{2^D-1} \sum_{j=J+1}^{\infty}\sum_{k\in \mathbb{Z}^D} f_{ljk} \cdot \psi_{ljk}(\wh G_\sm (z))\cdot  \num (z)\, \dd z;\\
 \wh{\m J}_{m,h}(f) &=\int f_2(\wh G_\sm(z))\cdot  \widetilde \nu_{\sm,\wh Q_{\sm}}(z)\, \dd z \\
 &=  \int \sum_{l=1}^{2^D-1} \sum_{j=J+1}^{+\infty}\sum_{k\in \mathbb{Z}^D} f_{ljk} \cdot  \psi_{ljk}(\wh G_\sm(z)) \cdot  \widetilde \nu_{\sm,\wh Q_{\sm}}(z)\, \dd z. 
 \end{aligned}
\end{equation*}
Now we expand the two measures relative to the wavelet basis, take the difference, and apply inequality~\eqref{eqnwaveletb} to obtain
that with probability at least $1-n^{-c}$,
\begin{equation*}
\begin{aligned}
&\quad \sup_{f\in C_1^\gamma(\mb R^D)}\, \big|{\m J}_{m,h}(f) - \wh{\m J}_{m,h}(f)\big|\\
&\leq C\, \sqrt{\frac{\log n}{n}} \, \sum_{l=1}^{2^D-1} \sum_{j=J+1}^{+\infty}\sum_{k'\in \mathbb{S}} \sum_{k\in \mathbb{Z}^D}|f_{ljk}| \cdot  \int \left|\psi_{ljk}(\wh G_\sm(z)) \cdot \phi_{k'}(z)\right| \,\dd z \\
&  \quad+ C\, \sqrt{\frac{\log n}{n}}\,  \sum_{l=1}^{2^D-1} \sum_{j=J+1}^{+\infty} \sum_{l'=1}^{2^{d}-1}\sum_{j'=0}^J\sum_{k'\in \mathbb{S}_{l'j'}}\sum_{k\in \mathbb{Z}^D} |f_{ljk}| \cdot\int \left|\psi_{ljk}(\wh G_\sm(z))\cdot \psi_{l'j'k'}(z) \right| \,\dd z\\
&  \quad+   \sum_{l=1}^{2^D-1} \sum_{j=J+1}^{+\infty} \sum_{l'=1}^{2^{d}-1}\sum_{j'=J+1}^{+\infty}\sum_{k'\in \mathbb{S}_{l'j'}}\sum_{k\in \mathbb{Z}^D}|f_{ljk}| \cdot |\theta^{\wh{Q}_{[m]}}_{l'j'k'}|\cdot \int \left|\psi_{ljk}(\wh G_\sm(z))\cdot \psi_{l'j'k'}(z)\right| \,\dd z.
\end{aligned}
\end{equation*}
\noindent For the first term, since for any $z \in \mathbb{R}^d$, there exists a constant $C$ such that $\sum_{k'\in \mathbb{S}} |\phi_{k'}(z)| \leq C$ and $\sum_{k\in \mathbb{Z}^D}|\psi_{ljk}(\wh G_\sm(z))|\leq C\,2^{\frac{Dj}{2}}$ (each $\phi_k$ or $\psi_{ljk}$ is compactly supported), we can get 
\begin{equation*}
\begin{aligned}
 &\quad \sqrt{\frac{\log n}{n}} \, \sum_{l=1}^{2^D-1} \sum_{j=J+1}^{+\infty}\sum_{k'\in \mathbb{S}} \sum_{k\in \mathbb{Z}^D}|f_{ljk}| \cdot  \int \left|\psi_{ljk}(\wh G_\sm(z)) \cdot \phi_{k'}(z)\right| \,\dd z \\
 &\leq C\, \sqrt{\frac{\log n}{n}}\,  \sum_{l=1}^{2^D-1} \sum_{j=J+1}^{+\infty}  2^{-j\gamma} \sum_{k'\in \mathbb{S}}\int |\phi_{k'}(z)|\, \dd z\\
 &\leq C_1\, \sqrt{\frac{\log n}{n}} \, \sum_{j=J+1}^{+\infty} 2^{-j\gamma} \leq  C_2\ \big(\frac{\log n}{n}\big)^{\frac{1}{2}+\frac{ \gamma}{2\alpha+d}},
\end{aligned}  
\end{equation*}
where in the first inequality we used the bound $|f_{ljk}|\leq C\,2^{-\frac{Dj}{2}-j\gamma}$.

\noindent Similarly, for the second term,  using the additional fact that for any $z \in \mathbb{R}^d$, there exists a constant $C$ such that $\sum_{k' \in\mb S_{l'j'}}|\psi_{l'j'k'}(z)|\leq C\,2^{\frac{dj}{2}}$ (each $\psi_{l'j'k'}$ is compactly supported),  we have
\begin{equation*}
\begin{aligned}
 &  \quad\sqrt{\frac{\log n}{n}}\,  \sum_{l=1}^{2^D-1} \sum_{j=J+1}^{+\infty} \sum_{l'=1}^{2^{d}-1}\sum_{j'=0}^J\sum_{k'\in \mathbb{S}_{l'j'}}\sum_{k\in \mathbb{Z}^D} |f_{ljk}| \cdot\int \left|\psi_{ljk}(\wh G_\sm(z))\cdot \psi_{l'j'k'}(z) \right| \,\dd z\\
 &\leq C\, \sqrt{\frac{\log n}{n}}\, \sum_{l=1}^{2^D-1} \sum_{j=J+1}^{+\infty} \sum_{l'=1}^{2^{d}-1}\sum_{j'=0}^J\sum_{k'\in \mathbb{S}_{l'j'}} 2^{-\frac{Dj}{2}-j\gamma}  \int  \sum_{k\in \mathbb{Z}^D}\left|\psi_{ljk}(\wh G_\sm(z))\right|\cdot \left|\psi_{l'j'k'}(z) \right| \,\dd z\\
 &\leq C_1\, \sqrt{\frac{\log n}{n}} \,\sum_{l=1}^{2^D-1} \sum_{j=J+1}^{+\infty} \sum_{l'=1}^{2^{d}-1}\sum_{j'=0}^J 2^{-j\gamma}  \int  \sum_{k'\in \mathbb{S}_{l'j'}} \left|\psi_{l'j'k'}(z) \right| \,\dd z\\
 &\leq C_2\, \sqrt{\frac{\log n}{n}}\, \sum_{j=J+1}^{+\infty} \sum_{j'=0}^J 2^{\frac{dj'}{2}} 2^{-j\gamma}
 \leq C_3\, \big(\frac{\log n}{n}\big)^{\frac{ \alpha+\gamma}{2 \alpha+d}},
 \end{aligned}
\end{equation*}
 where in the first inequality we used the bound $|f_{ljk}|\leq C\,2^{-\frac{Dj}{2}-j\gamma}$.

\noindent For the third term, we similarly get
\begin{equation*}
\begin{aligned}
&  \quad \sum_{l=1}^{2^D-1} \sum_{j=J+1}^{+\infty} \sum_{l'=1}^{2^{d}-1}\sum_{j'=J+1}^{+\infty}\sum_{k'\in \mathbb{S}_{l'j'}}\sum_{k\in \mathbb{Z}^D}|f_{ljk}| \cdot |\theta^{\wh{Q}_{[m]}}_{l'j'k'}|\cdot \int \left|\psi_{ljk}(\wh G_\sm(z))\cdot \psi_{l'j'k'}(z)\right| \,\dd z\\
 &\leq C\, \sum_{l=1}^{2^D-1} \sum_{j=J+1}^{+\infty} \sum_{l'=1}^{2^{d}-1}\sum_{j'=J+1}^{+\infty} 2^{-\frac{Dj}{2}-j\gamma} \cdot 2^{-\frac{dj'}{2}-j' \alpha}
 \int  \sum_{k\in \mathbb{Z}^D} \left|\psi_{ljk}(\wh G_\sm(z))\right|\sum_{k'\in \mathbb{S}_{l'j'}}\left|\psi_{l'j'k'}(z)\right| \,\dd z\\
&\leq C_1\,   \sum_{j=J+1}^{+\infty} \sum_{j'=J+1}^{+\infty} 2^{-j\gamma} \cdot 2^{-j' \alpha}\leq C_2\, (\frac{\log n}{n})^{\frac{ \alpha+\gamma}{2 \alpha+d}},
\end{aligned}
\end{equation*}
  where in the first inequality we used the bound $|f_{ljk}|\leq C\,2^{-\frac{Dj}{2}-j\gamma}$ and $|\theta_{l'j'k'}^{\wh Q_{[m]}}| \leq C\,2^{-\frac{dj'}{2}-j'\alpha}$.
 
 \noindent Putting all pieces together, we can reach the desired inequality.
\paragraph{Smoothness regularized estimator $\widetilde \nu_{\sm,\wh Q_{\sm}}$ constructed based on  kernel density estimation:}
Note that for any $f\in C^{\gamma}_1(\mb R^D)$ and $x\in \mb R^D$
\begin{equation*}
\begin{aligned}
   f_2(x)=\Pi^{\perp}_J f(x)&= \sum_{l=1}^{2^D-1} \sum_{j=J+1}^{\infty}\sum_{k\in \mathbb{Z}^D} f_{ljk}\, \psi_{ljk}(x)\\
   &\overset{(i)}{\leq} C\, \sum_{j=J+1}^{\infty} 2^{-j\gamma}\leq C_1\, \big(\frac{\log n}{n}\big)^{\frac{\gamma}{2\alpha+d}},
   \end{aligned}
\end{equation*}
 where $(i)$ uses that the support of $\psi_{ljk}$ is contained in $B_{2^{-j}C_0}(2^{1-j}k)$, $\|\psi_{ljk}\|_{\infty}\leq C_0 2^{\frac{Dj}{2}}$, and  $|f_{ljk}|\leq C_0\, 2^{-\frac{Dj}{2}-j\gamma}$.
 We claim that it suffices to show that with probability $1-n^{-c}$,
 \begin{equation}\label{KDErate}
  \int \big|\widetilde \nu_{\sm,\wh Q_{\sm}}(z)-  \nu^\ast_{\sm,\wh Q_{\sm}}(z)\big|\,\dd z\leq C\, \big(\frac{\log n}{n}\big)^{\frac{\alpha}{2\alpha+d}}.
 \end{equation}
 Indeed, under~\eqref{KDErate}, we have  for any $f\in C^{\gamma}_1(\mb R^D)$,
 \begin{equation*}
 \begin{aligned}
     &\big|\wh{\m J}_{m,h}(f)-{\m J}_{m,h}(f)\big|\\
     &=\Big|\int f_2(\wh{G}_\sm (y))\widetilde \nu_{\sm,\wh Q_{\sm}}(y) \dd y- \int f_2(\wh{G}_\sm (y)) \nu^\ast_{\sm,\wh Q_{\sm}}(y)\,\dd y\Big|\\
     &\leq \underset{x\in \mb R^D}{\sup}|f_2(x)| \int \big|\widetilde \nu_{\sm,\wh Q_{\sm}}(y)-  \nu^\ast_{\sm,\wh Q_{\sm}}(y)\big|\,\dd y\\
     &\leq C\,\big(\frac{\log n}{n}\big)^{\frac{\alpha+\gamma}{2\alpha+d}}.
      \end{aligned}
 \end{equation*}
 Now we prove claim~\eqref{KDErate} by following the standard analysis for kernel density estimator~\citep{10.1214/aoms/1177704472}. Recall 
 \begin{equation*}
    \wt \nu_{\sm,\wh{Q}_\sm}(y)=\frac{1}{h^d|I_2|} \sum_{i\in I_2}\bigg[ \Big(\prod_{j=1}^d \bar{k}\Big(\frac{\wh{Q}_{\sm,j}(X_i)-y_j}{h}\Big)\Big)\cdot \rho_m(X_i)\bigg],
  \end{equation*}
 since ${\rm supp}(\bar{k})\subset [0,1]$ and $\wh{Q}_{[m]}\in C^{\beta}_L(\mb R^D;\mb R^d)$, we have for any $y\notin [-L-1,L+1]^d$, $  \wt \nu_{\sm,\wh{Q}_\sm}(y)=\nu^\ast_{\sm,\wh{Q}_\sm}(y)=0$. Thus we only need to show that with probability $1-n^{-c}$, for any $y\in[-L-1,L+1]^d$, 
 \begin{equation*} 
    \big|\widetilde \nu_{\sm,\wh Q_{\sm}}(y)-  \nu^\ast_{\sm,\wh Q_{\sm}}(y)\big|\leq C\, \big(\frac{\log n}{n}\big)^{\frac{\alpha}{2\alpha+d}}.
 \end{equation*}
 Firstly, we bound the difference between the expectation of $\widetilde \nu_{\sm,\wh Q_{\sm}}(y)$ and $\nu^\ast_{\sm,\wh Q_{\sm}}(y)$.
\begin{equation*}
     \begin{aligned}
         &\Bigg|\mathbb{E}_{\mu^{\ast}} \left[\Big( \frac{1}{h^d} \prod_{j=1}^d \bar{k}\Big(\frac{\wh{Q}_{\sm,j}(X)-y_j}{h}\Big)\Big) \cdot\rho_m(X)\right]\\
         &-\mathbb{P}_{\mu^{\ast}}(X\in \widetilde{S}_m)\cdot \nu^{\ast}_\sm (\wh{l}_\sm^{-1}(y))\cdot \rho_m\big(G_\sm^{\ast}(\wh{l}_\sm^{-1}(y))\big)\cdot \Big({\rm det}\big({\bold{J}^T_{\wh{l}_\sm^{-1}(y)}}\bold{J}_{\wh{l}_\sm^{-1}}(y)\big)\Big)^{\frac{1}{2}}\Bigg|\\
         &\overset{(i)}{=}\mathbb{P}_{\mu^{\ast}}(X\in \widetilde{S}_m)\cdot\Bigg|\nu^{\ast}_\sm (\wh{l}_\sm^{-1}(y))\cdot \rho_m\big(G_\sm^{\ast}(\wh{l}_\sm^{-1}(y))\big)\cdot \Big({\rm det}\big({\bold{J}^T_{\wh{l}_\sm^{-1}(y)}}\bold{J}_{\wh{l}_\sm^{-1}}(y)\big)\Big)^{\frac{1}{2}}\\
         &-\int\frac{1}{h^d} \Big(\prod_{j=1}^d \bar{k}\big(\frac{z_j-y_j}{h}\big)\Big) \cdot\nu^{\ast}_\sm (\wh{l}_\sm^{-1}(z))\cdot \rho_m\big(G_\sm^{\ast}(\wh{l}_\sm^{-1}(z))\big)\cdot \Big({\rm det}\big({\bold{J}^T_{\wh{l}_\sm^{-1}(z)}}\bold{J}_{\wh{l}_\sm^{-1}}(z)\big)\Big)^{\frac{1}{2}}\,\dd z\Bigg|\\
         &\overset{(ii)}{=} \mathbb{P}_{\mu^{\ast}}(X\in \widetilde{S}_m)\cdot\Big| \int \Big(\prod_{j=1}^d \bar{k}(t_j) \Big)\cdot \big(\nu^{\diamond}(ht+\widetilde y)-\nu^{\diamond}(\widetilde y)\big) \,\dd t\Big| \\
         &\overset{(iii)}{\leq }\mathbb{P}_{\mu^\ast}(X\in \wt{S}_m)\cdot\bigg| \int \Big(\prod_{j=1}^d \bar{k}(t_j) \Big)\cdot \sum_{\eta\in \mb N_0^d\atop 1\leq |\eta|\leq \lfloor \alpha\rfloor} (v^{\diamond})^{(\eta)}(\widetilde y)\cdot(ht)^{\eta} \,\dd t\bigg|+C\, h^{\alpha}\\
         &\overset{(iiii)}{=}C\,  h^{ \alpha},
     \end{aligned}
 \end{equation*}
 where  $(i)$ uses $\mu^\ast|_{\wt{S}_m}=[G^\ast_{[m]}]_{\#}\nu^\ast_{[m]}$;   $(ii)$ let $t=(t_1,t_2,\cdots,t_d)$ with $t_j=\frac{z_j-y_j}{h}$,  $\nu^{\diamond}(y)=\nu^{\ast}_\sm (\wh{l}_\sm^{-1}(y))\cdot \rho_m\big(G_\sm^{\ast}(\wh{l}_\sm^{-1}(y))\big)\cdot \Big({\rm det}\big({\bold{J}^T_{\wh{l}_\sm^{-1}(y)}}\bold{J}_{\wh{l}_\sm^{-1}}(y)\big)\Big)^{\frac{1}{2}}$ and uses the fact $\int \bar{k}(t) \,\dd t=1$; $(iii)$ uses the conclusion of Lemma~\ref{lemma:density_regularity} and the fact that $\bar{k}$ is compactly supported; $(iiii)$ uses  $\int x^j k(x)=0$ for $j\in [\lceil \alpha\rceil]$. Then it remains to bound the difference between $\widetilde \nu_{\sm,\wh Q_{\sm}}(y)$ and its expectation. Since for any $y\in [-L-1,L+1]^d$,
 \begin{equation*}
 \begin{aligned}
     &\frac{1}{h^{2d}} \int \Big(\prod_{j=1}^d \bar{k}^2\Big(\frac{\wh{Q}_{\sm,j}(X)- y_j}{h}\Big) \Big)\cdot \rho_m^2 (X)\, \dd \mu^{\ast}\\
     &=\mathbb{P}_{\mu^{\ast}}(X\in \widetilde{S}_m)\\
     &\cdot\int_{z\in \mb B_{\sqrt{d}h}(y)}\frac{1}{h^{2d}}  \Big(\prod_{j=1}^d \bar{k}^2\big(\frac{z_j-{y}_j}{h}\big) \Big)\cdot \nu^{\ast}_\sm (\wh{l}_\sm^{-1}(z))\cdot \rho_m^2\big(G_\sm^{\ast}(\wh{l}_\sm^{-1}(z))\big)\cdot \Big({\rm det}\big({\bold{J}^T_{\wh{l}_\sm^{-1}(z)}}\bold{J}_{\wh{l}_\sm^{-1}}(z)\big)\Big)^{\frac{1}{2}}\,\dd z\\
     &\leq C\, \frac{1}{h^d},
     \end{aligned}
 \end{equation*}
 and for any $x\in \mb R^D$,
 \begin{equation*}
    \frac{1}{h^d} 
    \cdot\Big(\prod_{j=1}^d \bar{k} \Big(\frac{\wh{Q}_{\sm,j}(x)- y_j}{h} \Big) \Big)\cdot \rho_m (x)\leq C\, \frac{1}{h^d}.
 \end{equation*}
Now let $\m N_{\frac{1}{n^2}}\subset  [-L-1,L+1]^d$ be a $\frac{1}{n^2}$-covering set $[-L-1,L+1]^d$, where $|N_{\frac{1}{n^2}}|\leq C\, n^{2d}$, then by a similar union bound argument plus Bernstein's inequality as the proof of~\eqref{eqn:union+Bernstein}, it holds with probability at least $1-n^{-c}$ that for any $\widetilde{y} \in \m N_{\frac{1}{n^2}}$, it satisfies that 
\begin{equation*}
     \left|\mathbb{E}_{\mu^{\ast}} \left[\frac{1}{h^d} \Big(\prod_{j=1}^d \bar{k} \Big(\frac{\wh{Q}_{\sm,j}(X)-\widetilde{y}_j}{h} \Big) \Big)\cdot  \rho_m(X)\right]-\wt{\nu}_{\sm,\wh{Q}_\sm}(\widetilde{y})\right|\leq C\,  \sqrt{\frac{\log n}{n}} h^{-\frac{d}{2}}+\frac{\log n}{n} h^{-d}.
 \end{equation*}
 Then by the uniformly Lipschitzness of $\bar{k}(x)$ and $h=\big(\frac{\log n}{n}\big)^{\frac{1}{2 \alpha+d}}$, it holds with probability at least $1-n^{-c}$ that 
 \begin{equation*}
 \begin{aligned}
    &\underset{y\in [-L-1,L+1]^d}{\sup} \left|\mathbb{E}_{\mu^{\ast}} \left[\frac{1}{h^d} \Big(\prod_{j=1}^d \bar{k} \Big(\frac{\wh{Q}_{\sm,j}(X)-{y}_j}{h} \Big)  \Big)\cdot \rho_m(X)\right]-\wt{\nu}_{\sm,\wh{Q}_\sm}(y)\right|\\
    &\leq \underset{\wt y\in \m N_{\frac{1}{n^2}}}{\sup} \left|\mathbb{E}_{\mu^{\ast}} \left[\frac{1}{h^d} \Big(\prod_{j=1}^d \bar{k} \Big(\frac{\wh{Q}_{\sm,j}(X)-\wt{y}_j}{h} \Big) \Big)\cdot  \rho_m(X)\right]-\wt{\nu}_{\sm,\wh{Q}_\sm}(\wt y)\right|+C\, h^{-1-d}\frac{1}{n^2}\\
    &\leq C_1\,\big(\frac{\log n}{n}\big)^{\frac{ \alpha}{2 \alpha+d}}.
     \end{aligned}
 \end{equation*}
 We can then obtain the desired result by putting all pieces together.
\subsubsection{Proof of Lemma~\ref{lemma:smooth_corr}}\label{sec:proof_smooth_corr}
  For any $m\in\mb M$, by applying the Taylor expansion to any $f \in C_1^{\gamma}(\mathbb{R}^D)$, we obtain that
 \begin{equation}\label{uplargebeta0}
\begin{aligned}
 &\sup_{f\in C_1^\gamma(\mb R^D)} \bigg| \underbrace{\mathbb{E}_{\mu^{\ast}} [f(X)\cdot \rho_m(X)] - \mathbb{E}_{\mu^{\ast}} \big[f\big(\wh{G}_\sm(\wh{Q}_\sm (X))\big)\cdot \rho_m(X)\big]}_{\m J_{m,s}} \\
& \qquad\qquad\qquad +\sum_{j\in \mathbb{N}_0^D\atop 1\leq |j|\leq \lfloor\gamma\rfloor}\mathbb{E}_{\mu^{\ast}}\Big[\,\frac{1}{j!}\,f^{(j)}(X)\,\big[\wh{G}_\sm(\wh{Q}_\sm(X))-X\big]^{j}\, \rho_m(X)\Big]\bigg|\\
& \qquad\qquad
\leq C\, \mathbb{E}_{\mu^{\ast}}\big[\|\wh{G}_\sm(\wh{Q}_\sm(X))-X\|_2^{\gamma}\,\rho_m(X)\big] 
\leq C\, \Big(\frac{\log n}{n}\Big)\vee \Big(\frac{\log n}{n}\Big)^{\frac{\gamma\beta}{d}}
  \end{aligned}
\end{equation}
holds with probability at least $1-n^{-c}$ with respect to the randomness in $(\wh G_\sm, \wh Q_\sm)$ (or samples from $I_1$), where the last inequality is due to Lemma~\ref{lemma:mani_est}. In the rest of the proof, we restrict ourselves to the high probability event where the inequality in Lemma~\ref{lemma:mani_est} holds for any $\eta\in [\lceil2\gamma\rceil]$.

Since $-\wh {\m J}_{m,s}(f)$ is the sample version (samples from $I_2$) of the sum in the second line of~\eqref{uplargebeta0}, it remains to derive a high probability bound to
\begin{equation*} 
\begin{aligned}
& V_n(j) :\,=\underset{f\in C^{\gamma}_1(\mb R^D)}{\sup} \Big|\,\mathbb{E}_{\mu^{\ast}}\big[f^{(j)}(X)\, \big(\wh{G}_\sm(\wh{Q}_\sm(X))-X \big)^{j}\, \rho_m(X)\big]\\
&\qquad\qquad\qquad\qquad\quad - \ \underbrace{\frac{1}{|I_2|}\sum_{i\in I_2} f^{(j)}(X_i)\,\big(\wh{G}_\sm(\wh{Q}_\sm(X_i))-X_i\big)^{j}\, \rho_m(X_i)}_{\small \mbox{sample version}}\,\Big|,
\end{aligned}
\end{equation*}
for each $j\in \mathbb{N}_0^D$ and $1\leq |j| \leq \lfloor \gamma\rfloor$.
Note that we can bound the second moment of the supreme (over $f\in C^{\gamma}_1(\mb R^D)$) of each term inside the sum above as
\begin{align*}
   &\quad \mb E_{\mu^\ast}\Big[ \underset{f\in C^{\gamma}_1(\mb R^D)}{\sup} \big|f^{(j)}(X_i)\big|^2\cdot\big|\big(\wh{G}_\sm(\wh{Q}_\sm(X_i))-X_i\big)^{j}\big|^2\, \rho^2_m(X_i)\Big]\\
    &\leq C\, \mb E_{\mu^\ast}\big[\|\wh{G}_\sm(\wh{Q}_\sm(X))-X\|_2^{2|j|}\,\rho_m(X)\big] \leq C_1\, \Big(\frac{\log n}{n}\Big)\vee \Big(\frac{\log n}{n}\Big)^{\frac{2|j|\beta}{d}},
\end{align*}
where the last step is due to Lemma~\ref{lemma:mani_est}. In addition, each term is almost surely bounded by a constant $C_0$. Therefore, we can apply the Talagrand concentration inequality to obtain that for all $t\geq 0$,
\begin{align*}
    \mb P\bigg\{ V_n(j)\geq \mb E_{\mu^\ast} [V_n(j)] + C\, \sqrt{\frac{t}{n}} \cdot \Big[\sqrt{\frac{\log n}{n}}\vee \Big(\frac{\log n}{n}\Big)^{\frac{|j|\beta}{d}}\Big]  + \frac{C\,t}{n}\,\bigg\} \leq 2\,e^{-t}.
\end{align*}
Therefore, with probability at least $1-n^{-c}$, it holds that
\begin{align*}
     V_n(j)\leq \mb E_{\mu^\ast} [V_n(j)] + C\, \sqrt{\frac{\log n}{n}} \cdot \Big[\sqrt{\frac{\log n}{n}}\vee \Big(\frac{\log n}{n}\Big)^{\frac{|j|\beta}{d}}\Big].
\end{align*}
It remains to bound the expectation $\mb E_{\mu^\ast} [V_n(j)]$, which by the standard symmetrization argument, satisfies
\begin{align*}
    \mb E_{\mu^\ast} [V_n(j)] \leq \frac{2}{\sqrt{|I_2|}}\, \mb E \bigg[ \,\underset{f\in C^{\gamma}_1(\mb R^D)}{\sup} \,\bigg|\frac{1}{\sqrt{|I_2|}}\sum_{i\in I_2} \varepsilon_i\, f^{(j)}(X_i)\,\big(\wh{G}_\sm(\wh{Q}_\sm(X_i))-X_i\big)^{j}\, \rho_m(X_i) \bigg|\bigg],
\end{align*}
where $\{\varepsilon_i\}_{i=1}^n$ are $n$ i.i.d.~Rademacher random variables, i.e.~$\mb P(\varepsilon_i=1)=\mb P(\varepsilon_i=-1)=0.5$.

Since given $\{X_i\}_{i\in I_1\cup I_2}$, the stochastic process inside the supreme is a sub-Gaussian process with intrinsic metric 
 \begin{equation*}
 \begin{aligned}
& \quad d_{n,j}(f, \widetilde f)=\sqrt{\frac{1}{|I_2|}\sum_{i\in I_2} \Big[\big(f^{(j)}(X_i)-\widetilde{f}^{(j)}(X_i)\big)\cdot \big(\wh{G}_\sm(\wh{Q}_\sm(X_i))-X_i\big)^{j}\cdot \rho_m(X_i)\Big]^2}\\
&\leq C\,  \bigg(\,\underset{z\in \mb B_1^d } {\sup} \, \big|f^{(j)}(G^{\ast}_\sm(z))-\widetilde{f}^{(j)}(G^{\ast}_\sm(z))\big| \bigg)\cdot  \sqrt{ \frac{1}{|I_2|}\sum_{i\in I_2}  \big\|\wh{G}_\sm(\wh{Q}_\sm(X_i))-X_i\big\|_2^{2|j|}\cdot\rho^2_m(X_i)}.
 \end{aligned}
\end{equation*}
 
We then state the following lemma for bounding the covering entropy of smooth functions with respect to metrics that only concern evaluations of functions in low-dimensional submanifolds. Its proof is provided in Section~\ref{sec:proof_coveringmanifold}.
  \begin{lemma}\label{coveringmanifold}
Let $\mathcal{X}_G=\big\{x\in \mathbb{R}^D:\, x=G(z), z\in \mb B_1^d\big\}$ be a $d$-dimensional submanifold induced by a Lipschitz continuous map $G:\,\mb R^d\to \mb R^D$, then it holds for any $\widetilde \gamma>0$ that
\begin{equation*}
\log N \big(C_1^{\widetilde{\gamma}}(\mathbb{R}^D), \,\|\cdot\|_{L^{\infty}(\mathcal{X}_{G})},\,\epsilon \big)\leq C\,\epsilon^{-\frac{d}{\widetilde{\gamma}}}, \quad \forall \epsilon>0,
\end{equation*}
where $N(\mathcal{F},\,\widetilde{d}, \,\epsilon)$ denotes the $\epsilon$-covering number of function space $\mathcal{F}$ with respect to pseudo-metric $\widetilde{d}$, and $\|f\|_{L^{\infty}(\mathcal{X}_{G})}=\underset{x\in \mathcal{X}_{G}} {\sup }\big|f(x)\big|$ denotes the functional supreme norm constrained on set $\m X_G$.
\end{lemma}
\noindent Let $\mathcal{K}_j=\sqrt{\frac{1}{|I_2|}\sum_{i\in I_2}  \big\|\wh{G}_\sm(\wh{Q}_\sm(X_i))-X_i\big\|_2^{2|j|} 
\cdot \rho_m(X_i)^2}$. From Lemma~\ref{coveringmanifold}, we can get 
$$\log N\big(C_1^{{\gamma}}(\mathbb{R}^D),\, d_{n,j},\, \epsilon\big)\leq C\, \Big(\frac{ \mathcal{K}_j}{\epsilon}\Big)^{\frac{d}{\gamma-|j|}},\quad\forall \epsilon>0,$$
where we used the fact that for all $f\in C_1^\gamma(\mb R^D)$, $f^{(j)}$ belongs to $C_1^{\gamma-|j|}(\mb R^D)$.
By Dudley’s entropy integral bound for bounding the expectation of the supreme of sub-Gaussian processes (see for example, Theorem 5.22 of~\cite{wainwright_2019}), we obtain
\begin{equation*}
    \begin{aligned}
    & \mb E_{\mu^\ast} [V_n(j)] \leq C\,\mathbb{E}_{\mu^\ast} \bigg[\underset{\delta\in [0,1]}{\min} \bigg\{\,\delta+\frac{1}{\sqrt{n}}\int_{\delta}^{ \m K_j} \left(\frac{ \m K_j}{\epsilon}\right)^{\frac{d}{2(\gamma-|j|)}}\, \dd \epsilon \bigg\}\bigg].
    \end{aligned}
\end{equation*}
By choosing $\delta=\m K_j\cdot\big(n^{-\frac{\gamma-|j|}{d}}\vee \frac{\log n}{\sqrt{n}}\big)$, we obtain
\begin{equation*}
    \begin{aligned}
   \mb E_{\mu^\ast} [V_n(j)] \leq C\, \Big[\Big(\frac{\log n}{n}\Big)^{\frac{|j|\beta}{d}}\vee \sqrt{\frac{\log n}{n}}\,\Big]\cdot\Big[n^{-\frac{\gamma-|j|}{d}}\vee \frac{\log n}{\sqrt{n}}\Big],
    \end{aligned}
\end{equation*}
where the last step used the fact that 
$$\mb E[\m K_j]\leq \sqrt{\mb E[\m K_j^2]}=\sqrt{\mb {E} \big[\big\|\wh{G}(\wh{Q}(X))-X\big\|_2^{2|j|} 
\rho_m^2(X)\big]}\leq C\,\Big(\frac{\log n}{n}\Big)^{\frac{|j|\beta}{d}}\vee \sqrt{\frac{\log n}{n}}$$
according to Lemma~\ref{lemma:mani_est}.
 Finally, putting all pieces together, we obtain by a simple union bound over all $j\in \mathbb{N}_0^D$, $1\leq |j|\leq \lfloor\gamma\rfloor$ (at most $C\,D^{\lfloor\gamma\rfloor}$ many) that with probability at least $1-c_1n^{-c_2}$,
\begin{equation}\label{uplargebeta1}
\begin{aligned}
 \underset{j\in \mathbb{N}_0^D\atop 1\leq |j|\leq \lfloor\gamma\rfloor}{\sup} V_n(j)
 &\leq C\, \underset{j\in \mathbb{N}_0^D\atop 1\leq |j|\leq \lfloor\gamma\rfloor}{\sup} \,\Big[\Big(\frac{\log n}{n}\Big)^{\frac{|j|\beta}{d}}\vee \sqrt{\frac{\log n}{n}}\,\Big]\cdot\Big[n^{-\frac{\gamma-|j|}{d}}\vee \frac{\log n}{\sqrt{n}}\Big]\\
 &\leq C\, \Big(\frac{\log n}{n}\Big)^{\frac{\beta+\gamma-1}{d}}\vee \sqrt{\frac{\log n}{n}}.
   \end{aligned}
\end{equation}

\section{Technical Results and Proofs}\label{App:technical}
In this subsection, we collect all technical results used in the proofs and their proofs.

\subsection{Lemma~\ref{lemmaempiricalmean} and its proof}
\begin{lemma}\label{lemmaempiricalmean}
 There exists a constant $C_1$ such that for any $\mu^\ast\in \m P^\ast(d,D,\alpha,\beta,L^\ast)$, it holds with probability larger than $1-n^{-1}$ that 
 \begin{equation*}
    \sup_{f\in C_1^\gamma(\mb R^D)}\Big| \mb E_{\mu^\ast}[f(X)] - \frac{1}{n} \sum_{i=1}^n f(X_i)\Big| \leq C_1 \,\Big(\sqrt{ \frac{\log n}{n}} \vee n^{-\frac{\gamma}{d}}\Big).
 \end{equation*}
 Moreover, there exist $\mu^\ast\in \m P^\ast(d,D,\alpha,\beta,L^\ast)$ and constant $C_2$ such that  
  \begin{equation*}
   \mb E  \bigg[\sup_{f\in C_1^\gamma(\mb R^D)}\Big| \mb E_{\mu^\ast}[f(X)] - \frac{1}{n} \sum_{i=1}^n f(X_i)\Big|\bigg] \geq C_2 \,\Big(\sqrt{ \frac{1}{n}} \vee \frac{n^{-\frac{\gamma}{d}}}{\log n}\Big).
 \end{equation*}
\end{lemma}
\begin{proof}
 We first prove the upper bound. We provide two proofs here: one is based on the usual chaining technique in the empirical process theory; the other is based on embedding the discrimator space $C_1^\gamma(\mb R^D)$ into Besov space $B_{\infty,\infty}^\alpha(\mb R^D)$ and truncating the wavelet expansion of $f$ to a proper degree.
 \begin{itemize}
     \item{Proof based on chaining technique:} Firstly by standard symmetrization argument, it satisfies 
     \begin{equation*}
         \mb E_{\mu^\ast} \bigg[\sup_{f\in C_1^\gamma(\mb R^D)}\Big| \mb E_{\mu^\ast}[f(X)] - \frac{1}{n} \sum_{i=1}^n f(X_i)\Big|\bigg]\leq \frac{2}{\sqrt{n}}\mb E  \bigg[\sup_{f\in C_1^\gamma(\mb R^D)}\Big| \frac{1}{\sqrt{n}} \sum_{i=1}^n \varepsilon_i f(X_i)\Big|\bigg]
     \end{equation*}
      where $\{\varepsilon_i\}_{i=1}^n$ are $n$ i.i.d.~Rademacher random variables. Given $\{X_i\}_{i\in [n]}$, the stochastic process inside the supreme is a sub-Gaussian process with intrinsic metric 
      \begin{equation*}
          d_n(f,\widetilde f)=\sqrt{\frac{1}{n}\sum_{i=1}^n (f(X_i)-\widetilde f(X_i))^2}\leq \underset{x\in {\rm supp}(\mu^\ast)}{\sup}|f(x)-\widetilde f(x)|.
      \end{equation*}
By Lemma~\ref{lemmaboundaryless}, there exist  a constant $M$ and $\beta$-smooth functions $\{G_{[m]}\}_{m\in [M]}$ such that ${\rm supp}(\mu^\ast)\subset \bigcup_{m=1}^M G_{[m]}(\mb B_1^d)$. Therefore by Lemma~\ref{coveringmanifold}, we can get 
$$\log N\big(C_1^{{\gamma}}(\mathbb{R}^D),\, d_{n},\, \epsilon\big)\leq C\, \Big(\frac{ 1}{\epsilon}\Big)^{\frac{d}{\gamma}},\quad\forall \epsilon>0.$$
 By Dudley’s entropy integral bound for bounding the expectation of the supreme of sub-Gaussian processes, we obtain
\begin{equation*}
    \begin{aligned}
    &  \mb E \bigg[\sup_{f\in C_1^\gamma(\mb R^D)}\Big| \mb E_{\mu^\ast}[f(X)] - \frac{1}{n} \sum_{i=1}^n f(X_i)\Big|\bigg] \leq C\, \underset{\delta\in [0,1]}{\min} \bigg\{\,\delta+\frac{1}{\sqrt{n}}\int_{\delta}^{1} \left(\frac{1}{\epsilon}\right)^{\frac{d}{2\gamma}}\, \dd \epsilon \bigg\}.
    \end{aligned}
\end{equation*}
By choosing $\delta=n^{-\frac{\gamma}{d}}\vee \frac{\log n}{\sqrt{n}}$, we obtain
\begin{equation*}
    \begin{aligned}
  \mb E \bigg[\sup_{f\in C_1^\gamma(\mb R^D)}\Big| \mb E_{\mu^\ast}[f(X)] - \frac{1}{n} \sum_{i=1}^n f(X_i)\Big|\bigg]\leq C\, \Big[n^{-\frac{\gamma}{d}}\vee \frac{\log n}{\sqrt{n}}\Big].
    \end{aligned}
\end{equation*}
 Then by Talagrand concentration inequality, similar as the proof for Lemma~\ref{lemma:smooth_corr} in Section~\ref{sec:proof_smooth_corr}, we can obtain that it holds with probability at least $1-n^{-1}$ that 
\begin{equation*}
    \begin{aligned}
 \sup_{f\in C_1^\gamma(\mb R^D)}\Big| \mb E_{\mu^\ast}[f(X)] - \frac{1}{n} \sum_{i=1}^n f(X_i)\Big| \leq C_1\, \Big[n^{-\frac{\gamma}{d}}\vee \frac{\log n}{\sqrt{n}}\Big].
    \end{aligned}
\end{equation*}
  \item{Proof based on Wavelet expansion:}
 Recall $f\in C^{\gamma}_1(\mb R^D)$ admits the following wavelet expansion
 \begin{equation*}
  f(x)=\underbrace{\sum_{k\in \mathbb{Z}^D} b_k\, \phi_k(x)+\sum_{l=1}^{2^D-1} \sum_{j=1}^{J}\sum_{k\in \mathbb{Z}^D} f_{ljk}\, \psi_{ljk}(x)}_{\Pi_Jf} \ \  + \  \  \underbrace{\sum_{l=1}^{2^D-1} \sum_{j=J+1}^{\infty}\sum_{k\in \mathbb{Z}^D} f_{ljk}\, \psi_{ljk}(x)}_{\Pi^\perp_Jf},
\end{equation*}
and here we choose $J$ to be the largest integer such that $2^J \leq (\frac{n}{\log n})^{\frac{1}{d}}$. Then it holds for all $k\in\mb Z^D$,\, $j\in\mb N_0$ and $l\in[2^D-1]$ that
\begin{align*}
    |b_k| \leq C \quad \mbox{and}\quad |f_{ljk}| \leq C\,2^{-\frac{Dj}{2}-j\gamma}.
\end{align*}
We can then write 
\begin{equation*}
\begin{aligned}
    \sup_{f\in C_1^\gamma(\mb R^D)}\Big| \mb E_{\mu^\ast}[f(X)] - \frac{1}{n} \sum_{i=1}^n f(X_i)\Big|&\leq  \sup_{f\in C_1^\gamma(\mb R^D)}\Big| \mb E_{\mu^\ast}[\Pi_Jf(X)] - \frac{1}{n} \sum_{i=1}^n \Pi_Jf(X_i)\Big|\\
    &+\sup_{f\in C_1^\gamma(\mb R^D)}\Big[\big| \mb E_{\mu^\ast}[\Pi_J^{\perp}f(X)] \big|+\Big| \frac{1}{n} \sum_{i=1}^n \Pi_J^\perp f(X_i)\Big|\Big].
    \end{aligned}
\end{equation*}
Let $\m M={\rm supp}(\mu^\ast)$, define 
 \begin{equation*}
\begin{aligned}
&\widetilde{\mathbb{S}}=\big\{k \in\mathbb{Z}^D:\, {\rm supp}(\phi_{k})\cap \m M\neq \emptyset\big\};\\
&\widetilde{\mathbb{S}}_{lj}=\big\{k \in \mathbb{Z}^D:\,{\rm supp}(\psi_{ljk})\cap \m M\neq \emptyset\big\}.
 \end{aligned}
\end{equation*}
Then there exists a constant $C$ such that $|\widetilde{\mathbb{S}}|\leq C$. Moreover, as by Lemma~\ref{lemmaboundaryless}, $\m M\subset \bigcup_{m=1}^M G_{[m]}(\mb B_1^d)$ for some $\beta$-smooth functions $\{G_{[m]}\}_{m\in [M]}$ and the support of $\psi_{ljk}$ is contained in $\mb B_{2^{-j}C_0}\big(2^{1-j}k)$ for some finite constant $C_0$, we can get that $|\widetilde{\mathbb{S}}_{lj}|\leq C 2^{dj}$. 

  Under this notation, we have
\begin{equation}\label{eq1lemma31}
\begin{aligned}
   & \sup_{f\in C_1^\gamma(\mb R^D)}\Big| \mb E_{\mu^\ast}[\Pi_Jf(X)] - \frac{1}{n} \sum_{i=1}^n \Pi_Jf(X_i)\Big|\\
   &\leq \sum_{k\in \widetilde{\mathbb{S}}} |b_k|\cdot\Big|\mathbb{E}_{\mu^{\ast}} [\phi_k(X)]-\frac{1}{n} \sum_{i=1}^n\phi_k(X_i)\Big|\\
   &\qquad+\sum_{l=1}^{2^D-1}\sum_{j=1}^{J}\sum_{k\in\widetilde{\mathbb{S}}_{lj}}|f_{ljk}|\cdot \Big|\mathbb{E}_{\mu^{\ast}} [\psi_{ljk}(X)]-\frac{1}{n} \sum_{i=1}^n\psi_{ljk}(X_i)\Big|.
\end{aligned}
\end{equation}
 By a similar union bound argument plus Bernstein's inequality as the proof of~\eqref{eqn:union+Bernstein}, we obtain
  that with probability at least $1-n^{-1}$,
\begin{equation*}
    \begin{aligned}
   &\underset{k \in\widetilde{ \mathbb{S}}}{\sup} \Big|\mathbb{E}_{\mu^{\ast}} [\phi_k(X)]-\frac{1}{n} \sum_{i=1}^n\phi_k(X_i)\Big|
   \leq C_1\sqrt{\frac{\log n}{n}},\\
   & \mbox{and} \ \  \Big|\mathbb{E}_{\mu^{\ast}} [\psi_{ljk}(X)]-\frac{1}{n} \sum_{i=1}^n\psi_{ljk}(X_i)\Big|\leq C_1\bigg(\frac{\log n}{n}\cdot 2^{\frac{Dj}{2}}+\sqrt{\frac{\log n}{n}}\sqrt{\mathbb{E}_{\mu^{\ast}}\big[\psi_{ljk}^2(X)\big]}\bigg)
 \end{aligned}
\end{equation*}
holds for all $1\leq l\leq 2^D-1$, $1\leq j\leq J$, $k \in \widetilde{\mathbb{S}}_{lj}$, 
where the second inequality used the property that $\|\psi_{ljk}\|_\infty \leq C2^{\frac{Dj}{2}}$ so that the right hand side contains the term $2^{\frac{Dj}{2}}$.
By combining these two inequalities with inequality~\eqref{eq1lemma31} and using $2^{dJ}\leq \frac{n}{\log n}$ and $|f_{ljk}|\leq C 2^{-\frac{Dj}{2}-j\gamma}$, we obtain that with probability at least $1-n^{-1}$,
\begin{equation*} 
\begin{aligned}
   & \quad  \sup_{f\in C_1^\gamma(\mb R^D)}\Big| \mb E_{\mu^\ast}[\Pi_Jf(X)] - \frac{1}{n} \sum_{i=1}^n \Pi_Jf(X_i)\Big|\\
   &\leq  C\, \sqrt{\frac{\log n}{n}}+C\,
  \big(\frac{\log n}{n}\big)^{  \frac{\gamma\wedge d}{d}} + C\,\sqrt{\frac{\log n}{n}} \cdot \sum_{l=1}^{2^D-1}\sum_{j=1}^{J}  2^{-\frac{Dj}{2}-j\gamma}\sum_{k\in\widetilde{\mathbb{S}}_{lj}}\sqrt{\mathbb{E}_{\mu^{\ast}}\big[\psi_{ljk}^2(X)\big]}\\
   &\leq  C\,\sqrt{\frac{\log n}{n}}+ C\,\big(\frac{\log n}{n}\big)^{  \frac{ \gamma\wedge d}{d}}  + C\,\sqrt{\frac{\log n}{n}}\cdot\sum_{l=1}^{2^D-1}\sum_{j=1}^{J}  2^{-\frac{Dj}{2}-j\gamma+\frac{dj}{2}}\sqrt{\sum_{k\in\widetilde{\mathbb{S}}_{lj}}\mathbb{E}_{\mu^{\ast}}\big[\psi_{ljk}^2(X)\big]}\, .
\end{aligned}
\end{equation*}
Since for any $x\in \m M$, there are at most constant many $k$'s in $\widetilde {\mb {S}}_{lj}$ such that $\psi_{ljk}(x)\neq 0$ and $\|\psi_{ljk}\|_\infty \leq C2^{\frac{Dj}{2}}$, we obtain
\begin{equation*} 
\begin{aligned}
   \quad \mathbb{E}_{\mu^{\ast}}  \Big[\sum_{k\in\widetilde{\mathbb{S}}_{lj}}\psi_{ljk}^2(X)\Big] 
   \leq C\, 2^{Dj}.
\end{aligned}
\end{equation*}
Then we get that with probability at least $1-n^{-1}$,
\begin{align*}
    &  \quad  \sup_{f\in C_1^\gamma(\mb R^D)}\Big| \mb E_{\mu^\ast}[\Pi_Jf(X)] - \frac{1}{n} \sum_{i=1}^n \Pi_Jf(X_i)\Big|\\
    &\leq C \,\sqrt{\frac{\log n}{n}} + C\,\big(\frac{\log n}{n}\big)^{  \frac{ \gamma\wedge d}{d}}
    + C\,\sqrt{\frac{\log n}{n}} \cdot \sum_{l=1}^{2^D-1}\sum_{j=1}^{J}  2^{j\big(\frac{d}{2}-\gamma\big)}\\
    & \leq C\, \sqrt{\frac{\log n}{n}}+ C \,\big(\frac{\log n}{n}\big)^{\frac{\gamma}{d}} ,
\end{align*}
 where we have used $2^{J}\leq (\frac{n}{\log n})^{\frac{1}{d}}$ in the last step. It remains to bound
 \begin{equation*}
     \begin{aligned}
     &\sup_{f\in C_1^\gamma(\mb R^D)}\Big[\big| \mb E_{\mu^\ast}[\Pi_J^{\perp}f(X)] \big|+\Big| \frac{1}{n} \sum_{i=1}^n \Pi_J^\perp f(X_i)\Big|\Big]\\
    &\leq \sum_{l=1}^{2^D-1}\sum_{j=J+1}^{\infty}\sum_{k\in\widetilde{\mathbb{S}}_{lj}}|f_{ljk}|\cdot\Big[ \big|\mathbb{E}_{\mu^{\ast}} [\psi_{ljk}(X)]\big|+\Big|\frac{1}{n} \sum_{i=1}^n\psi_{ljk}(X_i)\Big|\Big]\\
    &\leq  C \sum_{l=1}^{2^D-1}\sum_{j=J+1}^{\infty} 2^{-\frac{Dj}{2}-j\gamma}  \sup_{x\in\m M} \big[\sum_{k\in\widetilde{\mathbb{S}}_{lj}} \psi_{ljk}(x)\big].
     \end{aligned}
 \end{equation*}
 Then by the bound   $\sup_{x\in\m M} \big[\sum_{k\in\widetilde{\mathbb{S}}_{lj}} \psi_{ljk}(x)\big]\leq C 2^{-\frac{Dj}{2}}$, we can obtain
 \begin{equation*}
     \begin{aligned}
     &\sup_{f\in C_1^\gamma(\mb R^D)}\Big[\big| \mb E_{\mu^\ast}[\Pi_J^{\perp}f(X)] \big|+\Big| \frac{1}{n} \sum_{i=1}^n \Pi_J^\perp f(X_i)\Big|\Big]\\
    &\leq  C \sum_{j=J+1}^{\infty} 2^{-j\gamma}\leq  C \,\big(\frac{\log n}{n}\big)^{\frac{\gamma}{d}}.
     \end{aligned}
 \end{equation*}
 Putting all pieces together, we can obtain  it holds with probability at least $1-n^{-1}$ that 
 \begin{equation*}
 \sup_{f\in C_1^\gamma(\mb R^D)}\Big| \mb E_{\mu^\ast}[ f(X)] - \frac{1}{n} \sum_{i=1}^n  f(X_i)\Big|\leq C\, \sqrt{\frac{\log n}{n}}+ C \,\big(\frac{\log n}{n}\big)^{\frac{\gamma}{d}}.
\end{equation*}
  \end{itemize}
 The desired upper bound then follows from the bounds derived by the above two methods. We then prove the lower bound.  Consider $\m M=\{x=(x_1,x_2,\cdots,x_D): \|x_{1:d+1}\|^2=1, x_{d+2:D}=\bold{0}_{D-d-1}\}$ and $\mu^\ast$ being uniform distribution on $\m M$, then $\mu^\ast\in \m P^\ast(d,D,\alpha,\beta, L^\ast)$.
 \begin{itemize}
      
 \item {Proof for the rate $n^{-\frac{\gamma}{d}}\cdot (\log n)^{-1}$:}
     To prove the desired result,  we will construct function $f^\ast(\cdot\,; X_{1:n})\in C^{\gamma}_1(\mb R^D)$ depend on the data $X_{1:n}$ so that 
     \begin{equation*}
        \big| \mb {E}_{\mu^\ast} [f^\ast(X; X_{1:n})]-n^{-1}\sum_{i=1}^n f^\ast(X_i; X_{1:n})\big|\geq C\, n^{-\frac{\gamma}{d}}\cdot (\log n)^{-1}
     \end{equation*}
   with high probability.  The function  $f^\ast(\cdot\,; X_{1:n})$ is constructed as follows: let 
   \begin{align*}
    k(t)=\left\{\begin{array}{l}
(1-t)^{\gamma+1} t^{\gamma+1}, \quad t \in(0,1), \\
0, \quad \text { o.w. }
\end{array}\right.
\end{align*}
  Then choose $m=c_0 n^{\frac{1}{d}}$ for a large enough constant $c_0$ and define 
  \begin{equation*}
      f^\ast(y;X_{1:n})=\frac{1}{m^{\gamma}\cdot\log n}\cdot \sum_{i=1}^n \sigma_i(y),
  \end{equation*}
    with 
    \begin{equation*}
     \sigma_i(y)=\prod_{j=1}^{D} k\big(m(y_j-X_{i,j}+\frac{1}{2m})\big),
  \end{equation*}
     where $X_{i,j}$ and $y_j$ denote the $j$-th dimension of vectors $X_i$ and $y$ respectively.  Then we have the following lemma which implies the desired result.
     \begin{lemma}\label{lemmafstar}
       There exist constant $c_1,c_2$ so that  it holds with probability larger than $1-\frac{1}{n}$ that 
       \begin{enumerate}
           \item $ f^\ast(\cdot\,;X_{1:n})\in C^{\gamma}_{c_1}(\mb R^D)$;
           \item   $\big| \mb {E}_{\mu^\ast} [f^\ast(X; X_{1:n})]-n^{-1}\sum_{i=1}^n f^\ast(X_i; X_{1:n})\big|\geq c_2\, n^{-\frac{\gamma}{d}}\cdot (\log n)^{-1}$.
       \end{enumerate}
 
     \end{lemma}

\item Proof for the rate $n^{-\frac{1}{2}}$.
Define $\chi:\mathbb{R}\to \mathbb{R}$ is defined by $\chi(t)=e^{-1/t}$ for $t>0$ and $\chi(t)=0$ for $t\leq 0$. For $x\in \mathbb{R}^D$, we define $f(x)=\sum_{j=1}^{D} x_j\cdot \frac{\chi(3/2-\|x\|_2)}{\chi(3/2-\|x\|_2)+\chi(\|x\|_2-1)}$.  Then
\begin{equation*}
  \mb {E}\,\Big|\frac{1}{n}\sum_{i=1}^n f(X_i)-\mb E_{\mu^\ast} [f(X)] \Big|= \frac{1}{\sqrt{n}}\mb {E}\,\Big|\frac{1}{\sqrt{n}}\sum_{i=1}^n \sum_{j=1}^D X_{i,j}\Big|.
\end{equation*}
Define $\widetilde Y$ to be the random variable $\widetilde Y=\sum_{j=1}^D Y_j$ where $Y=(Y_1,Y_2,\cdots, Y_D)$ is a random vector from $\mu^\ast$. Then $\mb E(\widetilde Y)=0$ and $\sigma^2=\mb E [\widetilde Y^2]>0$. So we have 
\begin{equation*}
    \begin{aligned}
    &\frac{1}{\sqrt{n}}\mb {E}\,\Big|\frac{1}{\sqrt{n}}\sum_{i=1}^n \sum_{j=1}^D X_{i,j}\Big|\\
    &=\frac{1}{\sqrt{n}}\int_{0}^{+\infty} \mb P\Big(\Big|\frac{1}{\sqrt{n}}\sum_{i=1}^n \sum_{j=1}^D X_{i,j}\Big|\geq t\Big)dt\\
    &\geq \frac{1}{\sqrt{n}}  \mb P\Big(\Big|\frac{1}{\sqrt{n}}\sum_{i=1}^n \sum_{j=1}^D X_{i,j}\Big|\geq 1\Big)\\
    &\overset{(i)}{\geq}\frac{1}{\sqrt{n}}\Big( \mb P (-1\leq \m N(0,\sigma^2)\leq 1)-\frac{C}{\sqrt n}\Big)\geq  \frac{C_1}{\sqrt{n}},
    \end{aligned}
\end{equation*} 
where inequality (i) is due to Berry-Essen theorem~\citep{Rai__2019}. Proof is completed.

 \end{itemize}
\end{proof}
\subsection{Proof of Lemma~\ref{lemmafstar}}
     We first show the smoothness of $f^\ast(\cdot\,;\,X_{1:n})$.   Recall $m=c_0 n^{\frac{1}{d}}$ where $c_0$ is a large enough constant so that 
     \begin{equation*}
         \underset{x\in \m M}{\sup}\,\bigg[\mu^\ast\Big(\prod_{j=1}^{D}\big[x_j-\frac{2}{m},x_j+\frac{2}{m}\big]\Big)\bigg]\leq \frac{1}{n}.
     \end{equation*}
 For any multi-index $a\in \mb N_0^D$ with $|a|\leq \gamma$, it holds that ${\rm supp}\big(\prod_{j=1}^D k^{(a_j)}(y_j)\big)\subset [0,1]^D$ and thus ${\rm supp}\big(\sigma_i(y)\big)\subset \m A_i= \prod_{j=1}^D \big[-X_{i,j}-\frac{1}{2m}, X_{i,j}+\frac{1}{2m}\big]$. So for any $y\in \m M$, if there exists $i,k \in [n]$ so that $\sigma_i(y)\neq 0$ and $\sigma_k(y)\neq 0$, then $X_k\in \prod_{j=1}^D \big[X_{i,j}-\frac{1}{m}, X_{i,j}+\frac{1}{m}\big]$.  We claim that it suffices to show that it holds with probability at least $1-n^{-1}$ that  for any $y\in \m M$,  
 \begin{equation}\label{hpbound}
     \frac{1}{n}\sum_{i=1}^n \bold{1}\Big(X_i\in \prod_{j=1}^D\big[y_j-\frac{1}{m}, y_j+\frac{1}{m}\big]\Big)\leq C\cdot \frac{\log n}{n}.
 \end{equation}
 Indeed, under the above statement, it holds that for any $y\in \m M$ and multi-index $a\in \mb N_0^{D}$ with $|a|\leq \gamma$, 
 \begin{equation*}
     \Big|\big\{k\in [n]\,:\, y\in {\rm supp}\big(\sigma_k^{(a)}\big)\big\}\Big|\leq C\cdot \log n,
 \end{equation*}
 and hence $f^\ast(\cdot\,;\, X_{1:n})\in C^{\gamma}_{c_1}(\mb R^D)$ for a constant $c_1$. Now we prove equation~\eqref{hpbound}.  Let $\m N_{\frac{1}{n}}$ denotes the minimal $n^{-1}$-covering set of $\m M$, then $|\m N_{\frac{1}{n}}|\leq C n^d$. Since 
 \begin{equation*}
     \underset{x\in \m M}{\sup}\,\bigg[\mu^\ast\Big(\prod_{j=1}^{D}\big[x_j-\frac{2}{m},x_j+\frac{2}{m}\big]\Big)\bigg]\leq \frac{\log n}{n},
 \end{equation*}
 by a similar union bound argument plus Bernstein's inequality as the proof of~\eqref{eqn:union+Bernstein}, we can get it holds with probability at least $1-n^{-1}$ that for any $\widetilde y\in \m N_{\frac{1}{n}}$ 
 \begin{equation*}
     \frac{1}{n}\sum_{i=1}^n \bold{1}\Big(X_i\in \prod_{j=1}^D\big[\widetilde y_j-\frac{2}{m}, \widetilde y_j+\frac{2}{m}\big]\Big)\leq C\cdot \frac{\log n}{n}.
 \end{equation*}
 Thus for any $y\in \m M$, there exists  $\widetilde y\in \m N_{\frac{1}{n}}$ so that 
 \begin{equation*}
    \frac{1}{n}\sum_{i=1}^n \bold{1}\Big(X_i\in \prod_{j=1}^D\big[ y_j-\frac{1}{m},   y_j+\frac{1}{m}\big]\Big)\leq  \frac{1}{n}\sum_{i=1}^n \bold{1}\Big(X_i\in \prod_{j=1}^D\big[\widetilde y_j-\frac{2}{m}, \widetilde y_j+\frac{2}{m}\big]\Big)\leq C\cdot \frac{\log n}{n}.
 \end{equation*}
 The proof for the first statement is then completed. For the second statement, 
 \begin{equation*}
 \begin{aligned}
   & \frac{1}{n}\sum_{i=1}^n f^\ast(X_i;X_{1:n})-\mb E[f(X;X_{1:n})]\\
    &=(\log n)^{-1}\Big[ \frac{1}{n}\sum_{i=1}^n \sum_{j=1}^n \frac{1}{m^\gamma}\sigma_j(X_i)-\mb E \big[ \sum_{j=1}^n \frac{1}{m^\gamma}\sigma_j(X)\big] \Big]\\
    &\geq  \frac{1}{m^\gamma\cdot \log n}\Big[ \sum_{j=1}^n \big(\frac{1}{n}\sigma_j(X_j)-\mu^\ast(\m A_j)\mb {E}_{\mu^\ast|_{\m A_j}} [\sigma_j(X)]\big) \Big]\\
    &\geq C n^{-\frac{\gamma}{d}} (\log n)^{-1}\underset{i\in [n]}{\inf}\bigg[\Big(k\big(\frac{1}{2}\big)\Big)^D-\mb {E}_{Y\sim \mu^\ast|_{\m A_i}} \Big[\prod_{l=1}^D k\big(m(Y_l-X_{i,l}+\frac{1}{2m})\big)\Big] \bigg],\\
     \end{aligned}
 \end{equation*}
 where the last inequality uses the fact $\mu^\ast(\m A_j)\leq \frac{1}{n}$ and $m=c_0 n^{\frac{1}{d}}$.  The desired conclusion then follows from the fact that $\mu^\ast$ is uniform distribution on $\m M$ and function $k(t)$ achieves its unique maximum at point $t=\frac{1}{2}$.
 
\subsection{Lemma~\ref{lemma:gen_model_as_mix} and its proof}\label{Sec:gen_model_as_mix}
\begin{lemma}\label{lemma:gen_model_as_mix}
  For any probability measure $\mu$ in the generative model class $\m S^\ast=\m S^\ast(D,d,\alpha,\beta,\ms O_M,L)$ defined in Appendix~\ref{sec:gen_model_class}, where the property 2 of  $\m S^\ast$ is satisfied with pairs $(\nu_{[m]}, G_{[m]})$ for $m\in [M]$, then $\mu$ can be expressed as a mixture of generative model with rejection sampling:
   \begin{align*}
        \mu =\sum_{m=1}^M w_{[m]} \,\mathcal{A}(G_{[m]},\nu_{[m]},\rho_m),
   \end{align*}
  where (1) $\{w_{[m]}\}_{m\in [M]} $ are the mixing weights given by $w_{[m]}=\mathbb{E}_{\mu}[\rho_m(X)]$ for $m\in [M]$ and $\{\rho_m\}_{m\in [M]}$ is the partition of unity subordinate to $\ms O_M$ defined in Section~\ref{sec:pou}; (2) $\mathcal{A}(G_{[m]}, \nu_{[m]}, \rho_m)$ is the probability measure induced by the data generating process where $X\sim [G_{[m]}]_{\#}\nu_{[m]}$ is accepted with probability $\rho_m(X)\in[0,1]$.
\end{lemma}

\begin{proof}
For each $m\in [M]$, since the partition of unity $\{\rho_m\}_{m\in [M]} $ is subordinate to the open cover $\ms O_M=\{\mb B_{r_m}(a_m)^\circ\}_{m \in {M}}$, by $\mb B_{r_m}(a_m)\subset \wt S_m$, we have ${\rm supp}(\rho_m)\subset \wt S_m$ for any $m\in [M]$. Then by the property that $\mu|_{\widetilde S_m} = [G_{[m]}]_\# \nu_{[m]}$,
 \begin{align*}
     w_{[m]} = \int \rho_m\,\dd\mu = \mb P(X\in\widetilde S_m) \,\int \rho_m\, \dd \big([G_{[m]}]_{\#}\nu_{[m]}\big) = \mb P(X\in\widetilde S_m)\,\int \dd \big(\rho_m\, [G_{[m]}]_{\#}\nu_{[m]}\big),
 \end{align*}
 where recall that for any measure $\nu$, $\rho_m\, \nu$ stands for the measure whose Radon-Nikodym derivative relative to $\nu$ is $\rho_m$.
Therefore, by the definition of $\mathcal{A}(G_{[m]}, \nu_{[m]}, \rho_m)$, we have that
 for any measurable function $f: \mathbb{R}^D\to \mathbb{R}$,
\begin{align*}
    \int f\, \dd \m A(G_{[m]},\nu_{[m]},\rho_m) =&\, \frac{\mb P(X\in\widetilde S_m)}{w_{[m]}}\,\int f \,\dd \big(\rho_m\, [G_{[m]}]_{\#}\nu_{[m]}\big) \\
    = &\, \frac{\mb P(X\in\widetilde S_m)}{w_{[m]}}\,\int f\,\rho_m \,\dd \big([G_{[m]}]_{\#}\nu_{[m]}\big).
\end{align*}
Moreover, we have, by using the partition of unity $\{\rho_m\}_{m\in [M]} $, that
     \begin{align*}
     &\int f\, \dd \mu =\sum_{m=1}^M\int  f\, \rho_m\, \dd \mu = \sum_{m=1}^M \mb P(X\in \widetilde S_m)\, \int  f\, \rho_m\, \dd \big([G_{[m]}]_\# \nu_{[m]}\big).
     \end{align*}
By combining the two preceding displays, we obtain $\mu =\sum_{m=1}^M w_{[m]} \,\mathcal{A}(G_{[m]},\nu_{[m]},\rho_m)$.
\end{proof}

\subsection{Lemma~\ref{le8} and its proof}\label{Sec:le8}
 \begin{lemma}\label{le8}
There exists a set $\{\omega^{(1)},\cdots , \omega^{(H)}\}\subset \{0,1\}^{[m^d]}$ such that $\log H \geq \frac{m^d}{8}-\log 2$ and for any $j,k\in [H]$ and $j\neq k$ it holds that  $\frac{m^d}{4}\leq \rho(\omega^{(j)},\omega^{(k)})\leq \frac{3m^d}{4}$, where $\rho(\omega,\omega')$ denotes the Hamming distance between $\omega$ and $\omega'$ on the hypercube.
 \end{lemma}
\begin{proof}
The proof of Lemma 8 is a simple modification of the proof of the Varshamov-Gilbert lemma~\citep{Tsybakov2009}. We include it here for completeness. Let $\{\omega^{(1)}, \cdots. \omega^{(H)}\}$ be the largest set satisfying that for any $j,k\in [H]$ and $j\neq k$, $\frac{m^d}{4}\leq \rho(\omega^{(j)},\omega^{(k)})\leq \frac{3m^d}{4}$. Let $\bar{\omega}=1-\omega$, it holds that 
\begin{equation*}
\{0,1\}^{m^d}\subset \left\{\bigcup\limits_{i=1}^H B_{\rm H}\left(w^{(i)},\frac{m^d}{4}\right)\right\}\cup  \left\{ \bigcup\limits_{i=1}^H B_{\rm H}\left(\bar{w}^{(i)},\frac{m^d}{4}\right)\right\},
\end{equation*}
where $B_{\rm H}(\omega,r)$ denote the ball centered at $\omega$ with radius $r$ with respect to the Hamming distance.
Therefore, we can obtain that
\begin{equation*}
\begin{aligned}
2^{m^d}&\leq  \sum_{i=1}^H \left|B_{\rm H}\left(w^{(i)},\frac{m^d}{4}\right)\right|+\sum_{i=1}^H\left| B_{\rm H}\left(\bar{w}^{(i)},\frac{m^d}{4}\right)\right|  \\
&=2H\sum_{i=0}^{\frac{m^d}{4}}\binom{m^d}{i}\\
&=2H\times2^{m^d}\times\mathbb{P}\left({\rm Binomial}(m^d, \frac{1}{2})\leq \frac{1}{4}m^d\right)\\
&=2H\times2^{m^d}\times\mathbb{P}\left({\rm Binomial}(m^d, \frac{1}{2})\geq \frac{3}{4}m^d\right)\\
&=2H\times2^{m^d}\times\mathbb{P}\left({\rm Binomial}(m^d, \frac{1}{2})-\frac{1}{2}m^d\geq \frac{1}{4}m^d\right)
\end{aligned}
\end{equation*}
Hoeffding's inequality~\citep{doi:10.1080/01621459.1963.10500830} then yields,
\begin{equation*}
2^{m^d}\leq 2\times2^{m^d}\times H \times\exp(-\frac{m^d}{8}),
\end{equation*}
implying $\log H\geq \frac{m^d}{8}-\log 2$.
\end{proof}

\subsection{Proof of Lemma~\ref{lemma2}}\label{Sec:Proof_lemma2}
We first prove the first part on the KL divergence upper bound.
Consider the following partition of ambient space $\mb R^D$:
\begin{equation*}
\begin{aligned}
&A_1=\Bigg\{x\in \mathbb{R}^d\, \Bigg|\, x_{1:d}\in \left[-\sqrt{\frac{1}{2d}}+\frac{1}{m}\sqrt{\frac{2}{d}},\ \sqrt{\frac{1}{2d}}+\frac{1}{m}\sqrt{\frac{2}{d}}\right]^d,\\
& \qquad\qquad\qquad\qquad x_{d+1}=G_{{\omega}^{(i)}}(x_{1:d}); \,x_{d+2:D}=\bold{0}_{D-d-1}\Bigg\}.\\
&A_2=\Bigg\{x\in \mathbb{R}^d\, \Bigg|\, x_{1:d}\in \mb B_1^d\setminus \left[-\sqrt{\frac{1}{2d}}+\frac{1}{m}\sqrt{\frac{2}{d}},\ \sqrt{\frac{1}{2d}}+\frac{1}{m}\sqrt{\frac{2}{d}}\right]^d; \\
& \qquad\qquad\qquad\qquad x_{d+1}=G_0(x_{1:d}); \,x_{d+2:D}=\bold{0}_{D-d-1}\Bigg\},\\
&A_3=\mathcal{M}_0\setminus\wt{\mathcal{M}}_0,\quad\mbox{and}\quad A_4=\mathbb{R}^D \setminus (A_1\cup A_2\cup A_3).
 \end{aligned}
\end{equation*}
 Under this partition, we claim that for any fixed $i\in [H]$, 
\begin{enumerate}
\item if  $x\in A_1$, then $\frac{d \mu_{i}}{d\bar{\mu}}(x)=2$;
\item if $x\in A_2\cup A_3$, then $\frac{d \mu_{i}}{d\bar{\mu}}(x)=1$;
\item for $A_4$, then $\mu_i(A_4)=\bar\mu(A_4)=0$,
 \end{enumerate}
where recall that $\frac{\dd P}{\dd Q}$ denotes the Radon–Nikodym derivative of $P$ with respect to $Q$. 

In fact, by construction, all $\mu_i$'s are the same outside $A_1$, which leads to 2 and 3. To see 1, note that
by construction, for each index $\xi\in\{0,1\}^{[m]^d}$, there are equal numbers of $0$'s and $1$'s in $\{{\omega}_{\xi}^{(0)},\cdots, {\omega}_{\xi}^{(H)}\}$. So for each $x\in A_1$, half $\mu_i(\dd x)$'s (density with respect to the volume measure of $\m M_i$) are equal and the rest half of $\mu_i(\dd x)$'s are zero, implying that $\frac{d \mu_{i}}{d\bar{\mu}}(x)=2$.
Now let us bound $D_{\rm KL}(\mu_{i}\,||\,\bar{\mu})$. Using above properties, we obtain $D_{\rm KL}(\mu_{i}\,||\,\bar{\mu})=\int_{A_1\cup A_2\cup A_3} \log \frac{d\mu_i}{d\bar{\mu}}\, \dd\mu_i=\log 2 \int_{A_1}   \dd\mu_i\leq \log 2$.
 
Next, we upper bound $\underset{j,k\in [H]\atop j\neq k} {\inf} d_{\gamma}(\mu_{j},\mu_{k})$ by constructing discriminator $f$ for discriminating $\mu_{j}$ and  $\mu_{k}$ so that $\int f(x)\, \dd\mu_{j}-\int f(x)\,\dd\mu_{k}$ is large, for each distinct index pair $j$ and $k$. Fix any pair of $j,k\in [H]$ with $j\neq k$, by construction we have $\rho(\omega^{(j)}, \omega^{(k)}) \geq \frac{m^d}{4}$. Define
\begin{equation*}
\widetilde{f}(z)=\sum_{\xi\in[m]^d}\left(\frac{1}{m}\right)^{\gamma} v_{\xi} \psi_{\xi}(z),
\end{equation*}
where
\begin{equation*}
v_{\xi}=\left\{\begin{array}{l}
1, \quad  \omega_{\xi}^{(j)}=1\text{ and }\omega_{\xi}^{(k)}=0;\text{ or }\omega_{\xi}^{(j)}=\omega_{\xi}^{(k)},\\
-1, \quad   \omega_{\xi}^{(j)}=0\text{ and }\omega_{\xi}^{(k)}=1.
\end{array}\right.
\end{equation*}
By the definition of $g_{\omega}(z)$, there exists a constant $c$ such that  for any $j \in[H]$,  it holds that ${\rm supp}(\mu_{\omega^{(j)}})\subset  \mathbb{R}^d\times\{x_{d+1}: |x_{d+1}-\sqrt{2-\|x_{1:d}\|^2}|\leq \frac{c}{m^{\beta}}\}\times \{(x_{d+2},\cdots,x_{D})^T=\bold{0}_{D-d-1}\}$.  Define function $h: \mathbb{R}\to \mathbb{R}$ by $h(x)=\max(- \frac{c}{m^{\beta}},\,\min(\frac{c}{m^{\beta}},x))$, then $h$ is a 1-Lipschitz function over $\mathbb{R}$. Recall that $\chi:\mathbb{R}\to \mathbb{R}$ is defined by $\chi(t)=e^{-1/t}$ for $t>0$ and $\chi(t)=0$ for $t\leq 0$. For $z\in \mathbb{R}^d$, we define $q(z)=\sqrt{2-\|z\|^2}\cdot\frac{\chi(3/2-\|z\|_2)}{\chi(3/2-\|z\|_2)+\chi(\|z\|_2-1)}$. Note that when $z\in \mb B_1^d$, $q(z)=\sqrt{2-\|z\|^2}$ and we multiply $\sqrt{2-\|z\|^2}$ by $\frac{\chi(3/2-\|z\|_2)}{\chi(3/2-\|z\|_2)+\chi(\|z\|_2-1)}$ to smoothly extend $\sqrt{2-\|z\|^2}$ from $\mb B_1^d$ to the entire space. Now define 
\begin{equation*}
f(x)=\widetilde{f}(x_{1:d})\, h\big(x_{d+1}-q(x_{1:d})\big)\, m^{\gamma-\gamma\beta+\beta}.
\end{equation*}
We then prove that $f$ is $\gamma$-smooth with bounded H\"{o}lder norm. Since for any $x, x' \in \mathbb{R}^D$, it holds that $|h\big(x_{d+1}-q(x_{1:d})\big)|\leq \frac{c}{m^{\beta}}$ and $|h\big(x'_{d+1}-q(x'_{1:d})\big)|\leq \frac{c}{m^{\beta}}$. Therefore, we have
\begin{equation*}
|h\big(x_{d+1}-q(x_{1:d})\big)-h\big(x_{d+1}-q(x'_{1:d})\big)|\leq (2c)^{1-\gamma} \frac{1}{m^{\beta(1-\gamma)}} |h\big(x_{d+1}-q(x_{1:d})\big)-h\big(x'_{d+1}-q(x'_{1:d})\big)|^{\gamma}.
\end{equation*}
Moreover, for any $z,z'\in \mathbb{R}^d$, there exists a constant $c_1$ such that 
\begin{equation*}
|\widetilde{f}(z)-\widetilde{f}(z')|\leq c_1 \frac{1}{m^{\gamma-1}}\|z-z'\|_2.
\end{equation*}
Therefore, in the case $\|z-z'\|_2\leq \frac{1}{m}$, then we have $\|z-z'\|_2\leq \frac{1}{m^{1-\gamma}}\|z-z'\|_2^\gamma$, and thus $|\widetilde{f}(z)-\widetilde{f}(z')|\leq c_1 \|z-z'\|_2^\gamma$; in the case $\|z-z'\|_2> \frac{1}{m}$, since there exists a constant $c_2$ such that $\sup_{z\in \mathbb{R}^d}|\widetilde{f}(z)|\leq \frac{c_2}{m^{\gamma}}$, it holds that $|f(z)-f(z')|\leq 2c_2 \|z-z'\|_2^{\gamma}$. Putting pieces together, we have that for any $x, x' \in \mathbb{R}^D$, there exists a constant $c_3$ such that 
\begin{equation*}
\begin{aligned}
|f(x)-f(x')|&\leq m^{\gamma-\gamma\beta+\beta}\bigg(\Big|\widetilde{f}(x_{1:d}) \cdot \Big(h\big(x_{d+1}-q(x_{1:d})\big)-h\big(x'_{d+1}-q(x'_{1:d})\big)\Big)\Big|\\
&\qquad\qquad+\Big|h\big(x'_{d+1}-q(x'_{1:d})\big) \cdot (\widetilde{f}(x_{1:d})-\widetilde{f}(x_{1:d}'))\Big|\bigg)\\
&\leq (2c)^{1-\gamma}\,c_2 \, \|x_{1:d+1}-x_{1:d+1}'\|^{\gamma}+c\,(2c_2\vee c_1)\, m^{\gamma-\gamma\beta\,}\|x_{1:d}-{x_{1:d}'}\|^{\gamma}\\
&\leq c_3\,\|x-x'\|^{\gamma},
\end{aligned}
\end{equation*}
where the last inequality is due to $\beta>1$. Consequently, there exists a constant $c_4$ such that $\frac{1}{c_4}f(x)\in C^{\gamma}_1(\mathbb{R}^D)$ (recall that we only consider $\gamma<1$). Furthermore, since for any $\omega\in \{0,1\}^{[m]^d}$, the support of $g_{\omega}(z)$ is contained in $B_{{3}/{4}}^d$, we can get
\begin{equation*}
\begin{aligned}
d_{\gamma}(\mu_{j}, \mu_{k})&\geq \frac{1}{c_4} \cdot\Big(\int f(x) \,\dd\mu_{j}-\int f(x) \,\dd\mu_{k}\Big)\\
&=\frac{\wt{C}}{c_4\cdot C}\, m^{\gamma-\gamma\beta+\beta} \Big(\int \widetilde f(x) \,\dd[\mu_{j}|_{{M}_{\omega^{(j)}}}]-\int \widetilde f(x) \,\dd[\mu_{k}|_{{M}_{\omega^{(k)}}}]\Big)\\
&=\frac{\wt{C}}{c_4\cdot C\, }m^{\gamma-\gamma\beta+\beta} \int \widetilde{f}(z)\cdot \big(g_{\omega^{(j)}}(z)-g_{\omega^{(k)}}(z)\big) \,\nu_0(z)\,\dd z\\
&=\frac{\wt{C}}{c_4\cdot C}\int m^{-\gamma\beta} \sum_{\xi\in [m]^d}v_{\xi} \psi_{\xi}(z) \sum_{\xi_1\in[m]^d} \big(\omega_{\xi_1}^{(j)}-\omega_{\xi_1}^{(k)}\big)\,\psi_{\xi_1}(z)\, \nu_0(z)\, \dd z \\
&\geq \frac{\wt{C}}{c_4\cdot C}\, m^{-\gamma\beta}\, \underset{z\in B_{{3}/{4}}^d}{\inf}\nu_0(z)\, \sum_{\xi\in[m]^d}\int |\omega_{\xi}^{(j)}-\omega_{\xi}^{(k)}| \,\psi_{\xi}^2(z)\,\dd z
\geq c'\,n^{-\frac{\gamma\beta}{d}},
 \end{aligned}
\end{equation*}
which completes the proof.

   \subsection{Proof of Lemma~\ref{lemmaboundaryless}}
 Fix any $\mu\in \m P^{\ast}$, let $\m M ={\rm supp}(\mu)$. Then there exist positive constants $\tau,L$ such that 
 \begin{enumerate}
     \item $\m M \subset \mb B_L^D$;
     \item for any $x_0\in \m M$, there exists $\lambda=\lambda(x_0)\in \Lambda$ such that
     \begin{enumerate}
         \item $U_{\lambda}\supset \m M\cap B_{\tau}(x_0)$ and $V_{\lambda}=\phi_{\lambda}(U_{\lambda})\supset B_\tau(\phi_{\lambda}(x_0))$;
         \item $\phi_{\lambda}^{-1}\in C^{\beta}_L(V_{\lambda})$ and $\mu\circ \phi_{\lambda}^{-1}\in C^{\alpha}_L(V_{\lambda})$, where we also use $\mu$ to denote the density of probability measure $\mu$ with respect to the volume measure of $\m M$;
         \item $[\inf_{z\in V_{\lambda}}\lambda_{\rm min}(J_{\phi^{-1}_{\lambda}}(z)^TJ_{\phi^{-1}_{\lambda}}(z))]^{-1}\leq L$.
          \end{enumerate}
 \end{enumerate}
  We will construct local parametrization $(\mb B_{r'}(x_0)\cap\m M, Q_{x_0})$ of $\m M$ in a neighborhood of $x_0$, where $Q_{x_0}(x)$ projects $x$ to the tangent space of $\m M$ at $x_0$  and $Q_{x_0}^{-1}$ can be $\beta$-smoothly extended to whole spaces with bounded H\"{o}lder norms. More specifically, we have the following lemma.
 \begin{lemma}\label{parametrization}
  There exist positive constants $(\tau_1,L_1)$ such that for any $x_0\in \m M$, define $Q_{x_0}: \mathbb{R}^D\to \mathbb{R}^d$ as $Q_{x_0}(x)=W_{x_0}^T(x-x_0)$ where $W_{x_0}\in \mathbb{R}^{D\times d}$ is an arbitrary orthonormal basis of the tangent space of $\m M$ at $x_0$, then there exists a set $\widetilde U_{x_0} $ satisfying $\mb B_{\tau_1}(x_0)\cap \m M\subset \widetilde U_{x_0}\subset  \m M$ and function $G_{x_0}\in C_{L_1}^{\beta}(\mb R^d;\mb R^D)$ so that  
  \begin{enumerate}
      \item $G_{x_0}(\mb B_1^d)=\widetilde U_{x_0}$ and for any $z\in \mb B_1^d$,    $Q_{x_0}(G_{x_0}(z))=z$; 
      \item $\mu\circ G_{x_0}|_{\mb B_1^d}\in C^{\alpha}_{L_1}(\mb B_1^d)$ and for any $z\in \partial \mb B_1^d$,  $\|G_{x_0}(z)-x_0\|\geq \tau_1$.
      \end{enumerate}
 \end{lemma}
 \noindent Then consider any open cover $\ms O_{M}=\{\mb B_{r_m}(a_m)^\circ\}_{m\in [M]} $ of $\mb B_{L+1}^D$ where $\min\{r_1,r_2,\cdots,r_M\}\leq \frac{\tau_1}{4}$. For an  arbitrary $m\in [M]$ such that $\mb B_{r_m+\frac{\tau_1}{4}}(a_m)\cap \mathcal{M}\neq \emptyset$, there exists $x_0\in \mb B_{ r_m+\frac{\tau_1}{4}}(a_m)\cap\mathcal{M}$ such that $\mb B_{\tau_1}(x_0)\supset \mb B_{r_m+\frac{\tau_1}{4}}(a_m)$. Then by Lemma~\ref{parametrization}, $\mb B_{\tau_1}(x_0)\cap \m M\subset\widetilde{U}_{x_0}$ and  we can let $\widetilde{S}_m$ be any set containing  $\mb B_{r_m+\frac{\tau_1}{4}}(a_m)$ such that $\widetilde{S}_m\cap \m M=\widetilde{U}_{x_0}$. Then we can define $G_{[m]}=G_{x_0}$.  Moreover,  by Sobolev extension theorem~\citep{stein2016singular}, there exists $\bar{Q}_{x_0}\in C^{\beta}_{L_1}(\mathbb{R}^D)$  such that $\bar{Q}_{x_0}|_{\widetilde{U}_{x_0}}=Q_{x_0}|_{\widetilde{U}_{x_0}}$ and we can define $Q_{[m]}=\bar{Q}_{x_0}$ Therefore, by the assumption that the density $\mu$ w.r.t. the volume measure of $\mathcal{M}$  is bounded  from above and below, $\mu|_{\widetilde{S}_m}$ can be written as the push forward measure $[{G}_{[m]}]_{\#}\nu_{[m]}$, where 
\begin{equation*}
    \nu_{[m]}(z)=\frac{\mu(G_{[m]}(z))\sqrt{{\rm det}(\bold{J}_{G_{[m]} }(z)^T\bold{J}_{G_{[m]}}(z))}}{\int_{\mb B_1^d}\mu(G_{[m]}(z))\sqrt{{\rm det}(\bold{J}_{G_{[m]}}(z)^T\bold{J}_{G_{[m]}}(z))}\,\dd z},\quad z\in \mb B_1^d.
\end{equation*}
 Then by Lemma~\ref{parametrization}, $|\log\nu_{[m]}(z)|$ is uniformly bounded above over $z\in \mb B_1^d$ and $\nu_{[m]}\in C^{\alpha}_{L_1}(\mb B_1^d)$ for a constant $L_1$ (recall that $\beta\geq \alpha+1$). Moreover, by the Caffarelli’s global regularity theory~\citep{Villani2009,10.2307/2118564} which states that for $\alpha$-smooth probability densities $\mu_P,\mu_Q$ supported on $\mb B_1^d$ and bounded from above and below on their supports,  the optimal transport from $\mu_P$ to $\mu_Q$ is $(\alpha+1)$-smooth, we can write $\nu_{[m]}$  as a push forward measure ${V_{[m]}}_{\#}\nu_0$ with $V_{[m]}$ being $\alpha$-smooth and $\nu_0$ being the uniform distribution on $\mb B_1^d$. Furthermore,  by the third statement of Lemma~\ref{parametrization} and Lipschitz-continuity of $G_{[m]}$, there exists $0<\epsilon<1$ so that $G_{[m]}(\mb B_1^d\setminus \mb B_{1-\epsilon}^d)\cap \mb B_{r_m}(a_m)=\emptyset$. On the other hand, for any  $m\in [M]$ such that $\mb B_{r_m+\frac{\tau_1}{4}}(a_m)\cap \mathcal{M}= \emptyset$, we can choose $\wt{S}_m=\mb B_{r_m+\frac{\tau_1}{4}}(a_m)\cup \wt{S}_{m^\ast}$, where $m^\ast$ can be any $m^\ast\in [M]$ so that  $\mb B_{r_{m^\ast}+\frac{\tau_1}{4}}(a_{m^\ast})\cap \mathcal{M}\neq \emptyset$. Proof is completed.
 
 \subsection{Proof of Lemma~\ref{parametrization}}
Fix an $x_0\in \m M$, write $z_0=\phi_{\lambda(x_0)}(x_0)$, $\psi_{\lambda(x_0)}=\phi_{\lambda(x_0)}^{-1}$ and $W_{x_0}\in \mathbb{R}^{D\times d}$ be an arbitrary orthonormal basis of the tangent space of $\m M$ at $x_0$.
 For ease of notation, we suppress the subscript $\lambda(x_0)$ and $x_0$ in $\phi_{\lambda(x_0)}, \psi_{\lambda(x_0)}, W_{x_0}$ in the following analysis for a fixed $x_0\in \m M$. 
 
 Recall $Q_{x_0}: \mathbb{R}^D\to \mathbb{R}^d$ defined as $Q_{x_0}(x)=W^T(x-x_0)$. Write $Q=Q_{x_0}$, then given  $y=Q(x)$ for some $x\in \m M$, we may recover $x$ by considering the solution  $\bar{z}$ of the following equation with respect to $z$:
\begin{equation}\label{eqnsolvez}
    W^T(\psi(z)-x_0)=y,
\end{equation}
and $x=\psi(\bar{z})$ if a unique solution exists. Now we show that equation~\eqref{eqnsolvez} has a unique solution in a small neighborhood of $z_0$ when $\|y\|$ is small. 

Firstly, by the $\beta$-smoothness of $\psi$, there exists a positive constant $L_1$ such that for any $y\in \mb B_{\tau}^d$,
\begin{equation}\label{eqnfnorm}
    \|(W^T\bold{J}_{\psi}(z_0+y))^{-1}\|_F\leq L_1,
\end{equation}
and for any $y,y'\in \mb B_{\tau}^d$
\begin{equation}\label{eqntaylorexp}
    \|W^T\psi(z_0+y)-W^T\psi(z_0+y')-W^T\bold{J}_{\psi}(z_0+y')(y-y')\|_2\leq L_1\|y-y'\|^{\beta},
\end{equation}
where recall $\bold{J}_f(x)$  denotes the Jacobian matrix of function $f$ at $x$. Then for $y\in \mb B_{\tau_2}^d$, where $\tau_2$ is a small enough positive constant that will be chosen later, we construct a solution $\bar{z}$ to equation~\eqref{eqnsolvez} as follows:  define $z^0=z_0$ and recursively define $z^k=z^{k-1}-(W^T\bold{J}_{\psi}(z^{k-1}))^{-1}(W^T(\psi(z^{k-1})-x_0)-y)$ for $k=1,2,\cdots$. Then define a sequence $\{b_k\}_{k\in \mb N}$ as $b_k=L_1^{\frac{1+\beta}{1-\beta}}(L_1^{\frac{1+\beta}{\beta-1}}\tau_2)^{\beta^k}$. Then if $\tau_2$ satisfies that $\sum_{k=0}^{+\infty} L_1b_k\leq \frac{\tau}{2}\wedge \frac{1}{4L_1^{2(\beta-1)}}$, we can obtain that $\|W^T(\psi(z^k)-x_0)-y\|_2\leq b_k$ and $\|z^{k+1}-z^{k}\|_2\leq L_1b_k$ for $k\in \mathbb{N}$. Hence $\lim_{k\to +\infty}z^k=\bar{z}$ exists, $ W^T(\psi(\bar{z})-x_0)=y$ and $\|\bar{z}-z_0\|_2\leq \tau'=\frac{\tau}{2}\wedge \frac{1}{4L_1^{2(\beta-1)}}$.

Now we show that for any $y\in \mb B_{\tau_2}^d$, the equation $W^T(\psi(z)-x_0)=y$ has a unique solution on $\mb B_{\tau'}(z_0)$. Suppose there are two solutions $\bar{z}, \bar{z}'$ on $\mb B_{\tau'}(z_0)$. Then $\|\bar{z}-\bar{z}'\|_2\leq2\tau'\leq \frac{1}{2L_1^{2(\beta-1)}} $ and
$W^T(\psi(\bar{z})-\psi(\bar{z}'))=0$. Then by equation~\eqref{eqntaylorexp}, we can obtain that
\begin{equation*}
    \|W^T\bold{J}_{\psi}(\bar{z}')(\bar{z}-\bar{z}')\|_2\leq L_1\|\bar{z}-\bar{z}'\|^{\beta},
\end{equation*}
which leads to 
\begin{equation*}
    \frac{\|\bar{z}-\bar{z}'\|_2}{\|(W^T\bold{J}_{\psi}(\bar{z}'))^{-1}\|_2}\leq L_1 \|\bar{z}-\bar{z}'\|^{\beta}.
\end{equation*}
Then combined with equation~\eqref{eqnfnorm}, we can obtain that \begin{equation*}
    \|\bar{z}-\bar{z}'\|_2\geq \frac{1}{L_1^{2(\beta-1)}},
\end{equation*}
which cause contradiction.  So we can define a function $ q: \mb B_{\tau_2}^d\to \mathbb{R}^d$ so that $q(y)$ is defined as the unique solution of $W^T(\psi(z)-x_0)=y$ over $z\in \mb B_{\tau'}(z_0)$. Consider $\widetilde{V}=\mb B_{\tau_2}^d$ and $\widetilde{U}=\psi(q(\widetilde{V}))=\{\psi(z): z\in \mb B_{\tau'}(z_0),\, \|W^T(\psi(z)-x_0)\|_2\leq \tau_2\}$. Next we show that there exists a positive constant $\tau_1$ such that $\mb B_{\tau_1}(x_0)\cap \m M\subset\widetilde{U}$.  First we know that $\mb B_\tau(x_0)\cap \m M\subset U=U_{\lambda(x_0)}$.  Then consider $\tau_1\leq \tau$ and any $x\in \mb B_{\tau_1}(x_0)\cap \m M$, we have $\|x-\psi(z_0)\|=\|x-x_0\|\leq \tau_1$. Define $\widetilde{z}^0=z_0$ and recursively define $\widetilde{z}^k=\widetilde{z}^{k-1}-\bold{J}_{\psi}(\widetilde{z}^{k-1})(\psi(\widetilde{z}^{k-1})-x)$ for $k\in \mathbb{N}^+$, then similar to the above analysis for the function $q(y)$, when $\tau_1$ is small enough, we can obtain that $\phi(x)=\lim_{k\to +\infty}\widetilde{z}^k$ and $\|\phi(x)-z_0\|\leq \tau'\wedge \frac{\tau_2}{\sup_{z\in \mb B_{\tau'}(z_0)}\|W^T\bold{J}_{\psi}(z)\|_F}$. Thus, $\mb B_{\tau_1}(x_0)\cap \m M\subset\widetilde{U}$ for a small enough positive constant $\tau_1$.

Then define 
$G: \mb B_1^d\to \widetilde{U}$ as $G(z)=\psi(q(\tau_2z))$  and $G^{-1}=Q: \widetilde{U}\to \mb B_1^d$ as $Q(x)=\frac{1}{\tau_2}W^T(x-x_0)$.  Since $\bold{J}_q(z)=(W^T\bold{J}_{\psi}(q(z)))^{-1}$, we can obtain that $\bold{J}_G(z)=\tau_2\bold{J}_{\psi}(q(\tau_2z))(W^T\bold{J}_{\psi}(q(\tau_2z)))^{-1}$. Then by $\psi\in C^{\beta}_{L}(V)$ and  $[\inf_{z\in V}\lambda_{\rm min}(\bold{J}_{\psi}(z)^T\bold{J}_{\psi}(z))]^{-1}\leq L$, we can obtain that $G\in C^{\beta}_{L^\ast}(\mb B_1^d)$ and $q\in C^{\beta}_{L^\ast}(\mb B_{\tau_2}^d)$ for a constant $L^\ast$, then by Sobolev extension theorem~\citep{stein2016singular}, there exists $\bar{G}\in C^{\beta}_{L^\ast}(\mathbb{R}^d)$ such that $\bar{G}|_{\mb B_1^d}=G$.  Moreover, by the assumption that $\mu\circ \psi\in C^{\alpha}_L(V)$, we can obtain  $\mu\circ G \in C^{\alpha}_{L^\ast}(\mb B_1^d)$.
Finally, since $G(\partial \mb B_1^d)=\{\psi(z): z\in \mb B_{\tau'}(z_0),\, \|W^T(\psi(z)-x_0)\|_2=\tau_2\}$. We can obtain that for any $x\in G(\partial \mb B_1^d)$, $\|x-x_0\|\geq \tau_2$. Proof is completed. 
\subsection{Proof of Lemma~\ref{lemma:pointwise_error}}\label{Sec:proof_lemma:pointwise_error}
Let $\delta=b_1\left(\frac{\log \wt{n}}{\wt{n}}\right)^{\frac{1}{d+L_0}}$, where $b_1\geq 1$ is a sufficiently large positive constant to be specified later. Let $\mathcal{A}_z=\{z\in \mb B_1^d:\,\|z\|_2\leq 1-\delta,\,  G^{\ast}_\sm(z)\in \mb B_{r_m+0.25/L}(a_m)\}$ be a proper subset of $\mb B_1^d$. Let $h_{\wt{n}}=(\frac{\log \wt{n}}{\wt{n}})^{\frac{1}{d}}$ and $\m N_{h_{\wt{n}}}\subset \mathcal{A}_z$ be a minimal $h_{\wt{n}}$-covering set of $\mathcal{A}_z$ under the $\ell_2$ distance, where its cardinality satisfies $|\m N_{h_{\wt{n}}}|\leq C \frac{\wt{n}}{\log \wt{n}}$. For any $\widetilde z\in \m N_{h_{\wt{n}}}$, define $\delta_{\widetilde z}=b_2 (\frac{\log \wt{n}}{\wt{n}(1-\|{\widetilde z}\|_2)^{L_0}})^{\frac{1}{d}}$ with $b_2=\frac{1}{2} b_1^{\frac{L_0}{d}+1}$. 

We claim that it suffices to show that for sufficiently large $b_1$, it holds with probability at least $1-n^{-c}$ that for any $\widetilde{z}\in \m N_{h_{\wt{n}}}$, 
\begin{equation}\label{leup2eq2}
 \sum_{j\in \mathbb{N}_0^d\atop |j|\leq \lfloor\beta\rfloor}\frac{1}{j!}\, \|\wh{f}_\sm^{(j)}(\widetilde{z})- G^{*,(j)}_\sm(\widetilde{z})\|_{2}\, (\delta_{\widetilde{z}})^{|j|} \leq C \left(\frac{\log \wt{n}}{\wt{n}(1-\|\widetilde{z}\|_2)^{L_0}}\right)^{\frac{\beta}{d}}.
 \end{equation}
In fact, if this inequality holds, then we can apply a standard argument of approximation by the $h_{\wt{n}}$-covering set. Concretely, for any $z\in \mathcal{A}_z$, there exists $\widetilde{z}\in \m N_{h_{\wt{n}}}$ such that  $\|z-\widetilde{z}\|_2\leq h_{\wt{n}}= (\frac{\log \wt n}{\wt n})^{\frac{1}{d}}$, we can obtain by applying Taylor expansion to $G^{\ast}_\sm(z)-\wh{f}_\sm(z)$ that 
\begin{equation*}
 \begin{aligned}
 \|G^{\ast}_\sm(z)-\wh{f}_\sm(z)\|_2&\leq  C \sum_{j\in \mathbb{N}_0^d\atop |j| \leq \lfloor\beta\rfloor}
 \frac{1}{j!}\,\|\wh{f}_\sm^{(j)}(\widetilde{z})- G_\sm^{*,(j)}(\widetilde{z})\|_{2}\,\Big(\frac{\log \wt n}{\wt n}\Big)^{\frac{|j|}{d}} \,+ \, C\, \Big(\frac{\log \wt n}{\wt n}\Big)^{\frac{\beta}{d}}\\
 &\leq C\left(\frac{\log \wt{n}}{\wt{n}(1-\|\widetilde{z}\|_2)^{L_0}}\right)^{\frac{\beta}{d}} \leq C_1\left(\frac{\log \wt{n}}{\wt{n}(1-\|z\|_2)^{L_0}}\right)^{\frac{\beta}{d}},
     \end{aligned}
\end{equation*}
where the last step follows by choosing $b_1$ large enough such that $|\|z\|_2 - \|\widetilde z\|_2| \leq \|z-\widetilde z\|_2 \leq h_{\wt{n}} \leq \frac{1}{2}\delta \leq \frac{1}{2} (1-\|z\|_2)$, which implies $1-\|\widetilde z\|_2 \geq \frac{1}{2}(1-\|z\|_2)$.
Moreover, for any $z\in \mathcal{A}'_{z}=\{z\in \mb B_1^d:\, 1-\delta\leq \|z\|_2\leq 1,\, G^{\ast}_\sm(z)\in S_m\}$, there exists $\widetilde z\in \m N_{h_{\wt{n}}}$ such that $\|z-\widetilde{z}\|_2\leq C(\frac{\log \wt{n}}{\wt{n}})^{\frac{1}{d+L_0}}$. Thus for any $z\in \mathcal{A}'_{z}$, there exists $\widetilde z\in \m N_{h_{\wt n}}$ such that
\begin{equation*}
\begin{aligned}
  &\|G_\sm^{\ast}(z)-\wh{f}_\sm(z)\|_2\leq  C \sum_{j\in \mathbb{N}_0^d\atop |j|\leq \lfloor\beta\rfloor}\frac{1}{j!}\,\|\wh{f}_\sm^{(j)}(\widetilde{z})- G_\sm^{*,(j)}(\widetilde{z})\|_{2}\, \Big(\frac{\log \wt n}{\wt n}\Big)^{\frac{|j|}{d+L_0}} \,+ \, \Big(\frac{\log \wt n}{\wt n}\Big)^{\frac{\beta}{d+L_0}} \\ &\leq  C\sum_{j\in \mathbb{N}_0^d\atop |j|\leq \lfloor\beta\rfloor}\frac{1}{j!}\,\|\wh{f}_\sm^{(j)}(\widetilde{z})- G^{*,(j)}_\sm(\widetilde{z})\|_{2}\,\underset{\bar z\in \m N_{h_{\wt n}}}{\inf}(\delta_{\bar z})^{|j|} \,+ \, \Big(\frac{\log \wt n}{\wt n}\Big)^{\frac{\beta}{d+L_0}} \leq C\left(\frac{\log \wt{n}}{\wt{n}}\right)^{\frac{\beta}{d+L_0}}.
  \end{aligned}
   \end{equation*}
   
Now let us prove inequality~\eqref{leup2eq2}.
Since $S^{\dagger}_m=\mb B_{r_m+0.5/L}(a_m)\subset \widetilde{S}_m$, we have that $(G_\sm^\ast,Q_\sm^\ast)$ is a feasible solution to the optimization problem~\eqref{estimator1} in step one of submanifold estimation and $G_\sm^\ast\big(Q_\sm^\ast(X_i)\big)=X_i$ for all $X_i\in S_m^\dagger$. Therefore, we obtain by the optimality of the estimator $(\wh G_\sm,\wh Q_\sm)$ that
\begin{align*}
     \frac{1}{|I_1|}\sum_{i\in I_1} \|X_i-\wh{G}_\sm(\wh{Q}_\sm(X_i))\|_2^2\, \bold{1}_{S^{\dagger}_m}(X_i)=0.
\end{align*}
In particular, by restricting the sum to those $X_i$ in $\mb B_{\delta_{\widetilde{z}}}(\widetilde{z})$ for a fixed $\widetilde{z}\in \m N_{h_{\wt{n}}}$, we further obtain (recall that $\wh f_\sm = \wh{G}_\sm\circ\wh{Q}_\sm\circ G^{\ast}$)
\begin{align*}
      \frac{1}{|I_1|}\sum_{i\in I_1}  \|G^{\ast}_\sm( Q^{\ast}_\sm(X_i))-\wh{f}_\sm\, ( Q^{\ast}_\sm(X_i))\|^2_2 \cdot \bold{1}_{\mb B_{\delta_{\widetilde{z}}}(\widetilde{z})}(Q^{\ast}_\sm(X_i)) \cdot \bold{1}_{ S_{m}^{\dagger}}(X_i)=0.
\end{align*}
To proceed, we utilize the property that $\bold{1}_{\mb B_{\delta_{\widetilde{z}}}(\widetilde{z})}(Q^{\ast}_\sm(X_i))\cdot \bold{1}_{ S_{m}^{\dagger}}(X_i) = \bold{1}_{\mb B_{\delta_{\widetilde{z}}}(\widetilde{z})}(Q^{\ast}_\sm(X_i)) \cdot \bold{1}_{ \widetilde{S}_{m}}(X_i)$ where recall $S_m^{\dagger}\subset \mb B_{r_m+1/L}(a_m) \subset \widetilde{S}_m$. To see this, we only need to show that $Q^\ast_\sm(X_i) \in \mb B_{\delta_{\widetilde{z}}}(\widetilde z)$ plus $X_i\in \widetilde S_m$ imply $X_i\in S_m^\dagger$. In fact, by the Lipschitzness of $G^{\ast}_\sm$,  we have $\|X_i-G^{\ast}(\widetilde{z})\|_2\leq L\sqrt{D} \delta_{\widetilde{z}}$, which combined with $ G^{\ast}(\widetilde{z})\in \mb B_{r_m+0.25/L}(a_m)$ implies $X_i\in \mb B_{r_m+0.5/L}(a_m) ={S}_m^\dagger$ when $n\geq (4L^2b_1\sqrt{D})^{4(d+L_0)}$. Based on this replacement of indicator functions, we further obtain
\begin{align*}
     \frac{1}{|I_1|}\sum_{i\in I_1}  \|G^{\ast}_\sm( Q^{\ast}_\sm(X_i))-\wh{f}_\sm\, ( Q^{\ast}(X_i))\|^2_2 \cdot \bold{1}_{\mb B_{\delta_{\widetilde{z}}}(\widetilde{z})}(Q^{\ast}_\sm(X_i)) \cdot \bold{1}_{ \widetilde S_{m}}(X_i)=0.
\end{align*}
By applying the Taylor expansion to $G^{\ast}_\sm(z)-\wh{f}_\sm(z)$ around $\widetilde z$ in the preceding display and using the fact that $G^{\ast}_\sm-\wh{f}_\sm \in C_{C_0}^\beta(\mb B_1^d;\,\mb R^D)$ with some sufficiently large constant $C_0$, we can get the following \emph{localized basic inequality} after some algebra calculation
\begin{equation}\label{leup2eq1}
   \begin{aligned}
      & U_n(\widetilde z,\,\wh{f}_\sm):\,= \\
      &\frac{1}{|I_1|}\sum_{i\in I_1}  \bigg\|\sum_{j\in \mathbb{N}_0^d\atop |j|\leq \lfloor\beta\rfloor} \frac{1}{j!} \, \big(G_\sm^{*,(j)}(\widetilde{z}) - \wh{f}_\sm^{(j)}(\widetilde{z})\big) \, ( Q_\sm^{\ast}(X_i)-\widetilde{z})^{j}\bigg\|_2^2 \cdot \bold{1}_{\mb B_{\delta_{\widetilde{z}}}(\widetilde{z})}(Q_\sm^{\ast}(X_i))  
       \cdot \bold{1}_{ \widetilde{S}_{m}}(X_i)
       \\
 \leq &\, c\bigg((\delta_{\widetilde{z}})^{2\beta}+(\delta_{\widetilde{z}})^{\beta} \sum_{j\in \mathbb{N}_0^d\atop |j|\leq \lfloor\beta\rfloor}\frac{1}{j!}\, \big\|G^{*,(j)}_\sm(\widetilde{z})-\wh{f}_\sm^{(j)}(\widetilde{z})\big\|_2\,(\delta_{\widetilde{z}})^{|j|}\bigg) \cdot 
 \frac{1}{|I_1|}\sum_{i\in I_1}   \bold{1}_{\mb B_{\delta_{\widetilde{z}}}(\widetilde{z})}(Q_\sm^{\ast}(X_i))
 \cdot \bold{1}_{ \widetilde{S}_{m}}(X_i).
\end{aligned} 
\end{equation}
The second factor on the right hand side of~\eqref{leup2eq1} can be bounded by applying a simple union bound argument and Bernstein's inequality for binomials as follows. First, we can bound the probability
\begin{equation}\label{eqnboundprob}
\begin{aligned}
 &\mathbb{P}_{\mu^{\ast}}\big(X\in \widetilde{S}_{m}, \, Q^{\ast}_\sm(X)\in \mb B_{\delta_{\widetilde{z}}}(\widetilde{z})\big)
 = \mb P_{\mu^\ast}(X\in \widetilde S_m)\cdot \mb P_{\mu^\ast}(Q_\sm^{\ast}(X)\in \mb B_{\delta_{\widetilde{z}}}(\widetilde{z})\,|\,X\in \widetilde S_m) \\
 &\overset{(i)}{=} \mb P_{\mu^\ast}(X\in \widetilde S_m)\,\int_{\mb B_{\delta_{\widetilde{z}}}(\widetilde{z})} \nu^\ast_\sm(z) \,\dd z
 \leq C\,(1-\|\widetilde z\|_2)^{L_0}\,\delta_{\widetilde{z}}^d \cdot \mb P_{\mu^\ast}(X\in \widetilde S_m)\\
 &\leq C_1\, b_2^d\, \frac{\log \widetilde n}{\widetilde n} \cdot \mb P_{\mu^{\ast}}(X\in \widetilde{S}_m)\leq C_1\, b_2^d\, \frac{\log n}{n},
 \end{aligned}
\end{equation}
where step (i) follows by the fact that $Q^\ast_\sm(X)$ given $X\in \widetilde S_m$ is distributed as $\nu^\ast_\sm$, and the last step follows since $\widetilde n = n\cdot\mb P_{\mu^{\ast}}(X\in \widetilde{S}_m)$.
 Since the random variable $\bold{1}_{\mb B_{\delta_{\widetilde{z}}}(\widetilde{z})}(Q^{\ast}(X))\cdot \bold{1}_{ \widetilde{S}_{m}}(X)$ is uniformly bounded by $1$, and inequality~\eqref{eqnboundprob} implies its variance to be bounded by $C_1\, b_2^d\, \frac{\log n}{n}$, we may apply the Bernstein inequality and a simple union bound argument over all $\widetilde z \in \m N_{h_{\wt{n}}}$ (with $|\m N_{h_{\wt{n}}}|\leq C\frac{\wt n}{\log \wt n}$) to obtain that with probability at least $1-n^{-c}$,
 \begin{equation}\label{eqn:union+Bernstein}
\underset{\widetilde{z}\in \m N_{h_{\wt{n}}}}{\sup} \bigg| \frac{1}{|I_1|}\sum_{i\in I_1} \bold{1}_{\mb B_{\delta_{\widetilde{z}}}(\widetilde{z})}(Q_\sm^{\ast}(X_i)) \cdot \bold{1}_{ \widetilde{S}_{m}}(X_i)-  \mathbb{P}_{\mu^{\ast}} \big(X\in \widetilde{S}_{m},\, Q^{\ast}_\sm(X)\in \mb B_{\delta_{\widetilde{z}}}(\widetilde{z})\big)\bigg|\leq C_2 b_2^{\frac{d}{2}}\cdot\frac{\log n}{n},
\end{equation}
which together with~\eqref{eqnboundprob} leads to 
\begin{equation}\label{eqn:binomial_union}
\underset{\widetilde{z}\in \m N_{h_{\wt{n}}}}{\sup}  \bigg[\frac{1}{|I_1|}\sum_{i\in I_1} \bold{1}_{\mb B_{\delta_{\widetilde{z}}}(\widetilde{z})}(Q_\sm^{\ast}(X_i))\cdot \bold{1}_{ \widetilde{S}_{m}}(X_i)\bigg]\leq C\, b_2^d\cdot \frac{\log n}{n}.
\end{equation}

To analyze the quantity $U_n(\widetilde z,\,\wh{f}_\sm)$ on the left hand side of the localized basic inequality~\eqref{leup2eq1}, we will resort to the following lemma.
A proof of the lemma is provided in Section~\ref{sec:proof_lemma:empirical_proc}, which is quite technical and involved, based on applying the chaining and peeling techniques in empirical process theory to analyze the supreme of the empirical process $U_n(\widetilde z,\,f)$ indexed by $\widetilde z\in \m N_{h_{\widetilde n}}$ and a $\beta$-smooth function $f\in C_{L}^\beta(\mb B_1^d;R^D)$.
\begin{lemma}\label{lemma:empirical_proc}
  With probability at least $1-n^{-c}$, the following inequality holds for any $\beta$-smooth function $f\in C_{L}^\beta(\mb B_1^d;R^D)$ and $\widetilde z \in \m N_{h_{\widetilde n}}$,
  \begin{align*}
      \big|U_n(\widetilde z,\,f) - \mb E_{\mu^\ast}[U_n(\widetilde z,\,f)]\big| \leq  C\, b_2^{\frac{d}{2}}\cdot\frac{\log n}{n}\cdot \bigg\{\bar{\delta}^2+\Big[ \sum_{j\in \mathbb{N}_0^d\atop |j|\leq \lfloor\beta\rfloor}\frac{1}{j!} \, \|{f}^{(j)}(\widetilde{z})- G^{*,(j)}_\sm(\widetilde{z})\|_{2}\,(\delta_{\widetilde{z}})^{|j|}\Big]^2\bigg\},
  \end{align*}
  where $\bar{\delta} = \big(\frac{\log \widetilde n}{\widetilde n}\big)^{\frac{\beta}{d}}$ and the expectation is taken with respect to the randomness in $\{X_i\}_{i\in I_1}$.
\end{lemma}
\noindent Before applying this lemma, notice that for any $\widetilde{z}\in \m N_{h_{\wt{n}}}$, we can bound the expectation $\mb E_{\mu^\ast}[U_n(\widetilde z,\,\wh{f}_\sm\,)]$, where $f$ has been plugged-in with $\wh{f}_\sm$\,, by
\begin{equation}\label{eqn:emp_exp}
 \begin{aligned}
& \quad \mb E_{\mu^\ast}[U_n(\widetilde z,\,\wh{f}_\sm\,)]\\
&=\mathbb{E}_{\mu^{\ast}} \bigg[\Big\|\sum_{j\in \mathbb{N}_0^d\atop |j|\leq \lfloor\beta\rfloor} \frac{1}{j!} \,\big(G^{*,(j)}_\sm(\widetilde{z}) - \wh{f}_\sm^{(j)}(\widetilde{z}) \big)\, (Q^{\ast}_\sm(X)-\widetilde{z})^{j}\Big\|_2^2 \cdot  \bold{1}_{\mb B_{\delta_{\widetilde{z}}}(\widetilde{z})}(Q^{\ast}_\sm(X)) \cdot \bold{1}_{\widetilde{S}_m}(X)\bigg]\\
&\overset{(i)}{\geq} \underset{z\in \mb B_{\delta_{\widetilde{z}}}(\widetilde{z})}{\inf} \nu^{\ast}_\sm(z)\cdot\mathbb{P}(X\in \widetilde{S}_m) \,\int_{z\in \mb B_{\delta_{\widetilde{z}}}(\widetilde{z})} \Big\|\sum_{j\in \mathbb{N}_0^d\atop |j|\leq \lfloor\beta\rfloor}\frac{1}{j!} \,\big( G^{*,(j)}_\sm(\widetilde{z})-\wh{f}_\sm^{(j)}(\widetilde{z})\big)
\, (z-\widetilde{z})^{j}\Big\|_2^2\,\dd z\\
 &\overset{(ii)}{=} \delta_{\widetilde{z}}^d \, \underset{z\in \mb B_{\delta_{\widetilde{z}}}(\widetilde{z})}{\inf} \nu^{\ast}_\sm(z)\cdot\mathbb{P}(X\in \widetilde{S}_m) \,\int_{\mb B_1^d}  \Big\|\sum_{j\in \mathbb{N}_0^d\atop |j|\leq \lfloor\beta\rfloor}\frac{1}{j!}\,\delta_{\widetilde{z}}^j \big( G^{*,(j)}_\sm(\widetilde{z})-\wh{f}_\sm^{(j)}(\widetilde{z})\big)\, z^{j}\Big\|_2^2 \, \dd z \\
 &\geq C\, b_2^d\cdot \frac{\log n}{n} \cdot \int_{\mb B_1^d}   \Big\|\sum_{j\in \mathbb{N}_0^d\atop |j|\leq \lfloor\beta\rfloor}\frac{1}{j!}\delta_{\widetilde{z}}^j \,\big( G^{*,(j)}_\sm(\widetilde{z})-\wh{f}_\sm^{(j)}(\widetilde{z})\big)\,z^{j}\Big\|_2^2 \,\dd z,
   \end{aligned}
\end{equation}
where step (i) uses the fact that $Q_\sm^\ast(X)$ given $X\in \widetilde S_m$ is distributed as $\nu^\ast_\sm$, step (ii) follows by applying the change of variable of $\frac{z-\widetilde z}{\delta_{\widetilde z}} \to z$, and the last step uses $\widetilde n= n\,\mathbb{P}(X\in \widetilde{S}_m)$ and
 the property that $\nu^\ast_\sm(z) \geq c(1-\|\widetilde z\|_2)^{L_0}$ for all $z\in \mb B_{\delta_{\widetilde{z}}}(\widetilde{z})$ as long as $b_1$ in the definition of $\delta$ is sufficiently large.
Now using the fact that for any $d$-variate polynomial $\mathcal{S}(y)=\sum_{j\in \mb N_0^d,\,|j|\leq k} a_{j} y^{j}$, $y\in\mb R^d$, there exists some positive constant $C(d,k)$ only depending on $(d,k)$ such that 
\begin{equation*}
\int_{\mb B_1^d} \mathcal{S}^2(y) \, \dd y \geq C(d,k) \sum_{j\in \mb N_0^d,\,|j|\leq k} a_{j}^2,
  \end{equation*}
we can obtain that
\begin{equation}\label{eqn:exp_low_b}
 \begin{aligned}
  \int_{\mb B_1^d}   \Big\|\sum_{j\in \mathbb{N}_0^d\atop |j|\leq \lfloor\beta\rfloor}\frac{1}{j!}\delta_{\widetilde{z}}^j \,\big( G^{*,(j)}_\sm(\widetilde{z})-\wh{f}_\sm^{(j)}(\widetilde{z})\big)\,z^{j}\Big\|_2^2 \,\dd z
   \geq c\, \bigg(\sum_{j\in \mathbb{N}_0^d\atop |j|\leq \lfloor\beta\rfloor}\frac{1}{j!}\,\|\wh{f}_\sm^{(j)}(\widetilde{z})- G^{*,(j)}_\sm(\widetilde{z})\|_{2}\, (\delta_{\widetilde{z}})^{|j|}\bigg)^2.
     \end{aligned}
\end{equation}
Finally, by combining equations~\eqref{leup2eq1},~\eqref{eqnboundprob},~\eqref{eqn:emp_exp},~\eqref{eqn:exp_low_b} and Lemma~\ref{lemma:empirical_proc}, we obtain that with probability at least $1-n^{-c}$, for any $\widetilde{z}\in \m N_{h_{\wt{n}}}$, 
\begin{equation*}
 \begin{aligned}
 &\quad b_2^d \cdot \frac{\log n}{n} \cdot \bigg(\sum_{j\in \mathbb{N}_0^d\atop |j|\leq \lfloor\beta\rfloor}\frac{1}{j!}\, \|\wh{f}_\sm^{(j)}(\widetilde{z})- G_\sm ^{*,(j)}(\widetilde{z})\|_{2}\, (\delta_{\widetilde{z}})^{|j|}\bigg)^2\\
 &\leq C b_2^{\frac{d}{2}} \cdot \frac{\log n}{n} \cdot \bigg(\sum_{j\in \mathbb{N}_0^d\atop |j|\leq \lfloor\beta\rfloor}\frac{1}{j!}\, \|\wh{f}_\sm^{(j)}(\widetilde{z})- G_\sm ^{*,(j)}(\widetilde{z})\|_{2}\, (\delta_{\widetilde{z}})^{|j|}\bigg)^2 
 \, +\,  C b_2^{\frac{d}{2}}\cdot \Big(\frac{\log \wt{n}}{\wt{n}}\Big)^{\frac{2\beta}{d}}\cdot \frac{\log n}{n}\\
& \qquad +\, C b_2^d\cdot \frac{\log n}{n}\cdot  \bigg((\delta_{\widetilde{z}})^{2\beta}+(\delta_{\widetilde{z}})^{\beta} \sum_{j\in \mathbb{N}_0^d\atop |j|\leq \lfloor\beta\rfloor}\frac{1}{j!}\, \| G^{*,(j)}_\sm(\widetilde{z})-\wh{f}_\sm^{(j)}(\widetilde{z})\|_{2}(\delta_{\widetilde{z}})^{|j|}\bigg).
\end{aligned}
\end{equation*}
Consequently, the claimed inequality~\eqref{leup2eq2} follows from the above by choosing a sufficiently large $b_1$ (recall that $b_2=\frac{1}{2}b_1^{\frac{L_0}{d}+1}$) and the definition that $\delta_{\widetilde z}=b_2 (\frac{\log \wt{n}}{\wt{n}(1-\|{\widetilde z}\|_2)^{L_0}})^{\frac{1}{d}}$.


\subsection{Proof of Lemma~\ref{lemma:empirical_proc}}\label{sec:proof_lemma:empirical_proc}
Since $f\in C_{L}^\beta(\mb B_1^d;R^D)$, for any $z\in \mb B_1^d$ and $j\in \mathbb{N}_0^d$ with $|j|\leq \lfloor\beta\rfloor$, it holds that $\|{f}^{(j)}(z)\|_2\leq \sqrt{D} L=C_0$. For any fixed $\widetilde{z}\in \m N_{h_{\wt{n}}}$ and $\widetilde\delta>0$, let 
\begin{equation*}
\begin{aligned}
\bar{\mathcal{T}}(\widetilde{\delta})=\Big\{T=\{T_j\}_{j\in \mathbb{N}_0^d, \, |j|\leq \lfloor\beta\rfloor}\in [-C_0,\,C_0]^{D\times \binom{d+\lfloor \beta\rfloor-1}{d}} :\, \sum_{j\in \mathbb{N}_0^d\atop |j|\leq \lfloor\beta\rfloor}\frac{1}{j!}\, \big\|T_j-G^{*,(j)}_\sm(\widetilde{z})\big\|_2\, (\delta_{\widetilde{z}})^{|j|}\leq \widetilde{\delta}\Big\}.
\end{aligned}
\end{equation*}
We also define the following supreme of an empirical process indexed by $T\in \bar{\mathcal{T}}(\widetilde{\delta})$,
\begin{equation*}
 \begin{aligned}
&\quad Z_n(\widetilde{\delta})=\\
 &\underset{T\in \bar{\mathcal{T}}(\widetilde{\delta}) }{\sup}\Bigg| \, \mathbb{E}_{\mu^{\ast}} \bigg[\Big\|\sum_{j\in \mathbb{N}_0^d\atop |j|\leq \lfloor\beta\rfloor} \frac{1}{j!}\,\big(G^{*,(j)}_\sm(\widetilde{z}) -T_j\big)\, (Q_\sm^{\ast}(X)-\widetilde{z})^{j} \Big\|_2^2\cdot  \bold{1}_{\mb B_{\delta_{\widetilde{z}}}(\widetilde{z})}(Q^{\ast}_\sm(X)) \cdot \bold{1}_{\widetilde{S}_m}(X)\bigg]\\
& \ \ - \frac{1}{|I_1|} \sum_{i\in I_1} \bigg[\Big\|\sum_{j\in \mathbb{N}_0^d\atop |j|\leq \lfloor\beta\rfloor} \frac{1}{j!}\,\big(G^{*,(j)}_\sm(\widetilde{z}) -T_j\big)\, (Q_\sm^{\ast}(X_i)-\widetilde{z})^{j}\Big\|_2^2\cdot  \bold{1}_{\mb B_{\delta_{\widetilde{z}}}(\widetilde{z})}(Q^{\ast}_\sm(X)) \cdot \bold{1}_{\widetilde{S}_m}(X) \bigg]\Bigg|,
  \end{aligned}
\end{equation*}
and $R_n(\widetilde{\delta})=\mathbb{E}_{\mu^\ast}\big[ Z_n(\widetilde{\delta}) \big]$. We will first prove a concentration inequality for a fixed radius $\widetilde{\delta}>0$, and then using the peeling technique to allow the radius to be random, which leads to the desired result.

To apply the Talagrand concentration inequality (see, for example, Theorem 3.27 of~\cite{wainwright_2019}) for bounding the difference $|Z_n(\widetilde{\delta}) - R_n(\widetilde \delta)|$ for a fixed $\widetilde{\delta}>0$, we notice that each additive component in the second empirical sum above has second moment uniformly bounded by
\begin{equation*}
 \begin{aligned}
& \mathbb{E}_{\mu^{\ast}} \bigg[\underset{T\in \bar{\mathcal{T}}(\widetilde{\delta}) }{\sup}\Big(\Big\|\sum_{j\in \mathbb{N}_0^d\atop |j|\leq \lfloor\beta\rfloor} \frac{1}{j!}\,\big(G^{*,(j)}_\sm(\widetilde{z}) -T_j\big)\, (Q_\sm^{\ast}(X)-\widetilde{z})^{j} \Big\|_2^4\cdot  \bold{1}_{\mb B_{\delta_{\widetilde{z}}}(\widetilde{z})}(Q^{\ast}_\sm(X)) \cdot \bold{1}_{\widetilde{S}_m}(X)\Big)\bigg]\\
&\leq \underset{z\in \mb B_{\delta_{\widetilde{z}}}(\widetilde{z})\atop T\in \bar{\mathcal{T}}(\widetilde{\delta})} {\sup} \Big\|\sum_{j\in \mathbb{N}_0^d\atop |j|\leq \lfloor\beta\rfloor} \frac{1}{j!}\,\big(G^{*,(j)}_\sm(\widetilde{z}) -T_j\big)\, (z-\widetilde{z})^{j} \Big\|_2^4 \cdot \mathbb{P}_{\mu^{\ast}}\big(X\in \widetilde{S}_{m},\, Q_\sm^{\ast}(X)\in \mb B_{\delta_{\widetilde{z}}}(\widetilde{z})\big)\\
&\leq C \underset{T\in \bar{\mathcal{T}}(\widetilde{\delta}) }{\sup}\bigg(\sum_{j\in \mathbb{N}_0^d\atop |j|\leq \lfloor\beta\rfloor}\frac{1}{j!}\,\|T_j- G^{*,(j)}_\sm(\widetilde{z})\|_{2}\, (\delta_{\widetilde{z}})^{|j|}\bigg)^4\cdot  b_2^d \cdot \frac{\log n}{n}\\
&\leq C\, b_2^d\, \widetilde \delta^4\cdot \frac{\log n}{n},
  \end{aligned}
\end{equation*}
where we have used inequality~\eqref{eqnboundprob} to bound $\mathbb{P}_{\mu^{\ast}}\big(X\in \widetilde{S}_{m},\, Q_\sm^{\ast}(X)\in \mb B_{\delta_{\widetilde{z}}}(\widetilde{z})\big)$.
Moreover, each additive component can be almost surely bounded by 
\begin{equation*}
 \begin{aligned}
& \quad \underset{z\in \mb B_{\delta_{\widetilde{z}}}(\widetilde{z})}{\sup} \Big\|\sum_{j\in \mathbb{N}_0^d\atop |j|\leq \lfloor\beta\rfloor} \frac{1}{j!}\,\big(G^{*,(j)}_\sm(\widetilde{z}) -T_j\big)\, (Q_\sm^{\ast}(X)-\widetilde{z})^{j} \Big\|_2^2\\
&\leq C\,  \bigg(\sum_{j\in \mathbb{N}_0^d\atop |j|\leq \lfloor\beta\rfloor}\frac{1}{j!} \, \|T_j- G^{*,(j)}(\widetilde{z})\|_{2}\, (\delta_{\widetilde{z}})^{|j|}\bigg)^2
\leq C\, \widetilde \delta^2.
   \end{aligned}
\end{equation*}
Based on these two bounds, we can apply the Talagrand concentration inequality to obtain that for any $s\geq 0$,
\begin{equation}~\label{Talagrand}
\mathbb{P} \big(Z_n(\widetilde{\delta}) \geq R_n(\widetilde{\delta}) +s^2\big )\leq 2 \exp\left(-\frac{c\,ns^4}{s^2 \,\widetilde{\delta}^2 + b_2^d\,\widetilde{\delta}^4 \cdot \frac{\log n}{n}}\right).
   \end{equation}
It remains to bound the expectation $R_n(\widetilde{\delta})$ via the symmetrization technique and chaining. 
By a standard symmetrization, we can get 
\begin{equation*}
\begin{aligned}
&R_n(\widetilde{\delta})\leq \frac{2}{\sqrt{|I_1|}}\, \mathbb{E}\Bigg[\underset{T\in \bar{\mathcal{T}}(\widetilde{\delta})}{\sup}\\
&\qquad \Bigg|\frac{1}{\sqrt{|I_1|}} \sum_{i\in I_1} \varepsilon_i\bigg[\Big\|\sum_{j\in \mathbb{N}_0^d\atop |j|\leq \lfloor\beta\rfloor} \frac{1}{j!}\,\big(G^{*,(j)}_\sm(\widetilde{z}) -T_j\big)\, (Q_\sm^{\ast}(X)-\widetilde{z})^{j} \Big\|_2^2\cdot \bold{1}_{\mb B_{\delta_{\widetilde{z}}}(\widetilde{z})}(Q_\sm^{\ast}(X_i)) \cdot \bold{1}_{\widetilde{S}_m}(X_i)\bigg]\Bigg|\Bigg],
\end{aligned}
\end{equation*}
where $\{\varepsilon_i\}_{i=1}^n$ are $n$ i.i.d.~copies from the Rademacher distribution, i.e.~$\mb P(\varepsilon_i = 1) = \mb P(\varepsilon_i= -1) = 0.5$. 
Since given $\{X_i\}_{i\in I_1}$, the stochastic process inside the supreme is a sub-Gaussian process with intrinsic metric 
 \begin{equation*}
 \begin{aligned}
& \quad d_n^2(T,\,\widetilde{T})\\
&=\frac{1}{|I_1|} \sum_{i\in I_1}  \bigg(\Big\|\sum_{j\in \mathbb{N}_0^d\atop |j|\leq \lfloor\beta\rfloor} \frac{1}{j!}\,\big(G^{*,(j)}_\sm(\widetilde{z}) -T_j\big)\, (Q_\sm^{\ast}(X_i)-\widetilde{z})^{j} \Big\|_2^2 \\
&\qquad\qquad\qquad - \Big\|\sum_{j\in \mathbb{N}_0^d\atop |j|\leq \lfloor\beta\rfloor} \frac{1}{j!}\,\big(G^{*,(j)}_\sm(\widetilde{z}) -\widetilde T_j\big)\, (Q_\sm^{\ast}(X_i)-\widetilde{z})^{j} \Big\|_2^2\bigg)^2 \cdot \bold{1}_{\mb B_{\delta_{\widetilde{z}}}(\widetilde{z})}(Q^{\ast}_\sm (X_i))) \cdot \bold{1}_{\widetilde{S}_m}(X_i)\\
 &\leq C\,\widetilde{\delta}^4\, \frac{1}{|I_1|} \sum_{i\in I_1}   \bold{1}_{\mb B_{\delta_{\widetilde{z}}}(\widetilde{z})}(Q_\sm^{\ast}(X_i))) \cdot \bold{1}_{\widetilde{S}_m}(X_i),
 \end{aligned}
\end{equation*}
for any $T,\widetilde{T}\in\bar{\mathcal{T}}(\widetilde \delta)$, where the last step uses the definition of $\bar{\mathcal{T}}(\widetilde \delta)$. The above combined with inequality~\eqref{eqnboundprob} implies 
\begin{equation*}
\mathbb{E}_{\mu^\ast}\Big[ \underset{T,\widetilde{T}\in \bar{\mathcal{T}}(\delta)}{\sup} d_n^2(T,\widetilde{T})\Big]\leq C\,b_2^d\, \widetilde{\delta}^4 \cdot\frac{\log n}{n}\quad \mbox{and}\quad
d_n(T,\widetilde{T}) \leq C \widetilde{\delta} \sum_{j\in \mathbb{N}_0^d\atop |j|\leq \lfloor\beta\rfloor}\frac{1}{j!}\, \|T_j-\widetilde{T}_{j}\|_{2}\, \delta_{\widetilde{z}}^{|j|}.
 \end{equation*}
 Lastly, let $\m K_n(\delta)= \underset{T,\widetilde{T}\in \bar{\mathcal{T}}(\delta)}{\sup} d_n^2(T,\widetilde{T})$, by applying the standard chaining via Dudley's inequality, we can get 
\begin{equation}\label{dudley}
 \begin{aligned}
 R_n(\widetilde{\delta})&\leq C\, \frac{1}{\sqrt{n}}\, \mathbb{E}_{\mu^\ast} \Big[ \int_{0}^{ \m K_n(\widetilde\delta)}\sqrt{\log \frac{\widetilde{\delta}}{u}}\,\dd u\Big]\\
 &= C\, \frac{1}{\sqrt{n}}\, \mathbb{E}_{\mu^\ast} \Big[ \m K_n(\widetilde\delta)\cdot\int_{0}^{1}\sqrt{\log \frac{\widetilde{\delta}}{u\cdot \m K_n(\widetilde\delta)}}\,\dd u\Big]\\
 &= C\, \frac{1}{\sqrt{n}}\, \mathbb{E}_{\mu^\ast} \Big[   \m K_n(\widetilde\delta)\cdot \bold{1}( \m K_n(\widetilde\delta)\leq b_2^{\frac{d}{2}}\widetilde \delta^2\sqrt{\frac{\log n}{n}})\int_{0}^{1}\sqrt{\log \frac{\widetilde{\delta}}{u \cdot\m K_n(\widetilde\delta)}}\,\dd u\Big]\\
 &+C\, \frac{1}{\sqrt{n}}\, \mathbb{E}_{\mu^\ast} \Big[   \m K_n(\widetilde\delta)\cdot \bold{1}( \m K_n(\widetilde\delta)> b_2^{\frac{d}{2}}\widetilde \delta^2\sqrt{\frac{\log n}{n}})\int_{0}^{1}\sqrt{\log \frac{\widetilde{\delta}}{u \cdot\m K_n(\widetilde\delta)}}\,\dd u\Big]\\
 &\leq C_1\, b_2^{\frac{d}{2}} \cdot \frac{\log ({n}/{\widetilde\delta})}{n}\cdot \widetilde{\delta}^2,
    \end{aligned}
\end{equation}
where we have used the fact that the $u$-covering entropy of $\bar{\mathcal{T}}(\widetilde \delta)$ relative to metric $d_n$ is at most $C_2\log\frac{\widetilde\delta}{u}$ for $u\in(0,1)$ where $C_2$ depends on $(d,D)$ (at most polynomial dependence on $D$). By combining this with inequality~\eqref{Talagrand}, we obtain that for all $t\geq 1$,
\begin{align}\label{fix_delta_tala}
    \mathbb{P} \Big(Z_n(\widetilde{\delta}) \geq  C\, t^2\,b_2^{\frac{d}{2}} \cdot \frac{\log (n/\widetilde \delta)}{n}\cdot \widetilde{\delta}^2\Big )\leq 2 \exp\Big(-c\,t^2\, \log (n/\widetilde \delta)\Big).
\end{align}

Finally, we apply the peeling technique to extend the above high probability bound on $Z_n(\widetilde\delta)$ to the random radius $\widetilde\delta=\sum_{j\in \mathbb{N}_0^d\atop |j|\leq \lfloor\beta\rfloor}\frac{1}{j!}\, \big\|\wh f^{(j)}_\sm-G^{*,(j)}_\sm(\widetilde{z})\big\|_2\, (\delta_{\widetilde{z}})^{|j|}$. Specifically, we first set the basic level $\bar{\delta} =(\frac{\log \wt{n}}{\wt{n}})^{\frac{\beta}{d}}$, and for $s=1,\cdots, S$ with $S\leq C\log \frac{1}{\bar{\delta}}$, define sets
\begin{equation*}
\begin{aligned}
\widetilde{\mathcal{T}}_0&=\Big\{T=\{T_j\}_{j\in \mathbb{N}_0^d, |j|\leq \lfloor\beta\rfloor}\in [-C_0,\,C_0]^{D\times \binom{d+\lfloor \beta\rfloor-1}{d}}:\,  \sum_{j\in \mathbb{N}_0^d\atop |j|\leq \lfloor\beta\rfloor}\frac{1}{j!}\, \|T_j- G^{*,(j)}(\widetilde{z})\|_{2}\, (\delta_{\widetilde{z}})^{|j|}\leq \bar{\delta}\Big\};\\
\widetilde{\mathcal{T}}_s&=\Big\{T=\{T_j\}_{j\in \mathbb{N}_0^d, |j|\leq \lfloor\beta\rfloor}\in [-C_0,\,C_0]^{D\times \binom{d+\lfloor \beta\rfloor-1}{d}}:\,  2^{s-1}\bar{\delta} \leq\sum_{j\in \mathbb{N}_0^d\atop |j|\leq \lfloor\beta\rfloor}\frac{1}{j!}\,\|T_j- G^{*,(j)}(\widetilde{z})\|_{2}\,(\delta_{\widetilde{z}})^{|j|}\leq 2^s\bar{\delta}\Big\}.
\end{aligned}
\end{equation*}
 By applying inequality~\eqref{fix_delta_tala} to $\widetilde\delta = 2^s\bar\delta$ for $s\in[S]$ with sufficiently large constant $t>0$, as $C_1\leq-\log(2^s\bar\delta)\leq C_2\log n$,  we obtain that 
 \begin{equation*}
\mathbb{P}\left(Z_n(\bar{\delta})\geq C\,b_2^{\frac{d}{2}}\,\frac{\log n}{n} \,\bar{\delta}^2\right) + \sum_{s=1}^S \mathbb{P}\left(Z_n(2^s \bar{\delta})\geq C\,b_2^{\frac{d}{2}} \,\frac{\log n}{n}\, 4^s \bar{\delta}^2\right)\leq n^{-(c+1)}.
\end{equation*}
Note that for any $T\in \widetilde{\mathcal{T}}_s$ and any $s\in\{0\}\cup [S]$, the event $Z_n(2^s \bar{\delta})\leq C\,b_2^{\frac{d}{2}} \,\frac{\log n}{n}\, 4^s \bar{\delta}^2$ implies
 \begin{align*}
     &\Bigg| \, \mathbb{E}_{\mu^{\ast}} \bigg[\Big\|\sum_{j\in \mathbb{N}_0^d\atop |j|\leq \lfloor\beta\rfloor} \frac{1}{j!}\,\big(G^{*,(j)}_\sm(\widetilde{z}) -T_j\big)\, (Q_\sm^{\ast}(X)-\widetilde{z})^{j} \Big\|_2^2\cdot  \bold{1}_{\mb B_{\delta_{\widetilde{z}}}(\widetilde{z})}(Q^{\ast}_\sm(X)) \cdot \bold{1}_{\widetilde{S}_m}(X)\bigg]\\
& \ \ - \frac{1}{|I_1|} \sum_{i\in I_1} \bigg[\Big\|\sum_{j\in \mathbb{N}_0^d\atop |j|\leq \lfloor\beta\rfloor} \frac{1}{j!}\,\big(G^{*,(j)}_\sm(\widetilde{z}) -T_j\big)\, (Q_\sm^{\ast}(X_i)-\widetilde{z})^{j}\Big\|_2^2\cdot  \bold{1}_{\mb B_{\delta_{\widetilde{z}}}(\widetilde{z})}(Q^{\ast}_\sm(X)) \cdot \bold{1}_{\widetilde{S}_m}(X) \bigg]\Bigg|\\
&\leq  c_1\,b_2^{\frac{d}{2}}\,\frac{\log n}{n}\, \Bigg\{\bar\delta^2 + \bigg( \sum_{j\in \mathbb{N}_0^d\atop |j|\leq \lfloor\beta\rfloor}\frac{1}{j!}\,\|T_j- G^{*,(j)}_\sm(\widetilde{z})\|_{2}\, (\delta_{\widetilde{z}})^{|j|}\bigg)^2\Bigg\}.
 \end{align*}
 Finally, since for any $f\in C_{L}^\beta(\mb B_1^d;R^D)$, $T_f:\,=\{T_{f,j}=f^{(j)}\}_{j\in \mathbb{N}_0^d,|j|\leq \lfloor\beta\rfloor}$ must belong to some $\widetilde {\m T}_s$, the claimed result is a consequence of the two preceding displays and a simple union bound over $\widetilde z \in \m N_{h_{\wt n}}$ where $|\m N_{h_{\wt{n}}}|\leq C\, \frac{\wt{n}}{\log \wt{n}} \leq C\, n$.

 \subsection{Proof of Lemma~\ref{lemma:density_regularity}}\label{sec:proof_lemma:density_regularity}
   Let $\wt{\m A}_z=\{z\in \mb B_1^d:\,\|z\|_2\leq 1-\epsilon, G^{\ast}_\sm(z)\in \mb B_{r_m+0.25/L}(a_m)\}$ for $\epsilon$ being a small positive constant independent of $n$ that will be specified later. From inequality~\eqref{leup2eq2}, we can get that (recall that $\wh{f}_\sm=\wh G_\sm\circ \wh l_\sm$)
\begin{equation*}
 \begin{aligned}
&\underset{z\in \wt{\m A}_z}{\sup}\|G^{\ast}_\sm(z)-\wh{f}_\sm(z)\|_2\leq C_{\epsilon} \Big(\frac{\log \wt{n}}{\wt{n}}\Big)^{\frac{\beta}{d}}, \quad\mbox{and}\\
&\underset{z\in \wt{\m A}_z}{\sup}\|\bold{J}_{G^{\ast}_\sm}(z)-\bold{J}_{\wh{f}_\sm}(z)\|_F \leq C_{\epsilon} \Big(\frac{\log \wt{n}}{\wt{n}}\Big)^{\frac{\beta-1}{d}},\quad\mbox{(since $\beta>1$)}
      \end{aligned}
\end{equation*} 
where $C_{\epsilon}$ is a constant depending on $\epsilon$. Moreover, since $G^{\ast}_\sm\in C^{\beta}_{C_0}(\mathbb{R}^d;\mathbb{R}^D)$ and $\wh{f}_\sm \in C^{\beta}_{C_0}(\mathbb{R}^d;\mathbb{R}^D)$ for some constant $C_0>0$, we can extend the supreme to $\Omega_m$ as
 \begin{equation}\label{eqndensitysmooth}
 \begin{aligned}
&\underset{z\in \Omega_m}{\sup}\,\|G^{\ast}_\sm(z)-\wh{f}_\sm(z)\|_2\leq C_{\epsilon} \Big(\frac{\log \wt{n}}{\wt{n}}\Big)^{\frac{\beta}{d}}+C_1\sqrt{Dd}\, \epsilon;\\
&\underset{z\in\Omega_m}{\sup}\,\|\bold{J}_{G_\sm^{\ast}}(z)-\bold{J}_{\wh{f}_\sm}(z)\|_F \leq C_{\epsilon} \Big(\frac{\log \wt{n}}{\wt{n}}\Big)^{\frac{\beta-1}{d}}+C_1\sqrt{D}\,d\, \epsilon.
      \end{aligned}
\end{equation} 
 By the fact that for $z\in \mb B_1^d$, it holds that $z=Q^{\ast}_\sm(G^{\ast}_\sm(z))$, we obtain $I_{d}=\bold{J}_{Q^{\ast}_\sm}(G^{\ast}_\sm(z)) \, \bold{J}_{G^{\ast}_\sm}(z)$. Since $Q^{\ast}_\sm$ is $L$-Lipschitz, $\bold{J}_{Q^{\ast}_\sm}(G^{\ast}_\sm(z))$ has bounded operator norm, which implies
 ${\rm det}({\bold{J}^T_{G^{\ast}_\sm}(z)}\,\bold{J}_{G^{\ast}_\sm}(z))\geq c$ for some positive constant $c>0$. Therefore, the second display in~\eqref{eqndensitysmooth} implies ${\rm det}(\bold{J}^T_{\wh{f}_\sm}(z)\,\bold{J}_{\wh{f}_\sm}(z))\geq \frac{c}{2}$ for all sufficiently small $\epsilon$ and sufficiently large $n$ (recall that $\wt n\geq \frac{1}{2} \sqrt{n\log n}$).
 Now by using the identity (using $\wh{f}_\sm=\wh G_\sm\circ \wh l_\sm$)
\begin{equation*}
 \begin{aligned}
 \bold{J}^T_{\wh{f}_\sm}(z) \,  \bold{J}_{\wh{f}_\sm}(z)
 &=\left(\bold{J}_{\wh{G}_\sm}(\wh{l}_\sm(z))  \, \bold{J}_{\wh{l}_\sm}(z)\right)^T\left(\bold{J}_{\wh{G}_\sm}(\wh{l}_\sm(z)) \, \bold{J}_{\wh{l}_\sm}(z)\right)\\
 &=\bold{J}_{\wh{l}_\sm}^T(z) \, \bold{J}_{\wh{G}_\sm}^T(\wh{l}_\sm(z))\, \bold{J}_{\wh{G}_\sm}(\wh{l}_\sm(z)) \, \bold{J}_{\wh{l}_\sm}(z),
      \end{aligned}
\end{equation*} 
 by taking determinant we further obtain (note that $\bold{J}_{\wh{l}_\sm}(z)$ is a square matrix)
\begin{equation*}
 \begin{aligned}
{\rm det}^2\left(\bold{J}_{\wh{l}_\sm}(z)\right) \cdot {\rm det}\left(\bold{J}_{\wh{G}_\sm}^T(\wh{l}_\sm(z)) \,  \bold{J}_{\wh{G}_\sm}(\wh{l}_\sm(z))\right)\geq \frac{c}{2}.
        \end{aligned}
\end{equation*} 
 Since both $\wh G_\sm$ and $\wh Q_\sm$ are $L$-Lipschitz, we can further deduce that $c_1\leq{\rm det}(\bold{J}_{\wh{l}_\sm}(z))\leq c_2$ for all $z\in \Omega_m$. In addition, since by definition $\wh l_\sm\in C_{C_0}^\beta(\mb R^d;\mb R^d)$ for some constant $C_0$ and $\beta>1$, for sufficiently small $\epsilon>0$, we have that 
 \begin{align}\label{eq:jacob_bound}
     \frac{1}{2}c_1\leq{\rm det}(\bold{J}_{\wh{l}_\sm}(z))\leq 2c_2
 \end{align}
  holds for all $z\in \Omega_{m,\epsilon}:\,=  \{z\in \mb B_{\epsilon}(\widetilde{z}):\, \widetilde{z}\in \Omega_m\}$, the $\epsilon$-enlargement of set $\Omega_m$. 
 
 We claim that $\wh l_\sm$ is globally invertible over $\Omega_{m,\epsilon}$ when $\epsilon$ is small enough. Otherwise, suppose there exist distinct $z_0$ and $z_1$ in $\Omega_{m,\epsilon}$ such that $\wh l_\sm(z_0)=\wh l_\sm(z_1)$.
 Since~\eqref{eq:jacob_bound} implies $\wh l_\sm$ to be locally invertible, meaning that there exists some constant $b_0>0$ independent of $\epsilon$ such that $\|z_0-z_1\|\geq b_0$. By the definition of $\Omega_{m,\epsilon}$ and the Lipschitzness of $\wh G_\sm$ and $\wh l_\sm$, there exist $\bar z_0$ and $\bar z_1$ in $\Omega_m$ such that (for sufficiently small $\epsilon$)
 \begin{align}
     \|\bar z_0 - \bar z_1\|\geq \frac{1}{2} b_0,\quad \|\wh{l}_\sm(\bar{z}_0)-\wh{l}_\sm(\bar{z}_1)\|_2\leq C \epsilon\quad \mbox{and}\quad 
     \|\wh f_\sm(\bar z_0)- \wh f_\sm(\bar z_1)\| \leq C\epsilon.
 \end{align}
 The third display above combined with the first display in~\eqref{eqndensitysmooth} implies $\|G^{\ast}_\sm(\bar{z}_0)-G^{\ast}_\sm(\bar{z}_1)\|_2\leq C_1\epsilon$. On the other hand, from the first display above and the Lipschitzness of $Q^\ast_\sm$, we have
 \begin{align*}
     \frac{1}{2} b_0 \leq \|\bar z_0 - \bar z_1\| = \|Q^{\ast}_\sm(G^{\ast}_\sm(\bar{z}_0)-Q_\sm^{\ast}(G_\sm^{\ast}(\bar{z}_1))\|_2\leq C\|G^{\ast}_\sm(\bar{z}_0)-G^{\ast}_\sm(\bar{z}_1)\|_2\leq CC_1\epsilon,
 \end{align*}
 which is a contradiction when $\epsilon$ is chosen small enough.
 
 Let $\wh{l}_\sm^{-1}:  \wh{l}_\sm(\Omega_{m,\epsilon/2})\to\Omega_{m,\epsilon/2}$ be the inverse of $\wh{l}_\sm$ over ${\Omega_{m,\epsilon/2}}$. By using~\eqref{eq:jacob_bound} and the inverse function theorem for H\"{o}lder space (see for example, Appendix A of ~\citep{Eldering2013}), we can conclude $\wh{l}_\sm^{-1} \in C^{ \alpha+1}_{C_0} (\wh{l}_\sm(\Omega_{m,\epsilon/2});\mathbb{R}^d)$ for some sufficiently large constant $C_0$. This completes the proof of the first part in the lemma.
 
 The expression of the density function of $\nu_{\sm,\wh Q_{\sm}}^\ast=[\wh Q_\sm]_\# (\rho_m\mu^\ast)$ is an immediate consequence by applying the change of variable of $z=\wh Q_\sm(x)$ or $x=G^\ast_\sm\circ \wh l_\sm^{-1}(z)$ to the right hand side of~\eqref{eqn:change_of_v} and using the definition that $[G^\ast_\sm]_\# \nu^\ast_\sm = \mu^\ast|_{\widetilde S_m} = \big[\mb P(X\in \widetilde S_m)\big]^{-1} \mu^\ast$ on $\widetilde S_m$. Moreover, since $G^\ast_\sm\in C_L^\beta(\mb R^d;\mb R^D)$, $\rho_m\in C^\infty(\mb B_L^D)$ is supported on $S_m = \mb B_{r_m}(a_m)$ and the restriction of $\nu^\ast_\sm$ on $\mb B_1^d$ belongs to $C_L^\alpha(\mb B_1^d)$, we can conclude that $\nu^\ast_\sm \cdot (\rho_m\circ G^\ast_\sm) \in C_L^\alpha(\mb R^d)$, where we have used the facts that $\beta \geq \alpha+1$ and either $\mb B_{r_m}(a_m)\cap G^\ast_\sm\big(\mb B_1^d\setminus \mb B_{1-\epsilon}^d\big)=\emptyset$ or $\nu_\sm^\ast\in C_L^\alpha(\mb R^d)$.
This together with $\wh{l}_\sm^{-1} \in C^{ \alpha+1}_{C_0} (\wh{l}_\sm(\Omega_{m,\epsilon/2});\mathbb{R}^d)$ implies 
  ${\nu}^{\ast}_{\wh{Q}}\in C^{ \alpha}_{C_1} (\mathbb{R}^d)$ for some constant $C_1$.

\subsection{Proof of Lemma~\ref{coveringmanifold}}\label{sec:proof_coveringmanifold}
The bound for $\epsilon\geq 1$ is trivial, so we only consider $\epsilon\in(0,1)$.
Choose $\Delta={\epsilon}^{\frac{1}{\widetilde{\gamma}}}$ and $m=\lceil\Delta^{-1}\rceil$. For each $\xi=(\xi_1,\xi_2, \cdots, \xi_{d})\in [m]^d$, define $z_{\xi}=\Delta\xi$ and $x_{\xi}=G(z_{\xi})$. For any integer $s\in\mb N_0$, denote $\epsilon_s=\frac{\epsilon}{\Delta^s}= \epsilon^{1-\frac{s}{\widetilde{\gamma}}}$. For any multi-index $k\in \mb N_0^d$, let $\alpha_{\xi}^{(k)}(f)=\big\lfloor\frac{f^{(k)}(x_{\xi})}{\epsilon_{|k|}}\big\rfloor$. 

Consider any two functions $f_1,f_2 \in C^{\widetilde{\gamma}}_1(\mathbb{R}^D)$. Suppose for any $\xi\in[m]^d$ and multi-index $k\in\mathbb{N}_0^d$ with $|k|\leq \lfloor\widetilde{\gamma}\rfloor$, it holds that $\alpha_{\xi}^{(k)}(f_1)=\alpha_{\xi}^{(k)}(f_2)$. Then by the Lipschitzness of $G$, we can get that for any $z\in[0,1]^d$,
\begin{equation*}
 \begin{aligned}
&\quad \big|f_1(G(z))-f_2(G(z))\big|\\
&\leq C\, \sum_{k\in \mathbb{N}_0^d, \, |k|\leq \lfloor\widetilde{\gamma}\rfloor}\frac{\big|f_1^{(k)}(x_{\xi})-f_2^{(k)}(x_{\xi})\big|}{k!}\, \Delta^{|k|}+\Delta^{\widetilde{\gamma}}\leq C_1\epsilon.
\end{aligned}
\end{equation*} 
Therefore, $\|f_1-f_2\|_{L^{\infty}(\mathcal{X}_{G})}\leq C\epsilon$. 

First consider $\xi^{[1]}=(1,1,\cdots,1)$. For any $k\in \mathbb{N}_0^d$ with $|k|\leq  \lfloor\widetilde{\gamma}\rfloor$, since $f\in C^{\widetilde{\gamma}}_1(\mathbb{R}^D)$, we have $\alpha^{(k)}_{\xi^{[1]}}\leq C\,\frac{1}{\epsilon_{|k|}}$. Therefore the total number $N^{(k)}_{\xi^{[1]}}$ of possible values of $\alpha^{(k)}_{\xi^{[1]}}$ (it must be an integer) is upper bounded by $ \frac{C_1}{\epsilon_{|k|}}$. Therefore, the logarithm of  total number of choices for $\{\alpha^{(k)}_{\xi^{[1]}}:\,k\in \mathbb{N}_0^d,\,|k|\leq  \lfloor\widetilde{\gamma}\rfloor\}$ is at most
\begin{equation*}
 \begin{aligned}
 \sum_{k\in \mathbb{N}_0^d, \, |k|\leq \lfloor\widetilde{\gamma}\rfloor} \log N^{(k)}_{\xi^{[1]}} \leq C\,  \sum_{k\in \mathbb{N}_0^d, \, |k|\leq \lfloor\widetilde{\gamma}\rfloor} \Big(1-\frac{|k|}{\widetilde{\gamma}}\Big)\log \frac{1}{\epsilon} \leq C_1 \log\frac{1}{\epsilon},
         \end{aligned}
\end{equation*} 
where constant $C_1$ only depends on $d$ and $\gamma$.

Next consider $\xi^{[2]}=(2,1,1,\cdots,1)$. For any $f\in C^{\widetilde{\gamma}}_1(\mathbb{R}^D)$ and $k\in \mathbb{N}_0^d$ with $|k|\leq  \lfloor\widetilde{\gamma}\rfloor$, note that
\begin{equation*}
\bigg|f^{(k)}(x_{\xi^{[2]}})-\sum_{\eta\in \mathbb{N}_0^d, \, |\eta|\leq \lfloor\widetilde{\gamma}\rfloor-|k|} \frac{\epsilon_{|k|+|\eta|}\, \alpha_{\xi^{[2]}}^{(k+\eta)}(f)}{\eta!}\,(x_{\xi^{[2]}}-x_{\xi^{[1]}})^{\eta}\bigg|\leq C\, \epsilon_{|k|}.
\end{equation*}
Therefore, given $\{\alpha^{(k)}_{\xi^{[1]}}:\,k\in \mathbb{N}_0^d,\,|k|\leq  \lfloor\widetilde{\gamma}\rfloor\}$ which consists of all $\alpha^{(k)}$ values at location $\xi^{[1]}$, the total number of possible values of $\{\alpha^{(k)}_{\xi^{[2]}}:\,k\in \mathbb{N}_0^d,\,|k|\leq  \lfloor\widetilde{\gamma}\rfloor\}$ at location $\xi^{[2]}$ is at most $\frac{ C\, \epsilon_{|k|}}{\epsilon_{|k|}}=C$, which is a constant. 

Similarly, by considering the rest grid points in $[m]^d$ through the order of $\xi^{[3]}=(3,1,1,\cdots,1)$, $\cdots$, $\xi^{[m^d]}=(m,m,\cdots,m)$ such that the Hamming distance between any two adjacent grid points equals one, we can conclude that given the $\alpha^{(k)}$ values on the previous grid point $\xi^{[j]}$, the $\alpha^{(k)}$ values on $\xi^{[j+1]}$ can take at most constant $C$ many values, for $j\in[m^d-1]$.
Therefore, we can conclude that 
$$\log N\big(C^{\widetilde \gamma}_1(\mathbb{R}^D),\, \|\cdot\|_{L^{\infty}(\mathcal{X}_{G})}, \,\epsilon\big)\leq  C_1\,\log\frac{1}{\epsilon} + C_2\,  m^d \leq C_3\,\Big(\frac{1}{\epsilon}\Big)^{\frac{d}{\widetilde{\gamma}}} \quad \forall \epsilon \in(0,1).$$


\section{Proof of Extensions of Main Results}\label{sec:proof_extension}

\subsection{Proof of Corollary~\ref{Lepski}}
The proof follows the standard analysis for Lepski's estimator~\citep{lepskii1991problem}.  Consider $\mu^\ast\in \m S^\ast(d,D,\alpha^\ast,\beta^\ast,\ms O_M,L)$,  by the analysis in Appendix~\ref{Sec:more_proofs_upperbound}, we only need to consider $m\in \mb M\cap \wh{\mb M}$. Let $\widetilde \beta=\max\{\beta\in \m B_1:\beta\leq \beta^\ast\}$ and $(\wh{G}^{\dagger}_{[m]},\wh{Q}^{\dagger}_{[m]})=(\wh{G}_{[m]}^{[\wh\beta_{[m]}]},\wh{Q}_{[m]}^{[\wh\beta_{[m]}]})$, then it holds that  for any $m\in \mb M\cap \wh{\mb M}$,
\begin{equation*}
  \sum_{i\in I_1} \|X_i-\wh G^{\dagger}_{[m]}\circ \wh Q^{\dagger}_{[m]}(X_i)\|_2^2 \cdot \bold{1}(X_i\in S_m^{\dagger})=0,
\end{equation*}
and $\wh G^{\dagger}_{[m]}\in C^{\widetilde \beta}_L(\mb R^d;\mb R^D)$, $\wh Q^{\dagger}_{[m]}\in C^{\widetilde \beta}_L(\mb R^D;\mb R^d)$. Then by replacing $\beta$ with $\widetilde \beta$ in the proof of Lemma~\ref{lemma:pointwise_error} in Section~\ref{Sec:proof_lemma:pointwise_error} and the proof of Lemma~\ref{lemma:smooth_corr} in Section~\ref{sec:proof_smooth_corr}, we can obtain it holds with probability at least $1-n^{-1}$ that for any $m\in \mb M\cap \wh{\mb M}$,
\begin{equation*}
 \begin{aligned}
      \sup_{f\in C_1^\gamma(\mb R^D)} \big|\wh{\m J}^{\dagger}_{m,s}(f) - {\m J}_{m,s}(f)\big| &\leq   C\,\sqrt{\frac{\log n}{n}} + C\, \Big(\frac{\log n}{n}\Big)^{\frac{\gamma\widetilde \beta}{d}}+ C\,\Big(\frac{\log n}{n}\Big)^{\frac{\gamma+\widetilde \beta-1}{d}}\\
       &\leq   C_1\,\sqrt{\frac{\log n}{n}} + C_1\, \Big(\frac{\log n}{n}\Big)^{\frac{\gamma  \beta^\ast}{d}}+ C_1\,\Big(\frac{\log n}{n}\Big)^{\frac{\gamma+  \beta^\ast-1}{d}},
  \end{aligned}
  \end{equation*}
 where the last ineuqlity is due to $\beta^\ast-\widetilde\beta\leq \frac{c}{\log n}$. Then it suffice to show that it holds with probability larger than $1-\frac{1}{n}$ that for any $m\in  \mb M\cap \wh{\mb M}$, 
   \begin{align*}
      \sup_{f\in C_1^\gamma(\mb R^D)} \big|\wh {\m J}_{m,h}^{[\wh\alpha_{[m]}]}(f)+\wh {\m J}_{m,l}^{[\wh\alpha_{[m]}]}(f) - {\m J}_{m,h}(f)-{\m J}_{m,l}(f)\big| \leq   C\,\sqrt{\frac{\log n}{n}} + C\,\big(\frac{\log n}{n}\big)^{\frac{\alpha^\ast+\gamma}{2\alpha^\ast+d}}.
     \end{align*}
  Firstly by replacing $\alpha$ with any $\alpha_k\leq \alpha^*$ in proofs of Lemma~\ref{lemma:low_freq} and Lemma~\ref{lemma:high_freq} in Section~\ref{sec:proof_low_freq} and Section~\ref{sec:proof_high_freq}, we can get it holds with probability at least $1-n^{-1}$ that  for any  $\alpha_k\in \{\alpha\in \m B_2: \alpha\leq \alpha^\ast\}$ and $m \in  \mb M\cap \wh{\mb M}$,
\begin{equation*}
\begin{aligned}
    \sup_{f\in C_1^\gamma(\mb R^D)} \big|\wh {\m J}_{m,h}^{[ \alpha_k]}(f)+\wh {\m J}_{m,l}^{[ \alpha_k]}(f) - {\m J}_{m,h}(f)-{\m J}_{m,l}(f)\big|&\overset{(i)}{\leq} c_1\,\sqrt{\frac{\log n}{n}} \vee\big(\frac{\log n}{n}\big)^{\frac{ \alpha_k\wedge (\wt\beta-1)+\gamma}{2\alpha_k+d}}\\
    &\overset{(ii)}{\leq} c_2\, \sqrt{\frac{\log n}{n}} \vee\big(\frac{\log n}{n}\big)^{\frac{ \alpha_k+\gamma}{2\alpha_k+d}},
    \end{aligned}
\end{equation*}  
where $(i)$ uses $\wh Q^{\dagger}_{[m]}\in C^{\widetilde \beta}_L(\mb R^D;\mb R^d)$ and $(ii)$ uses $\alpha_k\leq \alpha^\ast\leq \beta^\ast-1\leq \wt\beta-1+c\,(\log n)^{-1}$.    Then let $ \alpha_{k^\ast}=\max\{\alpha\in \m B_2:\alpha\leq \alpha^\ast\}$, if $ \wh\alpha_{[m]}< \alpha_{k^\ast}$, then there exists $k< k^\ast$  so that 
     \begin{equation*}
     \begin{aligned}
         c_0  \sqrt{\frac{\log n}{n}}\vee \big(\frac{\log n}{n}\big)^{\frac{\alpha_k+\gamma}{2\alpha_k+d}} &< \underset{f\in C^{\gamma}_1(\mb R^D)}{\sup}\big|\wh{\m J}^{[\alpha_{k^\ast}]}_{m,h}(f)+\wh{\m J}^{[ \alpha_{k^\ast}]}_{m,l}(f)-\wh{\m J}^{[\alpha_k]}_{m,h}(f)-\wh{\m J}^{[\alpha_k]}_{m,l}(f)\big|\\
         &\leq \underset{f\in C^{\gamma}_1(\mb R^D)}{\sup}\big|\wh{\m J}^{[\alpha_{k^\ast}]}_{m,h}(f)+ \m J^{[ \alpha_{k^\ast}]}_{m,l}(f)- \m J_{m,h}(f)-  \m J_{m,l}(f)\big|\\
         &+\underset{f\in C^{\gamma}_1(\mb R^D)}{\sup}\big|\wh{\m J}^{[\alpha_{k}]}_{m,h}(f)+ \m J^{[ \alpha_{k}]}_{m,l}(f)- \m J_{m,h}(f)-  \m J_{m,l}(f)\big|.\\
              \end{aligned}
     \end{equation*}
Thus when $c_0\geq 2c_2$, we can obtain
\begin{equation*}
\begin{aligned}
   & \mb P_{{\mu^\ast}^{\otimes n}}(\wh\alpha_{[m]}< \alpha_{k^\ast})\\
    &\leq    \mb P_{{\mu^\ast}^{\otimes n}}\Big( \underset{f\in C^{\gamma}_1(\mb R^D)}{\sup}\big|\wh{\m J}^{[\alpha_{k^\ast}]}_{m,h}(f)+ \m J^{[ \alpha_{k^\ast}]}_{m,l}(f)- \m J_{m,h}(f)-  \m J_{m,l}(f)\big|\geq c_2\, \sqrt{\frac{\log n}{n}}\vee \big(\frac{\log n}{n}\big)^{\frac{\alpha_{k^\ast}+\gamma}{2\alpha_{k^\ast}+d}} \Big)\\
    &+ \mb P_{{\mu^\ast}^{\otimes n}}\Big(\exists\, k<k^\ast, \underset{f\in C^{\gamma}_1(\mb R^D)}{\sup}\big|\wh{\m J}^{[\alpha_{k}]}_{m,h}(f)+ \m J^{[ \alpha_{k}]}_{m,l}(f)- \m J_{m,h}(f)-  \m J_{m,l}(f)\big|\geq c_2\, \sqrt{\frac{\log n}{n}}\vee \big(\frac{\log n}{n}\big)^{\frac{\alpha_k+\gamma}{2\alpha_k+d}} \Big)\\
    &\leq \frac{1}{n}.
    \end{aligned}
\end{equation*}
Then if $ \alpha_{k^\ast}\leq \wh\alpha_{[m]}$,
\begin{equation*}
    \begin{aligned}
     \sup_{f\in C_1^\gamma(\mb R^D)} \big|\wh {\m J}_{m,h}^{[\wh\alpha_{[m]}]}(f)+\wh {\m J}_{m,l}^{[\wh\alpha_{[m]}]}(f) - \wh {\m J}_{m,h}^{[\alpha_{k^\ast}]}(f)-\wh {\m J}_{m,l}^{[\alpha_{k^\ast}]}(f)\big| \leq   c_0\,\sqrt{\frac{\log n}{n}} + c_0\,\big(\frac{\log n}{n}\big)^{\frac{\alpha_{k^\ast}+\gamma}{2\alpha_{k^\ast}+d}}.
    \end{aligned}
\end{equation*}
Thus 
\begin{equation*}
    \begin{aligned}
     &\sup_{f\in C_1^\gamma(\mb R^D)} \big|\wh {\m J}_{m,h}^{[\wh\alpha_{[m]}]}(f)+\wh {\m J}_{m,l}^{[\wh\alpha_{[m]}]}(f) - {\m J}_{m,h}(f)-{\m J}_{m,l}(f)\big|\\
     &\leq \sup_{f\in C_1^\gamma(\mb R^D)} \big|\wh {\m J}_{m,h}^{[\wh\alpha_{[m]}]}(f)+\wh {\m J}_{m,l}^{[\wh\alpha_{[m]}]}(f) - \wh {\m J}_{m,h}^{[\alpha_{k^\ast}]}(f)-\wh {\m J}_{m,l}^{[\alpha_{k^\ast}]}(f)\big|\\
    & + \sup_{f\in C_1^\gamma(\mb R^D)} \big|\wh {\m J}_{m,h}^{[\alpha_{k^\ast}]}(f)+\wh {\m J}_{m,l}^{[\alpha_{k^\ast}]}(f) - {\m J}_{m,h}(f)-{\m J}_{m,l}(f)\big|.
        \end{aligned}
\end{equation*}
The desired result then follows from $\alpha^\ast-\frac{c}{\log n}\leq \alpha_{k^\ast}\leq\alpha^\ast$.

\subsection{Proof of Corollary~\ref{twosample}}
Firstly for any $\mu_1=\mu_2\in \mathcal{S}^{\ast}$, it holds that 
\begin{equation*}
    \begin{aligned}
    &\mathbb{E}_{\mu_1^{\otimes n},\mu_2^{\otimes n}} (\Phi_{\gamma,c})\\
    &=\mathbb{P}_{\mu_1^{\otimes n},\mu_2^{\otimes n}} \Big( \underset{f\in C^{\gamma}_1(\mathbb{R}^D)}{\sup}  \big(\wh{\mathcal{J}} (f;\,X_{1:n})-\wh{\mathcal{J}} (f;\,Y_{1:n})\big)
 \geq c\,\delta_n^{\ast}\Big)\\
    &\leq  \mathbb{P}_{\mu_1^{\otimes n},\mu_2^{\otimes n}} \Big( \underset{f\in C^{\gamma}_1(\mathbb{R}^D)}{\sup}\big(\wh{\mathcal{J}} (f;\,X_{1:n})-\int f(X) \, \dd \mu_1\big)\\
 &+\underset{f\in C^{\gamma}_1(\mathbb{R}^D)}{\sup}\big(\wh{\mathcal{J}} (f;\,Y_{1:n})-\int f(Y) \, \dd \mu_2\big)\geq c\,\delta_n^{\ast}\Big)\\
    &\leq \mathbb{P}_{\mu_1^{\otimes n}} \Big(\underset{f\in C^{\gamma}_1(\mathbb{R}^D)}{\sup}\big(\wh{\mathcal{J}} (f;\,X_{1:n})-\int f(X) \, \dd \mu_1\big)\geq \frac{c}{2}\,\delta_n^{\ast}\Big)\\
   & + \mathbb{P}_{\mu_2^{\otimes n}} \Big(\underset{f\in C^{\gamma}_1(\mathbb{R}^D)}{\sup}\big(\wh{\mathcal{J}} (f;\,Y_{1:n})-\int f(Y) \, \dd \mu_2\big)\geq \frac{c}{2}\,\delta_n^{\ast}\Big).\\
    \end{aligned}
\end{equation*}
Then by Theorem~\ref{upperboundgenerative}, for any constant $r$, there exists a constant $c$ such that 
\begin{equation*}
\begin{aligned}
&\mathbb{P}_{\mu_1^{\otimes n}} \Big(\underset{f\in C^{\gamma}_1(\mathbb{R}^D)}{\sup}\big(\wh{\mathcal{J}} (f;\,X_{1:n})-\int f(X) \, \dd \mu_1\big)\geq \frac{c}{2}\,\delta_n^{\ast}\Big)\\
   & + \mathbb{P}_{\mu_2^{\otimes n}} \Big(\underset{f\in C^{\gamma}_1(\mathbb{R}^D)}{\sup}\big(\wh{\mathcal{J}} (f;\,Y_{1:n})-\int f(Y) \, \dd \mu_2\big)\geq \frac{c}{2}\,\delta_n^{\ast}\Big)\leq \frac{n^{-r}}{2}.
    \end{aligned}
\end{equation*}
Moreover, for any $\mu_1,\mu_2\in \mathcal{S}^{\ast}$ with $d_{\gamma}(\mu_1,\mu_2)\geq c_1\,\delta_n^{\ast}$, it holds that 
\begin{equation*}
    \begin{aligned}
   & \mathbb{E}_{\mu_1^{\otimes n},\mu_2^{\otimes n}}(1-\Phi_{\gamma,c})\\
    &=\mathbb{P}_{\mu_1^{\otimes n},\mu_2^{\otimes n}} \Big(\underset{f\in C^{\gamma}_1(\mathbb{R}^D)}{\sup} \big(\wh{\mathcal{J}} (f;\,X_{1:n})-\wh{\mathcal{J}} (f;\,Y_{1:n})\big)< c\,\delta_n^{\ast}\Big)\\
    &\leq  \mathbb{P}_{\mu_1^{\otimes n},\mu_2^{\otimes n}} \Big(\underset{f\in C^{\gamma}_1(\mathbb{R}^D)}{\sup}\big(\wh{\mathcal{J}} (f;\,X_{1:n})-\int f(X) \, \dd \mu_1\big)\\
   & +\underset{f\in C^{\gamma}_1(\mathbb{R}^D)}{\sup}\big(\wh{\mathcal{J}} (f;\,Y_{1:n})-\int f(Y) \, \dd \mu_2\big)\geq (c_1-c)\, \delta_n^{\ast}\Big)\\
    &\leq \mathbb{P}_{\mu_1^{\otimes n}} \Big(\underset{f\in C^{\gamma}_1(\mathbb{R}^D)}{\sup}\big(\wh{\mathcal{J}} (f;\,X_{1:n})-\int f(X) \, \dd \mu_1\big)\geq \frac{c_1-c}{2}\,\delta_n^{\ast}\Big)\\
    &+ \mathbb{P}_{\mu_2^{\otimes n}} \Big(\underset{f\in C^{\gamma}_1(\mathbb{R}^D)}{\sup}\big(\wh{\mathcal{J}} (f;\,Y_{1:n})-\int f(Y) \, \dd \mu_2\big)\geq \frac{c_1-c}{2}\,\delta_n^{\ast}\Big).\\
    \end{aligned}
\end{equation*}
Then by Theorem~\ref{upperboundgenerative}, for any constant $r$ and $c$, there exists a constant $c_1$ such that 
\begin{equation*}
\begin{aligned}
& \mathbb{P}_{\mu_1^{\otimes n}} \Big(\underset{f\in C^{\gamma}_1(\mathbb{R}^D)}{\sup}\big(\wh{\mathcal{J}} (f;\,X_{1:n})-\int f(X) \, \dd \mu_1\big)\geq \frac{c_1-c}{2}\,\delta_n^{\ast}\Big)\\
    &+ \mathbb{P}_{\mu_2^{\otimes n}} \Big(\underset{f\in C^{\gamma}_1(\mathbb{R}^D)}{\sup}\big(\wh{\mathcal{J}} (f;\,Y_{1:n})-\int f(Y) \, \dd \mu_2\big)\geq \frac{c_1-c}{2}\,\delta_n^{\ast}\Big)\leq \frac{n^{-r}}{2}.
    \end{aligned}
\end{equation*}
Proof is completed.

 \subsection{Proof of Theorem~\ref{th:unbounded}}\label{proof:unbounded}
We first describe the estimator. Similar as the analysis in Section~\ref{sec:optimal_proce}, our proposed minimax-optimal estimator $\wh \mu^{\circ}$ is constructed in three steps. Let $[n]= I_1\cap I_2$ be a random splitting of the data indices into two sets with $|I_1|= \lceil n/2 \rceil$ and $|I_2| = n-|I_1|$.

\noindent {\bf Step 1}: (Submanifold estimation) Recall $K_1$ to be the $1/L$-enlargement of $K_0$ given by $K_1=\{x\in B_{1/L}(y)\,:\, y\in K_0\}$, we use the first half of data to compute 
\begin{equation*}
     \begin{aligned}
       (\wh{G},\wh{Q})=\underset{G\in {\ms{G}}, Q\in {\ms{Q}}}{\arg\min}\bigg( \frac{1}{|I_1|}\sum_{i\in I_1} \Big[\|X_i-G(Q(X_i))\|_2^2\bold{1}_{  K_{1}}(X_i)\Big]\bigg).
     \end{aligned}
 \end{equation*}
where recall $\ms{G}=C_{L}^{\beta}(\mathbb{R}^d;\, \mathbb{R}^D)$ and $\ms {Q}= C_{L}^{\beta}(\mathbb{R}^D;\, \mathbb{R}^d)$. Let $\wh{p}_0=\frac{1}{|I_2|}\sum_{i\in I_2}  \bold{1}_{K_0}(X_i)$ denote the sample frequency of falling into $K_0$.

\noindent {\bf Step 2}: (Surrogate functional construction) 
Follow the analysis in Section~\ref{sec:optimal_proce}, we construct a finite sample surrogate to $\int_{K_0} f(X)\,\dd \mu^*$ as follows. Firstly if $\wh{p}_0<\sqrt{\frac{\log n}{n}}$, then define $\wh{\m J}^{\circ}(f)=0$; otherwise, let $J$ be the largest integer such that $2^J \leq (\frac{n}{\log n})^{\frac{1}{2\alpha+d}}$, $\Pi_J f$ denote the projection of any $f\in\m F$ onto the first scale $J$ wavelet coefficients and $\Pi_J^\perp f = f-\Pi_J f$. The expectation of low frequency components of $f$: $\int_{K_0}\Pi_J f(\wh{G}(\wh{Q}(X)))\,\dd \mu^*$ is estimated with the empirical mean 
\begin{equation*}
\wh{\m J}^{\circ}_{l}(f)=\frac{1}{|I_2|} \sum_{i\in I_2} [\Pi_J f(\wh{G}(\wh{Q}(X_i)))\bold{1}_{K_0}(X_i)].
\end{equation*}
For the high frequency part, we construct kernel density estimator (KDE) to estimate the density of $\wh{Q}_{\#}(\bold{1}_{K_0}\cdot\mu^{\ast})$. More specifically, choose a kernel function $\bar{k}(x)\in C^{1+ \alpha}_{c_1}(\mathbb{R})$ such that  (1) ${\rm supp}(\bar{k}(x))\subset [0,1]$; (2) $\int \bar{k}(x)\dd x=1$; (3) for any $j\in \mb N^{+}$ with $ j\leq \lceil  \alpha\rceil$, $\int x^j \bar{k}(x)\, \dd x=0$. Define the kernel density estimator
 \begin{equation*}
      \bar{\nu}_{\wh{Q}}(y)=\frac{1}{h^d|I_2|} \sum_{i\in I_2} \Big(\prod_{j=1}^d \bar{k}\Big(\frac{\wh{Q}_j(X_i)-y_j}{h}\Big)\Big) \cdot\bold{1}_{K_{1/4}}(X_i).
  \end{equation*}
with bandwidth $h=\big(\frac{\log n}{n}\big)^{\frac{1}{2 \alpha+d}}$. Define 
  \begin{equation*}
    \wh{U}=\underset{i\in I_2: \,X_i \in K_{1/4}}{\bigcup} \mb B_{c_0(\frac{1}{\wh{p}_0}\cdot\frac{\log n}{n})^{\frac{1}{d}}}(X_i)
\end{equation*}
for a large enough constant $c_0$, we use
$\{ y=\wh{Q}(x)\,|\,x\in \wh{U}; \wh{G}(\wh{Q}(x))\in K_0\}$ to estimate the support of $\wh{Q}_{\#}(\bold{1}_{K_0}\cdot\mu^{\ast})$. Therefore, we can define the functional 
\begin{equation*}
\wh{\m J}^{\circ}_{h}(f)=\int_{\{ y=\wh{Q}(x)\,|\,x\in \wh{U}; \wh{G}(\wh{Q}(x))\in K_0\}}\Pi_J^{\perp}f(\wh{G}(y))\bar{\nu}_{\wh{Q}}(y)\,\dd y
\end{equation*}
to estimate $\int_{K_0}\Pi_J^{\perp} f(\wh{G}(\wh{Q}(X)))\,\dd\mu^*$. Finally we can construct a (sample version of) higher-order smoothness correction $\wh{\m J}^{\circ}_{s}(f)$ to $ \int_{K_0}  f\big(\wh G  \circ \wh Q (X)\big)\,\dd\mu^*$ for approximating $ \int_{K_0}f(X)\,\dd\mu^*$, defined as
\begin{equation*}
  \wh{\m J}^{\circ}_s(f)= -\frac{1}{|I_2|}\sum_{i\in I_2} \sum_{j\in \mathbb{N}_0^D\atop 1\leq |j|\leq \lfloor\gamma\rfloor}\left[\frac{1}{j!}f^{(j)}(X_i)(\wh{G}(\wh{Q}(X_i))-X_i)^{j}\bold{1}_{K_0}(X_i)\right]
\end{equation*}
Putting all pieces together,  we can construct. a estimator of $\int_{K_0}f(X)\,\dd\mu^\ast$ as
\begin{equation*}
 \wh{\mathcal{J}}^{\circ}(f)= \wh{\mathcal{J}}^{\circ}_{s}(f)+ \wh{\mathcal{J}}^{\circ}_l(f)+ \wh{\mathcal{J}}^{\circ}_h(f),
   \end{equation*} for arbitrary $f\in C^{\gamma}_1(\mathbb{R}^D)$.
   
 \noindent {\bf Step 3}: (Generative model estimation) Then we can  define the estimator 
 \begin{equation*}
 \begin{aligned}
 \wh{\mu}^{\circ}=\underset{\mu=G_{\#}\nu\atop G\in C^{\beta}_L(\mb R^d;\mb R^D), \nu\in \m{P}(\mb R^d)}{\inf} \underset{f\in C^{\gamma}_1(\mathbb{R}^D)}{\sup}
     \Big(\frac{1}{\wh{p}_0}\cdot\wh{\mathcal{J}}^{\circ}(f)-\int f(X)\,\dd\mu\Big).\\
      \end{aligned}
 \end{equation*}
We now show that $\wh{\mu}^{\circ}$ achieves the near-optimal rate. Consider a fixed $\mu^{\ast}\in \overline{\mathcal{S}}^{\ast}$, it suffices to show that 
\begin{align*} 
    \sup_{f\in C_1^\gamma(\mb R^D)}\Big|   \int_{K_0}f(X)\,\dd\mu^* - \wh{ \m J}^{\circ}(f)\Big| \leq C \, \Big(\frac{n}{\log n}\Big)^{-\frac{1}{2}}\vee \Big(\frac{n}{\log n}\Big)^{-\frac{ \alpha+\gamma}{2\alpha+d}}\vee \Big(\frac{n}{\log n}\Big)^{-\frac{\gamma\beta}{d}},
\end{align*}
where the regularized surrogate $\wh{ \m J}^{\circ}(f)=\wh{\m J}^{\circ}_l(f)+\wh{\m J}^{\circ}_h(f)+\wh{\m J}^{\circ}_s(f)$. We may similarly decompose the target functional into three terms as $\mb \int_{K_0}f(X) \,\dd\mu^* = \m J^{\circ}_l(f) + \m J^{\circ}_h(f)+ \m J^{\circ}_s(f)$, where
\begin{align*}
    \m J_l^{\circ}(f)  &=   \int_{K_0}\Pi_J f(\wh{G}(\wh{Q}(X)))\,\dd\mu^*,\\
     \m J_h^{\circ}(f)  &= \int_{K_0}\Pi_J^{\perp} f(\wh{G}(\wh{Q}(X)))\,\dd\mu^*,\\
      \m J_s^{\circ}(f)  &= \int_{K_0}f(X)\,\dd\mu^*-\int_{K_0}f\big(\wh G(\wh Q(X))\big)\,\dd\mu^*.
\end{align*}
Note that similar as the analysis in Appendix~\ref{Sec:more_proofs_upperbound}, we only need to consider the case that $\mathbb{P}_{\mu^{\ast}}(X\in K_0)\geq \frac{1}{2}\sqrt{\frac{\log n}{n}}$ and $\wh{p}_0\geq \sqrt{\frac{\log n}{n}}$.  Recall $\mu^\ast\in \overline{\m S}^\ast$, then there exists a compact set $\widetilde{K}\supseteq K_1$  such that $\mu^{\ast}|_{\widetilde{K}}$ can be expressed as a generative model $\mu^{\ast}|_{\widetilde{K}}=G^{\ast}_{\#}\nu^{\ast}$, where $G^{\ast}\in C^{\beta}_{L}(\mathbb{R}^d;\mb R^D)$. Denote  $\Omega^{\ast}={\rm supp}(\nu^{\ast})$, then there exists $Q^{\ast}:\mathbb{R}^D\to \mathbb{R}^d$ such that for any $z\in \Omega^{\ast}$, $Q^{\ast}(G^{\ast}(z))=z$ and  $Q^{\ast}\in C^{\beta}_{L}(\mathbb{R}^D;\mb R^d)$. 

The following three lemmas show that $\wh{\m J}^{\circ}_{s}$, $\wh{\m J}^{\circ}_{l}$ and $\wh{\m J}^{\circ}_{h}$ are good estimators for $\m J^{\circ}_{s}$, $\m J^{\circ}_{l}$ and $\m J^{\circ}_{h}$ respectively.
\begin{lemma}[Smoothness correction]\label{lemma:constrained_smooth_corr}
  When $\mu^\ast\in \overline{\m S}^\ast$, it holds with probability at least $1-n^{-c}$ that the functional $\wh{\m J}^{\circ}_{s}:\, C^\gamma(\mb R^D) \to \mb R$ satisfies
  \begin{align*}
      \sup_{f\in C_1^\gamma(\mb R^D)} \big|\wh{\m J}_{s}^{\circ}(f) - {\m J}_{s}^{\circ}(f)\big| \leq   C\,\sqrt{\frac{\log n}{n}} + C\, \Big(\frac{\log n}{n}\Big)^{\frac{\gamma\beta}{d}}+ C\,\Big(\frac{\log n}{n}\Big)^{\frac{\gamma+\beta-1}{d}}.
  \end{align*}
\end{lemma}
  
\begin{lemma}[Low frequency components]\label{lemma:constrained_low_freq}
 When $\mu^\ast\in \overline{\m S}^\ast$, it holds with probability at least $1-n^{-c}$ that the functional $\wh{\m J}_{l}^{\circ}:\, C^\gamma(\mb R^D) \to \mb R$ satisfies
  \begin{align*}
      \sup_{f\in C_1^\gamma(\mb R^D)} \big|\wh{\m J}_{l}^{\circ}(f) - {\m J}_{l}^{\circ}(f)\big| \leq C\,\sqrt{\frac{\log n}{n}} +  C\, \Big(\frac{\log n}{n}\Big)^{\frac{\alpha+\gamma}{2\alpha+d}}.
  \end{align*}
\end{lemma}
\begin{lemma}[High frequency components]\label{lemma:constrained_high_freq}
  When $\mu^\ast\in \overline{\m S}^\ast$, it holds with probability at least $1-n^{-c}$ that the functional $\wh{\m J}_{h}^{\circ}:\, C^\gamma(\mb R^D) \to \mb R$ satisfies
  \begin{align*}
      \sup_{f\in C_1^\gamma(\mb R^D)} \big|\wh{\m J}_{h}^{\circ}(f) - {\m J}_{h}^{\circ}(f)\big| \leq  C\, \Big(\frac{\log n}{n}\Big)^{\frac{\alpha+\gamma}{2\alpha+d}}.
  \end{align*}
\end{lemma}
By combining Lemmas~\ref{lemma:constrained_smooth_corr},~\ref{lemma:constrained_low_freq} and~\ref{lemma:constrained_high_freq}, similar as the proof of minimax upper bound in Theorem~\ref{maintheorem} in Section~\ref{sec:proofupperbound}, we can obtain the desired result.

\subsubsection{Proof of Lemma~\ref{lemma:constrained_smooth_corr}}
 Recall $K_r=\{x\in B_{|r|/L}(y)\,:\, y\in K_0\}$ for $r>0$ and $K_r=\{x\in\mb R^D\,:\,  B_{|r|/L}(x)\subset K_0\}$ for $r<0$. Define $\mathcal{M}^{\ast}=G^{\ast}(\Omega^{\ast})$ and $\mathcal{Z}_{a}=Q^{\ast}(\mathcal{M}^{\ast} \cap K_{a})$. Then $\mathcal{Z}_1\subset [-L,L]^d$ and for any $z_1,z_2\in \mathcal{Z}_1$,
 \begin{equation}\label{eqnunboundedlip}
    \|G^{\ast}(z_1)-G^{\ast}(z_2)\|_2\leq L \sqrt{D}\|z_1-z_2\|_2.
 \end{equation}
Let $\wt{n}=\mb P_{\mu^*}(X\in \wt K)\cdot n\geq \mb P_{\mu^*}(X\in K_0)\cdot n\geq \frac{1}{2}\sqrt{n\log n}$ and $\epsilon_n=c_0\,(\frac{\log \wt n}{\wt n})^\frac{1}{d}$, combined with~\eqref{eqnunboundedlip}, when $n$ is large enough, for any $z^{\ast}\in \mathcal{Z}_{1/2}$, it holds that 
 \begin{equation*}
     \bold{1}_{\mathcal{M}^{\ast}  \cap G^{\ast}(\mb B_{\epsilon_n}(z^{\ast}))}(X)\leq \bold{1}_ {\mathcal{M}^{\ast} \cap K_{1}}(X).
     \end{equation*}
 Therefore, for any $z^{\ast}\in \mathcal{Z}_{1/2}$,
 \begin{equation*}
      \frac{1}{\lceil\frac{n}{2}\rceil}\sum_{i\in I_1} \left[\|G^{\ast}(Q^{\ast}(X_i))-\wh{G}(\wh{Q}(G^{\ast}(Q^{\ast}(X_i))))\|_2^2\cdot \bold{1}_{\mathcal{M}^{\ast}  \cap G^{\ast}(\mb B_{\epsilon_n}(z^{\ast}))}(X_i)\right]=0.
  \end{equation*}
Let $\wh{l}=\wh{Q}(G^{\ast}(z))$, similar as the proof of Lemma~\ref{lemma:pointwise_error} in Section~\ref{Sec:proof_lemma:pointwise_error}, we can obtain that it holds with probability at least $1-\frac{1}{n^2}$ that
 \begin{equation}\label{eqnunbounded1}
 \begin{aligned}
   &\underset{z\in \mathcal{Z}_{1/2}}{\sup}\|G^{\ast}(z)-\wh{G}(\wh{l}(z)))\|_2\leq C\, \big(\frac{\log \wt n}{\wt n}\big)^{\frac{\beta}{d}};\\
  &\underset{x\in\mathcal{M}^{\ast} \cap K_{1/2}}{\sup}\|x-\wh{G}(\wh{Q}(x))\|_2\leq C\, \big(\frac{\log \wt n}{\wt n}\big)^{\frac{\beta}{d}},
      \end{aligned}
  \end{equation}
 and thus for any $\eta\in (0,\frac{d}{\beta}]$
 \begin{equation*}
     \int_{K_0} \|X-\wh{G}(\wh{Q}(X))\|^{\eta}\,\dd\mu^\ast\leq C\,\mb P_{\mu^\ast}(X\in K_0)\cdot\big(\frac{\log \wt n}{\wt n}\big)^{\frac{\beta}{d}}\leq \big(\frac{\log n}{n}\big)^{\frac{\beta}{d}},
 \end{equation*}
 and for any  $\eta>\frac{d}{\beta}$,
  \begin{equation*}
     \int_{K_0} \|X-\wh{G}(\wh{Q}(X))\|^{\eta}\,\dd\mu^\ast\leq C\,C_1^{\eta-\frac{d}{\beta}}\,\big(\frac{\log n}{n}\big)^{\frac{\beta}{d}}.
 \end{equation*}
Then similar as the proof of Lemma~\ref{lemma:smooth_corr}, it holds with probability at least $1-\frac{2}{n^2}$ that for any $f\in C^{\gamma}_1(\mathbb{R}^D)$,
 \begin{equation*}
     \begin{aligned}
     & \big|{\m J}^{\circ}_s(f)-\wh{\m J}^{\circ}_s(f)\big|\\
      &=\Bigg|\int_{K_0} f(X)\,\dd\mu^\ast- \int_{K_0} f(\wh{G}(\wh{Q}(X)) \, \dd \mu^{\ast}\\
      & + \frac{1}{|I_2|}\sum_{i\in |I_2|} \sum_{j\in \mathbb{N}_0^D\atop 1\leq |j|\leq \lfloor\gamma\rfloor}\left[\frac{1}{j!}f^{(j)}(X_i)(\wh{G}(\wh{Q}(X_i))-X_i)^{j}\bold{1}_{K_0}(X_i)\right]\Bigg|\\
      &\leq  C_1\,\sqrt{\frac{\log n}{n}} + C_1\, \Big(\frac{\log n}{n}\Big)^{\frac{\gamma\beta}{d}}+ C_1\,\Big(\frac{\log n}{n}\Big)^{\frac{\gamma+\beta-1}{d}}.
     \end{aligned}
 \end{equation*}
 
  \subsubsection{Proof of Lemma~\ref{lemma:constrained_low_freq}}
 Same as the proof of Lemma~\ref{lemma:low_freq} in Section~\ref{sec:proof_low_freq}, for any $f\in C^{\gamma}_1(\mathbb{R}^D)$, it can be written as 
   \begin{equation*}
  f(x)=\underbrace{\sum_{k\in \mathbb{Z}^D} b_k \phi_k(x)+\sum_{l=1}^{2^D-1} \sum_{j=1}^{J}\sum_{k\in \mathbb{Z}^D} f_{ljk} \psi_{ljk}(x)}_{\Pi_J f}+\underbrace{\sum_{l=1}^{2^D-1} \sum_{j=J+1}^{+\infty}\sum_{k\in \mathbb{Z}^D} f_{ljk} \psi_{ljk}(x)}_{\Pi_j^{\perp} f},
\end{equation*}
where $\{\phi_k, \psi_{ljk}\,|\, l\in [2^D-1], j\in \mathbb{N}, k\in \mathbb{Z}^D\}$ is the orthonormal wavelet basis for Besov space on $\mathbb{R}^D$ and  recall $J$ is the largest integer such that $2^J \leq (\frac{n}{\log n})^{\frac{1}{2\alpha+d}}$. Then define 
  \begin{equation*}
\begin{aligned}
&\widetilde{\mathbb{S}}=\{k \in\mathbb{Z}^D\,| \,{\rm supp}(\phi_{k})\cap  \wh{G}([-L,L]^d)\neq \emptyset\};\\
&\widetilde{\mathbb{S}}_{lj}=\{k \in \mathbb{Z}^D,| \,{\rm supp}(\psi_{ljk})\cap \wh{G}([-L,L]^d)\neq \emptyset\}.\\
 \end{aligned}
\end{equation*}
 By changing $\rho_m(x)$ in the proof of Lemma~\ref{lemma:low_freq} in Section~\ref{sec:proof_low_freq} to $\bold{1}_{K_0}(x)$, we can get it holds with probability at least $1-\frac{1}{n^2}$ that  for any $f\in C_{1}^{\gamma}(\mathbb{R}^D)$, 
  \begin{equation*}
    \begin{aligned}
   & \left|\frac{1}{|I_2|} \sum_{i\in I_2} [\Pi_J f(\wh{G}(\wh{Q}(X_i)))\bold{1}_{K_0}(X_i)]-\mathbb{E}_{\mu^{\ast}} [\Pi_J f(\wh{G}(\wh{Q}(X)))\bold{1}_{K_0}(X)]\right|\\
   &\leq C\, \sum_{k\in \widetilde{\mathbb{S}}}  \left|\mathbb{E}_{\mu^{\ast}}[\phi_k(\wh{G}(\wh{Q}(X)))\bold{1}_{K_0}(X)]-\frac{1}{|I_2|} \sum_{i\in I_2}[\phi_k(\wh{G}(\wh{Q}(X_i)))\bold{1}_{K_0}(X_i)]\right|\\
   &+\sum_{l=1}^{2^D-1}\sum_{j=1}^J\sum_{k\in\widetilde{\mathbb{S}}_{lj}} 2^{-\frac{Dj}{2}-j\gamma}\left|\mathbb{E}_{\mu^{\ast}}[\psi_{ljk}(\wh{G}(\wh{Q}(X)))\bold{1}_{K_0}(X)]-\frac{1}{|I_2|} \sum_{i\in I_2}[\psi_{ljk}'(\wh{G}(\wh{Q}(X_i)))\bold{1}_{K_0}(X_i)]\right|\\
   &\leq C_1\,\sqrt{\frac{\log n}{n}}+\big(\frac{\log n}{n}\big)^{\frac{ \alpha+\gamma}{2 \alpha+d}}.
          \end{aligned}
\end{equation*}
 \subsubsection{Proof of Lemma~\ref{lemma:constrained_high_freq}}
 Recall $\mathcal{M}^{\ast}=G^{\ast}(\Omega^{\ast})$; $K_r=\{x\in B_{|r|/L}(y)\,:\, y\in K_0\}$ for $r>0$ and $K_r=\{x\in\mb R^D\,:\,  B_{|r|/L}(x)\subset K_0\}$ for $r<0$; and $\mathcal{Z}_{a}=Q^{\ast}(\mathcal{M}^{\ast} \cap K_{a})$.  Consider $\wh{l}=\wh{Q}\circ G^\ast$, then similar as the proof of Lemma~\ref{lemma:density_regularity} in Section~\ref{sec:proof_lemma:density_regularity}, we can obtain that
 there exist positive constants $(n_0,c_2,c_3)$ such that  when $n\geq n_0$, for any $z\in \mathcal{Z}_{1/2}$, it holds that
  \begin{equation*} 
   c_2\leq |{\rm det}(\bold{J}_{\wh{l}}(z))|\leq c_3.
  \end{equation*}
 Moreover, by the lipschitzness of $G^\ast$ and $\bigcup_{z\in \Omega^\ast, G^\ast(z)\in K_1} B_{1/L}(z)\subset \Omega^\ast$, there exists a constant $\epsilon>0$ such that 
\begin{equation*}
    \{z\in \mb B_{\epsilon}(z^{\ast})\,:\, z^{\ast}\in \mathcal{Z}_{{1}/{3}}\}\subset \mathcal{Z}_{1/2}.
\end{equation*}
Thus if there exist $z_1, z_2\in \mathcal{Z}_{\frac{1}{3}}$ such that $\|z_1-z_2\|\leq \epsilon$ and $\wh{l}(z_1)=\wh{l}(z_2)$, then by $\beta> 1$ and for any $t\in [0,1]$, $|{\rm det}(\bold{J}_{\wh{l}}(tz_1+(1-t)z_2)|\geq c_2>0$, we can obtain that there exists a positive constant $c_4$ such that $\|z_1-z_2\|\geq \epsilon\wedge c_4$. On the opposite side,  by equation~\eqref{eqnunbounded1}, it holds that 
\begin{equation*}
    \|z_1-z_2\|_2= \|Q^{\ast}(G^{\ast}(z_1))-Q^{\ast}(G^{\ast}(z_2))\|_2\leq \sqrt{d}L \|G^{\ast}(z_1)-G^{\ast}(z_2)\|_2\leq C \big(\frac{\log \wt n}{\wt n}\big)^{\frac{\beta}{d}}.
\end{equation*}
 Therefore, when $n$ is large enough, it holds with probability at least $1-\frac{1}{n^2}$ that $\wh{l}$ is one to one on $\mathcal{Z}_{\frac{1}{3}}$. So we can define $\wh{l}^{-1}: \wh{l}(\mathcal{Z}_{\frac{1}{3}})\to \mathcal{Z}_{\frac{1}{3}}$ as the inverse of $\wh{l}|_{\mathcal{Z}_{\frac{1}{3}}}$ and by inverse function theorem in holder space (see for example, Appendix A of ~\citep{Eldering2013}), it holds that  $\wh{l}^{-1}\in C^{\beta}_{C_1}(\wh{l}({\mathcal{Z}_{\frac{1}{3}}});\mb R^d)$. \\
  
\noindent Furthermore,  recall $\Pi_J^{\perp}(f)= \sum_{l=1}^{2^D-1} \sum_{j=J+1}^{+\infty}\sum_{k\in \mathbb{Z}^D} f_{ljk}\psi_{ljk}(x)$ where $2^J=c\,\big(\frac{n}{\log n}\big)^{\frac{1}{2\alpha+d}}$, $|f_{ljk}|\leq C 2^{-\frac{Dj}{2}-j\gamma}$ and $\|\sum_{k\in \mb Z^D}\psi_{ljk}\|_\infty \leq C2^{\frac{Dj}{2}}$. Then by the fact that there exists a constant $C_1$ such that for any $j\in [2^D-1]$ and $j\in \mb Z$, $\sum_{k \in\mb Z^D}|\psi_{ljk}(z)|\leq C_1\,2^{\frac{Dj}{2}}$ (recall that the support of $\psi_{ljk}$ is contained in $\mb B_{2^{-j}C}(2^{1-j}k)$), we can get  for any $x\in \mb R^D$,
\begin{equation*}
    |\Pi_J^{\perp}f(x)|\leq C\, \big(\frac{\log n}{n}\big)^{\frac{\gamma}{2 \alpha+d}},
\end{equation*}
and
\begin{equation*}
    \begin{aligned}
        &\left|\int \Pi_J^{\perp}f(\wh{G}(\wh{Q}(X))) \bold{1}_{K_0}(X) \, \dd \mu^{\ast}-\int \Pi_J^{\perp}f(\wh{G}(\wh{Q}(X))) \bold{1}_{\{X\in K_{1/4}\,:\, \wh{G}(\wh{Q}(X))\in K_0\}}(X)\, \dd \mu^{\ast}\right|\\
        &\leq C \big(\frac{\log n}{n}\big)^{\frac{\gamma}{2 \alpha+d}} \int \left|\bold{1}_{K_0}(X)-\bold{1}_{\{X\in K_{1/4}\,:\, \wh{G}(\wh{Q}(X))\in K_0\}}(X)\right|\, \dd \mu^{\ast}.
    \end{aligned}
\end{equation*}
By equation~\eqref{eqnunbounded1}, there exists a constant $c$ such that for $\delta_n=c(\frac{\log \wt n}{\wt n})^{\frac{\beta}{d}}$ with $\wt{n}=\mb P_{\mu^*}(X\in \wt K)\cdot n\geq \mb P_{\mu^*}(X\in K_0)\cdot n\geq  \frac{1}{2}\sqrt{n\log n}$, it holds that
\begin{equation*}
    \begin{aligned}
        &\left|\int \Pi_J^{\perp}f(\wh{G}(\wh{Q}(X))) \bold{1}_{K_0}(X) \, \dd \mu^{\ast}-\int \Pi_J^{\perp}f(\wh{G}(\wh{Q}(X))) \bold{1}_{\{X\in K_{1/4}\,:\, \wh{G}(\wh{Q}(X))\in K_0\}}(X)\, \dd \mu^{\ast}\right|\\
        &\overset{(i)}{\leq} C\,\big(\frac{\log n}{n}\big)^{\frac{\gamma}{2 \alpha+d}}\cdot\mb P_{\mu^*}(X\in \wt{K}) \cdot\int  \Big(\bold{1}_{K_{\delta_n}}(G^{\ast}(z))-\bold{1}_{K_{-\delta_n}}(G^{\ast}(z))\Big)\nu^{\ast}(z) \,\dd z\\
        &\overset{(ii)}{\leq} C_1 \,\big(\frac{\log n}{n}\big)^{\frac{\gamma+\beta}{2 \alpha+d}}\vee \sqrt{\frac{\log n}{n}}\\
        &\leq C_1  \,\big(\frac{\log n}{n}\big)^{\frac{\gamma+ \alpha}{2 \alpha+d}}\vee \sqrt{\frac{\log n}{n}},\\
        \end{aligned}
\end{equation*}
where $(i)$ is due to equation~\eqref{eqnunbounded1}, $(ii)$ is due to the fact that $\big|\mathbb{P}_{\nu^\ast}\big(\{z\in \Omega\,:\, G^\ast(z)\in K_0\}\big)- \mathbb{P}_{\nu^\ast}\big(\{z\in \Omega\,:\, G^\ast(z)\in K_r\}\big)\big|\leq L\,|r|$ holds for any $|r|\leq 1$ and the last step is due to $ \alpha\leq \beta-1$. It remains to bound the difference between $\int \Pi_J^{\perp}f(\wh{G}(\wh{Q}(X))) \bold{1}_{\{X\in K_{1/4}\,:\, \wh{G}(\wh{Q}(X))\in K_0\}}(X)\, \dd \mu^{\ast}$ and $$\wh{\m J}^{\circ}_{h}(f)= \int_{\{ y=\wh{Q}(x)\,:\,x\in \wh{U}; \wh{G}(\wh{Q}(x))\in K_0\}}\Pi_J^{\perp}f(\wh{G}(y))\bar{\nu}_{\wh{Q}}(y)\,\dd y$$ for arbitrary $f\in C^{\gamma}_1(\mathbb{R}^d)$, where recall
\begin{equation*}
    \wh{U}=\underset{i\in I_2: X_i \in K_{1/4}}{\bigcup} \mb B_{c_0(\frac{1}{\wh{p}_0}\cdot\frac{\log n}{n})^{\frac{1}{d}}}(X_i).
\end{equation*}
Next lemma shows that $\mathcal{M}^{\ast} \cap K_{1/4}\subset \wh{U}$ with high probability.
\begin{lemma}\label{lemmaunbounded1}
There exists a constant $c_1$ such that when $c_0\geq c_1$, it holds with probability at least $1-\frac{1}{n^2}$ that $\mathcal{M}^{\ast} \cap K_{1/4}\subset \wh{U}=\underset{i\in I_2: X_i \in K_{1/4}}{\bigcup} \mb B_{c_0(\frac{1}{\wh{p}_0}\cdot\frac{\log n}{n})^{\frac{1}{d}}}(X_i)$.
 \end{lemma}
\noindent Let 
\begin{equation}\label{def:omega}
    \wh{\Omega}=\{ y=\wh{Q}(x)\,:\,x\in \mathcal{M}^{\ast} \cap K_{1/4}; \wh{G}(\wh{Q}(x))\in K_0\}.
    \end{equation}
    Since 
 \begin{equation*}
 \begin{aligned}
 &\int \Pi_J^{\perp}f(\wh{G}(\wh{Q}(X))) \bold{1}_{\{X\in K_{1/4}, \wh{G}(\wh{Q}(X)))\in K_0\}}\, \dd \mu^{\ast}\\
 &=P_{\mu^{\ast}}(X\in \widetilde{K})\int_{\wh{\Omega}}\Pi_J^{\perp} f(\wh{G}(y)) {\nu^{\ast}(\wh{l}^{-1}(y))}{|{\rm det}(\bold{J}_{\wh{l}^{-1}}(y))|}\bold{1}_{\{y\,:\,\wh{l}^{-1}(y)\in Q^{\ast}(\mathcal{M}^{\ast} \cap K_{1/4})\}}(y) \,\dd y\\
 &=P_{\mu^{\ast}}(X\in \widetilde{K})\int_{\wh{\Omega}}\Pi_J^{\perp}f(\wh{G}(y)) {\nu^{\ast}(\wh{l}^{-1}(y))}{|{\rm det}(\bold{J}_{\wh{l}^{-1}}(y))|} \,\dd y.
  \end{aligned}
  \end{equation*}
 where recall that $\wh{l}^{-1}$ is the inverse of $\wh{l}|_{\mathcal{Z}_{\frac{1}{3}}}$ with $\wh{l}(z)=\wh{Q}(G^{\ast}(z))$. While the definition of $\wh{\Omega}$ in equation~\eqref{def:omega} is intractable due to the unknown $\mathcal{M}^{\ast} $. The next lemma shows that with high probability, $\wh{\Omega}=\{ y=\wh{Q}(x)\,:\,x\in \wh{U}; \wh{G}(\wh{Q}(x))\in K_0\}$.  
\begin{lemma}\label{lemmaunbounded2}
 There exists a positive constant $c$ such that it holds with probability at least $1-\frac{c}{n^2}$ that 
 \begin{equation*}
     \begin{aligned}
        \{ y=\wh{Q}(x)\,:\,x\in \mathcal{M}^{\ast} \cap K_{1/4}; \wh{G}(\wh{Q}(x))\in K_0\}= \{ y=\wh{Q}(x)\,:\,x\in \wh{U}; \wh{G}(\wh{Q}(x))\in K_0\}.
     \end{aligned}
 \end{equation*}
\end{lemma}
\noindent  Recall 
 \begin{equation*}
      \bar{\nu}_{\wh{Q}}(y)=\frac{1}{h^d|I_2|} \sum_{i\in I_2}  \Big(\prod_{j=1}^d \bar{k} \Big(\frac{\wh{Q}_j(X_i)-y_j}{h} \Big)  \Big)\cdot \bold{1}_{X_i\in K_{1/4}}.
  \end{equation*}
  where $h=\big(\frac{\log n}{n}\big)^{\frac{1}{2 \alpha+d}}$. The next lemma shows that with high probability, for any $y\in \wh{\Omega}$,  $\bar{\nu}_{\wh{Q}}(y)$ is close to $\mathbb{P}_{\mu^{\ast}}(X\in \widetilde{K})\cdot{\nu^{\ast}(\wh{l}^{-1}(y))}\cdot{|{\rm det}(\bold{J}_{\wh{l}^{-1}}(y))|}$. 
\begin{lemma}\label{lemmaunbounded3}
  There exists a positive constant $c$ such that it holds with probability at least $1-\frac{c}{n^2}$ that for any $y\in \wh{\Omega}$, 
 \begin{equation*}
     \begin{aligned}
      \left|\bar{\nu}_{\wh{Q}}(y)-\mathbb{P}_{\mu^{\ast}}(X\in \widetilde{K})\cdot{\nu^{\ast}(\wh{l}^{-1}(y))}\cdot{|{\rm det}(\bold{J}_{\wh{l}^{-1}}(y))|}\right|\leq C\, \big(\frac{\log n}{n}\big)^{\frac{ \alpha}{2 \alpha+d}}.
     \end{aligned}
 \end{equation*}
\end{lemma}
 \noindent Then by Lemma~\ref{lemmaunbounded3}, it holds with probability at least $1-\frac{c}{n^2}$ that 
 \begin{equation*}
     \begin{aligned}
         &\underset{f\in C^{\gamma}_1(\mathbb{R}^D)}{\sup}\left|\int_{\wh{\Omega}}\Pi_J^{\perp}f(\wh{G}(y))\bar{\nu}_{\wh{Q}}(y)\,\dd y-\mathbb{P}_{\mu^{\ast}}(X\in \widetilde{K})\int_{\wh{\Omega}}\Pi_J^{\perp}f(\wh{G}(y)){\nu^{\ast}(\wh{l}^{-1}(y))}{|{\rm det}(\bold{J}_{\wh{l}^{-1}}(y))|}\,\dd y\right|\\
         &\leq C\, \big(\frac{\log n}{n}\big)^{\frac{ \alpha}{2 \alpha+d}}\underset{f\in C^{\gamma}_1(\mathbb{R}^D)}{\sup}\int_{\wh{\Omega}}|\Pi_J^{\perp}f(\wh{G}(y))|\,\dd y\\
         &\leq C_1\,\big(\frac{\log n}{n}\big)^{\frac{ \alpha+\gamma}{2 \alpha+d}}.
     \end{aligned}
 \end{equation*}
 We can then obtain the desired result by putting all pieces together.

\subsubsection{Proof of Lemma~\ref{lemmaunbounded1}}
 Let $\wt{\epsilon}_n=a\,(\frac{1}{\wh{p}_0}\cdot\frac{\log n}{n})^{\frac{1}{d}}$. By Bernstein's inequality, we have  it holds with probability at least $1-n^{-3}$ that 
 \begin{equation*}
    \big| \wh{p}_0-\mb{P}_{\mu^*}(K_0)\big|\leq C\,\Big(\frac{\log n}{n}+\sqrt{\frac{\log n}{n}}\sqrt{\mb{P}_{\mu^*}(K_0)}\Big)\leq C_1\,\sqrt{\frac{\log n}{n}}\sqrt{\mb{P}_{\mu^*}(K_0)},
 \end{equation*}
 where the last inequality is due to $\mb{P}_{\mu^*}(K_0)\geq \frac{1}{2}\sqrt{\frac{\log n}{n}}$. So we have 
 \begin{equation}\label{diffp_0}
     \Big|\frac{\wh{p_0}}{\mb{P}_{\mu^*}(K_0)}-1\Big|\leq C_1\, \frac{\sqrt{\frac{\log n}{n}}}{\sqrt{{\mb{P}_{\mu^*}(K_0)}}}\leq C_2\, \big(\frac{\log n}{n}\big)^{\frac{1}{4}}.
 \end{equation}
 So when $n$ is large enough, we have it holds with probability larger than $1-n^{-3}$ that $\wt{\epsilon}_n=a\,(\frac{1}{\wh{p}_0}\cdot\frac{\log n}{n})^{\frac{1}{d}}\geq \frac{1}{2}a\,(\frac{1}{\mb{P}_{\mu^*}(K_0)}\cdot\frac{\log n}{n})^{\frac{1}{d}}=\epsilon_n$.
 Let $\bar{N}_{\epsilon_n}\subset \mathcal{M}^{\ast} \cap K_{1/4}$ be the minimal $\epsilon_n$-covering  set of $\mathcal{M}^{\ast} \cap K_{1/4}$, then $|\bar{N}_{\epsilon_n}|\leq C\, n$. Moreover, for any $\widetilde{x}\in \bar{N}_{\epsilon_n}$,  since there exist positive constants $c,c_1$ such that 
 \begin{equation*}
         \{z\in \Omega^\ast\,:\,\|z-Q^\ast(\widetilde{x})\|\leq c\,\epsilon_n\}\subset \{z\in \Omega^\ast\,:\,\|G^\ast(z)-\widetilde{x}\|\leq \epsilon_n\}    \subset \{z\in \Omega^\ast\,:\,\|z-Q^\ast(\widetilde{x})\|\leq c_1\,\epsilon_n\},
 \end{equation*}
  it holds that 
 \begin{equation*}
 \begin{aligned}
     &C\,a^d\,\frac{\mb P_{\mu^*}(\wt K)}{\mb P_{\mu^*}(K_0)}\cdot \frac{\log n}{n}\geq \mb{P}_{\mu^{\ast}}(\widetilde K)\int_{\{z\in \Omega^\ast\,:\,\|G^\ast(z)-\widetilde{x}\|\leq \epsilon_n\}} \nu^\ast(z) \,\dd z= \mathbb{P}_{\mu^{\ast}}(\mb B_{\epsilon_n}(\widetilde{x}))\geq C_1\,  a^d\,\frac{\mb P_{\mu^*}(\wt K)}{\mb P_{\mu^*}(K_0)}\cdot \frac{\log n}{n}.
     \end{aligned}
 \end{equation*}
 Then by Bernstein's inequality, it holds with probability 
 larger than  $1-\frac{1}{n^3}$ that 
 \begin{equation*}
     \left|\frac{1}{|I_2|} \sum_{i\in I_2}\bold{1}_{ \mb B_{\epsilon_n}(\widetilde{x})}(X_i)-\mathbb{P}_{\mu^{\ast}}(\mb B_{\epsilon_n}(\widetilde{x}))\right| \leq C\, \Big(1+a^{\frac{d}{2}}\sqrt{\frac{\mb P_{\mu^*}(\wt K)}{\mb P_{\mu^*}(K_0)}}\Big)\frac{\log n}{n}.
 \end{equation*}
 Therefore when $a$ is large enough, it holds with probability at least $1-\frac{1}{n^2}$ that for any $\widetilde{x}\in \bar{N}_{\epsilon_n}$, 
 \begin{equation*}
     \frac{1}{|I_2|} \sum_{i\in I_2}\bold{1}_{ \mb B_{\epsilon_n}(\widetilde{x})}(X_i)\geq \frac{1}{n}.
 \end{equation*}
 So  
 \begin{equation*}
     \underset{x\in \mathcal{M}^{\ast} \cap K_{1/4}}{\sup} \underset{i\in I_2}{\min} \|x-X_i\|_2 \leq  \epsilon_n+ \underset{x\in \bar{N}_{\epsilon_n}}{\sup}  \underset{i\in I_2}{\min} \|x-X_i\|_2  \leq 2\,\epsilon_n\leq 2\, \wt{\epsilon}_n.
 \end{equation*}
 Proof is completed.
 
 \subsubsection{Proof of Lemma~\ref{lemmaunbounded2}}
 Firstly by Lemma~\ref{lemmaunbounded1}, it holds that 
 \begin{equation*}
     \{ y=\wh{Q}(x)\,:\,x\in \mathcal{M}^{\ast} \cap K_{1/4}; \wh{G}(\wh{Q}(x))\in K_0\}\subset \{ y=\wh{Q}(x)\,:\,x\in \wh{U}; \wh{G}(\wh{Q}(x))\in K_0\}.
 \end{equation*}
 For the inverse side, Recall~\eqref{diffp_0}, it holds with probability at least $1-n^{-3}$ that for any $x\in \wh{U}$ such that $\wh{G}(\wh{Q}(x))\in K_0$, there exists  $i^{\ast}\in I_2$ such that $X_{i^{\ast}}\in K_{1/4}$, $\|x-X_{i^{\ast}}\|_2\leq C\, (\frac{\log n}{n})^{\frac{1}{2d}}$ (recall $\mb P_{\mu^*}(K_0)\geq \frac{1}{2}\sqrt{\frac{\log n}{n}}$) and $\|\wh{Q}(x)-\wh{Q}(X_{i^{\ast}})\|_2\leq C_1\, (\frac{\log n}{n})^{\frac{1}{2d}}$. Moreover, since $\wh{G}(\wh{Q}(x))\in K_0$,  by equation~\eqref{eqnunbounded1}, it holds with probability at least $1-n^{-2}$ that $\wh{G}(\wh{Q}(X_{i^{\ast}})), X_{i^{\ast}} \in K_{c(\frac{\log n}{n})^{\frac{1}{2d}}}$ for a constant $c$.\\
 \quad\\
 Then fix $y\in \mathbb{R}^d$ and $i\in [n]$ where $\|y-\wh{Q}(X_i)\|_2\leq C_1 (\frac{\log n}{n})^{\frac{1}{2d}}$ and $X_i\in  K_{c(\frac{\log n}{n})^{\frac{1}{2d}}}$. Define $z^{(0)}=Q^{\ast}(X_i)$ and recursively define $z^{(k)}=z^{(k-1)}+\bold{J}_{\wh{l}}(z^{(k-1)})^{-1} (y-\wh{l}(z^{(k-1)}))$ (recall $\wh l=\wh Q\circ G^\ast$). Since $ \beta>1$ and for any $z\in Q^{\ast}(\mathcal{M}^{\ast} \cap K_{1/2})$, it holds that $ |{\rm det}(\bold{J}_{\wh{l}}(z))|\geq c_2>0$, we can obtain that when $n$ is large enough, there exists $\bar{z}$ such that $\lim_{k\to +\infty} z^{(k)}=\bar{z}$ and $\|\bar{z}-z^{(0)}\|\leq C_2\, (\frac{\log n}{n})^{\frac{1}{2d}}$. So $y=\wh{Q}(\bar{x})$ where $\bar{x}=G^{\ast}(\bar{z})$.
 Moreover, since $G^\ast(z^{(0)})=X_i\in K_{c(\frac{\log n}{n})^{\frac{1}{2d}}}$, we have $\bar{x}=G^\ast(\bar{z})\in K_{\frac{1}{4}}$; also since $z^{(0)}\in \Omega^\ast$ with $G^\ast(z^{(0)})\in K_1$, we have $\bar{z}\in {\bigcup}_{z\in\Omega^\ast, \, G^\ast(z)\in K_1} B_{1/L}(z)\subset \Omega^\ast$ and thus $\bar{x}\in G^\ast(\Omega^\ast)=\m M^\ast$, which together leads to $\bar{x}\in \m M^\ast\cap K_{\frac{1}{4}}$. Therefore we can get it holds with probability  at least $1-\frac{2}{n^2}$ that $ \{ y=\wh{Q}(x)\,:\,x\in \mathcal{M}^{\ast} \cap K_{1/4}; \wh{G}(\wh{Q}(x))\in K_0\}= \{ y=\wh{Q}(x)\,:\,x\in \wh{U}; \wh{G}(\wh{Q}(x))\in K_0\}$.
 \subsubsection{Proof of Lemma~\ref{lemmaunbounded3}}
 Firstly by equation~\eqref{eqnunbounded1},  there exists a positive constant $c$ such that it holds with probability at least $1-\frac{1}{n^2}$ that for any $\widetilde{y} \in \wh{\Omega}$, there exists $\widetilde{z}\in \Omega^{\ast}$ such that $\widetilde{y}=\wh{Q}(G^{\ast}(\widetilde{z}))=\wh{l}(\wt{z})$ and $G^{\ast}(\widetilde{z})\in K_{c(\frac{\log n}{n})^{\frac{\beta}{2d}}}$ (recall $\wt n\geq \frac{1}{2}\sqrt{n\log n}$). Therefore, fix  $\widetilde{y} \in \wh{\Omega}$ and  $y \in \mb B_{\sqrt{d}h}(\widetilde{y})$, define $z^{(0)}=\widetilde{z}$ and recursively define $z^{(k)}=z^{(k-1)}+\bold{J}_{\wh{l}}(z^{(k-1)})^{-1}(y-\wh{l}(z^{(k-1)}))$.
 Since $ \beta>1$ and for any $z\in Q^{\ast}(\mathcal{M}^{\ast} \cap K_{1/2})$,  it holds that $ |{\rm det}(\bold{J}_{\wh{l}}(z))|\geq c_2>0$, when  $n$ is large enough, we can get there exists $\bar{z}$ such that $\lim_{k\to +\infty} z^{(k)}=\bar{z}$, $\wh{l}(\bar{z})=y$ and $\|\bar{z}-z^{(0)}\|\leq C\, h\leq C_1\, \big(\frac{\log n}{n}\big)^{\frac{1}{2\alpha+d}}$. So when $n$ is large enough, $\bar{z}\in \Omega^*$ and $G^*(\bar{z})\in K_{\frac{1}{4}}$, which further leads to  $\{y\in \mb B_{\sqrt{d}h}(\widetilde{y})\,:\, \widetilde{y}\in \wh{\Omega}\} \subset \wh{Q}(\mathcal{M}^{\ast} \cap K_{1/4})$.  Therefore, it holds with probability at least $1-\frac{c}{n^2}$ that for any $\widetilde{y}\in \wh{\Omega}$,
 \begin{equation*}
     \begin{aligned}
         &\left|\mathbb{E}_{\mu^{\ast}} \left[\frac{1}{h^d}\Big(\prod_{j=1}^d\bar{k}\Big(\frac{\wh{Q}_j(X)-\widetilde{y}_j}{h}\Big)\Big)\cdot\bold{1}_{ K_{1/4}}(X)\right]-\mathbb{P}_{\mu^{\ast}}(X\in \widetilde{K})\cdot{\nu^{\ast}(\wh{l}^{-1}(\widetilde{y}))}\cdot{|{\rm det}(\bold{J}_{\wh{l}^{-1}}(\widetilde{y}))|}\right|\\
         &\overset{(i)}{=}\mathbb{P}_{\mu^{\ast}}(X\in \widetilde{K})\left|\int\frac{1}{h^d}  \Big(\prod_{j=1}^d \bar{k}\big(\frac{z_j-\widetilde{y}_j}{h}\big)\Big)\cdot {\nu^{\ast}(\wh{l}^{-1}(z))}\cdot{|{\rm det}(\bold{J}_{\wh{l}^{-1}}(z))|}\,\dd z-{\nu^{\ast}(\wh{l}^{-1}(\widetilde{y}))}\cdot{|{\rm det}(\bold{J}_{\wh{l}^{-1}}(\widetilde{y}))|}\right|\\
         &\overset{(ii)}{=} \mathbb{P}_{\mu^{\ast}}(X\in \widetilde{K})\Big| \int \Big(\prod_{j=1}^d \bar{k}(t_j) \Big)\cdot \big(\nu^{\circ}(ht+\widetilde y)-v^{\diamond}(\widetilde y)\big) \,\dd t\Big| \\
         &\overset{(iii)}{\leq }\mathbb{P}_{\mu^{\ast}}(X\in \widetilde{K})\bigg| \int \Big(\prod_{j=1}^d \bar{k}(t_j)\Big)\cdot  \sum_{\eta\in \mb N_0^d\atop 1\leq |\eta|\leq \lfloor \alpha\rfloor} (\nu^{\circ})^{(\eta)}(\widetilde y)\cdot(ht)^{\eta} \,\dd t\bigg|+C\, h^{\alpha}\\
         &\overset{(iiii)}{=}C\,  h^{ \alpha},
     \end{aligned}
 \end{equation*}
 where  $(i)$ uses $\{y\in \mb B_{\sqrt{d}h}(\widetilde{y})\,:\, \widetilde{y}\in \wh{\Omega}\} \subset \wh{Q}(\mathcal{M}^{\ast} \cap K_{1/4})$ and $\mu^\ast|_{\wt{K}}=G^\ast_{\#}\nu^\ast$;   $(ii)$ let $t=(t_1,t_2,\cdots,t_d)$ with $t_j=\frac{z_j-\widetilde{y}_j}{h}$,  $\nu^{\circ}(z)={\nu^{\ast}(\wh{l}^{-1}(z))}\cdot{|{\rm det}(\bold{J}_{\wh{l}^{-1}}(z))|}$ and uses the fact $\int \bar{k}(t) \,\dd t=1$; $(iii)$ uses $\nu^{\circ}$ is $\alpha$-smooth and  $\bar{k}$ is compactly supported; $(iiii)$ uses  $\int x^j k(x)=0$ for $j\in [\lceil \alpha\rceil]$. 
 
\noindent Moreover, for any $\widetilde y\in \wh{\Omega}$,
 \begin{equation*}
 \begin{aligned}
     &\frac{1}{h^{2d}} \int \Big(\prod_{j=1}^d \bar{k}^2\Big(\frac{\wh{Q}_j(X)-\widetilde y_j}{h}\Big)\Big)\cdot \bold{1}_{ K_{1/4}} (X)\, \dd \mu^{\ast}\\
     &=\mathbb{P}_{\mu^{\ast}}(X\in \widetilde{K})\int_{z\in \mb B_{\sqrt{d}h}(\widetilde y)}\frac{1}{h^{2d}} \Big(\prod_{j=1}^d \bar{k}^2\big(\frac{z_j-\widetilde{y}_j}{h}\big)\Big)\cdot {\nu^{\ast}(\wh{l}^{-1}(z))}\cdot{|{\rm det}(\bold{J}_{\wh{l}^{-1}}(z))|}\,\dd z\\
     &\leq C\, \frac{1}{h^d},
     \end{aligned}
 \end{equation*}
 and 
 \begin{equation*}
     \frac{1}{h^d}\cdot\Big(\prod_{j=1}^d \bar{k}\big(\frac{\wh{Q}_j(X)-\widetilde{y}_j}{h}\big) \Big)\cdot \bold{1}_{ K_{1/4}} (X)\leq C\, \frac{1}{h^d}.
 \end{equation*}
Now let $N_{\frac{1}{n^2}}\subset\wh{\Omega}$ be a $\frac{1}{n^2}$-covering set $\wh{\Omega}$, where $|N_{\frac{1}{n^2}}|\leq C\, n^{2d}$, then by a similar union bound argument plus Bernstein's inequality as the proof of~\eqref{eqn:union+Bernstein}, it holds with probability at least $1-\frac{1}{n^2}$ that for any $\widetilde{y} \in N_{\frac{1}{n^2}}$, it satisfies that 
 \begin{equation*}
     \left|\mathbb{E}_{\mu^{\ast}} \left[\frac{1}{h^d}\Big(\prod_{j=1}^d  \bar{k}\Big(\frac{\wh{Q}_j(X)-\widetilde{y}_j}{h}\Big) \Big)\cdot \bold{1}_{ K_{1/4}} (X)\right]-\bar{\nu}_{\wh{Q}}(\widetilde{y})\right|\leq  \sqrt{\frac{\log n}{n}} h^{-\frac{d}{2}}+\frac{\log n}{n} h^{-d}.
 \end{equation*}
 Then by the uniformly Lipschitzness of $\bar{k}(x)$ and $h=\big(\frac{\log n}{n}\big)^{\frac{1}{2 \alpha+d}}$, it holds with probability at least $1-\frac{1}{n^2}$ that  for any $y\in \wh{\Omega}$,
 \begin{equation*}
     \left|\mathbb{E}_{\mu^{\ast}} \left[\frac{1}{h^d}\Big(\prod_{j=1}^d  \bar{k}\Big(\frac{\wh{Q}_j(X)- {y}_j}{h}\Big) \Big)\cdot \bold{1}_{ K_{1/4}} (X)\right]-\bar{\nu}_{\wh{Q}}({y})\right|\leq  C\,\big(\frac{\log n}{n}\big)^{\frac{ \alpha}{2 \alpha+d}}.
 \end{equation*}
 Then by combining all pieces, we have
 \begin{equation*}
    \underset{y\in \wh{\Omega}}{\sup} \left|\mathbb{P}_{\mu^{\ast}}(X\in \widetilde{K})\cdot{\nu^{\ast}(\wh{l}^{-1}(y))}\cdot{|{\rm det}(\bold{J}_{\wh{l}^{-1}}(y))|}-\bar{\nu}_{\wh{Q}}(y)\right|\leq C\, \big(\frac{\log n}{n}\big)^{\frac{ \alpha}{2 \alpha+d}}.
 \end{equation*}

\subsection{Noisy case}\label{app:noisy case}
\begin{corollary}\label{noisecorollary}
Suppose $\mu^{\ast}\in \mathcal{S}^{\ast}$, $X_{1:n}$ and $\epsilon_{1:n}$ are $n$ i.i.d. samples from $\mu^{\ast}$ and $\mu_{\epsilon}$ respectively, where $\mathbb{P}_{\mu_{\epsilon}}(\|\epsilon\|_2\leq n^{-\frac{1}{2}-\frac{\beta}{d}})=1$. Let $Y_i=X_i+\epsilon_i$ for any $i\in [n]$. If $D>d$, $\gamma>0$, $\alpha\geq 0$ and $\beta> 1$,
Use $\wh{\mu}^{\diamond}$ to denote the estimator $\wh{\mu}$ defined in Section~\ref{sec:optimal_proce} with $X_{1:n}$ being replaced by $Y_{1:n}$, then there exist  positive constants $c_1$ and  $n_0$ such that when $n\geq n_0$ it holds that
\begin{equation*}
\mathbb{E} [d_\gamma(\wh{\mu}^{\diamond},\mu^{\ast})]\leq C\, \Big(\big(\frac{\log n}{n}\big)^{\frac{\gamma\beta}{d}}\vee \big(\frac{\log n}{n}\big)^{\frac{\alpha+\gamma}{2\alpha+d}}\vee \big(\frac{\log n}{n}\big)^{\frac{1}{2}}\Big).
\end{equation*}
\end{corollary}
\begin{proof}
Recall $S_m=\mb B_{r_m}(a_m)\subset \mb B_{r_m+0.5/L}(a_m)=S_m^{\dagger}\subset \mb B_{r_m+1/L}(a_m)\subset \widetilde{S}_m$, then when $n$ is large enough,
  $\mathbb{P}_{\mu^{\ast}}(X\in S_m)\leq \mathbb{P}_{\mu^{\ast}\ast \mu_{\epsilon}}(Y\in S_m^{\dagger})\leq \mathbb{P}_{\mu^{\ast}}(X\in \widetilde{S}_{m})$. Similar as the proof for Theorem~\ref{upperboundgenerative} in Appendix~\ref{Sec:more_proofs_upperbound}, we fix an arbitrary $m\in[M]$ where $\mathbb{P}_{\mu^{\ast}}(X\in S^{\dagger}_m)\geq \frac{1}{2}\sqrt{\frac{\log n}{n}}$ and bound $ {\sup}_{f\in C^{\gamma}_1(\mathbb{R}^D)}\big(\int f(X) \rho_m(X) \, \dd \mu^{\ast}-\wh{\mathcal{J}}_m(f,Y_{1:n})\big)$ in the following proof. 
Use $(\wh{G}^{\diamond}_{[m]},\wh{Q}^{\diamond}_{[m]}, \widetilde{\nu}^{\diamond}_{[m],\wh{Q}^{\diamond}_{[m]}})$ to denote estimators $(\wh{G}_{[m]},\wh{Q}_{[m]},\widetilde{\nu}_{[m],\wh{Q}_{[m]}})$ with $X_{1:n}$ being replaced by $Y_{1:n}$, and here we consider $\widetilde{\nu}_{[m],\wh{Q}_{[m]}}$ constructed by wavelet expansion with $\gamma\vee \alpha$-smooth basis. Since there exists a constant $c_2$ such that
\begin{equation*}
    \begin{aligned}
    \frac{1}{|I_1|}\sum_{i\in I_1} \|Y_i-\wh{G}^{\diamond}_{[m]}(\wh{Q}^{\diamond}_{[m]}(Y_i))\|_2^2\cdot\bold{1}_{S_m^{\dagger}}(Y_i)&\leq \frac{1}{|I_1|}\sum_{i\in I_1} \|Y_i-G^{\ast}_{[m]}(Q^{\ast}_{[m]}(Y_i))\|_2^2\cdot\bold{1}_{S_m^{\dagger}}(Y_i)\\
    &\leq c\,\Big( \frac{1}{|I_1|}\sum_{i\in I_1} \|X_i-G^{\ast}_{[m]}(Q^{\ast}_{[m]}(X_i))\|_2^2\cdot\bold{1}_{\widetilde{S}_m }(X_i)+\|\epsilon_i\|_2^2\Big)\\
    &\leq c_2 \,n^{-1-\frac{2\beta}{d}};
    \end{aligned}
\end{equation*}
 and when $r$ is small enough, for any $\widetilde{z}\in \{z\in \mb B_1^d\,:\, G^{\ast}(z)\in \mb B_{r_m+0.25/L}(a_m)\}$,
\begin{equation*}
 \begin{aligned}
& \frac{1}{|I_1|}\sum_{i\in I_1} \|X_i-\wh{G}^{\diamond}_{[m]} (\wh{Q}^{\diamond}_{[m]} (X_i))\|_2^2\cdot\bold{1}_{\mb B_{r}(\widetilde{z})} (Q^{\ast}(X_i))\cdot\bold{1}_{S_m^{\dagger}}(X_i)\\
 &\leq  \frac{1}{|I_1|}\sum_{i\in I_1} \|X_i-\wh{G}^{\diamond}_{[m]} (\wh{Q}^{\diamond}_{[m]}(X_i))\|_2^2\cdot \bold{1}_{S_m^{\dagger}}(Y_i)
 \leq c_2 n^{-1-\frac{2\beta}{d}}.
 \end{aligned}
    \end{equation*}
 Then same as the proof of Lemma~\ref{lemma:mani_est} and~\ref{lemma:density_regularity}, we can get  it holds with probability at least $1-n^{-c}$ that the density of $\nu^{\ast}_{[m],\wh{Q}^{\diamond}_{[m]}}$ belongs to $C^{\alpha}_{c_3}(\mathbb{R}^d)$; and for any $\eta\in [4\lceil\gamma\rceil]$, $\mathbb{E}_{\mu^{\ast}}\big[\|X-\wh{G}^{\diamond}_{[m]}(\wh{Q}^{\diamond}_{[m]}(X))\|_2^{\eta}\cdot\rho_m(X)\big] \leq c_3\,\big((\frac{\log n}{n})^{\frac{\eta\beta}{d}}\vee \frac{\log n}{n}\big)$; So same as the proof Lemma~\ref{lemma:smooth_corr}, it holds with probability at least $1-n^{-c}$ that 
 \begin{equation*}
 \begin{aligned}
& \underset{f\in C_1^{\gamma}(\mathbb{R}^D)}{\sup} \Big(\int f(X)\rho_m(X) \, \dd \mu^{\ast}-\int f(\wh{G}^{\diamond}_{[m]} (\wh{Q}^{\diamond}_{[m]}(X)))\rho_m(X) \, \dd \mu^{\ast}\\
&+ \frac{1}{|I_2|}\sum_{i\in I_2} \sum_{j\in \mathbb{N}_0^d\atop 1\leq |j|\leq \lfloor\gamma\rfloor}\frac{1}{j!}f^{(j)}(Y_i)(\wh{G}^{\diamond}_{[m]} (\wh{Q}^{\diamond}_{[m]}(Y_i))-Y_i)^{j}\rho_m(Y_i)\Big)\\
&\leq C\,\underset{f\in C_1^{\gamma}(\mathbb{R}^D)}{\sup} \Big(\int f(X) \rho_m(X)\, \dd \mu^{\ast}-\int f(\wh{G}^{\diamond}_{[m]}(\wh{Q}^{\diamond}_{[m]}(X))) \rho_m(X)\, \dd \mu^{\ast}\\
&+ \frac{1}{|I_2|}\sum_{i\in I_2} \sum_{j\in \mathbb{N}_0^d\atop 1\leq |j|\leq \lfloor\gamma\rfloor}\frac{1}{j!}f^{(j)}(X_i)(\wh{G}^{\diamond}_{[m]} (\wh{Q}^{\diamond}_{[m]} (X_i))-X_i)^{j}\rho_m(X_i)\Big)\\
&+\underset{f\in C_1^{\gamma}(\mathbb{R}^D)}{\sup} \bigg(\frac{1}{|I_2|}\sum_{i\in I_2} \sum_{j\in \mathbb{N}_0^d\atop 1\leq |j|\leq \lfloor\gamma\rfloor}\frac{1}{j!}\Big(f^{(j)}(Y_i)(\wh{G}^{\diamond}_{[m]} (\wh{Q}^{\diamond}_{[m]} (Y_i))-Y_i)^{j}\rho_m(Y_i)\\
&- f^{(j)}(X_i)(\wh{G}^{\diamond}_{[m]}(\wh{Q}^{\diamond}_{[m]}(X_i))-X_i)^{j}\rho_m(X_i)\Big)\bigg)
\\
&\leq C_1 \,\big(\frac{\log n}{n}\big)^{\frac{1}{2}}\vee \big(\frac{\log n}{n}\big)^{\frac{\alpha+\gamma}{2\alpha+d}}\vee \big(\frac{\log n}{n}\big)^{\frac{\gamma\beta}{d}}.
\end{aligned}
 \end{equation*}
Now we bound 
\begin{equation*}
\begin{aligned}
&\underset{f\in C_1^{\gamma}(\mathbb{R}^D)}{\sup} \Big( \frac{1}{|I_2|}\sum_{i\in I_2} \Pi_J f(\wh{G}^{\diamond}_{[m]}(\wh{Q}^{\diamond}_{[m]}(Y_i)))\rho_m(Y_i)+\int \Pi_J^{\perp}f(\wh{G}^{\diamond}_{[m]}(z))\widetilde{\nu}^\diamond_{\wh{Q}^{\diamond}_{[m]}}(z) \,\dd z\\
&-\int f(\wh{G}^{\diamond}_{[m]}(\wh{Q}^{\diamond}_{[m]}(X)))\rho_m(X)\, \dd \mu^{\ast}\Big). 
\end{aligned}
\end{equation*}
By Lemma~\ref{lemma:high_freq} and~\ref{lemma:low_freq}, we only need to bound
\begin{equation*}
\begin{aligned}
&\underset{f\in C_1^{\gamma}(\mathbb{R}^D)}{\sup} \Big( \frac{1}{|I_2|}\sum_{i\in I_2} \Pi_J f(\wh{G}^{\diamond}_{[m]}(\wh{Q}^{\diamond}_{[m]}(Y_i)))\rho_m(Y_i)\\
&-\frac{1}{|I_2|}\sum_{i\in I_2} \Pi_J f(\wh{G}^{\diamond}_{[m]}(\wh{Q}^{\diamond}_{[m]}(X_i)))\rho_m(X_i)\Big)\\
&+\underset{f\in C_1^{\gamma}(\mathbb{R}^D)}{\sup} \Big(\int \Pi_J^{\perp}f(\wh{G}^{\diamond}_{[m]}(z))\widetilde{\nu}^{\diamond}_{\wh{Q}^{\diamond}_{[m]}}(z) \,\dd z-\int \Pi_J^{\perp}f(\wh{G}^{\diamond}_{[m]}(z))\widetilde{\nu}_{\wh{Q}^{\diamond}_{[m]}}(z) \,\dd z\Big).
\end{aligned}
\end{equation*}
Recall
\begin{equation*}
\Pi_J f(x)=\sum_{k\in \mathbb{Z}^D}b_k\phi_k(x)+\sum_{l=1}^{2^D-1}\sum_{j=0}^{J} \sum_{k\in\mathbb{Z}^D} f_{ljk} \psi_{ljk}(x),
\end{equation*}
where $2^{dJ}=C\,n^{\frac{d}{2\alpha+d}}$. By the  fact that the basis $\phi_k$ and $\psi_{ljk}$
are $\gamma$-smooth,  there exists a constant $c$ such that for any $f\in C^{\gamma}_1(\mathbb{R}^D)$, it holds that $\Pi_J f\in C^{\gamma}_{c\log n}(\mathbb{R}^D)$ and $\Pi_J^{\perp} f\in C^{\gamma}_{c\log n}(\mathbb{R}^D)$. So it holds that 
\begin{equation*}
\begin{aligned}
&\underset{f\in C_1^{\gamma}(\mathbb{R}^D)}{\sup} \Big( \frac{1}{|I_2|}\sum_{i\in I_2} \Pi_J f(\wh{G}^{\diamond}_{[m]}(\wh{Q}^{\diamond}_{[m]}(Y_i)))\rho_m(Y_i)\\
&-\frac{1}{|I_2|}\sum_{i\in I_2} \Pi_J f(\wh{G}^{\diamond}_{[m]}(\wh{Q}^{\diamond}_{[m]}(X_i)))\rho_m(X_i)\Big)\leq C\, {n}^{-(\frac{1}{2}+\frac{\beta}{d})\cdot(\gamma\wedge 1)}\cdot\log n;
 \end{aligned}
\end{equation*}
\begin{equation*}
\begin{aligned}
&\underset{f\in C_1^{\gamma}(\mathbb{R}^D)}{\sup} \Big(\int \Pi_J^{\perp}f(\wh{G}^{\diamond}_{[m]}(z))\widetilde{\nu}^{\diamond}_{\wh{Q}^{\diamond}_{[m]}}(z) \,\dd z-\int \Pi_J^{\perp}f(\wh{G}^{\diamond}_{[m]}(z))\widetilde{\nu}_{\wh{Q}^{\diamond}_{[m]}}(z) \,\dd z\Big)\\
&\leq C\, \log n \cdot \underset{f\in C_1^{\gamma\wedge \beta}(\mathbb{R}^d)}{\sup} \Big(\int  f(z)\widetilde{\nu}^{\diamond}_{\wh{Q}^{\diamond}_{[m]}}(z) \,\dd z-\int  f(z)\widetilde{\nu}_{\wh{Q}^{\diamond}_{[m]}}(z) \,\dd z\Big)\\
&\leq C_1\,\log n\cdot\underset{|f_{ljk}|\leq (2^{-dj})^{\frac{\gamma\wedge \beta}{d}+\frac{1}{2}}\atop |b_k|\leq 1}{\sup} \int \Big(\sum_{k \in \mathbb{Z}^d} b_k \phi_k(z) +\sum_{l=1}^{2^{d}-1}\sum_{j=0}^{+\infty}\sum_{k\in \mathbb{Z}^d}f_{ljk} \psi_{ljk}(z)\Big) \Big(\widetilde{\nu}^{\diamond}_{\wh{Q}^{\diamond}_{[m]}}(z)-\widetilde{\nu}_{\wh{Q}^{\diamond}_{[m]}}(z)\Big)\,\dd z\\
&\leq C_1\, \log n \cdot\underset{|b_k|\leq 1}{\sup}\,\bigg[ \frac{1}{|I_2|} \sum_{i\in I_2} \sum_{k\in \mathbb{Z}^d} b_k \Big(\phi_k\big(\wh{Q}^{\diamond}_{[m]}(Y_i)\big)\rho_m(Y_i)-\phi_k\big(\wh{Q}^{\diamond}_{[m]}(X_i)\big) \rho_m(X_i)\Big)\bigg]\\
&+ C_1\,\log n \cdot\underset{|f_{ljk}|\leq (2^{-dj})^{\frac{\gamma\wedge\beta}{d}+\frac{1}{2}}}{\sup}\bigg[ \frac{1}{|I_2|}\sum_{i\in I_2} \sum_{l=1}^{2^d-1}\sum_{j=1}^{J}\sum_{k\in\mb{Z}^d} f_{ljk} \Big(\psi_k\big(\wh{Q}^{\diamond}_{[m]}(Y_i)\big)\rho_m(Y_i)\\
&\qquad\qquad\qquad\qquad-\psi_k\big(\wh{Q}^{\diamond}_{[m]}(X_i)\big)\rho_m(X_i)\Big)\bigg]\leq C\,(\log n)^2\, \big(n^{-\frac{1}{2}-\frac{\beta}{d}}\big)^{\gamma\wedge 1},\\
\end{aligned}
\end{equation*}
where the last inequality uses the fact that $\phi_k(\cdot)$ and $\psi_{ljk}(\cdot)$ are $\gamma$-smooth. Then similar as the proof of Theorem~\ref{upperboundgenerative}, we can get the desired  conclusion.

\end{proof}

 \end{document}